\documentclass[a4paper,11pt]{amsart}

\usepackage{amssymb}
\usepackage{amscd}
\usepackage{amsthm}
\usepackage{amsmath}
\usepackage[all]{xy}
\usepackage{extarrows}
\usepackage{rotating}
\usepackage{tikz-cd}
\usepackage{leftidx}
\usepackage{bm}
\usepackage{dynkin-diagrams}
\usepackage{mathrsfs}
\usepackage{hyperref}
\usepackage{rotating}
\usetikzlibrary{patterns}
\usepackage{csquotes}
\usepackage{textcomp}
\usepackage{wasysym}

\DeclareUnicodeCharacter{2212}{-}
\DeclareUnicodeCharacter{0008}{~}
\DeclareUnicodeCharacter{2011}{~}

\def\mathbi#1{\textbf{\em #1}}

\renewenvironment{proof}{{ \textbf{Proof}.}}{\qed}
\newtheorem{Thm}{Theorem}[section]
\newtheorem{Lem}[Thm]{Lemma}
\newtheorem{Def}[Thm]{Definition}
\newtheorem{Cor}[Thm]{Corollary}
\newtheorem{Prop}[Thm]{Proposition}
\newtheorem{Ex1}[Thm]{Example}
\newtheorem{Rem1}[Thm]{Remark}

\newtheorem{assumption}{Assumptions}
\newcommand{\Ima}{\mathrm{Im}}
\newcommand{\coker}{\mathrm{coker}}
\newcommand{\cone}{\mathrm{Cone}}
\newcommand{\Hom}{\mathrm{Hom}}
\newcommand{\thick}{\mathrm{thick}\,}
\newcommand{\add}{\mathrm{add}}
\newcommand{\per}{\mathrm{per}}
\newcommand{\Ext}{\mathrm{Ext}}
\newcommand{\End}{\mathrm{End}}
\newcommand{\obj}{\mathrm{obj}}
\newcommand*{\isoarrow}[1]{\arrow[#1,"\rotatebox{90}{\(\sim\)}"]}
\newcommand{\bijar}[1][]{%
	\ar[#1]
	\ar@<0.7ex>@{}[#1]|-*[@]{\sim}} 

\newcommand{\lG}{\mathrm{\textbf{L}}G^{*}}
\newcommand{\lQ}{\mathrm{\textbf{L}}Q^{*}}
\newcommand{\pvd}{\mathrm{pvd}}
\newcommand*{\corner}{\mbox{\LARGE{$\ulcorner$}}}
\newcommand{\dgcat}{\mathrm{dgcat}_{k}}

\newcommand*\circled[1]{\tikz[baseline=(char.base)]{
		\node[shape=circle,draw,inner sep=1pt] (char) {#1};}}

\newenvironment{Rem}{\begin{Rem1}\rm}{\end{Rem1}}
\newenvironment{Ex}{\begin{Ex1}\rm}{\end{Ex1}}

 \normalbaselines
\newcommand{\Si}{\Sigma}
\newcommand{\we}{\wedge}
\newcommand{\ten}{\otimes}
\newcommand{\lten}{\overset{\rm\textbf{L}}{\ten}}
\newcommand{\fib}{\mathrm{fib}}
\newcommand{\cof}{\mathrm{cof}}
\newcommand{\RHom}{\mathrm{\mathbf{R}Hom}}
\newcommand{\ra}{\rightarrow}

\newcommand{\iso}{\xrightarrow{_\sim}}
\newcommand{\liso}{\xleftarrow{_{\sim}}}
\newcommand{\id}{\mathbf{1}}

\newcommand{\Aus}{\mathrm{Aus}}

\newcommand{\ca}{{\mathcal A}}
\newcommand{\cb}{{\mathcal B}}
\newcommand{\cc}{{\mathcal C}}
\newcommand{\cd}{{\mathcal D}}
\newcommand{\ce}{{\mathcal E}}
\newcommand{\cf}{{\mathcal F}}

\newcommand{\ch}{{\mathcal H}}

\newcommand{\ck}{{\mathcal K}}

\newcommand{\cm}{{\mathcal M}}

\newcommand{\cp}{{\mathcal P}}

\newcommand{\cs}{{\mathcal S}}
\newcommand{\ct}{{\mathcal T}}
\newcommand{\cu}{{\mathcal U}}
\newcommand{\cv}{{\mathcal V}}

\newcommand{\cx}{{\mathcal X}}

\newcommand{\cz}{{\mathcal Z}}

\begin{document}

\title{Relative Cluster Categories and Higgs categories}



\author{Yilin Wu}
\address{
    School of Mathematical Sciences\\
    University of Science and Technology of China\\
    Hefei 230026, Anhui\\
    P. R. China
}
\email{wuyilinecnuwudi@gmail.com}



\dedicatory{}


\begin{abstract}
Cluster categories were introduced in 2006 by Buan-Marsh-Reineke-Reiten-Todorov in order to categorify acyclic cluster algebras without coefficients. Their construction was generalized by Amiot (2009) and Plamondon (2011) to arbitrary cluster algebras associated with quivers. A higher dimensional generalization is due to Guo (2011). Cluster algebras with coefficients are important since they appear in nature as coordinate algebras of varieties like Grassmannians,
double Bruhat cells, unipotent cells, etc.	The work of Geiss-Leclerc-Schröer often yields Frobenius exact categories which allow us to categorify such cluster algebras.

In this paper, we generalize the construction of (higher) cluster categories by Claire Amiot and by Lingyan Guo to the relative context. We prove the existence of an $ n $-cluster tilting object in a Frobenius extriangulated category, namely the Higgs category (generalizing the Frobenius categories of Geiss-Leclerc-Schröer), which is stably $ n $-Calabi--Yau and Hom-finite, arising from a left $ (n+1) $-Calabi--Yau morphism. Our results apply in particular to relative Ginzburg dg algebras coming from ice quivers with potential and higher Auslander algebras associated to $ n $-representation-finite algebras.

\end{abstract}

\maketitle
\tableofcontents

\section{Introduction}
Almost 20 years ago, Fomin and Zelevinsky~\cite{FZ2002} invented cluster algebras in order to create a combinatorial framework for the study of canonical bases \cite{Ka1990,Lu1990} in quantum groups and the study of total positivity in algebraic groups. Since the combinatorics of cluster algebras are very complicated, it is useful to model them categorically, so that more conceptual tools become available.

Consider a cluster algebra $ \ca $ without frozen variables and such that one of the clusters has an acyclic quiver $ Q $. Buan–Marsh–Reineke–Reiten–Todorov~\cite{BMRRT} introduced the cluster category $ \cc_{Q} $, given by the orbit category
$$ \cc_{Q}=\cd^{b}(kQ)/\tau\Si^{-1} ,$$
where $ \tau $ denotes the Auslander-Reiten translation of the derived category $ \cd^{b}(kQ) $ and $ \Si $ the shift functor on $ \cd^{b}(kQ) $. It was shown to be triangulated by a result of Keller~\cite{BK2005}.

Claire Amiot~\cite{Am2008} generalized the construction of the cluster category to finite-dimensional algebras $ A_{0} $ of global dimension $ \leqslant2 $. In her approach, in order to show that there is a triangle equivalence between $ \mathcal{C}_{A_{0}} $, constructed as a triangulated hull~\cite{BK2005}, and the quotient category $ \mathrm{per}\,\bm\Pi_{3}(A_{0})/\mathrm{pvd}\,\bm\Pi_{3}(A_{0}) $, where $ \bm\Pi_{3}(A_{0}) $ is the 3-Calabi--Yau completion~\cite{BK2011} of $ A_{0} $, she first studied the category $ \mathcal{C}_{\Pi}=\mathrm{per}\,\Pi/\mathrm{pvd}\,(\Pi)$ associated to a dg algebra $ \Pi $ with the following four properties:
\begin{itemize}
	\item[1)] $ \Pi $ is homologically smooth,
	\item[2)] $ \Pi $ is connective, i.e.\ the cohomology of $ \Pi $ vanishes in degrees $ >0 $,
	\item[3)] $ \Pi $ is 3-Calabi--Yau as a bimodule,
	\item[4)] $ H^{0}(\Pi) $ is finite-dimensional.
\end{itemize}
She proved that the category $ \mathcal{C}_{\Pi} $ is Hom-finite and $ 2 $-Calabi--Yau. Moreover, the image of the free dg module $ \Pi $ is a cluster tilting object in $ \mathcal{C}_{\Pi} $ whose
endomorphism algebra is the zeroth homology of $ \Pi $. Later, Lingyan Guo~\cite{LYG} generalized Amiot's construction to finite-dimensional algebras $ A $ of global dimension $\leqslant n $ and to dg algebras satisfying $ 1) $, $ 2) $, $ 4) $ and $ n $-Calabi--Yau as a bimodule.

For cluster algebras with non invertible frozen variables, there is a natural model for this purpose, a Frobenius category $ \ce $, i.e.\ an exact category with enough projective and injective objects, and such that the projective and injective objects coincide. Then by definition, each projective-injective object $ I $ satisfies
$$ \Ext^{i}_{\ce}(?,I)=\Ext^{i}_{\ce}(I,?)=0\,\,\text{for}\,\,i>0. $$
Thus, each projective-injective object $ I $ is in $ \add T $ for any cluster-tilting object $ T\in\ce $. By a result of Happel~\cite{Happ1988}, the stable category $ \underline{\ce} $, formed by taking the quotient by the ideal of morphisms factoring through a projective-injective object, is a triangulated category. The corresponding stable category $ \ce $ is 2-Calabi–Yau if there is a bifunctorial duality
$$ \Ext^{1}_{\ce}(X,Y)\cong D\Ext^{1}_{\ce}(Y,X) $$ for all $ X,Y\in\ce $.

Relative right Calabi--Yau structures were invented by Bertrand To\"{e}n in ~\cite[pp.\ 227-228]{BT2014}. Then, relative right and left Calabi--Yau structures were studied by Chris Brav and Tobias Dyckerhoff in \cite{BD2019}. A relative left $ n $-Calabi--Yau structure on a morphism $ f\colon B\rightarrow A $ between smooth dg algebras is the datum of a class $ [\xi] $ in negative cyclic homology $ H\!N_{n}(f) $ inducing certain dualities in $ \mathcal{D}(B^{e}) $ and $ \mathcal{D}(A^{e}) $ (see Section~\ref{section RCY}). In particular, if the dg algebra $ B $ is zero, then $ A $ is $ n $-Calabi--Yau as a bimodule. A canonical way to produce relative left Calabi--Yau structures is the deformed relative Calabi--Yau completion which was introduced by Yeung (\cite{WkY2016}). This generalized Keller's construction~\cite{BK2011} of deformed $ n $-Calabi--Yau completions to the relative context.

The main aim of this paper is to generalize the construction of Claire Amiot and Lingyan Guo to the relative context. We change the above properties of the dg algebra $ A $ to the following properties on a dg algebra morphism $ f\colon B\rightarrow A $ (not necessarily preserving the unit)
\begin{itemize}
	\item[1)] $ A $ and $ B $ are homologically smooth,
	\item[2)] $ A $ is connective, i.e.\ the cohomology of $ A $ vanishes in degrees $ >0 $,
	\item[3)] the morphism $ f\colon B\to A $ has a left $ (n+1) $-Calabi--Yau structure,
	\item[4)] $ H^{0}(A) $ is finite-dimensional.
\end{itemize}

Then we introduce the relative cluster category $ \mathcal{C}_{n}(A,B) $ associated with $ f\colon B\rightarrow A $ and show that it is Hom-finite under the above assumptions. We prove the existence of an $ n $-cluster tilting object in the Higgs category $ \mathcal{H} $ which is an extension closed subcategory of $ \mathcal{C}_{n}(A,B) $ and is stably $ n $-Calabi--Yau as a Frobenius extriangulated category in the sense of~\cite{NP2019}. However, there are many cases where $ H^{0}(A) $ is infinite-dimensional. Then the corresponding relative cluster category is Hom-infinite. Hom-infinite cluster categories were studied by Plamondon in~\cite{PGP2011}. Similarly, it would be interesting to explore Hom-infinite relative cluster categories. As this is not needed for the results in this paper, we leave it for future work.

The structure of the paper is as follows. In Section~\ref{Preliminaries}, we give some background material on dg categories and their invariants. Section~\ref{section RCY} is devoted to giving the definitions of relative Calabi--Yau structures and relative Calabi--Yau completions, and proving Proposition~\ref{reduced completion}, where we obtain a reduced version of the deformed relative Calabi--Yau completion for a dg functor between finitely cellular type dg categories. We also discuss the relation between relative Calabi--Yau completions and absolute Calabi--Yau completions, see Proposition~\ref{Relation with absolut CY completion}. Let $f\colon B\to A $ be a morphism (not necessarily preserving the unit element) between dg $ k $-algebras and let $ e=f(\boldmath{1}_{B}) $. In Section~\ref{section RCC}, under the above assumptions on $ f $, we define the \emph{relative $ n $-cluster category} $ \mathcal{C}_{n}(A,B) $ as the Verdier quotient of the perfect derived category $ \mathrm{per}A $ by the full subcategory $ \mathrm{pvd}_{B}(A) $ of the perfectly valued derived category $ \mathrm{pvd}(A) $ formed by the dg modules whose restriction to $ B $ is acyclic (see Definition~\ref{Relative cluster category}).
The relation between the triangulated categories involved can be summarized by the following commutative diagram
\begin{align*}
	\xymatrix{
		&&\mathrm{per}(eAe)\ar@{=}[r]\ar@{^{(}->}[d]&\mathrm{per}(eAe)\ar@{^{(}->}[d]\\
		\mathrm{pvd}_{B}(A)\ar@{^{(}->}"2,3"\ar@{->}_{\cong}[d]&&\mathrm{per}(A)\ar[r]^{\pi^{rel}}\ar@{->>}[d]^{p^{*}}&\mathcal{C}_{n}(A,B)\ar@{->>}[d]^{p^{*}}\\
		\mathrm{pvd}(\overline{A})\ar@{^{(}->}"3,3"&&\mathrm{per}(\overline{A})\ar[r]^{\pi}&\mathcal{C}_{n}(\overline{A}),
	}
\end{align*}
where $ \overline{A} $ is the homotopy cofiber of $ f\colon B\rightarrow A $, and the rows and columns are exact sequences of triangulated categories.

In Section~\ref{section SR and RFD}, we define the \emph{relative fundamental domain} $ \mathcal{F}^{rel} $ as a certain extension closed full subcategory of $ \mathrm{per}A $ (see Definition~\ref{RFD}). As in~\cite{Am2008} and~\cite{LYG}, the canonical quotient functor $ \pi^{rel}\colon\per A\rightarrow\mathcal{C}_{n}(A,B) $ induces a fully faithful embedding $ \pi^{rel}\colon\mathcal{F}^{rel}\hookrightarrow\mathcal{C}_{n}(A,B) $ (see Proposition~\ref{relative full faithful}). Then the \emph{Higgs category} $ \mathcal{H} $ is defined as the image of $ \mathcal{F}^{rel} $ in $ \mathcal{C}_{n}(A,B) $ (see Definition~\ref{Higgs category}). We show that it is closed under extensions in $ \mathcal{C}_{n}(A,B) $ (see Proposition~\ref{Extension closed subcategory}) and thus becomes an extriangulated category in the sense of~\cite{NP2019}. More precisely, we prove the following theorem.

\begin{Thm}$\mathrm{(Theorem~\ref{Higgs is Frobenius extrianglated category}\ and\   Proposition~\ref{addA is an n-cluster tiliting})}$
	The Higgs category $ \mathcal{H} $ is a Frobenius extriangulated category with projective-injective objects $ \mathcal{P}=\mathrm{add}(eA) $ and $ \mathrm{add}A $ is an $ n $-cluster-tilting subcategory of $ \mathcal{H} $ with $ \mathrm{End}_{\mathcal{H}}(A)=H^{0}(A) $. Moreover, the quotient functor $ p^{*}\colon\mathcal{C}_{n}(A,B)\rightarrow\mathcal{C}_{n}(\overline{A}) $ induces an equivalence of triangulated categories
	$$ \mathcal{H}/[\mathcal{P}]\iso
	\mathcal{C}_{n}(\overline{A}), $$
	where $ [\mathcal{P}] $ denotes the ideal of morphisms of $ \mathcal{H} $ which factor through objects in $ \cp $. 
\end{Thm}
In \cite{YW}, for $ n=3 $, we will define and study cluster characters on the Higgs category and the relative cluster category. 

We have the following results related to $ n $-angulated categories.
\begin{Thm}$\mathrm{(Theorem~\ref{addA is Frobenius n-exangulated category})}$
	Suppose that the $ n $-cluster tilting category $ \add\overline{A} $ of $ \cc_{n}(\overline{A}) $ satisfies $$ \Si^{n}\add\overline{A}=\add\overline{A} .$$ Then $ \add\overline{A} $ carries a canonical $ (n+2) $-angulated structure. Moreover, the $ n $-cluster-tilting subcategory $ \mathrm{add}A $ of $ \ch $ carries a canonical structure of Frobenius $ n $-exangulated category with projective-injective objects $ \mathcal{P}=\mathrm{add}(eA) $. The quotient functor $ p^{*}\colon\mathcal{C}_{n}(A,B)\rightarrow\mathcal{C}_{n}(\overline{A}) $ induces an equivalence of $ (n+2) $-angulated categories
	$$ \mathrm{add}A/[\mathcal{P}]\iso\mathrm{add}(\overline{A}). $$
\end{Thm}

In Section~\ref{section A is stalk}, under the hypotheses $ 1) $-$4) $, when the dg algebra $ A $ is concentrated in degree 0, we show that $ H^{0}(A) $ is of global dimension at most $ n+1 $ so that we have the equivalence
$$ \mathcal{D}^{b}(\mathrm{mod}H^{0}A)\iso \per A .$$
Moreover, $ H^{0}A $ is internally bimodule $ (n+1) $-Calabi–Yau with respect to the idempotent $ e=f(\boldmath{1}_{B}) $ in the sense of Pressland (see~\cite{MP2017}) and restriction induces an equivalence from the Higgs category $ \mathcal{H} $ to the category of Gorenstein projective modules over $ B'=eH^{0}(A)e $. More precisely, we have the following theorem.
\begin{Thm}$\mathrm{(Theorem~\ref{Higgs of stalk})}$
Let $ f\colon B\ra A $ be a dg algebra morphism which satisfies the hypotheses $ 1) $-$4) $. Let $ e=f(\boldmath{1}_{B}) $. Moreover, we assume that $ A $ is concentrated in degree $ 0 $. Then we have
	\begin{itemize}
		\item[a)] The algebra $ B'=eH^{0}(A)e $ is Iwanaga-Gorenstein of injective dimension at most $ g\leqslant n+1 $ as a $ B' $-module.
		\item[b)] Under the equivalence $ \mathcal{D}^{b}(\mathrm{mod}H^{0}A)\simeq\per A $, the subcategory $ \mathcal{F}^{rel} $ corresponds to the subcategory $ \mathrm{mod}_{n-1}(H^{0}A) $ of $ H^{0}A $-modules of projective dimension at most $ n-1 $.
		\item[c)] Via the equivalence $ \mathrm{res}\colon\mathcal{D}^{b}(\mathrm{mod}H^{0}A)\overset{\sim}{\longrightarrow} \per A $, the localization $ \pi^{rel}\colon\per A\rightarrow\mathcal{C}_{n}(A,B) $ identifies with the restriction functor $ \mathcal{D}^{b}(\mathrm{mod}H^{0}A)\rightarrow\mathcal{D}^{b}(\mathrm{mod}B') $, i.e.\ we have a commutative square
		\[
		\begin{tikzcd}
			\mathcal{D}^{b}(\mathrm{mod}H^{0}A) \arrow{r}\isoarrow{d} &\mathcal{D}^{b}(\mathrm{mod}B')\isoarrow{d}\\
			\per A\arrow{r} &\mathcal{C}_{n}(A,B). 
		\end{tikzcd}
		\]
		\item[d)] Under the equivalence $ \mathcal{D}^{b}(\mathrm{mod}B')\overset{\sim}{\longrightarrow}\mathcal{C}_{n}(A,B) $, the Higgs category $ \mathcal{H}\subseteq\mathcal{C}_{n}(A,B) $ corresponds to the subcategory $ \mathrm{gpr}B' $ of Gorenstein projective modules over $ B'=eH^{0}(A)e $. In particular, when $ B' $ is self injective, we have $ \mathcal{H}\cong\mathrm{mod}B'$.
		\item[e)] Let $ \cm=\add A\subset\ch $. Then the exact sequence of triangulated categories 
		$$ 0\ra\pvd_{B}(A)\ra\per A\ra\cc_{n}(A,B)\ra0 $$
		is equivalent to
		$$ 0\ra\ck^{b}_{\ch-ac}(\cm)\ra\ck^{b}(\cm)\ra\cd^{b}(\ch)\ra0. $$ In particular, the relative cluster category $ \cc_{n}(A,B) $ is equivalent to the bounded derived category $ \cd^{b}(\ch) $ of $ \ch $.
	\end{itemize}
\end{Thm}
We summarize the situation in the following commutative diagram
\begin{align*}
	\xymatrix{
		\per A\ar"1,4"^{\pi^{rel}}&&&\mathcal{C}_{n}(A,B)\\
		&\mathcal{D}^{b}(\mathrm{mod}H^{0}A)\ar[r]^{res}\ar"1,1"\ar@{}[lu]|*[@]{\sim}&\mathcal{D}^{b}(\mathrm{mod}B')\bijar[ru]&\\
		&\mathrm{mod}H^{0}A\ar[r]^{res}\ar@{^{(}->}[u]&\mathrm{mod}B'\ar@{^{(}->}[u]\\
		&\mathrm{mod}_{n-1}(H^{0}A)\ar[r]^>>>>>>>>>{\sim}\ar@{^{(}->}[u]\ar"5,1"\ar@{}[ld]|*[@]{\sim}&\mathrm{gpr}B'\ar@{^{(}->}[u]\bijar[rd]\\
		\mathcal{F}^{rel}\ar@{^{(}->}"1,1"\ar"5,4"^{\sim}&&&\mathcal{H}\ar@{^{(}->}"1,4"\,.
	}
\end{align*}

The paradigmatic example for $ A $ is the relative 3-Calabi--Yau completion of the Auslander algebra of a Dynkin quiver $ Q $ (cf.$ \ $below). Then $ B' $ is the preprojective algebra of $ Q $ and $ \mathcal{H} $ is equivalent to the module category of $ B' $. This motivates the terminology “Higgs category" because Higgs bundles~\cite{NH1987,CS1992} are the geometric version of modules over preprojective algebras.

In Section~\ref{section FQP}, we apply this general approach to Jacobi-finite ice quivers with potential. In this way, for each Jacobi-finite ice quiver with potential $ (Q,F,W) $, we associate a Frobenius extriangulated category $ \mathcal{H} $ endowed with a canonical cluster-tilting object (see Theorem~\ref{Higgs for frozen quiver}).

In the last section, we apply our main result to higher Auslander-Reiten theory. Let $ B_{0} $ be an $ n $-representation-finite algebra in the sense of Iyama-Oppermann~\cite{OISO2011}. Let $ \tau_{n}^{-1} $ be the higher inverse Auslander-Reiten translation of $ B_{0} $ and let $ A_{0}\coloneqq\mathrm{End}_{B_{0}}(\oplus_{i\geqslant 0}\tau_{n}^{-i}B_{0}) $ be the higher Auslander algebra of $ B_{0} $.

Then there is a natural fully faithful morphism
\begin{align*}
	\xymatrix{
		f_{0}\colon B_{0}\ar@{^{(}->}[r]& A_{0}
	}.
\end{align*}
The relative $ (n+2) $-Calabi--Yau completion of $ f_{0} $
$$ f\colon B=\bm{\Pi}_{n+1}(B_{0})\longrightarrow A=\bm{\Pi}_{n+2}(A_{0},B_{0}) $$ satisfies the assumptions 1)-4) and $ A $ is concentrated in degree 0. Moreover, $ H^{0}(f) $ is fully faithful (see Proposition~\ref{n-reprsentation finite case}). Let $ \widetilde{B_{0}} $ denote the higher preprojective algebra of $ B_{0} $ in the sense of Iyama-Oppermann~\cite{OISO2011}. We give a new proof (see Lemma~\ref{Lemma:self-injective}) of the fact, first proved in~\cite{OISO2011}, that
$ \widetilde{B_{0}} $ is a self-injective algebra. By our main results in Section~\ref{section SR and RFD} and Section~\ref{section A is stalk}, we have the following theorem.

\begin{Thm}$\mathrm{(Theorem~\ref{CTO in n-finite case})}$
Consider the relative cluster category $ \mathcal{C}_{n+1}(A,B) $ associated with $ f\colon B\rightarrow A $.
	\begin{itemize}
		\item[a)] The Higgs category $ \mathcal{H}\subseteq\mathcal{C}_{n+1}(A,B) $ is equivalent to $ \mathrm{mod}(\widetilde{B_{0}}) $ and the image of $ A $ in $\mathcal{H}$ is an $ (n+1) $-cluster-tilting object.
		\item[b)] We have a triangle equivalence $ \underline{\mathrm{mod}}(\widetilde{B_{0}})\iso\mathcal{C}_{n+1}(A_{0}/A_{0}eA_{0}) $, where $ e=f_{0}(\boldmath{1}_{B_{0}}) $. In particular, $ \underline{\mathrm{mod}}(\widetilde{B_{0}}) $ contains a canonical $ (n+1) $-cluster-tilting object.
	\end{itemize}
\end{Thm}
Notice that $ b) $ is the main result of~\cite{OISO2011}. We deduce it from $ a) $ thereby giving a new proof which is fundamentally different from that of~\cite{OISO2011}.

\section*{Notation and conventions}
Throughout this paper, $ k $ will denote an algebraically closed field. We denote by $ D=\Hom_{k}(−, k) $ the $ k $-dual. All modules are right modules unless stated otherwise. We say an algebra $ A $ is Noetherian if it is Noetherian as both a left and right module over itself. We denote by $ gf $ the composition of morphisms (or arrows) $ f\colon X\rightarrow Y $ and $ g\colon Y\rightarrow Z $.

Let $ \mathcal{T} $ be an additive category. We say that a morphism $ f\colon X\rightarrow Y $ is \emph{right minimal} if any morphism $ g\colon X\rightarrow X $ satisfying $ fg=f $ is an isomorphism. Dually we define a \emph{left minimal} morphism. For a collection $ \mathcal{X} $ of objects in $ \mathcal{T} $, we denote by $ \mathrm{add}\,\mathcal{X} $ the smallest full subcategory of $ \mathcal{T} $ which is closed under finite coproducts, summands and isomorphisms and contains $ \mathcal{X} $. 

Let $ \mathcal{X} $ be a subcategory of $ \mathcal{T} $. We say that a morphism $ f\colon X\rightarrow Y $ is a \emph{right $ \mathcal{X} $-approximation} of $ Y $ if $ X\in\mathcal{X} $ and $ \mathrm{Hom}_{\mathcal{T}}(X',f) $ is surjective for any $ X'\in\mathcal{X} $. We say that $ \mathcal{X} $ is \emph{contravariantly finite} if any object in $ \mathcal{T} $ has a right $ \mathcal{X} $-approximation. Dually, we define a \emph{left $ \mathcal{X} $-approximation} and a \emph{covariantly finite} subcategory. We say that $ \mathcal{X} $ is \emph{functorially finite} if it is contravariantly and covariantly finite.

When $ \mathcal{T} $ is a triangulated category, we denote by $ \Sigma $ the shift functor and by $ \mathrm{thick}\,\mathcal{X} $ the smallest thick subcategory of $ \mathcal{T} $ containing $ \mathcal{X} $. For collections $ \mathcal{X} $ and $ \mathcal{Y} $ of objects in $ \mathcal{T} $, we denote by $ \mathcal{X}*\mathcal{Y} $ the collection of objects $ Z\in\mathcal{T} $ appearing in a triangle $ X\rightarrow Z\rightarrow Y\rightarrow \Si X $ with $ X\in\mathcal{X} $ and $ Y\in\mathcal{Y} $. We set
$$ \mathcal{X}^{\perp_{\mathcal{T}}}\coloneqq\{T\in\mathcal{T}\,|\,\mathrm{Hom}(\mathcal{X},T)=0\}, $$
$$ {}^{\perp_{\mathcal{T}}}\mathcal{X}\coloneqq\{T\in\mathcal{T}\,|\,\mathrm{Hom}(T,\mathcal{X})=0\} .$$ When it does not cause confusion, we will simply write $ \mathcal{X}^{\perp} $ and $ ^{\perp}\mathcal{X} $.

Let $ \ct $ be a triangulated category. For two objects $ X $ and $ Y $ of $ \ct $ and an integer $ n $, by $ \Hom_{\ct}(X,\Si^{>n}Y)=0 $ (respectively, $ \Hom_{\ct}(X,\Si^{\geqslant n}Y)=0 $, $ \Hom_{\ct}(X,\Si^{<n}Y)=0 $, $ \Hom_{\ct}(X,\Si^{\leqslant n}Y)=0 $), we mean 
$$ \Hom_{\ct}(X,\Si^{i}Y)=0 $$ for all $ i>n $ (respectively, for all $ i\geqslant n $, $ i < n $, $ i\leqslant n $). 

Let $ \cx $ be a full subcategory of $ \ct $. We say that $ \cx $ is a thick subcategory of $ \ct $ if it is a triangulated subcategory of $ \ct $ which is closed under taking direct summands. In this case we denote by $ \ct/\cx $ the triangle quotient of $ \ct $ by $ \cx $. In general, we denote by $ \thick_{\ct}\cx $ (or simply $ \thick\cx $) the smallest thick subcategory of $ \ct $ which contains $ \cx $. 


\section*{Acknowledgments}

The author is supported by the China Scholarship Council (CSC, grant number: 201906140160), Innovation Program for Quantum Science and Technology (2021ZD0302902), Natural Science Foundation of China (Grant Nos. 12071137) and STCSM (No. 18dz2271000). He would like to thank his PhD supervisor Bernhard Keller for his his guidance, patience and kindness. He would also like to thank his PhD co-supervisor Guodong Zhou for the constant support and encouragement during his career in mathematics. He is grateful to Peigen Cao, Xiaofa Chen, Haibo Jin, Yann Palu, Kai Wang, and Yu Wang for many interesting discussions and useful comments. He is very grateful to the anonymous referee for suggesting various improvements.

\section{Preliminaries}\label{Preliminaries}
We recall some basic definitions related to dg categories and their invariants. We refer to Keller’s ICM address \cite{BK2006} for the details.

A \emph{differential graded or dg category} is a $ k $-category $ \mathcal{A} $ whose morphism spaces are dg $ k $-modules and whose compositions $ \mathcal{A}(y,z)\otimes\mathcal{A}(x,y)\to \mathcal{A}(x,z)$ are morphisms of dg $ k $-modules.  We denote the category of all (small) dg categories over $ k $ by $ \dgcat $. In particular, dg categories with one object can be identified with dg algebras $ A $, i.e.\ graded $ k $-algebras endowed with a differential $ d $ such that the Leibniz rule holds
\begin{align*}
	\xymatrix{
		d(f\circ g)=d(f)\circ g+(-1)^{p}f\circ d(g)
	}
\end{align*}
for all $ f\in A^{p} $ and all $ g $.

Let $ \mathcal{A} $ be a dg category. The \emph{opposite dg category} $ \mathcal{A}^{op} $ has the same objects as $ \mathcal{A} $ and its morphisms are defined by
$$ \mathcal{A}^{op}(X,Y)=\mathcal{A}(Y,X); $$
the composition of $ f\in\mathcal{A}^{op}(Y, X)^{p} $ with $ g\in\mathcal{A}^{op}(Z,Y)^{q} $ is given by $ (−1)^{pq}gf $.
The category $ Z^{0}(\mathcal{A}) $ has the same objects as $ \mathcal{A} $ and its morphisms are defined by
$$ (Z^{0}\mathcal{A})(X,Y)=Z^{0}(\mathcal{A}(X,Y)), $$ where $ Z^{0} $ is the kernel of $ d\colon \mathcal{A}(X,Y)^{0}\rightarrow\mathcal{A}(X,Y)^{1} $. The category $ H^{0}(\mathcal{A}) $ has the same objects as $ \mathcal{A} $ and its morphisms are defined by
$$ (H^{0}(\mathcal{A}))(X,Y)=H^{0}(\mathcal{A}(X,Y)), $$ where $ H^{0} $ denotes the 0-th homology of the complex $ \mathcal{A}(X,Y) $. We say that a morphism $ f\colon x\to y $ in $ Z^{0}(\ca) $ is a \emph{homotopy equivalence} if it becomes invertible in $ H^{0}(\ca) $.

\subsection{The category of dg categories}

Let $ \mathcal{A} $ and $ \mathcal{B} $ be dg categories. A \emph{dg functor}
$ G\colon \mathcal{B} \to\mathcal{A} $
is given by a map $ G\colon \obj(\mathcal{B})\to \obj(\mathcal{A}) $ and by morphisms of dg $ k $-modules $ G(x,y)\colon\mathcal{B}(x,y)\to \mathcal{A}(Gx,Gy),x,y\in\obj(\mathcal{B})$,
compatible with the composition and the units. The category of small dg categories $ \dgcat $ has the small dg categories as objects and the dg functors as morphisms.

The \emph{tensor product} $ \mathcal{A}\otimes\mathcal{B} $ has the class of objects $ \obj(\mathcal{A})\times \obj(\mathcal{B}) $ and the morphism spaces
$ \mathcal{A}\otimes\mathcal{B}((x,y),(x',y'))=\mathcal{A}(x,x')\otimes\mathcal{B}(y,y') $
with the natural compositions and units. The \emph{enveloping dg category} of $ \mathcal{A} $ is defined as $ \mathcal{A}\otimes\mathcal{A}^{op} $ and we denote it by $ \mathcal{A}^{e} $.

Let $ G, G'\colon\mathcal{B}\to\mathcal{A} $ be two dg functors. We define $ \mathcal{H}om(G,G')^{n} $ to be the $ k $-module formed by the families of morphisms
\begin{align*}
	\phi_{x}\in\mathcal{A}(Gx,G'x)^{n}
\end{align*}
such that $ G'(f)\circ\phi_{x}=(-1)^{pn}\phi_{y}\circ G(f) $ for all $ f\in\mathcal{B}^{p}(x,y) $. We define $ \mathcal{H}om(G,G') $ to be the graded $ k $-module with components $ \mathcal{H}om(G,G')^{n} $ and whose differential is induced by the differential of $ \mathcal{A}(Gx,G'x) $. The set of morphisms $ G\to G' $ is by definition in bijection with $ Z^{0}(\mathcal{H}om(G,G')) $. Thus,  we can form a dg category $ \mathcal{H}om(\mathcal{B},\mathcal{A}) $, which has the dg functors as objects and the morphism space $ \mathcal{H}om(G,G') $ for two dg functors $ G $ and $ G' $.

Endowed with the tensor product, the category $ \dgcat $ becomes a symmetric tensor category which admits an internal Hom-functor, i.e.
\begin{align*}
	\mathrm{Hom}(\mathcal{A}\otimes\mathcal{B},\mathcal{C})=\mathrm{Hom}(\mathcal{A},\mathcal{H}om(\mathcal{B},\mathcal{C})),
\end{align*}
for $ \mathcal{A}, \mathcal{B},\mathcal{C}\in \dgcat $.

\begin{Def}\rm
	A \emph{quasi-equivalence} is a dg functor $ G\colon\mathcal{B}\to\mathcal{A}$ such that
	\begin{itemize}
		\item[(1)] $ G(x,y)\colon\mathcal{B}(x,y)\to\mathcal{A}(G(x),G(y)) $ is a quasi-isomorphism for all objects $ x,y $ of $\mathcal{A}$;
		\item[(2)] the induced functor $ H^{0}(G)\colon H^{0}(\mathcal{B})\to H^{0}(\mathcal{A}) $ is an equivalence.
	\end{itemize}
	
\end{Def}
By \cite{GT2005}, there is a model structure on $ \dgcat $ with weak equivalences being quasi-equivalences. We denote by $\rm Ho(\dgcat) $ the corresponding homotopy category.
\begin{Thm}\cite[Theorem 0.1]{GT2005}\label{htp dg}\label{model on dgcat}
	Let $ k $ be a commutative ring. There is a cofibrantly generated model structure (Dwyer-Kan model structure) on $ \dgcat $ where a dg functor $ G\colon\mathcal{B}\to\mathcal{A} $ is
	\begin{itemize}
		\item a weak equivalence if $ G $ is quasi-equivalence;
		\item a fibration if 
		\subitem{1.} for all objects $ x,y\in\mathcal{B} $ the component $ G(x,y) $ is a degreewise surjection of chain complexes;
		\subitem{2.} for each isomorphism $ G(x)\to z $ in $ H^{0}(\mathcal{A}) $ there is a lift to an isomorphism in $ H^{0}(\mathcal{B}) $.
	\end{itemize}
\end{Thm}

\subsection{Drinfeld dg quotients}
Let $ \cb\subseteq\ca $ be a full dg subcategory. Denote by $ j\colon\cb\ra\ca $ the inclusion.
\begin{Def}\rm\cite[Section 3]{VD2004}
	The dg quotient category $ \ca/\cb $ is defined as follows:
	\begin{itemize}
		\item $ \obj(\ca/\cb)=\obj(\ca) $;
		\item freely add new morphisms $ \epsilon_{U}\colon U\ra U $ of degree $ −1 $ for each $ U\in\obj(\cb) $, and set $ d(\epsilon_{U})=\boldmath{1}_{U} $.
	\end{itemize}
\end{Def}
We denote by $ p\colon\ca\ra\ca/\cb $ the canonical functor. For any objects $ x $ and $ y $, we have a decomposition of graded $ k $-modules
$$ \ca/\cb(x,y)=\bigoplus_{n\geqslant0}\bigoplus_{U_{i}\in\obj(\cb)}\ca(U_{n},y)\ten_{k}k\epsilon_{U_{n}}\ten_{k}\cdots\ten_{k}k\epsilon_{U_{2}}\ten_{k}\ca(U_{1},U_{2})\ten_{k}k\epsilon_{U_{1}}\ten_{k}\ca(x,U_{1}) .$$ Using the formula $ d(\epsilon_{U})=\boldmath{1}_{U} $, one can easily get the differential on $ \ca/\cb(x,y) $.

\bigskip

Let $ G\colon\mathcal{B}\to\mathcal{A} $ be a dg functor. The \emph{homotopy cofiber} $ \mathcal{A}/\mathcal{B} $ \cite[Remark 4.0.2]{GT2010} of $ G $ is defined by the following homotopy push-out diagram in $ \dgcat $ with respect to the Dwyer-Kan model structure
\begin{align*}
	\xymatrix{
		\mathcal{B}\ar[r]^-{G}\ar[d]&\mathcal{A}\ar[d]^-{p}\\
		0\ar[r]&\mathcal{A}/\mathcal{B},
	}
\end{align*}
where $ 0 $ is the dg category of one object with endomorphism space being $ 0 $.
We will call $ \cb\rightarrow\ca\rightarrow \ca/\cb $ a \emph{homotopy cofiber sequence} in $ \dgcat $.

The homotopy cofiber $ \mathcal{A}/\mathcal{B} $ can be computed as the Drinfeld dg quotient of $ \mathcal{A} $ by its full dg subcategory $ \Ima(G) $, where $ \Ima(G) $ is the full dg subcategory of $ \mathcal{A} $ whose objects are the $ y\in\mathcal{A} $ such that there exist an object $ x $ in $ \mathcal{B} $ and an isomorphism $ G(x)\cong y $ in $ \rm H^{0}(\mathcal{A}) $.

\subsection{Homotopy between dg functors}
Let $ \mathcal{B} $ be a small dg category. The \emph{dg category} $ P(\mathcal{B}) $ is defined as follows: its \emph{objects}
are the homotopy equivalences $ f\colon x\to y $. The complexes of \emph{morphisms} are defined (as $ \mathbb{Z} $-graded $ k $-modules) by:
\begin{align*}
	\xymatrix{
		P(\mathcal{B})(f,g)=\mathcal{B}(x,w)\oplus\mathcal{B}(y,z)\oplus\mathcal{B}(x,z)[-1]
	},
\end{align*}
where $ f\colon x\to y $, $ g\colon w\to z $ are in $ P(\mathcal{B})$.

A homogeneous element of degree $ r $ of this graded $ k $-module can be represented by a matrix
$$
\left[
\begin{matrix}
	m_{1} & 0  \\
	h & m_{2}  
\end{matrix}
\right], 
$$
where $ m_{1}\in \mathcal{B}(x,w)^{r}$, $ m_{2}\in \mathcal{B}(y,z)^{r}$ and $ h\in \mathcal{B}(x,z)^{r-1}$.

The differential is given by
\begin{align*}
	d \left\{  \left[
	\begin{matrix}
		m_{1} & 0  \\
		h & m_{2}  
	\end{matrix}
	\right] \right\}= \left[
	\begin{matrix}
		d(m_{1}) & 0  \\
		d(h)+gm_{1}-(-1)^{r}(m_{2}f) & d(m_{2})  
	\end{matrix}
	\right].
\end{align*}
The Composition in $ P(\mathcal{B}) $ corresponds to matrix multiplication and the units to the identity matrices.

Then we have a dg inclusion functor
\begin{align*}
	\xymatrix{
		I\colon\mathcal{B}\ar[r]& P(\mathcal{B})
	}
\end{align*}
which sends an object $ x $ in $ \mathcal{B} $ to $ (x=x) $ and $ I(f)=\left[
\begin{matrix}
	f & 0  \\
	0 & f  
\end{matrix}
\right]$.

Moreover we have two projection functors
\begin{align*}
	P_{0},P_{1}\colon P(\mathcal{B})\to\mathcal{B}
\end{align*}
which are defined as follows
\begin{align*}
	P_{0}(f\colon x\to y)=x&&P_{0}\left\{
	\left[
	\begin{matrix}
		m_{1} & 0  \\
		h & m_{2}  
	\end{matrix}
	\right]
	\right\}=m_{1};\\
	P_{1}(f\colon x\to y)=y&&P_{1}\left\{
	\left[
	\begin{matrix}
		m_{1} & 0  \\
		h & m_{2}  
	\end{matrix}
	\right]
	\right\}=m_{2}
	.
\end{align*}

Then we obtain the following commutative diagram in $ \dgcat $ (\cite[Proposition 2.0.11]{GT2010})
\begin{align*}
	\xymatrix{
		\mathcal{B}\ar"1,3"^{\bigtriangleup=(id_{\mathcal{B}},id_{\mathcal{B}})}\ar"2,2"_{I}&&\mathcal{B}\times\mathcal{B}\\
		&P(\mathcal{B})\ar"1,3"_{P_{0}\times P_{1}}&,
	}
\end{align*}
where $ I $ is a quasi-equivalence and $ P_{0}\times P_{1} $ is a trivial fibration, with respect to the Dwyer-Kan model structure on $ \dgcat $.
This means that the dg category $ P(\mathcal{B}) $ is a path object for $ \mathcal{B}$.

\begin{Def}\label{homotopy def}\cite[Remark 2.0.12]{GT2010}\rm\, Let $ G, G'\colon \mathcal{B}\to\mathcal{A} $ be two dg functors, where $ \mathcal{B} $ is a cofibrant dg category. Two dg functors $ G $ and $ G' $ are homotopic if there exists a dg functor $ H\colon\mathcal{B}\to P(\mathcal{A}) $ that makes the following diagram commute
	\begin{align*}
		\xymatrix{
			&&\mathcal{A}\\
			\mathcal{B}\ar"1,3"^{G}\ar"3,3"_{G'}\ar"2,3"^{H}&&P(\mathcal{A})\ar[u]_{P_{0}}\ar[d]^{P_{1}}\\
			&&\mathcal{A}.
		}
	\end{align*}
	The dg functor $ H $ corresponds exactly to
	\begin{itemize}
		\item a homotopy equivalence $ \alpha(x)\colon G(x)\to G'(x) $ in $ \mathcal{A} $ for every object $ x $ in $ \mathcal{B} $, and
		\item a degree $ -1 $ morphism
		$h=h(x,y)\colon\mathcal{B}(x,y)\to\mathcal{A}(G(x),G'(y)) $, for all objects $ x $ and $ y $ in $ \mathcal{B} $,  such that
		$$ \alpha(y) G(f)-G'(f)\alpha(x)=d(h(f))+h(d(f)) $$ and $$ h(fg)=h(f) G(g)+(-1)^{n}G'(f)h(g),$$ where $ f $ and $ g $ are composable morphisms in $ \mathcal{B} $ with $ f $ of degree $ n $.
	\end{itemize}	
\end{Def}

\subsection{The derived category of a dg category}

Let $ \mathcal{A} $, $ \mathcal{B} $ be small dg categories. Let $ \cc_{dg}(k) $ be the dg category of $ k $-complexes \cite[Section 2.2]{BK2006}. 
A \emph{left dg $ \mathcal{A} $-module} is a dg functor $ L\colon\mathcal{A} \to \mathcal{C}_{dg}(k) $. A \emph{right dg $ \mathcal{A} $-module} is a dg functor $ M\colon\mathcal{A}^{op}\to \mathcal{C}_{dg}(k)$. A \emph{dg $ \mathcal{A}$-$\mathcal{B} $-bimodule} is a dg functor $ N\colon\mathcal{B}^{op}\otimes\mathcal{A}\to \mathcal{C}_{dg}(k) $. For each object $ X $ of $ \ca $, we have the right module represented by $ X $
$$ X^{\we}=\ca(?,X) .$$
The \emph{category of right dg modules} $ \mathcal{C}(\mathcal{A}) $ has as objects the right dg $ \mathcal{A} $-modules and as morphisms $ L\rightarrow M $ the morphisms of dg functors.

We identify $ \mathcal{A} $-bimodules with right $ \mathcal{A}^{e} $-modules via the morphism
\begin{align*}
	M\otimes\mathcal{A}^{e}=M\otimes\mathcal{A}\otimes\mathcal{A}^{op}\xrightarrow{\sim}\mathcal{A}^{op}\otimes M\otimes\mathcal{A}  
\end{align*}
taking $ m\otimes a\otimes b $ to $ (-1)^{|b|(|m|+|a|)}b\otimes m\otimes a $, and we denote by $ \cc(\mathcal{A}^{e}) $ the category of $ \mathcal{A} $-bimodules. There is a distinguished $ \mathcal{A} $-bimodule $ \mathcal{A}_{\triangle} $ given by morphisms in the category $ \mathcal{A} $, i.e.\ 
$ \mathcal{A}_{\triangle}(x,y)=\mathcal{A}(x,y)$. We call it the \emph{diagonal bimodule of $ \mathcal{A} $} and still denote it by $ \mathcal{A} $. 

A bimodule $ M\in\cc(\ca^{e}) $ is said to be \emph{semi-free} if there is a set of homogeneous elements $ {\xi_{i}\in M(x_{i},y_{i})},\,i\in S $, called a basis of $ M $, such that, for any pair $ (x,y)\in\obj(\ca)\times\obj(\ca) $, every object $ \eta\in M(x,y) $ can be written uniquely as a finite sum 
$$ \eta=\Sigma_{i\in S}f_{i}\circ\xi_{i}\circ g_{i}, $$ where $ g_{i}\in\ca(x,x_{i}) $ and $ f_{i}\in\ca(y_{i},y) $ and only finitely many of them are nonzero. When the basis set is finite, its cardinality is called the \emph{rank} of the semi-free module $ M $.

The dg category $ \cc_{dg}(\ca) $ is defined by $ \cc_{dg}(\ca)=\ch om(\ca^{op},\cc_{dg}(k)) $. We write $ \ch om(L,M) $ for the complex of morphisms from $ L $ to $ M $ in $ \cc_{dg}(\ca) $. For each $ X\in\ca $, we have a natural isomorphism $$ \ch om(X^{\we},M)\iso M(X) .$$


A morphism $ f\colon L\rightarrow M  $ is a \emph{quasi-isomorphism} if it induces an isomorphism in homology. Then the \emph{derived category} $ \mathcal{D}(\mathcal{A}) $ is the localization of the category $ \mathcal{C}(\mathcal{A}) $ with respect to the class of quasi-isomorphisms. The \emph{category of perfect objects}  per$(\mathcal{A}) $ associated with $ \mathcal{A} $ is the closure in $ \mathcal{D}(\mathcal{A}) $ of the set of representable functors $ X^{\wedge}=\mathcal{A}(?,X) $, $  X\in\mathcal{A} $, under shifts in both directions, extensions and taking direct factors. The \emph{category of perfectly valued modules} $ \mathrm{pvd}(\mathcal{A}) $ is the full subcategory of $ \mathcal{D}(\mathcal{A}) $ formed by the dg modules $ M $ such that each dg $ k $-module $ M (X) $, $ X\in\mathcal{A} $, is perfect, i.e.\ $ \sum_{p}\dim_{k} H^{p}(M(X)) $ is finite. 

Let $ f\colon\mathcal{B}\rightarrow\mathcal{A} $ be a dg functor. Then $ f $ induces the restriction functor $ f_{*}\colon\mathcal{C}(\mathcal{A})\rightarrow\mathcal{C}(\mathcal{B}) $ which is given by $ f_{*}(M)=M\circ f $. It fits into the usual triple of adjoint functors $ (f^{*},f_{*},f^{!}) $ between $ \mathcal{C}(\mathcal{A}) $ and $ \mathcal{C}(\mathcal{B}) $. We denote the corresponding adjoint functors between $ \mathcal{D}(\mathcal{A}) $ and $ \mathcal{D}(\mathcal{B}) $ by $ (\mathbf{L}f^{*},f_{*},\mathbf{R}f^{!}) $. 

The functor $ f $ also induces a dg functor $ f^{e}\colon\cb^{e}\ra\ca^{e} $ between their enveloping dg categories. By abuse of notation, we also denote the corresponding adjoint functors between $ \mathcal{D}(\ca^{e}) $ and $ \mathcal{D}(\cb^{e}) $ by $ (\mathbf{L}f^{*},f_{*},\mathbf{R}f^{!}) $. For a dg $ \cb $-bimodule $ M $, we have the following useful formula
\begin{equation*}
	\begin{split}
		\mathbf{L}f^{*}(M)=&M\lten_{\cb^{e}}\ca^{e}\\
		\cong&\ca\lten_{\cb}M\lten_{\cb}\ca.
	\end{split}
\end{equation*}
In particular, if we take $ M=\cb $, then $ \mathbf{L}f^{*}(\cb)\cong\ca\lten_{\cb}\ca $.

\begin{Def}\rm
	A dg category $\mathcal{A} $ is said to be \emph{(homologically) smooth} if the diagonal bimodule $ \mathcal{A} $ is perfect as a module over $ \mathcal{A}^{e} $, i.e.\ $ \mathcal{A} $ is in per$ (\mathcal{A}^{e}) $.	
\end{Def}
\begin{Def}\rm
	A dg category $\mathcal{A} $ is said to be \emph{proper} if $ \ca(X,Y)\in\per k $ for all objects $ (X,Y)\in\ca^{e} $.
\end{Def}

\begin{Def}\label{derived dual}\rm
	For any right $ \mathcal{A}^{e} $-module $ M $, we define its \emph{derived dual} $ M^{\vee} $ in $ \mathcal{D}(\mathcal{A}^{e}) $ as
	\begin{align*}
		M^\vee=\mathbf{R}\mathrm{Hom}_{\mathcal{A}^{e}}(M,\mathcal{A}^{e}).
	\end{align*}
	In particular, the $ \emph{inverse dualizing bimodule} $ of $ \mathcal{A} $ is defined as $ \mathcal{A}^{\vee} $.
\end{Def}

\begin{Def}\label{relative derived dual}\rm
	Let $ G\colon\mathcal{B}\to\mathcal{A} $ be dg functor between dg categories. The \emph{inverse dualizing bimodule of $ G $} in $ \mathcal{D}(\mathcal{A}^{e}) $ is defined as
	\begin{align*}
		\Theta_{G}=\mathbf{R}\mathrm{Hom}_{\mathcal{A}^{e}}(\cone(\mathcal{A}\lten_{\cb}\mathcal{A}\to \mathcal{A}),\mathcal{A}^{e}).
	\end{align*}
\end{Def}

\begin{Def}\rm\cite[Section 4.5]{BK2006}
	A dg category $ \ca $ is called \emph{strictly pretriangulated (=spt)} if it satisfies the following:
	\begin{itemize}
		\item  each object has a suspension, and $ \Si\colon\ca\ra\ca $ is dg dense (i.e.\ every object in the target category is dg isomorphic to some object in the image);
		\item each closed morphism of degree zero has a cone.
	\end{itemize}
\end{Def}
\begin{Prop}\cite[Lemma 2.3]{BK1999}
	Let $ \ca $ be a spt dg category. Then $ Z^{0}(\ca) $ has a canonical Frobenius exact structure, whose stable category coincides with $ H^{0}(\ca) $. Therefore, $ H^{0}(\ca) $ is canonically triangulated.
\end{Prop}
\begin{Def}\rm
	The \emph{pretriangulated hull} $ \ca^{pretr} $ is the smallest dg subcategory of $ \cc_{dg}(\ca) $ containing $ \ca $, closed under $ \Si^{\pm} $ and cones. As $ \ca^{pretr} $ is spt, the \emph{triangulated hull} $ \ca^{tr} $ of $ \ca $ is defined to be $ H^{0}(\ca^{pretr}) $.
\end{Def}

\begin{Thm}\cite[Theorem 3.4]{VD2004}\label{Thm:dg quotient to tri quo}
	Let $ \ca $ be a dg category and $ \cb \subseteq\ca $ a full dg subcategory. Then the canonical functor $$ \ca^{tr}/\cb^{tr}\iso(\ca/\cb)^{tr} $$ is a triangle equivalence.
\end{Thm}

\subsection{Relative Hochschild homology}
Let $ \mathcal{A} $ be a dg category. The $ \emph{bar resolution} $ $ (C^{bar}(\mathcal{A}),b') $ of the diagonal bimodule $ \mathcal{A} $ is the dg $ \mathcal{A} $-bimodule whose value at $ (x,y) $ is given by the total complex of the bicomplex whose $ (n,j) $-th entry is 
\begin{align*}
	\xymatrix{
		C_{n}^{bar}(x,y)^{(j)}\coloneqq\oplus_{x_{0},\cdots,x_{n-1}}\{\mathcal{A}(x_{n-1},y)\otimes_{k}\mathcal{A}(x_{n-2},x_{n-1})\otimes_{k}\cdots\otimes_{k}\mathcal{A}(x_{0},x_{1})\otimes_{k}\mathcal{A}(x,x_{0})\}^{(j)}
	},
\end{align*}
where the horizontal arrows are given by the Hochschild differential
\begin{align*}
	\xymatrix{
		d_{0}(a_{0}\otimes\cdots\otimes a_{n})\coloneqq\Sigma_{i=0}^{n-1}\pm a_{0}\otimes\cdots\otimes a_{i}a_{i+1}\otimes\cdots\otimes a_{n}
	}.
\end{align*}
and the vertical arrows are the differentials of the tensor products.

The augmentation is the morphism of bimodules
\begin{align*}
	\xymatrix{
		\epsilon_{\mathcal{A}}\colon C^{bar}(\mathcal{A})\ar[r]&\mathcal{A}
	}
\end{align*}
which is
\begin{align*}
	\xymatrix{
		\epsilon_{x,y}\colon\oplus_{z\in Obj(\mathcal{A})}\mathcal{A}(z,y)\otimes_{k}\mathcal{A}(x,z)\ar[r]&\mathcal{A}(x,y)\colon f\otimes f^{'}\ar@{|->}[r]&f\circ f^{'}
	}
\end{align*}
and 0 everywhere else. $ C^{bar}(\mathcal{A}) $ is a cofibrant replacement of $ \mathcal{A} $ in the category of $ \mathcal{A} $-bimodules.

\begin{Def}\rm
	Let $ \mathcal{A} $ be a small dg category. Then the \emph{Hochschild complex} of $ \mathcal{A} $ is defined as
	\begin{align*}
		\xymatrix{
			H\!H(\mathcal{A})=\mathcal{A}\otimes_{\mathcal{A}^{e}}C^{bar}(\mathcal{A})
		}
	\end{align*} and the \emph{Hochschild homology} $ H\!H_{\bullet}(\ca) $ of $ \mathcal{A} $ is the homology of this complex. More precisely,
	\begin{align*}
		\xymatrix{
			H\!H(\mathcal{A})=\bigoplus_{m\geqslant0}\{\bigoplus_{(x_{0},x_{1},\cdots,x_{m})\in Obj(\mathcal{A})}\mathcal{A}(x_{m},x_{0})\otimes(\Si\mathcal{A}(x_{m-1},x_{m}))\otimes\cdots\otimes(\Si\mathcal{A}(x_{0},x_{1}))\}
		}
	\end{align*}
	We denote by $ b $ the differential of $ H\!H(\mathcal{A}) $.
	
\end{Def}

Let $ G\colon\mathcal{B}\to\mathcal{A} $ be dg functor. Then $ G $ induces a canonical morphism of $ \mathcal{B} $-bimodules $ G_{\mathcal{B},\mathcal{A}}\colon C^{bar}(\mathcal{B})\to C^{bar}(\mathcal{A}) $ and we have the following commutative diagram of $ \mathcal{B} $-bimodules
\begin{align*}
	\xymatrix{
		\mathcal{B}\ar[r]^{G}&\mathcal{A}\\
		C^{bar}(\mathcal{B})\ar[r]^{G_{\mathcal{B},\mathcal{A}}}\ar[u]^{\epsilon_{\mathcal{B}}}& C^{bar}(\mathcal{A})\ar[u]_{\epsilon_{\mathcal{A}}}.
	}
\end{align*}
Thus, we have a canonical morphism of Hochschild complexes
\begin{align*}
	\xymatrix@=1.5cm{
		\gamma_{G}\colon H\!H(\mathcal{B})=\mathcal{B}\otimes_{\mathcal{B}^{e}}C^{bar}(\mathcal{B})\ar[r]^-{G\otimes G_{\mathcal{B},\mathcal{A}}}&\mathcal{A}\otimes_{\mathcal{B}^{e}}C^{bar}(\mathcal{A})\ar@{->>}[r]^-{can}& H\!H(\mathcal{A})=\mathcal{A}\otimes_{\mathcal{A}^{e}}C^{bar}(\mathcal{A})
	}
\end{align*}

\begin{Def}\label{relative Hochschild homology}\rm
	The \emph{Hochschild homology} $ H\!H_{\bullet}(G) $ of the dg functor $ G\colon\mathcal{B}\to\mathcal{A} $ is the homology of the \emph{relative Hochschild complex} which is defined as follows
	\begin{align*}
		\xymatrix{
			H\!H(G)=\cone(\gamma_{G}\colon H\!H(\mathcal{B})\to H\!H(\mathcal{A}))
		}.
	\end{align*}
\end{Def}

\subsection{Mixed complexes}
Let $ \Lambda $ be the dg algebra generated by an indeterminate $ \epsilon $ of chain degree $ -1 $ with $ \epsilon^{2}=0 $ and $ d\epsilon=0 $. The underlying complex of $ \Lambda $ is
$$ \cdots0\rightarrow0\rightarrow k\epsilon\rightarrow k\rightarrow0\cdots. $$
It carries the structure of graded commutative Hopf
algebra with coproduct given by $ \Delta(\epsilon)=1\ten\epsilon+\epsilon\ten1 $.

Then a \emph{mixed complex} over $ k $ is a dg right $ \Lambda $-module whose underlying dg $ k $-module is $ (M,b) $ and where $ \epsilon $ acts by $ B $. Suppose that $ M=(M,b,B) $ is a mixed complex. Then the \emph{shifted mixed complex} $ \Si M $ is the mixed complex such that $ (\Si M)_{p}=M_{p-1} $ for all $ p $, $ b_{\Si M}=-b $ and $ B_{\Si M}=-B $. Let $ f\colon M\rightarrow M' $ be a morphism of mixed complexes. Then the \emph{mapping cone} over $ f $ is the mixed complex
$$ \bigg(M'\oplus M,\begin{bmatrix}
	b_{M'} & f \\
	0 & -b_{M}
\end{bmatrix},\begin{bmatrix}
	B_{M'} & 0 \\
	0 & -B_{M}
\end{bmatrix}\bigg). $$
We denote by $ \cm ix $ the category of mixed complexes and by $ \cd\cm ix $ the derived category of the dg algebra $ \Lambda $.

Let $ \mathcal{A} $ be a dg category. We associate a precyclic chain complex $ C(\ca) $ (see~\cite{JLLoday_book}) with $ \ca $ as follows: For each $ n\in\mathbb{N} $, its $ n $-th term is
$$ \coprod\ca(x_{n},x_{0})\ten\ca(x_{n-1},x_{n})\ten\ca(x_{n-2},x_{n-1})\ten\cdots\ten\ca(x_{0},x_{1}),$$ where the sum runs over all sequences $ x_{0},\ldots,x_{n} $ of objects of $ \ca $. The degeneracy maps are given by
\begin{equation*}
	d_{i}(f_{n},\ldots,f_{i},f_{i-1},\ldots,f_{0})=\left\{
	\begin{aligned}
		(f_{n},\ldots,f_{i}f_{i-1},\ldots,f_{0})&&\text{if $ i>0 $,}\\
		(-1)^{n+\sigma}(f_{0}f_{n},\ldots,f_{1})&&\text{if $ i=0 $},
	\end{aligned}
	\right.
\end{equation*}
where $ \sigma=(\mathrm{deg}f_{0})(\mathrm{deg}f_{1}+\cdots+\mathrm{deg}f_{n-1}) $. The cyclic operator is given by
$$ t(f_{n-1},\ldots,f_{0})=(-1)^{n+\sigma}(f_{0},f_{n-1},f_{n-2},\cdots,f_{1}) .$$

Then we associate a mixed complex $ (M(\mathcal{A}),b,B) $ with this precyclic chain complex as follows: The underlying dg module of $ M(\ca) $ is the mapping cone over $ (1 − t) $ viewed as a morphism of complexes
$$ 1-t\colon (C(\ca),b')\ra(C(\ca),b) ,$$ where $ b=\sum_{i=0}^{n}(-1)^{i}d_{i} $ and $ b'=\sum_{i=0}^{n-1}(-1)^{i}d_{i} $.
Its underlying module is $ C(\ca)\oplus C(\ca) $; it is endowed with the grading whose $ n $-th component is $ C(\ca)_{n}\oplus C(\ca)_{n-1} $ and the differential is 
$$ \begin{bmatrix} b & 1-t \\ 0 & -b' \end{bmatrix} .$$
The operator $ B\colon M(\ca)\ra M(\ca) $ is
$$ \begin{bmatrix} 0 & 0 \\ N & 0 \end{bmatrix} ,$$ where $ N=\sum_{i=0}^{n}t^{i} $.

\begin{Def}\rm
	The \emph{cyclic homology} $ H\!C_{\bullet}(\ca) $ of $ \mathcal{A} $ is defined to be the homology of the \emph{cyclic chain complex} of $ \ca $
	$$ H\!C(\ca)=M(\ca)\lten_{\Lambda}k .$$
	
	The \emph{negative cyclic homology} $ H\!N_{\bullet}(\ca) $ of $ \mathcal{A} $ is defined to be the homology of the \emph{negative cyclic chain complex} of $ \ca $
	$$ H\!N(\ca)=\RHom_{\Lambda}(k,M(\ca)). $$
\end{Def}
\begin{Rem}
	The dg algebra $ \Lambda $ is the singular homology with coefficients in $ k $ of the circle $ S^{1} $. The circle action is captured algebraically in terms of the structure of a mixed complex so that the above constructions can be explained as homotopy orbit and homotopy fixed points of the Hochschild complex $ C_{\bullet}(\mathcal{A}) $ with the algebraic circle action (see~\cite{CK1987,JLLoday_book,MH2018}).
\end{Rem}

The augmentation morphism $ \Lambda\to k $ induces natural morphisms in $ \cd(k) $
$$ H\!N(\ca)\rightarrow H\!H(\ca)\rightarrow H\!C(\ca).$$

Let $ G\colon\mathcal{B}\to \mathcal{A} $ be a dg functor. It induces a canonical morphism $ \gamma_{G}\colon M(\mathcal{B})\to M(\mathcal{A}) $ between their mixed complexes. We denote by $ M(G) $ the mapping cone over $ \gamma_{G} $.
\begin{Def}\rm
	The \emph{cyclic homology} $ H\!C_{\bullet}(G) $ of $ G\colon\cb\rightarrow\ca $ is defined to be the homology of the \emph{cyclic chain complex group} of $ G $
	$$ H\!C(G)=M(G)\lten_{\Lambda}k. $$
	The \emph{negative cyclic homology} $ H\!N_{\bullet}(G) $ of $ G\colon\cb\rightarrow\ca $ is defined to be the homology of the \emph{negative cyclic chain complex} of $ G $
	$$ H\!N(G)=\RHom_{\Lambda}(k,M(G)). $$
\end{Def}

Similarly, the augmentation morphism $ \Lambda\to k $ induces natural morphisms in $ \cd(k) $
$$ H\!N(G)\rightarrow H\!H(G)\rightarrow H\!C(G).$$

\begin{Thm}\cite[Theorem 1.5]{BK1999}\label{Thm: Morita invariance of mixed complex}
	Let $ \ca $ and $ \cb $ be dg categories. Let $ G\colon\cb\rightarrow\ca $ be a \emph{Morita functor}, i.e.\ a dg functor such that $ G_{*}\colon\cd(\cb)\rightarrow\cd(\ca) $ is an equivalence. Then $ \gamma_{G}\colon M(\cb)\rightarrow M(\ca) $ is an isomorphism in $ \cd(\Lambda) $.
\end{Thm}

\section{Relative Calabi--Yau structures}\label{section RCY}

\subsection{Reminder on the derived category of morphisms} Let $I$ be the path $k$-category of
the quiver $1 \to 2$. The letter $I$ stands for `interval'. Let $\ca$ be a dg $k$-category. The objects of the derived category
$\cd(I^{op}\ten \ca)$ identify with morphisms $f\colon M_1 \to M_2$ of dg $\ca$-modules. Each such object gives
rise to a triangle
\[
\xymatrix{
	M_1 \ar[r]^{f} & M_2 \ar[r] & \cof(f) \ar[r] & \Sigma M_1
}
\]
of $\cd\ca$ which is {\em functorial} in the object $f$ of $\cd(I^{op}\ten\ca)$. Here, we write $ \cof $ for the homotopy cofiber, i.e.\ the cone of a morphism.

For two objects $f\colon M_1 \to M_2$ and $f'\colon M_1'\to M_2'$, consider a morphism of triangles
\[
\xymatrix{
	M_1 \ar[d]_a \ar[r]^f & M_2 \ar[d]^b \ar[r] & \cof(f) \ar[d]^c \ar[r] & \Sigma M_1 \ar[d]^{\Sigma a}\\
	M_1' \ar[r]^{f'} & M'_2 \ar[r]^-{g'} & \cof(f') \ar[r] & \Sigma M_1'}
\]
in the derived category $\cd\ca$. It is well-known that a given morphism $b\colon M_2 \to M_2'$
extends to such a morphism of triangles $(a,b,c)$ if and only if we have $g' \circ b \circ f=0$ and
that in this case, the pair $(a,b)$ lifts to a morphism of $\cd(I^{op}\ten \ca)$. The following easy
lemma makes this more precise. Here, we write $\fib$ for the homotopy fiber, i.e.\ the
desuspension of the cone of a morphism.

\begin{Lem}\label{lemma:morphisms in I*A}
	We have a canonical isomorphism bifunctorial in the objects
	$f$ and $f'$ of $\cd(I^{op}\ten \ca)$
	\[
	\RHom_{I^{op}\ten \ca}(f,f') \iso \fib(\RHom_\ca(M_2, M_2') \to \RHom_\ca(M_1, \cof(f')).
	\]
	More precisely, let $ g\colon N_{1}\rightarrow N_{2} $ be an object in $ \cd(I^{op}\ten\ca) $ and $ \beta\colon f'\rightarrow g $ a morphism in $ \cd(I^{op}\ten\ca) $, we have the following commutative diagram
	\begin{align*}
		\xymatrix{
			\RHom_{I^{op}\ten \ca}(f,f')\ar[r]^-{\sim}\ar[d]^{\beta_{*}}& \fib(\RHom_\ca(M_2, M_2')\to \RHom_\ca(M_1, \cof(f'))\ar[d]^{\beta_{*}}\\
			\RHom_{I^{op}\ten \ca}(f,g)\ar[r]^-{\sim}& \fib(\RHom_\ca(M_2, N_2)\to \RHom_\ca(M_1, \cof(g)).
		}
	\end{align*}
\end{Lem}
\begin{proof}
	We have isomorphisms of dg categories
	\begin{equation*}
		\begin{split}
			\cc_{dg}(I^{op}\ten\ca)=&\ch om(I\ten\ca^{op},\cc_{dg}(k))\\
			\simeq&\ch om(I,\ch om(\ca^{op},\cc_{dg}(k)))\\
			=&\ch om(I,\cc_{dg}(\ca)).
		\end{split}
	\end{equation*}	
	In this way, $ \cc_{dg}(\ca) $ identifies with the category of morphisms $ M_{1}\rightarrow M_{2} $ of dg $ \ca $-modules with the dg enhancement given by
	\begin{align*}
		\xymatrix{
			\ar @{} [dr] | \corner \ch om(f,f') \ar[r] \ar[d] &\ch om_{\ca}(M_{2},M'_{2}) \ar[d] \\
			\ch om_{\ca}(M_{1},M'_{1}) \ar[r] & \ch om_{\ca}(M_{1},M'_{2}).
		}
	\end{align*}
	
	The model structure on $ \cc(I^{op}\ten\ca) $ translates into a model structure on $ \ch om(I,\cc_{dg}(\ca)) $ whose weak equivalences are the componentwise quasi-isomorphisms and whose cofibrant objects are the graded split monomorphisms $ M_{1}\rightarrowtail M_{2} $ with cofibrant $ M_{1} $ and $ M_{2} $. Therefore, we can assume that $ f $ and $ f' $ are graded split injective morphisms between cofibrant dg $ \ca $-modules. Then we have the following commutative diagram in $ \cc(\ca) $
	\[
	\begin{tikzcd}
		M_{1}\arrow[r,rightarrowtail]\arrow[d]&M_{2}\arrow[r,twoheadrightarrow]\arrow[d]&\coker(f)\arrow[d]\\
		M'_{1}\arrow[r,rightarrowtail]&M'_{2}\arrow[r,twoheadrightarrow]&\coker(f'),
	\end{tikzcd}
	\]
	where the first row and second row are graded split exact sequences. It induces the the following commutative diagram	of complexes
	\[
	\begin{tikzcd}
		\ch om(\coker(f),M'_{2})\arrow[r,rightarrowtail]\arrow[d,equal]&\ch om(f,f') \arrow[r,twoheadrightarrow] \arrow[d,rightarrowtail]
		\arrow[dr, phantom, "\square"]
		&\ch om(M_{1},M'_{1})\arrow[d,rightarrowtail] \\
		\ch om(\coker(f),M'_{2})\arrow[r,rightarrowtail]&\ch om(M_{2},M'_{2})\arrow[r, twoheadrightarrow]\arrow[d,two heads]
		&\ch om(M_{1},M'_{2})\arrow[d,two heads]\\
		&\ch om(M_{1},\coker(f'))\arrow[r,equal]&\ch om(M_{1},\coker(f')),
	\end{tikzcd}
	\]
	where the upper right square is a bicartesian. Thus we have the exact sequence
	\[
	\begin{tikzcd}
		\ch om(f,f')\arrow[r,hook]&\ch om(M_{2},M'_{2})\arrow[r,two heads]&\ch om(M_{1},\coker(f'))
	\end{tikzcd}
	\]
	and the canonical isomorphism 
	\[
	\RHom_{I^{op}\ten \ca}(f,f') \iso \fib(\RHom_\ca(M_2, M_2') \to \RHom_\ca(M_1, \cof(f')).
	\]
	
	%
	
\end{proof}

Let $ G\colon\cb\rightarrow\ca $ be a dg functor. It induces the dg functor $$ \id\ten G\colon I^{op}\ten\cb\longrightarrow I^{op}\ten\ca ,$$ which we still denote by $ G $. It yields the adjunction
$$ \lG\colon\cd(I^{op}\ten\cb)\leftrightarrows\cd(I^{op}\ten\ca)\colon G_{*}. $$

\begin{Lem}\label{Lemma:functor property of morphism of path cat}
	Let $ f\colon M_{1}\rightarrow M_{2} $ and $ f'\colon M'_{1}\rightarrow M'_{2} $ be objects in $ \cd(I^{op}\ten\cb) $. We have the following commutative diagram
	\[
	\begin{tikzcd}
		\RHom_{I^{op}\ten\cb}(f,f')\arrow[r,"\sim"]\arrow[d,"\lG"]&\fib(\RHom_{\cb}(M_{2},M'_{2})\rightarrow\RHom_{\ca}(M_{1},\cof(f'))\arrow[d,"\lG"]\\
		\RHom_{I^{op}\ten\ca}(\lG(f),\lG(f'))\arrow[r,"\sim"]&\fib(\RHom_{\ca}(\lG(M_{2}),\lG(M'_{2}))\rightarrow\RHom_{\ca}(\lG(M_{1}),\cof(\lG(f')).
	\end{tikzcd}
	\]
	Similarly, let $ g\colon N_{1}\rightarrow N_{2} $ and $ g'\colon N'_{1}\rightarrow N'_{2} $ be objects in $ \cd(I^{op}\ten\ca) $. We have the following commutative diagram
	\[
	\begin{tikzcd}
		\RHom_{I^{op}\ten\ca}(g,g')\arrow[r,"\sim"]\arrow[d,"G_{*}"]&\fib(\RHom_{\cb}(N_{2},N'_{2})\rightarrow\RHom_{\ca}(N_{1},\cof(g'))\arrow[d,"G_{*}"]\\
		\RHom_{I^{op}\ten\cb}(G_{*}(g),G_{*}(g'))\arrow[r,"\sim"]&\fib(\RHom_{\ca}(G_{*}(N_{2}),G_{*}(N'_{2}))\rightarrow\RHom_{\ca}(G_{*}(N_{1}),\cof(G_{*}(g')).
	\end{tikzcd}
	\]
\end{Lem}
\begin{proof}
	We only show the first statement since the second one can be shown similarly. We can assume that $ f,f' $ are graded split injective morphisms between cofibrant dg $ \ca $-modules. Then it is easy to see that the following diagram commutes
	\[
	\begin{tikzcd}
		\ch om_{I^{op}\ten\cb}(f,f')\arrow[r,rightarrowtail]\arrow[d]&\ch om_{\cb}(M_{2},M'_{2})\arrow[r,twoheadrightarrow]\arrow[d]&\ch om_{\cb}(M_{1},\coker(f'))\arrow[d]\\
		\ch om_{I^{op}\ten\ca}(G^{*}(f),G^{*}(f'))\arrow[r,rightarrowtail]&\ch om_{\ca}(G^{*}(M_{2}),G^{*}(M'_{2}))\arrow[r,twoheadrightarrow]&\ch om_{\ca}(G^{*}(M_{1}),\coker(G^{*}(f'))).
	\end{tikzcd}
	\]
	Thus we get the first commutative diagram.
	
\end{proof}

\bigskip

Relative right Calabi--Yau structures were invented by Bertrand To\"{e}n in ~\cite[pp. 227-228]{BT2014}. 
Later, the theory of relative right and left Calabi--Yau structures was developed by Chris Brav and Tobias Dyckerhoff in \cite{BD2019}. 
\subsection{Relative right Calabi--Yau structures}
Let $ G\colon\cb\rightarrow\ca $ be a dg functor 
\footnote{The definition we will give actually makes sense even if we do not assume $ \ca $ and $ \cb $ to be proper.}. We denote by $ D\ca^{op} $ the dg $ \ca $-bimodule defined as follows:
$$ D\ca^{op}(X,Y)=D\ca(Y,X),\quad\forall(X,Y)\in\ca^{e}, $$ where $ D $ is the $ k $-linear dual $ \Hom_{k}(?,k) $. We call it the \emph{linear dual bimodule} of $ \ca $.
Similarly, we define the dg $ \cb $-bimodule $ D\cb^{op} $. The natural $ \cb $-bimodule morphism $ u_{G}\colon\cb\rightarrow G_{*}\ca $ induces a morphism between their linear dual bimodules
$$ G_{*}(D\ca^{op})\rightarrow D\cb^{op}. $$
It canonically lifts to an object $ u_{G}^{*} $ of $\cd(I^{op}\ten\cb^e)$. Similarly, its homotopy fiber 
$$ \fib(u_{G}^{*})\rightarrow G_{*}(D\ca^{op}) $$ lifts to an object $ \delta_{G} $ of $\cd(I^{op}\ten\cb^e)$. Each morphism $ \Si^{n-1}(u_{G})\rightarrow \delta_{G} $ gives rise to a morphism of triangles in $\cd(\cb^e)$
\begin{equation}\label{eq:right morphism of triangles}
	\xymatrix@=1.3cm{
		\Si^{n-1}\cb\ar[r]^-{\Si^{n-1}u_{G}}\ar[r]\ar[d]&\Si^{n-1}G_{*}\ca\ar[r]\ar[d]&\cof(\Si^{n-1}u_{G})\ar[r]\ar[d]&\Si^{n}\cb\ar[d]\\
		\fib(u_{G}^{*})\ar[r]^-{\delta_{G}}&G_{*}(D\ca^{op})\ar[r]^{u_{G}^{*}}&D\cb^{op}\ar[r]&\Si\fib(u_{G}^{*})
	}	
\end{equation}
We are therefore interested in morphisms $\Si^{n-1}u_{G}\to \delta_{G}$ in $\cd(I^{op}\ten\cb^e)$. 

\begin{Lem}\label{Lemma: right CY iso map}
	We have a canonical isomorphism
	$$ \RHom_{I^{op}\ten\cb^{e}}(u_{G},\delta_{G})\iso\fib(\Hom_{k}(\ca\lten_{\cb^{
			e}}\ca,k)\rightarrow \Hom_{k}(\cb\lten_{\cb^{e}}\cb,k)). $$
	Moreover this isomorphism is compatible with the composition of dg functors, i.e.\ if $ Q\colon\ca\rightarrow\cc $ is another dg functor, then we have the following commutative diagram
	\begin{align*}
		\xymatrix{
			\RHom_{I^{op}\ten\cb^{e}}(u_{Q\circ G},\delta_{Q\circ G})\ar[r]^-{\sim}\ar[d]&\fib(\Hom_{k}(\cc\lten_{\cb^{e}}\cc,k)\rightarrow \Hom_{k}(\cb\lten_{\cb^{e}}\cb,k))\ar[d]\\
			\RHom_{I^{op}\ten\cb^{e}}(u_{G},\delta_{G})\ar[r]^-{\sim}&\fib(\Hom_{k}(\ca\lten_{\cb^{e}}\ca,k)\rightarrow \Hom_{k}(\cb\lten_{\cb^{e}}\cb,k))
		}
	\end{align*}
	
\end{Lem}
\begin{proof}
	Using the standard adjunctions, we get 
	\begin{equation*}
		\xymatrix{
			\RHom_{\cb^{e}}(\cb,D\cb^{op})\simeq\Hom_{k}(\cb\lten_{\cb^{e}}\cb,k)
		}
	\end{equation*}
	and
	\begin{equation*}
		\begin{split}
			\RHom_{\cb^{e}}(G_{*}\ca,G_{*}(D\ca^{op}))\simeq&\RHom_{\ca^{e}}(\lG(G_{*}\ca),D(\ca^{op}))\\
			\simeq&\Hom_{k}(G_{*}(\ca)\lten_{\cb^{e}}\ca^{e},k)\\
			\simeq&\Hom_{k}(\ca\lten_{\cb^{e}}\ca,k).
		\end{split}
	\end{equation*}
	
	We know that the composition $ Q\circ G\colon\cb\rightarrow\ca\rightarrow\cc $ induces the following morphisms in $ \cd(I^{op}\ten\cb^{e}) $
	$$ u_{G}\rightarrow u_{Q\circ G},\quad\delta_{Q\circ G}\rightarrow\delta_{G} .$$
	
	Then the claim follows by Lemma~\ref{lemma:morphisms in I*A}.

\end{proof}

We therefore obtain the following chain of morphisms
\begin{align}\label{Diagram:right CY}
	\xymatrix{
		\Hom(H\!C(G),k)\ar[d]&\\
		\Hom(H\!H(G),k)\cong\fib(\Hom(\ca\lten_{\ca^e}\ca,k)\to\Hom(\cb\lten_{\cb^e} \cb,k))\ar[d]\\
		\fib(\Hom(\ca\lten_{\cb^e}\ca,k)\to \Hom(\cb\lten_{\cb^e}\cb,k))\ar[r]^-{\sim}&\RHom_{I^{op}\ten\cb^e}(u_{G},\delta_{G})
	}
\end{align}

\begin{Def}\cite[Definition 4.7]{BD2019}\rm\label{Def: relative right CY}
	A \emph{right $n$-Calabi--Yau structure} on the dg functor $ G\colon\mathcal{B}\to \mathcal{A} $ 
	is a class $ [\omega] $ in $ \Hom(H\!C_{n-1}(G),k) $ such that the associated morphism $\Sigma^{n-1}u_{G} \to \delta_{G}$ is invertible, i.e.\ its associated morphism of triangles (\ref{eq:right morphism of triangles}) is invertible.
\end{Def}

\subsection{Relative left Calabi--Yau structures} 
Let $G\colon \cb \to \ca$ be a dg functor. We assume that $ \cb $ is smooth. This ensures that the canonical morphism $$ \ca\lten_{\cb}\cb^{\vee}\lten_{\cb}\ca\ra(\ca\lten_{\cb}\cb\lten_{\cb}\ca)^{\vee} $$ is invertible in $ \cd(\ca^{e}) $. The composition of $\ca$ induces the morphism
\[
\ca\lten_\cb\ca \to \ca
\]
of $\cd(\ca^e)$. It canonically lifts to an object $\mu_{G}$ of $\cd(I^{op}\ten\ca^e)$. Similarly, its homotopy fiber
\[
\fib(\mu_{G}) \to \ca\lten_\cb\ca
\]
lifts to an object $\nu_{G}$ of $\cd(I^{op}\ten\ca^e)$. Notice that each morphism $ \Si^{n-1}\mu_{G}^\vee \to \nu_{G} $
gives rise to a morphism of triangles in $\cd(\ca^e)$

\begin{equation}\label{eq:left morphism of triangles}
	\xymatrix@=1.5cm{
		\Si^{n-1}\ca^\vee \ar[r]^-{\Si^{n-1}\mu_{G}^\vee} \ar[d] & 
		\Si^{n-1}(\ca \lten_\cb \ca)^\vee \ar[r] \ar[d] & \Si^{n-1} \cof(\mu_{G}^\vee)\ar[d] \ar[r]  & \Si^n \ca^{\vee}\ar[d]\\
		\fib(\mu_{G}) \ar[r]^-{\nu_{G}} &  \ca\lten_\cb\ca \ar[r]^-{\mu_{G}} & \ca \ar[r] & \cof(\mu_{G}).}
\end{equation}
We are therefore interested in morphisms $\Si^{n-1}\mu_{G}^\vee \to \nu_{G}$ in $\cd(I^{op}\ten\ca^e)$. 
\begin{Lem}\label{Lemma:morphism in left CY}
	We have a canonical morphism
	\[
	\fib(\ca\lten_{\cb^e} \ca \to \ca\lten_{\ca^e} \ca)\ra\RHom_{I^{op}\ten\ca^e}(\mu_{G}^\vee,\nu_{G}).
	\] It is invertible if $ \ca $ is smooth. Moreover this canonical morphism is compatible with the composition of dg functors, i.e.\ if $ Q\colon\ca\rightarrow\cc $ is another dg functor between smooth dg categories, then we have the following commutative diagram 
	\[
	\begin{tikzcd}
		\fib(\ca\lten_{\cb^e}\ca \to \ca\lten_{\ca^e} \ca)
		\arrow[r]\arrow[d]&\RHom_{I^{op}\ten\ca^e}(\mu_{G}^\vee,\nu_{G})\arrow[d]\\
		\fib(\cc\lten_{\cb^e}\cc\to \cc\lten_{\cc^e} \cc)
		\arrow[r]&\RHom_{I^{op}\ten\cc^e}(\mu_{Q\circ G}^\vee,\nu_{Q\circ G}).
	\end{tikzcd}
	\]

\end{Lem}
\begin{proof}
	By Lemma~\ref{lemma:morphisms in I*A}, we have
	\[
	\RHom_{I^{op}\ten\ca^e}(\mu_{G}^\vee,\nu_{G})\iso\fib(\RHom_{\ca^{e}}((\ca\lten_{\cb}\ca)^{\vee},\ca\lten_{\cb}\ca)\to\RHom_{\ca^{e}}(\ca^{\vee},\ca)).
	\]
	We have a canonical morphism
	$$ \ca\lten_{\ca^{e}}\ca\ra\RHom_{\ca^{e}}(\ca^{\vee},\ca), $$ which is invertible if $ \ca $ is smooth. Moreover, we have the isomorphisms
	\begin{equation*}
		\begin{split}
			\RHom_{\ca^{e}}((\ca\lten_{\cb}\ca)^{\vee},\ca\lten_{\cb}\ca)\simeq&\,\RHom_{\ca^{e}}(\lG(\cb^{\vee}),\ca\lten_{\cb}\ca)\\
			\simeq&\,\RHom_{\cb^{e}}(\cb^{\vee},G_{*}(\ca\lten_{\cb}\ca))\\
			\simeq&\,\cb\lten_{\cb^{e}}(\ca\lten_{\cb}\ca)\\
			\simeq&\,\ca\lten_{\cb^{e}}\ca,
		\end{split} 
	\end{equation*}
	where we use the smoothness of $ \cb $ for the first and the $ \boldmath{3} $rd isomorphism.
	Thus, we have a canonical morphism
	\[
	\fib(\ca\lten_{\cb^e} \ca \to \ca\lten_{\ca^e} \ca)\ra\RHom_{I^{op}\ten\ca^e}(\mu_{G}^\we,\nu_{G}),
	\]which is invertible if $ \ca $ is smooth.
	
	By Lemma~\ref{Lemma:functor property of morphism of path cat}, we have the following commutative diagram
	\[
	\begin{tikzcd}
		\fib(\ca\lten_{\cb^e}\ca \to \ca\lten_{\ca^e} \ca)
		\arrow[r]\arrow[d]&\RHom_{I^{op}\ten\ca^e}(\mu_{G}^\vee,\nu_{G})\arrow[d]\\
		\fib(\lQ(\ca\lten_{\cb^e}\ca)\to\lQ(\ca\lten_{\ca^e}\ca))
		\arrow[r]&\RHom_{I^{op}\ten\cc^e}(\lQ(\mu_{G}^\vee),\lQ(\nu_{G})),
	\end{tikzcd}
	\]
	where $ \lQ(\mu^{\vee}_{G}) $ is given by $ \ca^{\vee}\lten_{\ca^{e}}\cc^{e}\ra\cb^{\vee}\lten_{\cb^{e}}\cc^{e} $ and $ \lQ(\nu_{G}) $ is given by $ \lG(\fib(\mu_{G}))\ra\cb\lten_{\cb^{e}}\cc^{e} $.
	
	It is easy to see that we have natural morphisms $ \mu^{\vee}_{Q\circ G}\ra\lQ(\mu^{\vee}_{G}) $ and $ \lQ(\nu_{G})\ra\nu_{Q\circ G} $ in $ \cd(I^{op}\ten\cc) $. Then by Lemma~\ref{lemma:morphisms in I*A}, we get the following commutative diagrams
	\[
	\begin{tikzcd}
		\fib(\ca\lten_{\cb^e}\ca \to \ca\lten_{\ca^e} \ca)
		\arrow[r]\arrow[d]&\RHom_{I^{op}\ten\ca^e}(\mu_{G}^\vee,\nu_{G})\arrow[d]\\
		\fib(\lQ(\ca\lten_{\cb^e}\ca)\to\lQ(\ca\lten_{\ca^e}\ca))
		\arrow[r]\arrow[d]&\RHom_{I^{op}\ten\cc^e}(\lQ(\mu_{G}^\vee),\lQ(\nu_{G}))\arrow[d]\\
		\fib(\cc\lten_{\cb^{e}}\cc\to\cc\lten_{\ca^{e}}\ca)
		\arrow[r]\arrow[d]&\RHom_{I^{op}\ten\cc^e}(\mu^{\vee}_{Q\circ G},\lQ(\nu_{G}))\arrow[d]\\
		\fib(\cc\lten_{\cb^{e}}\cc\to\cc\lten_{\cc^{e}}\cc)
		\arrow[r]&\RHom_{I^{op}\ten\cc^e}(\mu^{\vee}_{Q\circ G},\nu_{Q\circ G}).
	\end{tikzcd}
	\]	

\end{proof}

We therefore obtain the following chain of morphisms
\begin{align}\label{Diagram: left CY}
	\xymatrix{
		H\!N(G)\ar[d]&\\
		H\!H(G)\cong\Si\fib(\cb\lten_{\cb^e} \cb \to \ca\lten_{\ca^e} \ca)\ar[d]\\
		\Si\fib(\ca\lten_{\cb^e}\ca\to \ca\lten_{\ca^e}\ca)\ar[r]&\Si\RHom_{I^{op}\ten\ca^e}(\mu_{G}^\vee,\nu_{G}).
	}
\end{align}

\begin{Def}\cite[Definition 4.11]{BD2019}\cite[Definition 4.13]{WkY2016}\label{Def Relative CY}\rm
	$ \, $A \emph{left $n$-Calabi--Yau structure} on the dg functor $ G\colon\mathcal{B}\to \mathcal{A} $ 
	is a relative negative cyclic class $ [\xi] $ in $  H\!N_{n}(G) $ such that
	\begin{itemize}
		\item[a)] the associated morphism $\Sigma^{n-1}\mu_{G}^\vee \to \nu_{G}$ is invertible and
		\item[b)] the morphism $\Si^{n-1}\cb^\vee \to \cb$ corresponding to the image of $[\xi]$ in
		$H\!H_{n-1}(\cb)$ is invertible in $\cd(\cb^e)$. 
	\end{itemize}
\end{Def}

Notice that the morphism $\mu_{G}^\vee \to \Sigma^{n-1}\nu_{G}$ is invertible if and only if
its associated morphism of triangles (\ref{eq:left morphism of triangles}) is invertible.
We point out that condition b) is not imposed by Brav--Dyckerhoff \cite{BD2019} but
is imposed by Yeung \cite{WkY2016}.

\begin{Rem}
	If we take the dg category $ \mathcal{B} $ to be the empty dg category $ \emptyset $, which is the initial object in the category of small dg categories $ \dgcat $, then the above definition coincides with the definition of an \emph{absolute left n-Calabi--Yau structure} on $ \mathcal{A} $~\cite{BK2011}. 
\end{Rem}

\begin{Prop}\cite[Corollary 7.1]{BD2019}\label{Relative structure of htp cofiber}
	Let $ f\colon\mathcal{B}\to\mathcal{A} $ be a dg functor between homologically smooth dg categories which carries a left $ n $-Calabi--Yau structure. Then there is a canonical left $ n $-Calabi--Yau structure on the cofiber $ \mathcal{A}/\mathcal{B} $.
\end{Prop}

\begin{Prop}\label{quasi-invariant}
	Let $ \mathcal{B}, \mathcal{A},\ca' $ be smooth dg categories. Let $ G\colon\mathcal{B}\to\mathcal{A} $ be a dg functor and $Q\colon\mathcal{A}\to \mathcal{A}' $ be a quasi-equivalence. The isomorphism 
	$$ H\!N_{n}(G)\ra H\!N_{n}(Q\circ G) $$ induced by $ Q $ yields a bijection between the left $ n $-Calabi--Yau structures on $ G $ and on $ Q\circ G $.
\end{Prop}
\begin{proof}
	By Theorem~\ref{Thm: Morita invariance of mixed complex}, the functor $ Q $ induces the following quasi-isomorphism of triangles in $ \cd(\cm ix) $
	\begin{align*}
		\xymatrix{
			M(\mathcal{B})\ar@{=}[d]\ar[r]&M(\mathcal{A})\ar[r]\ar[d]&M(G)\ar[r]\ar[d]&\Si M(\mathcal{B})\ar@{=}[d]\\
			M(\mathcal{B})\ar[r]&M(\mathcal{C})\ar[r]&M(Q\circ G)\ar[r]&\Si M(\mathcal{B}).
		}
	\end{align*}
	Combining with Lemma~\ref{Lemma:morphism in left CY}, the above diagram yields the following commutative diagrams in $ \cd(k) $
	\begin{align*}
		\xymatrix{
			H\!N(G)\ar[r]^{\sim}\ar[d]&H\!N(Q\circ G)\ar[d]\\
			H\!H(G)\ar[r]^{\sim}\ar[d]&H\!H(Q\circ G)\ar[d]\\
			\Si\fib(\ca\lten_{\cb^{e}}\ca\ra\ca\lten_{\ca^{e}}\ca)\ar[r]^{\sim}\ar[d]^{\simeq}&\Si\fib(\ca'\lten_{\cb^{e}}\ca'\ra\ca'\lten_{\ca'^{e}}\ca')\ar[d]_{\simeq}\\
			\Si\RHom_{I^{op}\ten\ca^e}(\mu_{G}^\vee,\nu_{G})\ar[r]^-{\Theta}&\Si\RHom_{I^{op}\ten\ca'^e}(\mu_{Q\circ G}^\vee,\nu_{Q\circ G}).
		}
	\end{align*}
	The map $ \Theta $ admits the following description. The quasi-equivalence $ Q $ induces a quasi-equivalence
	$$ \id\ten Q^{e}\colon I^{op}\ten\ca^{e}\ra I^{op}\ten\ca'^{e} ,$$ which we still denote by $ Q $. Then the extension along $ Q $ yields an equivalence
	$$ \lQ\colon\cd(I^{op}\ten\ca^{e})\iso\cd(I^{op}\ten\ca^{e}). $$
	The functor $ \lQ $ maps $ \mu_{G}^\vee $ to $ \mu_{Q\circ G}^\vee $ and $ \nu_{G} $ to $ \nu_{Q\circ G} $. Then the map $ \Theta $ is the map induced by $ \lQ $ on mapping complexes. In particular, $ \Theta $ preserves equivalences.
	Thus each left $ n $-Calabi--Yau structure on $ G $ induces a left $ n $-Calabi--Yau structure on $ Q\circ G $.
	Similarly, we can use the restriction functor $ Q_{*}\colon\cd(I^{op}\ten\ca'^{e})\iso\cd(I^{op}\ten\ca'^{e}) $ to show that each left $ n $-Calabi--Yau structure on $ Q\circ G $ induces a left $ n $-Calabi--Yau structure on $ G $.
	%
\end{proof}

\begin{Cor}\label{homotopy-invariant}
	Let $ \mathcal{B}, \mathcal{A} $ be two homologically smooth dg categories and moreover assume $ \mathcal{B} $ is cofibrant with respect to the Dwyer-Kan model structure (see Theorem~\ref{model on dgcat}).
	Let $ G,G'\colon\mathcal{B}\to \mathcal{A} $ be two homotopic dg functors. The canonical isomorphism
	$$ H\!N_{n}(G)\iso H\!N_{n}(G') $$
	induces a bijection between the relative left $ n $-Calabi--Yau structures on $ G $ and on $ G' $.	
\end{Cor}
\begin{proof}
	Since $ G $ and $ G' $ are homotopic, there exists a dg functor $ H\colon\mathcal{B}\to P(\mathcal{A}) $ that makes the following diagram commutative (see Definition~\ref{homotopy def})
	\begin{align*}
		\xymatrix{
			&&\mathcal{A}\\
			\mathcal{B}\ar"1,3"^{G}\ar"3,3"_{G'}\ar"2,3"^{H}&&P(\mathcal{A})\ar[u]_{P_{0}}\ar[d]^{P_{1}}\\
			&&\mathcal{A}.
		}
	\end{align*}
	
	We know that $ P_{0} $ and $ P_{1} $ are quasi-equivalences. They induce isomorphisms $ H\!N_{n}(G)\liso H\!N_{n}(H)\iso H\!N_{n}(G') $. Now the claim follows from the above Proposition~\ref{quasi-invariant}.
\end{proof}

\subsection{From left to right}
Let $ G\colon\cb\rightarrow\ca $ be a dg functor between smooth dg categories. Suppose that $ G $ carries a left $ n $-Calabi--Yau structure. We define $ \per_{dg}(\ca) $ to be the dg subcategory of $ \cc_{dg}(\ca) $ whose objects are the perfect cofibrant dg $ \ca $ modules and $ \pvd_{dg}(\ca) $ to be the dg subcategory of $ \cc_{dg}(\ca) $ whose objects are the perfectly valued cofibrant dg $ \ca $ modules. Similarly, we define $ \per_{dg}(\cb) $ and $ \pvd_{dg}(\cb) $. The restriction along $ G\colon\cb\rightarrow\ca $ induces a dg functor $ R\colon\ce=\pvd_{dg}(\ca)\rightarrow\cf=\pvd_{dg}(\cb) $.
\begin{Thm}\cite[pp. 389]{BD2019}
	The functor $ R\colon\ce\rightarrow\cf $	inherits a canonical right $ n $-Calabi--Yau structure, i.e.\ we have a class $ [\omega] $ in $ \Hom_{k}(H\!C_{n-1}(R),k) $ which yields an isomorphism of triangles in $ \cd(\ce^{e}) $
	\begin{equation}\label{Thm:right iso morphism}
		\xymatrix@=1.5cm{
			\Si^{n-1}\ce\ar[r]^-{\Si^{n-1}u_{R}}\ar[r]\ar[d]&\Si^{n-1}R_{*}\cf\ar[r]\ar[d]&\Si^{n-1}\cof(u_{R})\ar[r]\ar[d]&\Si^{n}\ce\ar[d]\\
			\fib(u_{R}^{*})\ar[r]^-{\delta_{R}}&R_{*}(D\cf^{op})\ar[r]^{u_{R}^{*}}&D\ce^{op}\ar[r]&\Si\fib(u_{R}^{*}),
		}	
	\end{equation}
	where $ R_{*} $ is the restriction along $ R^{e}\colon\ce^{e}\rightarrow\cf^{e} $.
\end{Thm}
\begin{proof}
	By the definition of $ \pvd_{dg}\cb $, we have a dg functor
	$$ (\per_{dg}\cb)^{op}\ten_{k}\cf\rightarrow\per_{dg}(k),\quad(P,M)\longmapsto\ch om_{\cb}(P,M). $$
	It yields a morphism in $ \cd\cm ix $
	$$ M((\per_{dg}\cb)^{op})\ten_{k} M(\cf)\iso M((\per_{dg}\cb)^{op}\ten_{k}\cf)\rightarrow M(\per_{dg}(k))\liso M(k)\simeq k, $$
	where the first quasi-isomorphism of $ \Lambda $-modules is due to \cite[Theorem 2.4]{CK1987}.

	By the adjunction between $ ?\ten_{k}M(\cf) $ and $ \Hom_{k}(M(\cf),?) $, we get a morphism in $ \cd\cm ix $
	$$ M(\cb)\iso M(\per_{dg}\cb)\iso M((\per_{dg}\cb)^{op})\rightarrow\Hom_{k}(M(\cf),k). $$
	
	Similarly, we get another morphism in $ \cd\cm ix $
	$$ M(\ca)\iso M(\per_{dg}\ca)\rightarrow\Hom_{k}(M(\ce),k). $$
	
	Those two maps fit into the following commutative diagram in $ \cd\cm ix $
	\begin{align*}
		\xymatrix{
			M(\cb)\ar[r]^-{\sim}\ar[d]^{\gamma_{G}}&M(\per_{dg}\cb) \ar[r]\ar[d]_{\gamma_{G^{*}}}&\Hom_{k}(M(\cf),k)\ar[d]^{\gamma_{R}^{*}}\\
			M(\ca)\ar[r]^-{\sim}& M(\per_{dg}\ca)\ar[r]&\Hom_{k}(M(\ce),k).
		}
	\end{align*}
	Applying the functor $ \RHom_{\Lambda}(k,?) $, it yields the following commutative diagram in $ \cc(k) $
	\begin{align*}
		\xymatrix{
			H\!N(\cb)=\RHom_{\Lambda}(k,M(\cb))\ar[r]\ar[d]^{\beta_{G}}&\RHom_{\Lambda}(k,\Hom_{k}(M(\cf),k))\simeq\Hom_{k}(H\!C(\cf),k)\ar[d]^{\beta'_{G}}\\
			H\!N(\ca)=\RHom_{\Lambda}(k,M(\ca))\ar[r]&\RHom_{\Lambda}(k,\Hom_{k}(M(\ce),k))\simeq\Hom_{k}(H\!C(\ce),k),
		}
	\end{align*}
where the isomorphisms on the right hand are due to the adjunction pair $ (k\lten_{\Lambda}?,\RHom_{k}(k,?)) $.

The above commutative square fits into the following commutative cube 
\[
\begin{tikzcd}[row sep=0.8em, column sep=0.01em]
	H\!N(\cb) \arrow[rr] \arrow[dr,"\beta_{G}"] \arrow[dd,swap] &&
	\RHom_{\Lambda}(k,\Hom_{k}(M(\cf),k)) \arrow[dd,dashed,"{\text{induced by $ \Lambda\ra k $}}"{yshift=8pt,description}] \arrow[dr,"\beta'_{G}"] \\
	& H\!N(\ca) \arrow[rr]\arrow[dd] &&
	\RHom_{\Lambda}(k,\Hom_{k}(M(\ce),k)) \arrow[dd] \\
	H\!H(\cb) \arrow[rr,dashed]\arrow[dr] && \RHom_{\Lambda}(\Lambda,\Hom_{k}(M(\cf),k))\simeq\Hom_{k}(\cf\lten_{\cf^{e}}\cf,k)\arrow[dr,dashed] \\
	&H\!H(\ca) \arrow[rr]&& \Hom_{k}(\ce\lten_{\ce^{e}}\ce,k).
\end{tikzcd}
\]
Moreover, the bottom of the cube above fits into the following commutative cube
\[
\begin{tikzcd}[row sep=1em, column sep = 1em]
	H\!H(\cb) \arrow[rr]\arrow[dr] \arrow[dd,swap] &&
	\Hom_{k}(\cf\lten_{\cf^{e}}\cf,k) \arrow[dd,dashed] \arrow[dr] \\
	&H\!H(\ca)\arrow[rr]\arrow[dd]&&
	\Hom_{k}(\ce\lten_{\ce^{e}}\ce,k) \arrow[dd] \\
	\ca\lten_{\cb^{e}}\ca\arrow[rr,dashed] \arrow[dr] && \Hom_{k}(\cf\lten_{\ce^{e}}\cf,k) \arrow[dr,dashed] \\
	& \ca\lten_{\ca^{e}}\ca \arrow[rr]&& \Hom_{k}(\ce\lten_{\ce^{e}}\ce,k). 
\end{tikzcd}
\]

Therefore we get the following commutative diagram in $ \cc(k) $
	\begin{align}\label{diagram: left to right}
		\xymatrix{
			H\!N(G)=\cone(\beta_{G})\ar[r]^-{\alpha}\ar[d]&\Hom_{k}(\Si^{-1}H\!C(R),k)=\cone(\beta'_{G})\ar[d]\\
			H\!H(G)\ar[r]\ar[d]&\Hom_{k}(\Si^{-1}H\!H(R),k)\ar[d]\\
			\cof(\ca\lten_{\cb^e}\ca\to \ca\lten_{\ca^e}\ca)\ar[d]^{\simeq\,\text{due to Lemma \ref{Lemma:morphism in left CY}}}\ar[r]&\cof(\Hom_{k}(\cf\lten_{\ce^{e}}\cf,k)\ra\Hom_{k}(\ce\lten_{\ce^{e}}\ce,k)\ar[d]^{\simeq\,\text{due to Lemma \ref{Lemma: right CY iso map}}}\\
			\Si\RHom_{I^{op}\ten\ca^{e}}(\mu_{G}^{\vee},\nu_{G})\ar[r]^{\Theta}&\RHom_{I^{op}\ten\ce^{e}}(\Si^{-1}u_{R},\delta_{R}).
		}
	\end{align}
The map $ \Theta $ is described below.

Consider the functor $ \Psi $ given by the composite
$$ \cc(\ca^{e})\ra\cc(\per_{dg}(\ca)^{e})\ra\cc(\ce\ten(\per_{dg}\ca)^{op})\ra\cc(\ce^{op}\ten\per_{dg}\ca)\ra\cc(\ce^{op}\ten\ce)\simeq\cc(\ce^{e})^{op}, $$ where the second and last functors are given by restriction along $ \ce\subseteq\per_{dg}\ca $, the first functor is given by the extension along Yoneda embedding and the third functor is given by
$$ M\mapsto M^{*},\,(a,p)\mapsto\RHom_{\ce}(M(?,p),\per_{dg}\ca(?,a)).$$
Then we obtain an induced functor 
\[
\begin{tikzcd}
	\mathrm{\mathbf{L}}\Psi\colon\cd(I^{op}\ten\ca^{e})\arrow[r]&\cd(I^{op}\ten\ce^{e})^{op}.
\end{tikzcd}
\]
Explicitly, this functor associates to a graded split monomorphism of $ \ca $-bimodules $ f\colon M_{1}\rightarrowtail M_{2} $ with cofibrant $ M_{1} $ and $ M_{2} $, the morphism of $ \ce $-bimodules given by
\begin{align*}
	\xymatrix{
		&&\RHom_{\ce}(M_{2}\ten_{\ca^{e}}\per_{dg}(\ca)^{e}(?,p'),\per_{dg}(\ca)(?,p))\ar"3,3"\\
		(p,p')\ar@{|->}"2,2"&\\
		&&\RHom_{\ce}(M_{1}\ten_{\ca^{e}}\per_{dg}(\ca)^{e}(?,p'),\per_{dg}(\ca)(?,p))
		}
\end{align*}
	
	Therefore, the functor $ \mathrm{\mathbf{L}}\Psi $ maps $ \mu_{G}\colon\ca\lten_{\cb}\ca\ra\ca $ to $ \ce\ra R_{*}(\cf) $ and $ \mu_{G}^{\vee}\colon\ca^{\vee}\ra(\ca\lten_{\cb}\ca)^{\vee} $ to $ R_{*}(D\cf^{op})\ra D\ce^{op} $. An explicit calculation shows that the map $ \Theta $ in diagram (\ref{diagram: left to right}) is the map induced by $ \mathrm{\mathbf{L}}\Psi $ on mapping complexes. 
	
	Suppose that the left $ n $-Calabi--Yau structure on $ G\colon\cb\ra\ca $ is induced by $ [\xi]\in H\!N_{n}(G) $. Then we have an isomorphism of triangles~(\ref{eq:left morphism of triangles}) in $ \cd(\ca^{e}) $. After applying the functor $ \mathrm{\mathbf{L}}\Psi $ to this diagram~(\ref{eq:left morphism of triangles}), we get an isomorphism of triangles~(\ref{Thm:right iso morphism}) in $ \cd(\ce^{e}) $ and this isomorphism is induced by the class $ \alpha([\xi]) $.
	
\end{proof}

\begin{Prop}\label{Relative Hom}
	Let $ G\colon\cb\ra\ca $ satisfy the above assumption. For $ L,M\in\mathcal{D}(\mathcal{A}) $, we put 
	$$ \mathcal{C}(L,M)=\cone(\RHom_{\mathcal{A}}(L,M))\rightarrow\RHom_{\mathcal{B}}(G_{*}(L),G_{*}(M)). $$ Suppose that $ L\in\mathrm{pvd}(\mathcal{A}) $ and $ M\in\mathcal{D}(\mathcal{A}) $. Then there is a bifunctorial isomorphism of triangles
	\begin{align}\label{Relative right from left}
		\xymatrix{
			D\mathcal{C}(L,M)\ar[r]&D\RHom_{\mathcal{B}}(G_{*}(L),G_{*}(M))\ar[r]^-{DG_{*}}&D\RHom_{\mathcal{A}}(L,M)\ar[r]&\Si D\mathcal{C}(L,M)\\
			\RHom_{\mathcal{A}}(M,\Si^{n-1}L)\ar[r]^-{G_{*}}\ar[u]_{\simeq}&\RHom_{\mathcal{B}}(G_{*}(M),\Si^{n-1}G_{*}(L))\ar[r]\ar[u]_{\simeq}&\mathcal{C}(M,\Si^{n-1}L)\ar[r]\ar[u]_{\simeq}&\RHom_{\mathcal{A}}(M,\Si^{n}L)\ar[u]_{\simeq}.
		}
	\end{align}
	
	If $ G_{*}(L)=0 $ or $ G_{*}(M)=0 $, then $ D\RHom_{\mathcal{A}}(L,M)\cong\RHom_{\mathcal{A}}(M,\Si^{n}L) $. In particular, the full subcategory $ \mathrm{pvd}_{\mathcal{B}}(\mathcal{A}) $ defined as the kernel of the restriction functor $ G_{*}\colon\pvd\ca\ra\pvd\cb $ is $ n $-Calabi--Yau as a triangulated category.
\end{Prop}

\begin{proof}
	Since $ G:\mathcal{B}\rightarrow\mathcal{A} $ has a relative left $ n $-Calabi--Yau structure, by Definition \ref{Def Relative CY} and diagram (\ref{eq:left morphism of triangles}), we have an isomorphism in $ \mathcal{D}(\mathcal{A}^{e}) $
	$$ \Si^{n}\ca^{\vee}\iso\cof(\ca\lten_{\cb}\ca\xrightarrow{\mu_{G}}\ca)=\cof(\ca\lten_{\cb}\cb\lten_{\cb}\ca\xrightarrow{\mu_{G}}\mathcal{A})\cong  \cone(\cb\lten_{\cb^{e}}\ca\xrightarrow{\mu_{G}}\mathcal{A}) ,$$ and an isomorphism in $ \mathcal{D}(\mathcal{B}^{e}) $
	$$ \Si^{n-1}\cb^{\vee}\iso\mathcal{B} .$$
	Let $ P_{L} $ and $ P_{M} $ be cofibrant resolutions of $ L $ and $ M $ respectively.
	By \cite[Lemma 4.1]{Bk2008}, we have
	\begin{equation*}
		\begin{split}
			D\RHom_{\mathcal{B}}(G_{*}(L),G_{*}(M))&\simeq\RHom_{\mathcal{B}}(G_{*}(M),\Si^{n-1}G_{*}(L))
		\end{split}
	\end{equation*}
	and
	\begin{equation*}
		\begin{split}
			D\RHom_{\mathcal{A}}(L,M)&\simeq D\ch om_{\ca}(P_{L},P_{M})\\
			&\simeq \ch om_{\ca}(P_{M}\otimes_{\mathcal{A}}^{\mathbf{L}}\ca^{\vee},P_{L})\\
			&\simeq \fib(\ch om_{\ca}(P_{M},\Si^{n}P_{L})\to \Hom_{\mathcal{D}(A)}(P_{M}\otimes_{\mathcal{A}}^{\mathbf{L}}\mathbf{L}G^{*}(\mathcal{B}),\Si^{n}P_{L}))\\
			&\simeq \fib(\mathbf{R}\Hom_{\ca}(M,\Si^{n}L)\to \ch om_{\ca}(P_{M}\otimes_{\mathcal{A}}^{\mathbf{L}}(\mathcal{A}\otimes^{\mathbf{L}}_{\mathcal{B}}\mathcal{A}),\Si^{n}P_{L}))\\
			&\simeq \fib(\mathbf{R}\Hom_{\ca}(M,\Si^{n}L)\to \ch om_{\ca}(P_{M}\otimes^{\mathbf{L}}_{\mathcal{B}}\mathcal{A},\Si^{n}P_{L}))\\
			&\simeq \fib(\mathbf{R}\Hom_{\ca}(M,\Si^{n}L)\to \mathcal{H}om_{\cb}(G_{*}(P_{M}),\Si^{n}G_{*}(P_{L}))\\
			&\simeq \fib(\mathbf{R}\Hom_{\ca}(M,\Si^{n}L)\to \mathbf{R}\Hom_{\cb}(G_{*}(M),\Si^{n}G_{*}(L)))\\
			&\simeq\mathcal{C}(M,\Si^{n-1}L).
		\end{split}
	\end{equation*}
	
	Thus, we get the bifunctorial isomorphism of triangles~(\ref{Relative right from left}). If $ G_{*}(L)=0 $ or $ G_{*}(M)=0 $, then we have the following functorial duality
	\begin{align*}
		\xymatrix{
			D\RHom_{\mathcal{A}}(L,M)\simeq \RHom_{\mathcal{A}}(M,\Si^{n}L)
		}.
	\end{align*}
	In particular, the kernel $ \mathrm{pvd}_{\mathcal{B}}(\mathcal{A}) $ of $ G_{*}\colon\pvd(\ca)\ra\pvd(\cb) $ is $ n $-Calabi--Yau as a triangulated category.

\end{proof}

Let $ \mathcal{B}\xrightarrow{G}\mathcal{A}\xrightarrow{Q}\mathcal{A}/\mathcal{B} $ be a homotopy cofiber sequence of small dg categories. By construction, the dg category $ \mathcal{A}/\mathcal{B} $ is the Drinfeld dg quotient of $ \mathcal{A} $ by its full dg subcategory $ \Ima(G) $, where $ \Ima(G) $ is the full dg subcategory of $ \mathcal{A} $ whose objects are the $ y $ in $ \mathcal{A} $ such that there exists an object $ x $ in $ \mathcal{B} $ and an isomorphism $ F(x)\cong y $ in $ \rm H^{0}(\mathcal{A}) $. We denote by $ i $ the dg inclusion $ \Ima(G)\hookrightarrow \mathcal{A} $.

\begin{Cor}\label{Relative CY duality}
	For any dg module $ N $ and any dg module $ M $ in $ \pvd(\mathcal{A}) $ whose restriction to $ \Ima G $ is acyclic, there is a canonical isomorphism
	\begin{align*}
		\xymatrix{
			D\Hom_{\mathcal{D}(\mathcal{A})}(M,N)\simeq \Hom_{\mathcal{D}(\mathcal{A})}(N,\Si^{n}M)
		}.
	\end{align*}
	
\end{Cor}

\begin{proof}
	Since the restriction of $ M $ to $ \Ima G $ is acyclic, we have $ G_{*}(M)=0 $. Then the claim follows from the above Proposition~\ref{Relative Hom}.
\end{proof}

\subsection{Relative Calabi--Yau completions}

Given a dg category $ \cb $, let $ (\mathrm{dgcat}_{k})_{\cb/} $ be the category of dg categories under $ \cb $.  The forgetful functor $ (\mathrm{dgcat}_{k})_{\cb/}\ra\cc(\cb^{e}) $, sending a dg functor $ G\colon\cb\ra\ca $ to the $ \cb $-bimodule given by $ (a,a')\mapsto\ca(G(a'),G(a)) $, has a left adjoint $ T_{\cb} $, that can be described as follows:

Given a $ \cb $-bimodule $ M $, the \emph{tensor category} $ T_{\mathcal{B}}(M) $ is defined as follows:
\begin{align*}
	\xymatrix{
		T_{\mathcal{B}}(M)=\mathcal{B}\oplus M \oplus (M\otimes_{\mathcal{B}}M) \oplus (M\otimes_{\mathcal{B}}M\otimes_{\mathcal{B}}M) \oplus \cdots}
\end{align*}
Thus, the dg category $ T_{\mathcal{B}}(M) $ has the same objects as $ \cb $ and morphism complexes
\begin{equation*}
	\begin{split}
		T_{\mathcal{B}}(M)(x,y)&=\mathcal{B}(x,y)\oplus M(x,y) \oplus
		\{\oplus_{z\in \mathcal{B}}M(z,y)\otimes_{k} M(x,z)\}\oplus\\
		&\{\oplus_{z_{1},z_{2}\in \mathcal{B}}M(z_{2},y)\otimes_{k}M(z_{1},z_{2})\otimes_{k} M(x,z_{1})\}\oplus\cdots 
	\end{split}
\end{equation*}
The dg structure on $ T_{\mathcal{B}}(M) $ is given by the differentials of $ \mathcal{B} $ and $ M $ and the multiplication is given by the concatenation product. This adjunction is Quillen and thus induces an adjunction between their homotopy categories. We will denote by $ \mathbf{L}T_{\cb} $ the left derived functor of $ T_{\cb}\colon\cc(\cb^{e})\ra(\mathrm{dgcat}_{k})_{\cb/} $.

An $ \cb $-bilinear (super-)derivation $ D $ of degree 1 on $ \mathbf{L}T_{\cb}(M) $ is determined by its restriction to the generating bimodule $ M $. Then it is easy to see that each morphism $ c\colon M\ra\Si\cb $ in $ \cd(\cb^{e}) $ gives rise to a \enquote*{deformation} $$ (\mathbf{L}T_{\cb}(M),d_{c}) $$ of $ \mathbf{L}T_{\cb}(M) $, obtained by adding the $ \ca $-bilinear (super-)derivation $ D_{c} $ determined by $ c $ to the differential of $ \mathbf{L}T_{\cb}(M) $.

Let $ G\colon\mathcal{B}\to \mathcal{A} $ be a dg functor between smooth dg categories and let $ [\xi] $ be an element in $ H\!H_{n-2}(G) $. Our objective is to define the \emph{deformed relative $ n $-Calabi--Yau completion} of $ G\colon\cb\to\ca $ with respect to the Hochschild homology class $ [\xi]\in H\!H_{n-2}(G) $.

The dg functor $ G\colon\cb\ra\ca $ induces a morphism of dg $ \ca $-bimodules $ \cb\lten_{\cb^{e}}\ca^{e}\ra\ca $. Let $ \Xi $ be the cofiber of its bimodule dual, i.e.\ $ \Xi=\cone(\ca^{\vee}\ra\lG(\cb)^{\vee}) =\cone(\ca^{\vee}\ra(\cb\lten_{\cb^{e}}\ca^{e})^{\vee}) $. Clearly, the the dualizing bimodule $ \Theta_{G}=(\cone(\cb\lten_{\cb^{e}}\ca^{e}\ra\ca))^{\vee} $ of $ G $ is quasi-isomorphic to $ \Si^{-1}\Xi $.

By the definition of Hochschild homology of $ G $, we have the following long exact sequence
$$ \cdots\ra H\!H_{n-2}(\cb)\ra H\!H_{n-2}(\ca)\ra H\!H_{n-2}(G)\ra H\!H_{n-3}(\cb)\ra\cdots .$$ Thus, the Hochschild homology class $ [\xi]=[(s\xi_{\mathcal{B}},\xi_{\mathcal{A}})]\in H\!H_{n-2}(G) $ induces an element $ [\xi_{\cb}]$ in $ H\!H_{n-3}(\cb) $.

Notice that since $ \cb,\ca $ are smooth, we have the following isomorphisms
\begin{equation*}
	\begin{split}
		H\!H_{n-3}(\cb)=&H^{3-n}(\cb\lten_{\cb^{e}}\cb)\\
		\simeq& H^{3-n}(\RHom_{\cb^{e}}(\cb^{\vee},\cb))\\
		\simeq&
		\Hom_{\cd(\cb^{e})}(\Si^{n-2}\cb^{\vee},\Si\cb)\\
		\Hom_{\cd(\ca^{e})}(\Si^{n-2}\Xi,\Si\ca)\simeq&\Hom_{\cd(\ca^{e})}(\cone(\ca^{\vee}\ra(\ca\lten_{\cb}\ca)^{\vee}),\Si^{3-n}\ca)\\
		\simeq&\Hom_{\cd(\ca^{e})}(\cone(\ca^{\vee}\ra(\ca\lten_{\cb}\cb\lten_{\cb}\ca)^{\vee}),\Si^{3-n}\ca)\\
		\simeq&\Hom_{\cd(\ca^{e})}(\cone(\ca^{\vee}\ra\ca\lten_{\cb}\cb^{\vee}\lten_{\cb}\ca),\Si^{3-n}\ca)\\
		\simeq&H^{3-n}(\RHom_{\ca^{e}}(\cone(\ca^{\vee}\ra\ca\lten_{\cb}\cb^{\vee}\lten_{\cb}\ca),\ca))\\
		\simeq&H^{2-n}(\cone(\RHom_{\ca^{e}}(\ca\lten_{\cb}\cb^{\vee}\lten_{\cb}\ca,\ca)\ra\RHom_{\ca^{e}}(\ca^{\vee},\ca)))\\
		\simeq&H^{2-n}(\cone(\cb\lten_{\cb^{e}}\ca\ra\ca\lten_{\ca^{e}}\ca)),
	\end{split}
\end{equation*}
where we use the isomorphism $ \cb\lten_{\cb^{e}}\cb\iso\RHom_{\cb^{e}}(\cb^{\vee},\cb) $ in the first computation and $ \cb\lten_{\cb^{e}}\ca\iso\RHom_{\cb^{e}}(\cb^{\vee},\ca)\iso\RHom_{\ca^{e}}(\ca\lten_{\cb}\cb^{\vee}\lten_{\cb}\ca,\ca) $ in the second one.

Thus, via the canonical morphism 
$$ H\!H_{n-2}(G)=H^{2-n}(\cone(\cb\lten_{\cb^{e}}\cb\ra\ca\lten_{\ca^{e}}\ca))\ra H^{2-n}(\cone(\cb\lten_{\cb^{e}}\ca\ra\ca\lten_{\ca^{e}}\ca)) ,$$
the homology class $ [\xi]=[(s\xi_{\mathcal{B}},\xi_{\mathcal{A}})] $ induces a morphism in $ \cd(\ca^{e}) $
$$ \xi\colon\Si^{n-2}\Xi\ra\Si\ca $$ and the homology class $ [\xi_{\cb}] $ induces a morphism in $ \cd(\cb^{e}) $
$$ \xi_{\cb}\colon\Si^{n-2}\cb^{\vee}\ra\Si\cb .$$ Moreover, we have the following commutative diagram in $ \cd(\ca^{e}) $ (\cite[Proposition 4.7]{WkY2016})
\begin{align*}
	\xymatrix{
		\lG(\Si^{n-2}\cb^{\vee})=\Si^{n-2}(\cb^{\vee}\lten_{\cb^{e}}\ca^{e})\ar[r]\ar[d]_{\xi_{\cb}}&\Si^{n-2}\Xi=\Si^{n-2}(\cone(\ca^{\vee}\ra\cb^{\vee}\lten_{\cb^{e}}\ca^{e}))\ar[d]^{\xi}\\
		\lG(\Si\cb)\ar[r]&\Si\ca,
	}
\end{align*}
where the first horizontal morphism is the canonical inclusion and the second one is the canonical morphism $ \lG(\cb)=\cb\lten_{\cb^{e}}\ca^{e}\simeq\ca\lten_{\cb}\ca\ra\ca $ of $ \ca $-bimodules.

Therefore, the morphism $ \xi_{\cb} $ gives rise to a \enquote*{deformation}
$$ \bm{\Pi}_{n-1}(\cb,\xi_{\cb}) $$ of $ \bm{\Pi}_{n-1}(\cb)=\mathbf{L}T_{\cb}(\Si^{n-2}\cb^{\vee}) $, obtained by adding $ \xi_{\cb} $ to the differential of $ \bm{\Pi}_{n-1}(\cb) $; the morphism $ \xi $ gives rise to a ‘deformation’
$$ \bm{\Pi}_{n}(\ca,\cb,\xi) $$ of $  \bm{\Pi}_{n}(\ca,\cb)=\mathbf{L}T_{\ca}(\Si^{n-2}\Xi) $, obtained by adding $ \xi $ to the differential of $ \mathbf{L}T_{\ca}(\Si^{n-2}\Xi) $; and the commutative diagram above gives rise to a dg functor 
\begin{align}\label{relative functor}
	\xymatrix{
		\widetilde{G}\colon\bm{\Pi}_{n-1}(\cb,\xi_{\cb})\ra\bm{\Pi}_{n}(\ca,\cb,\xi).
	}
\end{align}

A standard argument shows that up to quasi-isomorphism, the dg functor $ \widetilde{G} $ and the deformations $ \bm{\Pi}_{n-1}(\mathcal{B},\xi_{\mathcal{B}}) $, $ \bm{\Pi}_{n}(\mathcal{A},\mathcal{B},\xi) $ only depend on the class $ [\xi] $. When the class $ [\xi_{\cb}] $ vanishes (respectively, $ [\xi] $ vanishes), we will abbreviate $ \bm{\Pi}_{n-1}(\mathcal{B},\xi_{\mathcal{B}}) $ (respectively, $ \bm{\Pi}_{n}(\mathcal{A},\mathcal{B},\xi) $) to $ \bm{\Pi}_{n-1}(\mathcal{B}) $ (respectively, $ \bm{\Pi}_{n}(\mathcal{A},\mathcal{B}) $).

\begin{Def}\rm\cite[Definition 5.18]{WkY2016}
	Let $ G\colon\mathcal{B} \to \mathcal{A}$ be a dg functor between smooth dg categories. The dg functor $ \widetilde{G} $ (\ref{relative functor}) defined above is called the \emph{deformed relative $ n $-Calabi--Yau completion} of $ G\colon\mathcal{B} \to \mathcal{A}$ with respect to the Hochschild homology class $ [\xi] \in H\!H_{n-2}(G)$. If we take the class $ [\xi]$ be 0, then we will call it simply the \emph{relative $ n $-Calabi--Yau completion} of $ G\colon\mathcal{B} \to\mathcal{A}$.
\end{Def}
\begin{Rem}
  If we take $\mathcal{B}$ to be the empty category, then the deformed relative $ n $-Calabi--Yau completion is the deformed $ n $-Calabi--Yau completion of~\cite{BK2011}.
\end{Rem} 

\begin{Thm}\cite[Theorem 7.1]{WkY2016}\cite[Proposition 5.29]{BCS2020}
	Let $ G\colon\mathcal{B} \to \mathcal{A}$ be a dg functor between smooth dg categories and let $ [\xi]$ be an element in $ H\!H_{n-2}(G) $. If $ [\xi] $ has a negative cyclic lift, then each choice of such a lift gives rise to a canonical left $ n $-Calabi--Yau structure on the dg functor
	\[
	\begin{tikzcd}
		\widetilde{G}\colon \bm{\Pi}_{n-1}(\mathcal{B},\xi_{\mathcal{B}}) \ar[r]&\bm{\Pi}_{n}(\mathcal{A},\mathcal{B},\xi).
	\end{tikzcd}
	\]
\end{Thm} 

\subsection{Reduced relative Calabi--Yau completions}
Recall that a dg category $ \mathcal{A} $ over $ k $ is said to be \emph{semi-free} if there is a graded quiver $ Q=(Q_{0},Q_{1}) $ such that
the underlying graded $ k $-category of $ \mathcal{A} $ is freely generated by the arrows of $ Q $ over the vertex set $ Q_{0} $. We write this as $ \mathcal{A}=T_{kQ_{0}}(kQ_{1}) $.
\begin{Def}\cite[Section 2]{WkY2016}
	\rm
	A dg category $ \mathcal{A} $ is said to be \emph{cellular} if it is semi-free over some graded quiver $ Q=(Q_{0},Q_{1}) $
	that admits a filtration 
	\begin{align*}
		Q^{(1)}\subset Q^{(2)}\subset\cdots
	\end{align*}
	such that every generating arrow $ f\in Q^{(i)} $ has differential $ d(f) $ contained in the graded category $ T_{kQ_{0}}(kQ^{(i-1)}) $.
	
	We say that $ \mathcal{A} $ is \emph{finitely cellular} if the graded quiver $ (Q_{0},Q_{1}) $ is finite (i.e.\ both $ Q_{0} $ and $ Q_{1} $ are finite). 
	
	We say that $ \mathcal{A} $ is of \emph{finite cellular type} if it is quasi-equivalent to a finitely cellular dg category.
\end{Def}

Let $ G\colon\mathcal{B} \to \mathcal{A}$ be a dg functor between finitely cellular type dg categories. By \cite[Remark 24.2.8]{WKY2016-thesis}, we can assume that $ \mathcal{B} $ and $ \mathcal{A} $ are finitely cellular and $ G\colon\mathcal{B} \to \mathcal{A}$ is a semi-free extension, i.e.\ there is a finite graded quiver $ Q $ and a subquiver $ F\subseteq Q $ such that the underlying graded $ k $-category of $ \cb $ and $ \ca $ are isomorphic to $ T_{kF_{0}}(kF_{1}) $ and $ T_{kQ_{0}}(kQ_{1}) $, respectively. We abbreviate $ R_{2}=kF_{0} $ and $ R_{1}=kQ_{0} $. Then we have a short exact sequence of $ \mathcal{B} $-bimodules
\begin{align*}
	\xymatrix{
		0\ar[r]&\Omega^{1}(\mathcal{B})\ar[r]^-{\alpha}&\mathcal{B}\otimes_{R_{2}}\mathcal{B}\ar[r]^-{m}&\mathcal{B}\ar[r]&0,
	}
\end{align*}
where \emph{the bimodule of differentials} $ \Omega^{1}(\mathcal{B}) $ is generated by $ \{D(f)|f\in F_{1}\} $, the map $ \alpha $ is given by $ D(f)\mapsto f\otimes 1_{x}-1_{y}\otimes f $ where $ f\colon x\ra y $
and the map $ m $ is the composition map in $ \mathcal{B} $. 

We define $ \Omega^{1}(\ca) $ similarly.
We put $$ \xymatrix{\mathcal{P}_{B}= \cone(\Omega^{1}(\mathcal{B})\ar[r]^-{\alpha}&\mathcal{B}\otimes_{R_{2}}\mathcal{B}}) $$ and $$ \xymatrix{\mathcal{P}_{A}= \cone(\Omega^{1}(\mathcal{A})\ar[r]^-{\alpha}&\mathcal{A}\otimes_{R_{1}}\mathcal{A}}). $$ Then $P_{\mathcal{B}} $ and $ P_{\mathcal{A}} $ are cofibrant replacements of the bimodules $ \mathcal{B} $ and $ \mathcal{A} $ respectively. The $ \mathcal{B} $-bimodule $ \mathcal{P}_{\mathcal{B}}^{\vee} $ is cellular of finite rank, with basis $ \{f^{\vee}_{\mathcal{B}}|f\in F_{1}\} \cup \{c_{x,\mathcal{B}}|x\in F_{0}\} $ where the arrow $ f^{\vee}_{\mathcal{B}} $ has degree $ |f^{\vee}_{\mathcal{B}}| =1-|f|$, and points in the opposite direction to $ f $; the loop $ c_{x,\mathcal{B}} $ has degree $ |c_{x,\mathcal{B}}|=0 $ , and is based at $ x $. Similarly, the $ \mathcal{A} $-bimodule $ \mathcal{P}_{\mathcal{A}}^{\vee} $ is also cellular of finite rank, with basis $ \{g^{\vee}_{\mathcal{A}}|g\in Q_{1}\} \cup \{c_{y,\mathcal{A}}|y\in Q_{0}\} $ where the arrow $ g^{\vee}_{\mathcal{A}} $ has degree $ |g^{\vee}_{\mathcal{A}}| =1-|g|$, and points in the opposite direction to $ g $; the loop $ c_{y,\mathcal{A}} $ has degree $ |c_{y,\mathcal{A}}|=0 $ , and is based at $ y $. 

The natural map $ \alpha_{G}\colon G^{*}(P_{\mathcal{B}})\to P_{\mathcal{A}} $ in $\mathcal{C}(\mathcal{A}^{e}) $ induces the dual map $ \alpha_{G}^{\vee}\colon P_{\mathcal{A}}^{\vee}\to G^{*}(P_{\mathcal{B}})^{\vee} $ in $ \mathcal{C}(\mathcal{A}^{e}) $. This $ \alpha_{G}^{\vee} $ is given as follows:
\begin{itemize}
	\item $ \alpha_{G}^{\vee}(c_{y,\mathcal{A}})=c_{y,\mathcal{B}} $ if $ y $ belongs to $ F_{0} $; otherwise, $ \alpha_{G}^{\vee}(c_{y,\mathcal{A}})=0 $,
	\item $ \alpha_{G}^{\vee}(g_{\mathcal{A}}^{\vee})=g_{\mathcal{B}}^{\vee} $ if $ g $ belongs to $ F_{1} $; otherwise, $ \alpha_{G}^{\vee}(g_{\mathcal{A}}^{\vee})=0 $.
\end{itemize}

Clearly, the morphism $ \alpha_{G}^{\vee}  $ is a graded split surjection of $ \ca $-bimodules. Let $ \mathcal{K}$ be the kernel of $ \alpha_{G}^{\vee} $. Then $ \mathcal{K} $ is cellular of finite rank, with basis $ \{g_{\mathcal{A}}^{\vee},c_{y,\mathcal{A}}\ |\ g\in N_{1}=Q_{1}\setminus F_{1},\ y\in N_{0}=Q_{0}\setminus F_{0} \} $. We have a split exact sequence in the category of graded $ \mathcal{A} $-bimodules, i.e.\ there exist two graded bimodule morphisms $ s_{G}\colon G^{*}(P_{\mathcal{B}})^{\vee} \to P_{\mathcal{A}}^{\vee}$, $ r_{\mathcal{K}}\colon P_{\mathcal{A}}^{\vee}\to\mathcal{K}$ such that $ \alpha_{G}^{\vee}\circ s_{G}=\id_{G^{*}(P_{\mathcal{B}})^{\vee}} $, $ r_{\mathcal{K}} \circ i_{\mathcal{K}}=\id_{\mathcal{K}} $, $ s_{G}\circ \alpha_{G}^{\vee}+i_{\mathcal{K}}\circ r_{\mathcal{K}}=\id_{P_{\mathcal{A}}^{\vee}} $. We summarize the notations in the diagram
\begin{equation}
	\begin{tikzcd}\label{Split of bimodules}
		0\arrow[r]&\mathcal{K}\arrow[r,shift left=1ex,"i_{\ck}"]&P_{\mathcal{A}}^{\vee}\arrow[r,shift left=1ex,"\alpha_{G}^{\vee}"]\arrow[l,shift left=1ex,"r_{\ck}"]&G^{*}(P_{\mathcal{B}})^{\vee}\arrow[r]\arrow[l,shift left=1ex,"s_{G}"]&0.
	\end{tikzcd}
\end{equation}
We choose the graded morphisms $ r_{\ck} $ and $ s_{G} $ are given as follows:
\begin{itemize}
	\item The graded morphism $ s_{G} $ maps $ g^{\vee}_{\cb} $ to $ g^{\vee}_{\ca} $ and maps $ c_{x,\cb} $ to $ c_{x,\ca} $.
	\item The graded morphism $ r_{\ck} $ maps $ g^{\vee}_{\ca} $ to $ g^{\vee}_{\ca} $ if $ g $ is in $ N_{1} $; otherwise, we put $ r_{\ck}(g^{\vee}_{\ca})=0 $. Moreover, it maps $ c_{y,\ca} $ to $ c_{y,\ca} $ if $ y $ is in $ N_{0} $; otherwise, we put $ r_{\ck}(c_{y,\ca})=0 $. 
\end{itemize}

The above exact sequence yields a triangle in $ D(\mathcal{A}^{e}) $
\begin{align}
	\xymatrix{
		P_{\mathcal{A}}^{\vee}\ar[r]^-{\alpha_{G}^{\vee}}&G^{*}(P_{\mathcal{B}})^{\vee}\ar[r]^-{u}&\Si\mathcal{K}\ar[r]&,
	}
\end{align}
where $ u $ is equal to $ r_{\mathcal{K}}\circ d_{P_{\mathcal{A}}^{\vee}}\circ s_{G} $. Thus, we get the following isomorphism of triangles in $\mathcal{D}(\mathcal{A}^{e}) $

\begin{align}\label{quasidiagram}
	\xymatrix{
		P_{\mathcal{A}}^{\vee}\ar"1,3"^{\alpha_{G}^{\vee}}\ar[d]^{\id}&&G^{*}(P_{\mathcal{B}})^{\vee}\ar"1,6"^{u}\ar[d]^{\id}&&&\Si\mathcal{K}\ar[r]\ar@{^{(}->}[d]^{v}&\\
		P_{\mathcal{A}}^{\vee}\ar"2,3"^{\alpha_{G}^{\vee}}&&G^{*}(P_{\mathcal{B}})^{\vee}\ar"2,6"^{l}&&&\Xi\ar[r]&,
	}
\end{align}
where $ \Xi=\cone(P_{\mathcal{A}}^{\vee}\to G^{*}(P_{\mathcal{B}})^{\vee}) $ and $ v $ is the quasi-isomorphism induced by the inclusion of $ \ck $ into $ P_{\ca}^{\vee} $. Here the morphism $ P_{\mathcal{A}}^{\vee}\to G^{*}(P_{\mathcal{B}})^{\vee} $ of $ \ca $-modules is a cofibrant replacement of $ \ca^{\vee}\ra\lG(\cb)^{\vee} $.

Now we consider the derived tensor category $
\mathbf{L}T_{\mathcal{A}}(\Si^{n-1}\mathcal{K})
$. Let $ [\xi]=[(s\xi_{\mathcal{B}},\xi_{\mathcal{A}})] $ be an element in $ H\!H_{n-2}(G) $. By the above section, the homology class $ [\xi] $ induces a morphism in $ \cd(\ca^{e}) $
$$ \xi\colon\Si^{n-2}\Xi\ra\Si\ca $$ and the homology class $ [\xi_{\cb}] $ induces a morphism in $ \cd(\cb^{e}) $
$$ \xi_{\cb}\colon\Si^{n-2}\cb^{\vee}\ra\Si\cb .$$

Since the $ \ca $-bimodule $ \ck $ is cofibrant, we have $ \mathbf{L}T_{\mathcal{A}}(\Si^{n-1}\mathcal{K})=T_{\mathcal{A}}(\Si^{n-1}\mathcal{K}) $.

We define $ \xi_{\mathcal{K}} $ as the following composition
$$ \xymatrix{
	\xi_{\mathcal{K}}\colon\,\Si^{n-1}\mathcal{K}\ar@{^{(}->}[r]^-{v}& \Si^{n-2}\Xi\ar"1,4"^{\xi}&&\Si\mathcal{A},
}$$ where $ v $ is the conical inclusion.
Then it determines an $ \ca $-bilinear derivation $ d'_{\ck} $ on $ T_{\mathcal{A}}(\Si^{n-1}\mathcal{K}) $ and we get a \enquote*{deformation}
$$ T_{\ca}(\Si^{n-1}\mathcal{K},\xi_{\mathcal{K}}) $$ of $ T_{\mathcal{A}}(\Si^{n-1}\mathcal{K}) $, obtained by adding $ d'_{\ck} $ to the differential of $
T_{\mathcal{A}}(\Si^{n-1}\mathcal{K})
$.

Then the canonical inclusion of dg $ \ca $-bimodules
$ \xymatrix{
	\Si^{n-1}\mathcal{K}\ar@{^{(}->}[r]^<<<<{v}&\Si^{n-2}\Xi}$ induces a fully faithful dg functor
$$ \xymatrix{\Psi\colon T_{\ca}(\Si^{n-1}\mathcal{K},\xi_{\mathcal{K}})\ar@{^{(}->}[r]&\bm{\Pi}_{n}(\mathcal{A},\mathcal{B},\xi)}.$$

Next we will construct a dg functor from $ \bm{\Pi}_{n-1}(\mathcal{B},\xi_{\mathcal{B}}) $ to $ T_{\ca}(\Si^{n-1}\mathcal{K},\xi_{\mathcal{K}}) $.

Firstly, we have the following diagram 
\begin{align*}
	\xymatrix{
		\Si^{n-2}G^{*}(P_{\mathcal{B}})^{\vee}\ar"1,3"^{u}\ar[d]_{\id}&&\Si^{n-1}\mathcal{K}\ar@{^{(}->}[d]^{v}\\
		\Si^{n-2}G^{*}(P_{\mathcal{B}})^{\vee}\ar"2,3"^{l}\ar[d]_{G^{*}(\xi_{\mathcal{B}})}&&\Si^{n-2}\Xi\ar[d]^{\xi}\\
		\Si G^{*}(\mathcal{B})\ar"3,3"^{j_{G}}&&\Si\mathcal{A}\,,
	}
\end{align*} where the upper square is commutative up to homotopy and the lower square is commutative. The homotopy is given by
$$ \xymatrix{H^{'}\colon\Si^{n-2}G^{*}(P_{\mathcal{B}})^{\vee}\ar"1,3"^-{\Si^{n-2}s_{G}^{\vee}}&& \Si^{n-2}P_{\mathcal{A}}^{\vee}\ar@{^{(}->}"1,5"^-{inclusion}&&\Si^{n-3}\Xi},$$
where $ s^{\vee}_{G} $ is the map defined in~(\ref{Split of bimodules}).

Combining those two diagrams, we get the following diagram commutative up to homotopy

\begin{align*}
	\xymatrix{
		\Si^{n-2}G^{*}(P_{\mathcal{B}})^{\vee}\ar"1,3"^-{u}\ar[d]^{G^{*}(\xi_{\mathcal{B}})}&&\Si^{n-1}\mathcal{K}\ar[d]^{\xi\circ v}\\
		\Si G^{*}(\mathcal{B})\ar"2,3"^-{j_{G}}&&\Si\mathcal{A}
	}
\end{align*}
where the homotopy is given by
$$ \xymatrix{H\colon\Si^{n-2}G^{*}(P_{\mathcal{B}})^{\vee}\ar"1,3"^-{\Si^{n-2}s_{G}^{\vee}}&& \Si^{n-2}P_{\mathcal{A}}^{\vee}\ar@{^{(}->}"1,5"^-{inclusion}&&\Si^{n-3}\Xi\ar[r]^{-\xi}&\mathcal{A}}.$$ 

Then the following diagram commutes strictly
\begin{align*}
	\xymatrix{
		\Si^{n-2}G^{*}(P_{\mathcal{B}})^{\vee}\ar"1,3"^-{(-H,u)^{T}}\ar[d]_{(d_{G^{*}(P_{\mathcal{B}})^{\vee}},G^{*}(\xi_{\mathcal{B}}))^{T}}&&\mathcal{A}\oplus\Si^{n-1}\mathcal{K}\ar[d]^{(d_{\mathcal{A}},\xi\circ v)}\\
		\Si^{n-1} G^{*}(P_{\mathcal{B}})^{\vee}\oplus \Si G^{*}(\mathcal{B})\ar"2,3"^-{(H,j_{G})}&&\Si\mathcal{A}.
	}
\end{align*}

Thus, the above commutative diagram induces a dg functor
\begin{align}
	\xymatrix{
		\bm{G}_{rel}\colon\bm{\Pi}_{n-1}(\mathcal{B},\xi_{\mathcal{B}})\ar[r]&\bm{\Pi}_{n}^{red}(\ca,\cb,\xi)
	}
\end{align}
where we put $ \bm{\Pi}_{n}^{red}(\ca,\cb,\xi)=T_{\ca}(\Si^{n-1}\mathcal{K},\xi_{\mathcal{K}}) $.
A standard argument shows that up to quasi-isomorphism, the dg functor $ \bm{G}_{rel} $ and the deformed dg category $ \bm{\Pi}_{n}^{red}(\ca,\cb,\xi) $ only depend on the class $ [\xi] $ and the dg functor $ G\colon\cb\ra\ca $.

We call the dg functor $ \bm{G}_{rel} $ the \emph{reduced deformed relative n-Calabi--Yau completion} of $ G\colon\mathcal{B} \to \mathcal{A}$ with respect to the Hochschild homology class $ [\xi] \in H\!H_{n-2}(G)$.

\begin{Prop}\label{reduced completion}
	Let $ G\colon\mathcal{B}\to \mathcal{A} $ be a dg functor between dg categories of finitely cellular type and let $ [\xi]=[(s\xi_{\mathcal{B}},\xi_{\mathcal{A}})] $ be an element in $ H\!H_{n-2}(G) $ which has a negative cyclic lift. Then we have the following diagram which is commutative up to homotopy and where $ \Psi $ is a quasi-equivalence.
	\begin{align}
		\xymatrix{
			\bm{\Pi}_{n-1}(\mathcal{B},\xi_{\mathcal{B}})\ar[r]^{\widetilde{G}}\ar[rd]_{\bm{G}_{rel}}&\bm{\Pi}_{n}(\mathcal{A},\mathcal{B},\xi)\\
			&\bm{\Pi}_{n}^{red}(\ca,\cb,\xi)\ar@{^{(}->}[u]_{\Psi}
		}
	\end{align}
	Thus,  the dg functor $ \bm{G}_{rel}\colon\bm{\Pi}_{n-1}(\mathcal{B},\xi_{\mathcal{B}})\to \bm{\Pi}_{n}^{red}(\ca,\cb,\xi) $ has a canonical left $ n $-Calabi--Yau structure.	
\end{Prop}
\begin{proof}
	Since the map $ v $ in
	diagram (\ref{quasidiagram}) is a  quasi-isomorphism between cofibrant dg $ \ca^{e} $-modules, the map $ v $ is a homotopy equivalence. Then we can construct a homotopy inverse of $ \Psi $. Thus the dg functor $ \Psi $ is a quasi-equivalence.

	Suppose that $ \mathcal{B} $ and $ \mathcal{A} $ are finitely cellular and $ G\colon\mathcal{B} \to \mathcal{A}$ is a semi-free extension, i.e.\ there is a finite graded quiver $ Q $ and a subquiver $ F\subseteq Q $, cf. above. We know that the bimodules
	$$ \xymatrix{\mathcal{P}_{B}= \cone(\Omega^{1}(\mathcal{B})\ar[r]^-{\alpha}&\mathcal{B}\otimes_{R_{2}}\mathcal{B}}) $$
	and 
	$$\xymatrix{\mathcal{P}_{A}= \cone(\Omega^{1}(\mathcal{A})\ar[r]^-{\alpha}&\mathcal{A}\otimes_{R_{1}}\mathcal{A}}) $$
	are cofibrant replacements of the bimodules $ \mathcal{B} $ and $ \mathcal{A} $ respectively. Therefore, the $ \mathcal{B} $-bimodule $ \Si^{n-2}\mathcal{P}_{\mathcal{B}}^{\vee} $ is cellular of finite rank, with basis
	$ \{f^{\vee}_{\mathcal{B}}\,|\,f\in F_{1}\} \cup \{c_{x,\mathcal{B}}\,|\,x\in F_{0}\} $ where the arrow $ f^{\vee}_{\mathcal{B}} $ has degree $ |f^{\vee}_{\mathcal{B}}|=3-n-|f|$, and points in the opposite direction to $ f $; the loop $ c_{x,\mathcal{B}} $ has degree $ |c_{x,\mathcal{B}}|=2-n $ , and points from $ x $ to $ x $. 
	
	Similarly, the $ \mathcal{A} $-bimodule $ \Si^{n-1}\mathcal{P}_{\mathcal{A}}^{\vee} $ is also cellular of finite rank, with basis
	$ \{g^{\vee}_{\mathcal{A}}|g\in Q_{1}\}\cup \{c_{y,\mathcal{A}}|y\in Q_{0}\} $ where the arrow $ g^{\vee}_{\mathcal{A}} $ has degree $ |g^{\vee}_{\mathcal{A}}| =2-n-|g|$, and points in the opposite direction to $ g $; the loop $ c_{y,\mathcal{A}} $ has degree $ |c_{y,\mathcal{A}}|=1-n $ , and points from $ y $ to $ y $. 
	
	Then the homotopy (see Definition~\ref{homotopy def}) between 
	$ \Psi\circ\bm{G}_{rel} $ and $ \widetilde{G} $ is given as follows:
	\begin{itemize}
		\item For each object $ x $ in $ R_{1} $, we have $ \Psi\circ\bm{G}_{rel}(x)=\widetilde{G}(x)=x $, i.e, $\alpha(x)$ is the identity map in $ \bm\Pi_{n}(\mathcal{A},\mathcal{B},\xi) $.
		\item For all objects $ x $ and $ y $ in $ R_{1} $, the degree $ -1 $ map
		
		$$ h=h(x,y)\colon\bm\Pi_{n-1}(\mathcal{B},\xi_{\mathcal{B}})(x,y)\to\bm{\Pi}_{n}(\mathcal{A},\mathcal{B},\xi)(x,y) $$ is obtained from the following map of degree $ -1 $,
		\begin{align*}
			\xymatrix{
				h_{2}\colon\Si^{n-2}G^{*}(\mathcal{P}_{\mathcal{B}}^{\vee})\ar[r]&\Si^{n-1}\mathcal{P}_{\mathcal{A}}^{\vee}
			}
		\end{align*}
		where $ h_{2} $ is given by $ f_{\mathcal{B}}^{\vee}\to f_{\mathcal{A}}^{\vee} $, and $ c_{x,\mathcal{B}}\to c_{x,\mathcal{A}} $.
	\end{itemize}
	
	By Proposition~\ref{quasi-invariant} and Corollary~\ref{homotopy-invariant}, the dg functor $ \bm{G}_{rel}\colon\bm{\Pi}_{n-1}(\mathcal{B},\xi_{\mathcal{B}})\to \bm{\Pi}_{n}^{red}(\ca,\cb,\xi) $ has a canonical left $ n $-Calabi--Yau structure.

\end{proof}

\subsection{Relation with the absolute Calabi--Yau completion}

Let $ G\colon\mathcal{B}\rightarrow\mathcal{A} $ be a dg functor between smooth dg categories. In~\cite[Section 5.2.3]{BCS2020}, Bozec--Calaque--Scherotzke defined the following tensor category over $ \mathcal{A} $
$$ \bm\Pi_{n}(G)=T_{\mathcal{A}}(\Si^{n-1}\cb^{\vee}\otimes_{\mathcal{B}^{e}}^{\mathbf{L}}\mathcal{A}^{e}) .$$ 

Let $ \ca/\cb $ be the homotopy cofiber of $ G $, i.e.\ we have the following homotopy push-out diagram in $ \dgcat $ with Dwyer-Kan model structure~\cite{GT2005}
\begin{align*}
	\xymatrix{
		\mathcal{B}\ar[r]\ar[d]&\mathcal{A}\ar[d]\\
		0\ar[r]&\mathcal{A}/\mathcal{B}.
	}
\end{align*}

\begin{Prop}\label{Relation with absolut CY completion}
	The following sequence is a homotopy cofiber sequence in $ dgcat_{k} $
	$$ \bm\Pi_{n-1}(\mathcal{B})\rightarrow\bm\Pi_{n}(\mathcal{A},\mathcal{B})\rightarrow\bm\Pi_{n}(\mathcal{A}/\mathcal{B}), $$
where $ \Pi_{n}(\mathcal{A},\mathcal{B})=\Pi_{n}(\mathcal{A},\mathcal{B},\xi=0) $.
	
\end{Prop}

\begin{proof}
	By~\cite[Corollary 5.24]{BCS2020}, the dg functor $ \bm\Pi_{n-1}(\mathcal{B})\rightarrow\bm\Pi_{n}(\mathcal{A},\mathcal{B}) $ is the following composition
	$$ \bm\Pi_{n-1}(\mathcal{B})\rightarrow\bm\Pi_{n-1}(G)\rightarrow\bm\Pi_{n}(\mathcal{A},\mathcal{B}).$$

	Consider the diagram
	\[
	\begin{tikzcd}
		\cb\arrow[r,"G"]\arrow[d]\arrow[dr, phantom, "\circled{1}"]&\ca\arrow[d]&\bm{\Pi}_{n-1}(\ca)\arrow[dl]\arrow[r]\arrow[d, phantom, "\circled{3}"]&\ca\arrow[dl]\\
		\bm{\Pi}_{n-1}(\cb)\arrow[r]\arrow[d]\arrow[dr, phantom, "\circled{2}"]&\bm{\Pi}_{n-1}(G)\arrow[r]\arrow[d]&\bm{\Pi}_{n-1}(\ca,\cb)\\
		0\arrow[r]&\ca/\cb.
	\end{tikzcd}
	\]
	The square $ \circled{1} $ is a homotopy push-out by~\cite[
	Lemma 5.27]{BCS2020}. Since the rectangle around $ \circled{1} $ and $ \circled{2} $ is a homotopy push-out, it follows that so is $ \circled{2} $. By~\cite[Corollary 5.24]{BCS2020}, the square $ \circled{3} $ is also a homotopy push-out.
	
	Therefore the homotopy cofiber of $ \bm\Pi_{n-1}(\mathcal{B})\rightarrow\bm\Pi_{n}(\mathcal{A},\mathcal{B}) $ is the homotopy push-out of the following diagram
	\begin{align*}
		\xymatrix{
			\bm\Pi_{n-1}(\mathcal{A})\ar[r]\ar[d]&\mathcal{A}\\
			\bm\Pi_{n-1}(G)\ar[d]&\\
			\mathcal{A}/\mathcal{B}&.
		}
	\end{align*}
	
	It is easy to see that the composition $ \bm\Pi_{n-1}(\mathcal{A})\rightarrow \bm\Pi_{n-1}(G)\rightarrow\mathcal{A}/\mathcal{B} $ is equal to $ \bm\Pi_{n-1}(\mathcal{A})\xrightarrow{} \mathcal{A}\rightarrow\mathcal{A}/\mathcal{B} $. Consider the diagram
	\[
	\begin{tikzcd}
		\bm\Pi_{n-1}(\mathcal{A})\arrow[r]\arrow[d]\arrow[dr, phantom, "\circled{4}"]&\mathcal{A}\arrow[d]\\
		\mathcal{A}\arrow[d]\arrow[r]&\bm\Pi_{n}(\mathcal{A})\\
		\mathcal{A}/\mathcal{B}&.
	\end{tikzcd}
	\]
	
	The square $ \circled{4} $ is a homotopy push-out by~\cite[Proposition 5.6]{BK2011}. By~\cite[Theorem 4.6]{BK2011}, the following diagram is a homotopy push-out
	\begin{align*}
		\xymatrix{
			\mathcal{A}\ar[d]\ar[r]&\bm\Pi_{n}(\mathcal{A})\ar[d]\\
			\mathcal{A}/\mathcal{B}\ar[r]&\bm\Pi_{n}(\mathcal{A}/\mathcal{B}).
		}
	\end{align*} 
	Thus, the sequence $$ \bm\Pi_{n-1}(\mathcal{B})\rightarrow\bm\Pi_{n}(\mathcal{A},\mathcal{B})\rightarrow\bm\Pi_{n}(\mathcal{A}/\mathcal{B}) $$ is a homotopy cofiber sequence in $ \dgcat $.
	
\end{proof}

\section{Relative cluster categories}\label{section RCC}
Let $f\colon\ B\to A $ be a morphism (not necessarily unital) between differential graded (=dg) $ k $-algebras. We consider the following assumptions.

\begin{assumption}\rm\label{Relative assumption}
	Suppose that the morphism $ f\colon B\to A $ satisfies the following properties:
	\begin{itemize}
		\item[1)] $ A $ and $ B $ are smooth,
		\item[2)] $ A $ is connective, i.e.\ the cohomology of $ A $ vanishes in degrees $ >0 $,
		\item[3)] the morphism $ f\colon B\to A $ has a left $ (n+1) $-Calabi--Yau structure,
		\item[4)] $ H^{0}(A) $ is finite-dimensional.
	\end{itemize}	
	
\end{assumption}

Let $ \mathrm{pvd}(A) $ be the perfectly valued derived category of $ A $, i.e.\ $ \mathrm{pvd}(A) $ is the full subcategory of $ \cd(A) $ whose objects are the perfectly valued dg $ A $-modules. Since $ A $ is homologically smooth, $ \mathrm{pvd}(A) $ is a full subcategory of $ \mathrm{per}A $ (see \cite[Lemma 4.1]{Bk2008}). We denote by $ e $ the idempotent $ f(\boldmath{1}_{B}) $ and by $ i\colon eAe\hookrightarrow A $ the canonical inclusion of dg algebras.

\begin{Def}\rm\label{Relative cluster category}
	Let  $ \mathrm{pvd}_{B}(A) $ be the full triangulated subcategory of $ \mathrm{pvd}(A) $ defined as the kernel of the restriction functor $ i_{*}\colon\cd(A)\to\cd(eAe) $. The \emph{relative $ n $-cluster category} $ \mathcal{C}_{n}(A,B) $ is defined as the following Verdier quotient
	$$ 
	\mathcal{C}_{n}(A,B)=\mathrm{per}A/\mathrm{pvd}_{B}(A).
	$$
\end{Def}
We denote by $ \pi^{rel} $ the canonical quotient functor $ \mathrm{per}A \to \mathcal{C}_{n}(A,B) $.

\subsection{Gluing $ t $-structures}
Let $ G\colon\mathcal{B}\to\mathcal{A} $ be a dg functor. Let $ \mathcal{A}/\mathcal{B} $ be the homotopy cofiber of $ G $ in $ \dgcat $. Then the dg category $ \mathcal{A}/\mathcal{B} $ can be computed as the Drinfeld dg quotient of $ \mathcal{A} $ by its full dg subcategory $ \Ima(G) $, where $ \Ima(G) $ is the full dg subcategory of $ \mathcal{A} $ whose objects are the $ y\in\mathcal{A} $ such that there exists $ x\in\mathcal{B} $ and an isomorphism $ G(x)\cong y $ in $ \rm H^{0}(\mathcal{A}) $. We denote by $ i $ the dg inclusion functor $ \Ima(G)\hookrightarrow\mathcal{A} $ and by $ p $ the quotient functor $ \mathcal{A}\twoheadrightarrow\mathcal{A}/\mathcal{B} $.

\begin{Prop}\label{Dg quotoent to Recollemnet}\cite[Theorem 5.1.3]{CC2019}
	We have the following recollement of derived categories
	
	\begin{align}\label{Recollement}
		\xymatrix{
			&\mathcal{D}(\mathcal{A}/\mathcal{B} )\ar[r]^<<<<{p_{*}=p_{!}}&\mathcal{D}
			(\mathcal{A})\ar[r]^<<<<{i_{*}=i_{!}}\ar@/^2pc/[l]^{p^{!}}\ar@/_2pc/[l]_{p^{*}}&\mathcal{D}(\Ima(G) )\ar@/^2pc/[l]^{i^{!}}\ar@/_2pc/[l]_{i^{*}}.
		}
	\end{align}
	The respective triangle functors are explicitly given as follows:
	\begin{align*}
		\xymatrix{
			p^{*}=?\otimes_{\mathcal{A}}^{\mathbf{L}}\mathcal{A}/\mathcal{B}& p_{*}=\mathbf{R}\Hom_{\mathcal{A}/\mathcal{B}}(\mathcal{A}/\mathcal{B},?)\simeq\ ?\otimes_{\mathcal{A}/\mathcal{B}}^{\mathbf{L}}\mathcal{A}/\mathcal{B}=p_{!} & p^{!}=\mathbf{R}\Hom_{\mathcal{A}}(\mathcal{A}/\mathcal{B},?)\\
			i^{*}=?\otimes_{\rm Im(G)}^{\mathbf{L}}\mathcal{A}& i_{*}=\mathbf{R}\Hom_{\mathcal{A}}(\mathcal{A},?)\simeq\ ?\otimes_{\mathcal{A}}^{\mathbf{L}}\mathcal{\mathcal{A}}=i_{!} & i^{!}=\mathbf{R}\Hom_{\Ima(G)}(\mathcal{A},?)
		}
	\end{align*}
	Consequently, we have a triangle equivalence up to direct summands
	\begin{align*}
		\xymatrix{
			\mathrm{per}(\mathcal{A})/\mathrm{per}(\Ima(G))\ar[r]^<<<<<{p^{*}}&\mathrm{per}(\mathcal{A}/\mathcal{B}).
		}
	\end{align*}
\end{Prop}

\begin{Thm}\cite[Gluing $ t $-structures] {BBD1982}\label{Gluing t-structure}
	Suppose that we have the following recollement of triangulated categories
	\begin{align*}
		\xymatrix{
			\mathcal{U}\ar[r]^{i}&\mathcal{T}\ar[r]^{e}\ar@/_1.3pc/[l]_{q}\ar@/^1.3pc/[l]^{p}&\mathcal{V}\ar@/_1.3pc/[l]_{j}\ar@/^1.3pc/[l]^{r}
		}.
	\end{align*} Let $ (\mathcal{U}^{\leqslant0},\mathcal{U}^{\geqslant0})  $ be a t-structure in $ \mathcal{U} $ and $ (\mathcal{V}^{\leqslant0},\mathcal{V}^{\geqslant0})  $ be a t-structure in $ \mathcal{V} $.
	Then we have a canonical t-structure in $ \mathcal{T} $ defined as follows:
	\begin{align*}
		\xymatrix{
			\mathcal{T}^{\leqslant n}=\{X\in\mathcal{T}|e(X)\in\mathcal{V}^{\leqslant n} \ and \  q(X)\in\mathcal{U}^{\leqslant n}\}\\
			\mathcal{T}^{\geqslant n}=\{X\in\mathcal{T}|e(X)\in\mathcal{V}^{\geqslant n} \ and \  p(X)\in\mathcal{U}^{\geqslant n}\}
			.}
	\end{align*}	
\end{Thm}
We say that the t-structure $ (\ct^{\leqslant n},\ct^{\geqslant n}) $ on $ \ct $ is \emph{glued} from the given t-structure on $ \cu $ and $ \cv $.

\bigskip

For any object $ X $ in $\mathcal{T} $, the canonical distinguished triangle for $ X $ with respect to the glued $ t $-structure can be constructed as follows: Let $ X $ be an object in $\mathcal{T} $. We have a distinguished triangle in $ \mathcal{V} $,
$$ \tau^{\mathcal{V}}_{\leqslant0}(e(X))\to e(X)\to\tau_{\mathcal{V}}^{\geqslant1}(e(X))\to\Si\tau^{\mathcal{V}}_{\leqslant0}(e(X)).$$
Hence we obtain a distinguished triangle
$$ Y\xlongrightarrow{f} X\to r(\tau^{\mathcal{V}}_{\geqslant1}e(X))\to \Si Y, $$ where $ X\to r(\tau^{\mathcal{V}}_{\geqslant1}e(X)) $ is the composition $ X\to r(e(X))\to r(\tau^{\mathcal{V}}_{\geqslant1}e(X)) .$

Similarly, we have a distinguished triangle in $ \mathcal{U} $,
$$ \tau^{\mathcal{U}}_{\leqslant0}(q(Y))\to q(Y)\to\tau^{\mathcal{U}}_{\geqslant1}(q(Y))\to\Si\tau^{\mathcal{U}}_{\leqslant0}(q(Y)) .$$ Hence we obtain a distinguished triangle
$$ Z\xlongrightarrow{g} Y\to i(\tau^{\mathcal{U}}_{\geqslant1}q(Y))\to \Si Z, $$ where $ Y\to i(\tau_{\mathcal{U}}^{\geqslant1}q(Y)) $ is the composition $ Y\to i(q(Y))\to i(\tau^{\mathcal{U}}_{\geqslant1}q(Y)) .$
Thus, we have the following octahedron
\begin{align*}
	\xymatrix{
		Z\ar[r]^{g}\ar@{=}[d]&Y\ar[d]^{f}\ar[r]&i(\tau^{\mathcal{U}}_{\geqslant1}q(Y))\ar[d]\ar[r]&\Si Z\ar@{=}[d]\\
		Z\ar[r]^{f\circ g}&X\ar[r]\ar[d]&U\ar[r]\ar[d]&\Si Z\ar[d]^{\Si g}\\
		&r(\tau^{\mathcal{V}}_{\geqslant1}e(X))\ar@{=}[r]\ar[d]&r(\tau^{\mathcal{V}}_{\geqslant1}e(X))\ar[r]\ar[d]&\Si Y\\
		&\Si Y\ar[r]&\Si i(\tau^{\mathcal{U}}_{\geqslant1}q(Y))\quad.
	}
\end{align*}

Then one can show that we have $ Z\in\mathcal{T}^{\leqslant0} $ and $ U\in\mathcal{T}^{\geqslant1} $. Thus, for any $ X\in\mathcal{T} $, the canonical distinguished triangle for $ X $ with respect to the glued t-structure is given by
$$ Z\to X\to U\to \Si Z.$$

\bigskip

\bigskip
Let $ e=f(\boldmath{1}_{B}) .$ We denote by $ \overline{A} $ the homotopy cofiber of $ f\colon B\ra A $.
Consider the following homotopy cofiber sequence in $ \dgcat $
\begin{align*}
	\xymatrix{
		B\ar[r]^-{f}\ar[d]&A\ar[d]^-{p}\\
		0\ar[r]&\overline{A}.
	}
\end{align*} 
Then we have the following immediate Proposition.
\begin{Prop}\cite[Corollary 7.1]{BD2019}\label{Prop: smooth of cofiber}
	The homotopy cofiber $ \overline{A} $ is homologically smooth and it has a canonical $ (n+1) $-Calabi--Yau structure.
\end{Prop}

\begin{Prop}\label{Prop: cofiber is fd}
	The homotopy cofiber $ \overline{A} $ is connective and $ H^{0}(\overline{A}) $ is finite-dimensional.
\end{Prop}
\begin{proof}
	By the construction of the Drinfeld dg quotient and the assumption that $ A $ is connective, the dg algebra $ \overline{A} $ is also connective. By~\cite[Theorem 5.8]{BCL2018}, the 0-th cohomology $ H^{0}(\overline{A}) $ is isomorphic to $ H^{0}(A)/\langle e\rangle $. Thus, the algebra $ H^{0}(\overline{A}) $ is finite-dimensional.
	
\end{proof}

\begin{Cor}\label{Recollement of dg algebras}
	We have the following recollement
	\begin{align}\label{dg algebra recollement}
		\xymatrix{
			\mathcal{D}(\overline{A})\ar[r]^{p_{*}}&\mathcal{D}(A)\ar[r]^-{i_{*}}\ar@/_2pc/[l]_{p^{*}}\ar@/^2pc/[l]^{p^{!}}&\mathcal{D}(eAe)\ar@/_2pc/[l]_{i^{*}}\ar@/^2pc/[l]^{i^{!}}
		},
	\end{align}
	where the respective triangle functors are explicitly given as follows
	\begin{align*}
		\xymatrix{
			p^{*}=?\otimes_{A}^{\mathbf{L}}\overline{A}& p_{*}=\mathbf{R}\Hom_{\overline{A}}(\overline{A},?)\simeq\ ?\otimes_{\overline{A}}^{\mathbf{L}}\overline{A}=p_{!} & p^{!}=\mathbf{R}\Hom_{A}(\overline{A},?)\\
			i^{*}=?\otimes_{eAe}^{\mathbf{L}}eA& i_{*}=\mathbf{R}\Hom_{A}(eA,?)\simeq\ ?\otimes_{A}^{\mathbf{L}}Ae=i_{!} & i^{!}=\mathbf{R}\Hom_{eAe}(Ae,?).
		}
	\end{align*}
	Consequently, we have a triangle equivalence
	\begin{align*}
		\xymatrix{
			i^{*}\colon\mathrm{per}(eAe)\iso\langle eA\rangle_{\mathrm{per}A}
		}
	\end{align*}
and a triangle equivalence up to direct summands
	\begin{align*}
		\xymatrix{
			p^{*}\colon\mathrm{per}(A)/\mathrm{per}(eAe)\ra\mathrm{per}(\overline{A}),
		}
	\end{align*}
	where $ \langle eA\rangle_{\per A} $ is the thick subcategory of $ \mathrm{per}A $ generated by $ eA $.
\end{Cor}
\begin{proof}
	This is a special case of Proposition~\ref{Dg quotoent to Recollemnet}. For more details we refer the reader to \cite[Corollary 2.12]{KY2016}.
%
	
\end{proof}


\begin{Def}\rm
	Let $ \mathcal{A} $ be an abelian $ k $-category. For $ i\in\mathbb{Z} $ and for a complex $ M $ of objects in $ \mathcal{A} $, we define the \emph{standard truncations} $ \tau_{\leqslant i}M $ and $ \tau_{>i}M $ by
	
	\[(\tau_{\leqslant i}M)^{j}=\left\{
	\begin{aligned}
		M ^{j}&&\text{if}&&j<i \\
		\ker(d^{i}_{M})&&\text{if}&&j=i \\
		0 &&\text{if}&& j>i
	\end{aligned}
	\right.\qquad\qquad 
	(\tau_{> i}M)^{j}=\left\{
	\begin{aligned}
		0&&\text{if}&&j<i \\
		\frac{M}{\ker(d^{i}_{M})}&&\text{if}&&j=i \\
		M^{j} &&\text{if}&& j>i
	\end{aligned}
	\right.\]
	
	Their respective differentials are inherited from $ M $.
	Notice that $ \tau_{\leqslant i}(M)$ is a subcomplex of $ M $ and $ \tau_{>i}(M) $ is the corresponding quotient complex. Thus we have a sequence, which is componentwise short exact,
	
	$$ 0\to\tau_{\leqslant i}(M)\to M\to\tau_{>i}(M)\to0 .$$

	Moreover, taking standard truncations behaves well with respect to cohomology, i.e.\ we have
	\[H^{j}(\tau_{\leqslant i}M)=\left\{
	\begin{aligned}
		H^{j}(M)&&\text{if}&&j\leqslant i, \\
		0 &&\text{if}&& j>i.
	\end{aligned}
	\right.\]
\end{Def}

\subsection{Relative $ t $-structure}
Let $ f\colon B\rightarrow A $ be a dg $ k $-algebra morphism satisfying the Assumptions~\ref{Relative assumption}. Then the map of complexes $ \tau_{\leqslant 0}A\to A $ is a quasi-isomorphism of dg algebras. Thus, we may assume that the components $ A^{p} $ vanish for all $ p>0 $. Then the canonical projection $ A\rightarrow H^{0}(A) $ is a homomorphism of dg algebras. We view a module over $ H^{0}(A)  $ as a dg module over $ A $ via this homomorphism. This defines a natural functor $ \mathrm{Mod}H^{0}(A)\to \mathcal{D}(A) $ which induces an equivalence from $ \mathrm{Mod}H^{0}(A) $ onto the heart of the canonical t-structure on $ \cd(A) $ whose left aisle (see~\cite{BK1988}) is the full subcategory on the dg modules $ M $ such that $ H^{p}M=0 $ for all $ p>0 $.

Let $ \mathrm{Mod}_{B}H^{0}(A) $ be the full subcategory of $ \mathrm{Mod}H^{0}(A) $ whose objects are the right $ H^{0}(A) $-modules $ X $ such that the restriction of $ X $ to $ H^{0}(eAe) $ vanishes. Thus, we get a natural functor $ i\colon \mathrm{Mod}_{B}H^{0}(A)\to \mathcal{D}(A) $.

On $ \mathcal{D}(\overline{A}) $ we take the canonical t-structure with heart $ \heartsuit=\mathrm{Mod}H^{0}(\overline{A}) $ and on $ \mathcal{D}(eAe) $ we take the trivial t-structure  whose left aisle is $ \mathcal{D}(eAe) $.
We deduce the following corollary from Theorem~~\ref{Gluing t-structure}.

\begin{Cor}\label{Relative t-structure}
	There is a t-structure on $ \mathcal{D}(A) $ obtained by gluing the canonical t-structure on $ \mathcal{D}(\overline{A}) $ with the trivial t-structure on $ \mathcal{D}(eAe) $ through the recollement diagram (\ref{dg algebra recollement}). We denote by $ (\mathcal{D}(A)_{rel}^{\leqslant 0},	\mathcal{D}(A)_{rel}^{\geqslant 0}) $ the glued t-structure on $ \cd(A) $.
	Here, for any $ k\in\mathbb{Z} $, 
	\begin{align*}
		\xymatrix{
			\mathcal{D}(A)_{rel}^{\leqslant k}=\{X\in\mathcal{D}(A)|H^{l}(p^{*}X)=0,\,\forall l>k\},
		}
	\end{align*}
	\begin{align*}
		\xymatrix{
			\mathcal{D}(A)_{rel}^{\geqslant k}=\{X\in\mathcal{D}(A)|i_{*}(X)=0,H^{l}(p^{!}X)\cong H^{l}(X)=0,\,\forall l<k\}.
		}
	\end{align*}
	and so the heart $ \heartsuit^{rel} $ of this glued $ t $-structure is equivalent to $ \mathrm{Mod}_{B}H^{0}(A) $. 
	
\end{Cor}
\begin{proof}
	The functor $ p_{*}\colon\mathcal{D}(\overline{A})\to \ker(i_{*}) $ is an equivalence of triangulated categories. So the restrictions of the adjoints $ p^{*} $ and $ p^{!} $ to $ \ker(i_{*}) $ give quasi-inverses of $ p_{*}\colon\mathcal{D}(\overline{A})\to \ker(i_{*}) $. Thus, we have
	\begin{equation}
		\begin{split}
			\heartsuit^{rel}=\mathcal{D}(A)_{rel}^{\leqslant 0}\cap\mathcal{D}(A)_{rel}^{\geqslant 0} &= \{X\in\mathcal{D}(A)|i_{*}(X)=0,H^{l}(p^{!}X)=H^{l}(p^{*}X)=0,\ \forall l\neq 0\}\\
			&= \{X\in\mathcal{D}(A)|i_{*}(X)=0,H^{l}(X)=0,\ \forall l\neq 0\}.
		\end{split}
	\end{equation}
	
	The morphism of dg algebras $ A\to H^{0}(A) $ induces a natural functor $ i\colon\mathrm{Mod}_{B}H^{0}(A)\to\heartsuit^{rel} $. Let $ X $ be an object in $ \heartsuit^{rel}\subseteq \ker(i_{*}) $. Then $ X $ is concentrated in degree 0 and $ X $ is isomorphic to an object $ X' $ in $ \mathrm{Mod}H^{0}(A) $. Since we know that $ i_{*}(X) $ is acyclic, $ X' $ is also in $ \mathrm{Mod}_{B}H^{0}(A) $. This shows the denseness of $ i $. The fully faithfulness follows from the following commutative square with three fully faithful functors
	\begin{align*}
		\xymatrix{
			\mathrm{Mod}_{B}H^{0}(A)\ar@{^{(}->}[r]\ar[d]&\mathrm{Mod}H^{0}(A)\ar[d]^{\simeq}\\
			\heartsuit^{rel}\ar@{^{(}->}[r]&\heartsuit.
		}
	\end{align*}
\end{proof}

\bigskip
We will call $ (\mathcal{D}(A)_{rel}^{\leqslant 0},	\mathcal{D}(A)_{rel}^{\geqslant 0}) $ the \emph{relative t-structure} on $ \mathcal{D}(A) $. We illustrate this glued t-structure in the following picture
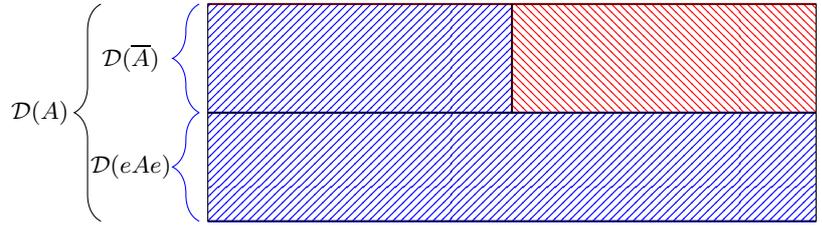
\begin{align*}
	\begin{tikzpicture}[scale=1.6]
		\draw [thick,red] (0,1.8)--(5,1.8);
		\draw [thick,blue] (0,0.9)--(5,0.9);
		\draw [thick,blue] (0,0)--(5,0);
		\draw[thick,red](2.5,1.8)--(2.5,0.9);
		\draw[decorate,decoration={brace,amplitude=10pt},xshift=-25pt,yshift=0pt](0,0)--(0,1.8)node [black,midway,xshift=-0.8cm] {\footnotesize $\cd(A)$};
		\draw[decorate,blue,decoration={brace,amplitude=10pt},xshift=-2pt,yshift=0pt](0,0)--(0,0.9)node [black,midway,xshift=-0.9cm] {\footnotesize $\cd(eAe)$};
		\draw[decorate,blue,decoration={brace,amplitude=10pt},xshift=-2pt,yshift=0pt](0,0.9)--(0,1.8)node [black,midway,xshift=-0.9cm] {\footnotesize $\cd(\overline{A})$};
		\draw[pattern=north east lines, pattern color=blue] (0,0) rectangle (5,0.9);
		\draw[pattern=north east lines, pattern color=blue] (0,0.9) rectangle (2.5,1.8);
		\draw[pattern=north west lines, pattern color=red] (2.5,0.9) rectangle (5,1.8);
	\end{tikzpicture}\,,
\end{align*} 
where the blue region represents the subcategory $ \cd(A)_{rel}^{\leqslant0} $ and the red region represents the subcategory $ \cd(A)_{rel}^{\geqslant0} $.

By Corollary~\ref{Recollement of dg algebras}, the canonical triangle for an object $ X\in\mathcal{D}(A) $ with respect to the glued t-structure can be constructed as follows: Let $ X $ be an object in $\mathcal{D}(A) $. We have the following canonical triangle
\begin{align*}
	\xymatrix{
		i^{*}(i_{*}X)\ar[r]&X\ar[r]&p_{*}(p^{*}X)\ar[r]&\Si i^{*}(i_{*}X)
	}.
\end{align*}
For the object $ p^{*}X\in\mathcal{D}(\overline{A}) $, we have the following canonical triangle triangle
\begin{align*}
	\xymatrix{
		\tau_{\leqslant k}(p^{*}X)\ar[r]&p^{*}X\ar[r]&\tau_{>k}(p^{*}X)\ar[r]&\Si \tau_{\leqslant k}(p^{*}X).
	}
\end{align*}
Then we get a triangle in $ \mathcal{D}(A) $
\begin{align*}
	\xymatrix{
		p_{*}(\tau_{\leqslant k}(p^{*}X))\ar[r]&p_{*}(p^{*}X)\ar[r]&p_{*}(\tau_{>k}(p^{*}X))\ar[r]&\Si p_{*}(\tau_{\leqslant k}(p^{*}X))
	}.
\end{align*}
Thus, by the octahedral axiom, there exists an object $ \tau_{\leqslant k}^{rel}X $ in $ \mathcal{D}_{rel}^{\leqslant k}(A) $ such that we have an isomorphism $ p^{*}(\tau_{\leqslant k}^{rel}X)\cong\tau_{\leqslant k}(p^{*}X) $ and the following morphism of distinguished triangles 
\begin{align*}
	\xymatrix{
		p_{*}(\tau_{\leqslant k}(p^{*}X))\ar[r]&0\ar[r]&\Si p_{*}(\tau_{\leqslant k}(p^{*}X))\ar[r]&\\
		\tau_{\leqslant k}^{rel}X\ar[u]\ar[r]&X\ar[u]\ar[r]&p_{*}(\tau_{>k}(p^{*}X))\ar[u]\ar[r]&\\
		i^{*}(i_{*}X)\ar[u]\ar[r]&X\ar[u]^{\boldmath{1_{X}}}\ar[r]&p_{*}(p^{*}X)\ar[u]\ar[r]&\\
		\Si^{-1}p_{*}(\tau_{\leqslant k}(p^{*}X))\ar[u]\ar[r]&0\ar[u]\ar[r]&p_{*}(\tau_{\leqslant k}(p^{*}X))\ar[u]\ar[r]&
		.}
\end{align*}

\begin{Def}\rm\label{Relative truncation functor}
	We define the \emph{relative truncation functor} $ \tau^{rel}_{>k} $ to be the following composition
	$$
	\xymatrix{
		\tau^{rel}_{> k}\colon\mathcal{D}(A)\ar[r]^<<<<<{p^{*}}&\mathcal{D}(\overline{A})\ar[r]^{\tau_{> k}}&\mathcal{D}(\overline{A})\ar[r]^{p_{*}}&\mathcal{D}(A).
	}$$
\end{Def}
Thus, for any $ X\in\mathcal{D}(A) $, we have a canonical triangle in $ \cd(A) $
$$ \tau_{\leqslant k}^{rel}X\to X\to \tau_{>k}^{rel}\to\Si\tau_{\leqslant k}^{rel}X $$
such that $ \tau_{\leqslant k}^{rel}X $ belongs to $ \mathcal{D}(A)_{rel}^{\leqslant k} $ and $ \tau^{rel}_{> k}(X)=p_{*}(\tau_{>k}(p^{*}X)) $ belongs to $ \mathcal{D}(A)_{rel}^{\geqslant k+1} $. Moreover, the object $ \tau^{rel}_{> k}(X) $ lies in $ \pvd_{B}(A) $ since it is the essential image of $ p_{*} $.

%
%

\subsection{The restriction of the relative $ t $-structure}

\begin{Prop}\cite[Proposition 2.5]{KY2016}
	For each $ p\in\mathbb{Z} $, the space $ H^{p}(A) $ is finite dimensional. Consequently, the category $ \mathrm{per}A $ is Hom-finite.
\end{Prop}

\begin{Prop}\label{Relative t-structure on perA}
	The relative t-structure on $ \mathcal{D}(A) $ restricts to $ \mathrm{per}A $.
\end{Prop}
\begin{proof}
	Let $ X $ be in $ \mathrm{per}A $ and look at the canonical triangle with respect to the relative t-structure on $ \mathcal{D}(A) $
	$$ \tau^{rel}_{\leqslant0}X\rightarrow X\rightarrow\tau^{rel}_{>0}X\rightarrow\Si\tau^{rel}_{\leqslant0}X ,$$ where $ \tau^{rel}_{>0}X=p_{*}(\tau_{>0} p^{*}(X)) $. By Proposition~\ref{Prop: cofiber is fd}, the algebra $ H^{0}(\overline{A}) $ is finite-dimensional. Then by~\cite[Proposition 2.5]{KY2016}, the category $ \per(\overline{A}) $ is also Hom-finite. Thus, the space $$ H^{l}(\tau^{rel}_{>0}X)=\Hom_{\cd(\ca)}(A,\Si^{l}p_{*}\tau_{>0}p^{*}X)\simeq\Hom_{\cd(\overline{A})}(\overline{A},\Si^{l}\tau_{>0}p^{*}X) $$ equals zero or $ H^{l}(\tau_{>0}p^{*}X) $ which is finite-dimensional. Thus, the object $ \tau^{rel}_{>0}X $ is in $ \mathrm{pvd}(A) $ and so in $ \mathrm{per}A $. Since $ \mathrm{per}A $ is a triangulated subcategory, it follows that $ \tau^{rel}_{\leqslant0}X $ also lies in $ \mathrm{per}A $.
	
\end{proof}

\begin{Prop}\label{Relative t-structure on D_{fd,B}}
	Let $ \mathrm{pvd}_{B}(A)_{rel}^{\leqslant0} $ be the full subcategory of $ \mathcal{D}(A)_{rel}^{\leqslant0} $ whose objects are the $ M\in \mathrm{pvd}(A) $ whose restriction along $ i\colon eAe\hookrightarrow A $ is acyclic. Then $ (\mathrm{pvd}_{B}(A)_{rel}^{\leqslant0},\mathcal{D}(A)_{rel}^{\geqslant0}) $ is a t-structure on $ \mathrm{pvd}_{B}(A) $ and the corresponding heart is equivalent to $ \mathrm{mod}_{B}H^{0}(A) $, where $ \mathrm{mod}_{B}H^{0}(A) $ is the full subcategory of $ \mathrm{Mod}_{B}H^{0}(A) $ whose objects are the finite-dimensional $ H^{0}(A) $-modules. Moreover, the triangulated category $ \mathrm{pvd}_{B}(A) $ is generated by its heart.
	
\end{Prop}
\begin{proof}
	Let $ n\in\mathbb{Z} $. For any object $ X\in \mathrm{pvd}_{B}(A) $, we have the following triangle
	\begin{align*}
		\xymatrix{
			\tau^{rel}_{\leqslant 0}X\ar[r]&X\ar[r]&\tau^{rel}_{>0}X\ar[r]&
		}
	\end{align*} with $ \tau^{rel}_{\leqslant 0}X\in\mathcal{D}(A)_{rel}^{\leqslant 0} $ and $ \tau^{rel}_{>0}X\in\mathcal{D}(A)_{rel}^{>0}\subseteq \mathrm{pvd}_{B}(A) $. So the object $ \tau^{rel}_{\leqslant 0}X $ is also in $ \mathrm{pvd}_{B}(A) $. This is the triangle required to show that $ (\mathrm{pvd}_{B}(A)_{rel}^{\leqslant0},\mathcal{D}(A)_{rel}^{\geqslant0}) $ is a t-structure.
	
	To show the second statement, let $ M $ be an object in $ \mathrm{pvd}_{B}(M) $. Let $ k\leqslant m $ be integers such that $ H^{l}(M)\neq0 $ only for $ l\in[k,m] $. We use induction on $ m-k $. If $ m-k=0 $, then a shift of $ M $ is in the heart. Now suppose $ m-k>0 $. Then the relative truncations yield a triangle in $ \mathrm{pvd}_{B}(A) $
	$$ \tau^{rel}_{\leqslant k}M\rightarrow M\rightarrow \tau^{rel}_{> k}M\rightarrow \Si\tau^{rel}_{\leqslant k}M. $$
	The homology of $ \tau^{rel}_{\leqslant k}M $ is concentrated in degree $ k $. Thus, the object $ \tau^{rel}_{\leqslant k}M $ belongs to a shifted copy of the heart. Moreover, the homology of $ \tau^{rel}_{> k}M $ is bounded between degrees $ k+1 $ and $ m $. By the induction hypothesis, the object $ \tau^{rel}_{> k}M $ is contained in the triangulated subcategory generated by the heart. Therefore the same holds for $ M $.
\end{proof}
\bigskip

Recall that we have defined $ \cc_{n}(A,B)=\per A/\pvd_{B}(A) $.
\begin{Prop}\cite[Proposition 7.1.4]{Am2008}\label{Home space in RCA}
	Under the projection functor $ \pi^{rel}\colon \mathrm{per}A\to\mathcal{C}_{n}(A,B) $, for any $ X $ and $ Y $ in $ \mathrm{per}A $, we have 
	\begin{align*}
		\xymatrix{
			\Hom_{\mathcal{C}_{n}(A,B)}(\pi^{rel} X,\pi^{rel} Y)=\varinjlim_{k\leqslant0} \Hom_{\mathcal{D}(A)}(\tau_{\leqslant k}^{rel}X,\tau_{\leqslant k}^{rel}Y).
		}
	\end{align*}
\end{Prop}
\begin{proof}
	Let $ X $ and $ Y $ be in $ \per A $. An element of $ \varinjlim_{k\leqslant0} \Hom_{\mathcal{D}(A)}(\tau_{\leqslant k}^{rel}X,\tau_{\leqslant k}^{rel}Y) $ is an equivalence class of morphisms $ \tau_{\leqslant k}^{rel}X\to\tau_{\leqslant k}^{rel}Y $. Two morphisms $ f\colon \tau_{\leqslant k}^{rel}X\to\tau_{\leqslant k}^{rel}Y $ and $ g\colon \tau_{\leqslant m}^{rel}X\to\tau_{\leqslant m}^{rel}Y $ with $ m\geqslant k $ are equivalent if there is a commutative square
	\begin{align*}
		\xymatrix{
			\tau_{\leqslant k}^{rel}X\ar[r]^{f}\ar[d]&\tau_{\leqslant k}^{rel}Y\ar[d]\\
			\tau_{\leqslant m}^{rel}X\ar[r]^{g}&\tau_{\leqslant m}^{rel}Y,
		}
	\end{align*}
	where the vertical arrows are the canonical morphisms.
	
	Suppose that $ f $ is a morphism $ f\colon \tau_{\leqslant k}^{rel}X\to\tau_{\leqslant k}^{rel}Y $. We can form the following morphism from $ X $ to $ Y $ in $ \mathcal{C}_{n}(A,B) $
	\begin{align*}
		\xymatrix{
			&\tau_{\leqslant k}^{rel}X\ar@{=>}"2,1"\ar@{.>}[r]^{f}\ar"2,3"&\tau_{\leqslant k}^{rel}Y\ar@{.>}[d]\\
			X&&Y,
		}
	\end{align*}
	where the morphisms $ \tau_{\leqslant k}^{rel}X\to X $ and $ \tau_{\leqslant k}^{rel}Y\to Y $ are the canonical morphisms. Here we use the fact that the cone $ \tau^{rel}_{>k}X $ of the morphism $ \tau^{rel}_{\leqslant k}X\ra X $ lies in $ \pvd_{B}(A) $. Hence the above diagram defines a morphism in $ \cc_{n}(A,B) $.
	
	If $ f\colon \tau_{\leqslant k}^{rel}X\to\tau_{\leqslant k}^{rel}Y $ and $ g\colon \tau_{\leqslant m}^{rel}X\to\tau_{\leqslant m}^{rel}Y $ with $ m\geqslant k $ are equivalent, there is an equivalence of diagrams
	\begin{align*}
		\xymatrix{
			&\tau_{\leqslant k}^{rel}X\ar@{=>}"2,1"\ar@{.>}[r]^{f}\ar"2,3"&\tau_{\leqslant k}^{rel}Y\ar@{.>}[d]\\
			X&&Y\\
			&\tau_{\leqslant m}^{rel}X\ar@{.>}[r]^{g}\ar@{=>}"2,1"\ar"2,3"&\tau_{\leqslant m}^{rel}Y\ar@{.>}[u]\,.
		}
	\end{align*}
	Thus, we have a well-defined map from $ \varinjlim_{k\leqslant0} \Hom_{\mathcal{D}(A)}(\tau_{\leqslant k}^{rel}X,\tau_{\leqslant k}^{rel}Y) $ to $ \Hom_{\mathcal{C}_{n}(A,B)}(\pi^{rel} X,\pi^{rel} Y) $ which is  injective.
	
	Let $ h\colon X\to Y $ be a morphism in $ \Hom_{\mathcal{C}_{n}(A,B)}(\pi^{rel} X,\pi^{rel}Y) $. Suppose that $ h $ can be represented by the following right fraction
	\begin{align*}
		\xymatrix{
			&X'\ar@{=>}"2,1"_{s}\ar"2,3"^{h'}&\\
			X&&Y.
		}
	\end{align*}
	Let $ X'' $ be the cone of $ s $. It is an object of $ \mathrm{pvd}_{B}(A) $ and therefore lies in $ \mathcal{D}^{rel}_{>n} $ for some $ l\ll 0 $. Therefore there are no morphisms from $ \tau_{\leqslant l}^{rel}X $ to $ X'' $ and we have the following factorization
	\begin{align*}
		\xymatrix{
			&\tau_{\leqslant l}^{rel}X\ar@{.>}"2,1"_{s'}\ar[d]\ar"2,3"^{0}&\\
			X'\ar[r]&X\ar[r]&X''\ar[r]&\Si X'.
		}
	\end{align*}
	We obtain an isomorphism of diagrams
	\begin{align*}
		\xymatrix{
			&X'\ar@{=>}"2,1"_{s}\ar"2,3"^{h'}&\\
			X&&Y\\
			&\tau_{\leqslant l}^{rel}X\ar@{=>}"2,1"\ar"1,2"_{s'}\ar"2,3"_{f'=h's'}&.
		}
	\end{align*}
	Since $ \tau_{\leqslant l}^{rel}X $ is in $ \mathcal{D}(A)_{rel}^{\leqslant l} $ and $ \tau_{> l}^{rel}Y $ is in  $ \mathcal{D}(A)_{rel}^{>l} $, the morphism $ f'\colon\tau_{\leqslant l}^{rel}X\to Y $ induces a morphism $ f\colon\tau_{\leqslant l}^{rel}X\to \tau_{\leqslant l}^{rel}Y $ which lifts the given morphism. Thus the map from $ \varinjlim_{l\leqslant0} \Hom_{\mathcal{D}(A)}(\tau_{\leqslant l}^{rel}X,\tau_{\leqslant l}^{rel}Y) $ to $ \Hom_{\mathcal{C}_{n}(A,B)}(\pi^{rel} X,\pi^{rel} Y) $ is surjective.
	
\end{proof}

\subsection{SMC reduction}
Let $ \mathcal{T} $ be a Krull--Schmidt triangulated category and $ \mathcal{S} $ a subcategory of $ \mathcal{T} $. 

\begin{Def}\rm\cite[Deﬁnition 2.4]{Jin2020}
	We call $ \mathcal{S} $ a \emph{pre-simple-minded collection (pre-SMC)} if for any $ X,Y\in\mathcal{S} $, the following conditions hold.
	\begin{itemize}
		\item[(1)] $ \Hom_{\mathcal{T}}(X,\Si^{<0}Y)=0 $;
		\item[(2)] $ \mathrm{dim}_{k}\Hom_{\mathcal{T}}(X,Y)=\delta_{X,Y}.$
	\end{itemize}	
	We call $ \mathcal{S} $ a \emph{simple-minded collection (SMC)} if $ \mathcal{S} $ is a pre-SMC and moreover, $ \thick(\mathcal{S})=\mathcal{T} $.
\end{Def}

Let $ \cs $ be a pre-SMC. The \emph{SMC reduction} of $ \mathcal{T} $ with respect to $ \mathcal{S} $ is defined as the following Verdier quotient~\cite[Section 3.1]{Jin2020}
$$ \mathcal{U}\coloneqq\mathcal{T}/\mathrm{thick}(\mathcal{S}) .$$

The subcategory $ \thick(\mathcal{S}) $ admits a natural t-structure $ (\mathcal{X}_{\mathcal{S}},\mathcal{Y}_{\mathcal{S}}) $, where $ \mathcal{X}_{\mathcal{S}} $ is the smallest extension closed subcategory of $ \mathcal{T} $ containing any non-negative shift of $ \mathcal{S} $ and $ \mathcal{Y}_{\mathcal{S}} $ is the smallest extension closed subcategory of $ \mathcal{T} $ containing any non-positive shift of $ \mathcal{S} $ (see~\cite[Corollary 3 and Proposition 4]{A2009},\cite[Proposition 5.4]{KOY2014} or \cite{RR2017}). Then the corresponding heart is denoted by $ \mathcal{H}_{\mathcal{S}} $. It is equal to the smallest extension closed subcategory of $ \mathcal{T} $ containing $ \mathcal{S} $.

Consider the following mild conditions:
\begin{itemize}
	\item[$ (R1) $] The heart $ \mathcal{H}_{\mathcal{S}} $ is contravariantly finite in the Hom-orthogonal subcategory $ (\Si^{>0}\mathcal{S})^{\perp} $ and covariantly finite in $ {}^{\perp}(\Si^{<0}\mathcal{S}) $.
	\item[$ (R2) $] For any $ X\in\mathcal{T} $, we have $ \mathrm{Hom}_{\ct}(X,\Si^{i}\mathcal{H}_{\cs})=0=\mathrm{Hom}_{\ct}(\mathcal{H}_{\cs},\Si^{i}X) $ for $ i\ll0 $.
\end{itemize}

\begin{Prop}\cite[Proposition 3.2]{Jin2020}\label{Equivalence condition of SM}
	The following are equivalent.
	\begin{itemize}
		\item[(1)] $ (\mathcal{X}_{\mathcal{S}},\mathcal{X}_{\mathcal{S}}^{\perp}) $ and $ ({}^{\perp}\mathcal{Y}_{\mathcal{S}},\mathcal{Y}_{\mathcal{S}}) $ are t-structures on $ \mathcal{T} $;
		\item[(2)] $ \mathcal{H}_{\mathcal{S}} $ satisfies the conditions $ (R1) $ and $ (R2) $.
	\end{itemize}
	
\end{Prop}

Let $ \mathcal{W} $ be the following subcategory of $ \mathcal{T} $
$$ \mathcal{W}\coloneqq(\Si^{\geqslant0}\mathcal{S})^{\perp}\cap{}^{\perp}(\Si^{\leqslant0}\mathcal{S}) .$$

\begin{Thm}\cite[Theorem 3.1]{Jin2020}\label{Thm:PSM reduction}
	Assume the assumptions $ (R1) $ and $ (R2) $ hold. Then the composition
	$$ \mathcal{W}\hookrightarrow\mathcal{T}\rightarrow\mathcal{U} $$ is a $ k $-linear equivalence $ \mathcal{W}\iso\mathcal{U} $.
	
\end{Thm}

In our case, since the $ k $-algebra $ H^{0}(A) $ is a finite dimensional $ k $-algebra, we can suppose that $ \boldmath{1}_{A} $ has a decomposition
$$ \boldmath{1}_{A}=e_{1}+\cdots+e_{n} $$into primitive orthogonal idempotents $ e_{i} $ such that $ e=f(\boldmath{1}_{B})=e_{1}+\cdots+e_{k} $ for some $ 0\leqslant k\leqslant n $. Then $ \mathrm{mod}_{B}H^{0}(A) $ is generated by $ \mathcal{S}=\{S_{k+1},S_{k+2},\cdots S_{n}\} $, where $ S_{i} $ is the simple $ H^{0}(A) $ module associated to the idempotent $ e_{i} $.

Then it is easy to see that $ \mathcal{S} $ is a simple-minded collection of $ \mathrm{pvd}_{B}(A) $ and is a pre-simple-minded collection of $ \mathrm{per}A $. 

\begin{Cor}\label{SMC reduction}
	The composition $ \mathcal{W}\hookrightarrow \mathrm{per}A\rightarrow \mathcal{C}_{n}(A,B)=\mathrm{per}A/\mathrm{pvd}_{B}(A) $ is a $ k $-linear equivalence $ \mathcal{W}\iso\mathcal{C}_{n}(A,B) $, where $ \mathcal{W} $ is the following subcategory of $ \mathrm{per}A $
	$$ \mathcal{W}=(\Si^{\geqslant0}\mathcal{S})^{\perp}\cap\,   
	^{\perp}(\Si^{\leqslant0}\mathcal{S}) .$$
	In particular, $ \cc_{n}(A,B) $ is idempotent complete.
\end{Cor}

\begin{proof}
	It suffices to check the conditions (R1) and (R2). For any $ X\in\per A $, it is easy to see that $ \Hom_{\per A}(X,\Si^{i}\ch_{S}) $ vanishes for $ i\ll0 $. By the relative Calabi--Yau duality (Corollary~\ref{Relative CY duality}), the space $ \mathrm{Hom}_{\per A}(\mathcal{H}_{\cs},\Si^{i}X) $ also vanishes for $ i\ll0 $. Therefore $ \ch_{S} $ satisfies the condition (R2). By Lemma~\ref{Lem:S is functorially ﬁnite} below, the category $ \mathrm{mod}_{B}H^{0}(A) $ is functorially finite in $ \per A $. So $ \ch_{S} $ satisfies the condition (R1). Then the claim follows from Theorem~\ref{Thm:PSM reduction}.
\end{proof}

\begin{Lem}\label{Lem:S is functorially ﬁnite}
	Let $ B\ra A $ be morphism between dg $ k $-algebras which satisfies the assumptions~\ref{Relative assumption}. Then $ \mathrm{mod}_{B}H^{0}(A) $ is functorially finite in $ \per A $.
\end{Lem}
\begin{proof}
	Let $ P $ be an object in $ \per A $. Since $ A $ is connective, there is a canonical co-t-structure $ ((\mathrm{per}A)_{\geqslant0},(\mathrm{per}A)_{\leqslant0}) $ on $ \mathrm{per}A $ (\cite[Proposition 2.8]{IY2018}), where $$ (\mathrm{per}A)_{\geqslant0}\coloneqq\bigcup_{n\geqslant0}\Si^{-n}\mathrm{add}A*\cdots*\Si^{-1}\mathrm{add}A*\mathrm{add}A\quad \text{and}\quad(\mathrm{per}A)_{\leqslant0}\coloneqq\bigcup_{n\geqslant0}\mathrm{add}A*\Si\mathrm{add}A*\cdots*\Si^{n}\mathrm{add}A.$$
	Then we have a canonical triangle in $ \per A $
	$$ \sigma_{>0}P\ra P\xrightarrow{t}\sigma_{\leqslant0}P\ra\Si\sigma_{>0}P $$ such that $ \sigma_{>0}P\in(\per A)_{>0} $ and $ \sigma_{\leqslant0}P\in(\mathrm{per}A)_{\leqslant0} $.	Consider the object $ X=\tau_{\geqslant0}(\sigma_{\leqslant0}P)=H^{0}(\sigma_{\leqslant0}P) $. It is easy to see that $ \tau_{\geqslant0}(\sigma_{\leqslant0}P) $ is in $ \mathrm{mod}H^{0}(A) $ and we have a canonical morphism $ f\colon P\xrightarrow{t}\sigma_{\leqslant0}P\ra X $.
	
	Let $ M $ be an object in $ \mathrm{mod}H^{0}(A) $ and $ g\colon P\ra M $ a morphism. Since the space $ \Hom_{\per A}(\sigma_{>0}P,M) $ vanishes, we have $ \Hom_{\per A}(P,M)\simeq\Hom_{\per A}(\sigma_{\leqslant0}P,M) $. Then there exists a morphism $ h\colon X\ra M $ such that the following diagram commutes
	\[
	\begin{tikzcd}
		P\arrow[r,"t"]\arrow[d,swap,"g"]&\sigma_{\leqslant0}P\arrow[r]\arrow[dl]&X\arrow[dll,"h"]\\
		M.&
	\end{tikzcd}
	\]
	This shows that $ \mathrm{mod}H^{0}(A) $ is covariantly finite in $ \per A $. By~\cite[Lemma 3.8]{Jin2020}, the subcategory $ \mathrm{mod}_{B}H^{0}(A) $ is functorially finite in $ \mathrm{mod}H^{0}(A) $. Thus, the subcategory $ \mathrm{mod}_{B}H^{0}(A) $ is also covariantly finite in $ \per A $. It remains to show $ \mathrm{mod}_{B}H^{0}(A) $ is contravariantly finite in $ \per A $. 
	
	Let $ N $ be an object in $ \mathrm{mod}_{B}H^{0}(A) $. Let $ g'\colon N\ra P $ be a morphism of dg $ A $-modules. By the relative Calabi--Yau duality (see Corollary~\ref{Relative CY duality}), the spaces 
	$$ \Hom_{\per A}(N,\tau_{\leqslant-n-2}P)\simeq D\Hom_{\per A}(\tau_{\leqslant-n-2}P,\Si^{n+1}N) $$ and $$ \Hom_{\per A}(N,\Si\tau_{\leqslant-n-2}P)\simeq D\Hom_{\per A}(\tau_{\leqslant-n-2}P,\Si^{n}N) $$ vanish. Thus, we have $ \Hom_{\per A}(N,P)\simeq\Hom_{\per A}(N,\tau_{\geqslant-n-1}P) $. We denote by $ g'' $ the composition $ N\xrightarrow{g'}P\ra\tau_{\geqslant-n-1}P $. Let $ I_{P} $ be a fibrant replacement of $ \tau_{\geqslant-n-1}P $. Then we have $ \Hom_{\per A}(N,P)\simeq\Hom_{\per A}(N,\tau_{\geqslant-n-1}P)\simeq\Hom_{\ch(A)}(N,I_{P}) $.
	
	Since $ \tau_{\geqslant-n-1}P $ has finite dimensional total homology, the dg module $ I_{P} $ also has finite dimensional total homology. We write $ I_{P} $ as a $ k $-complex and consider the following diagram
	\[
	\begin{tikzcd}
		\cdots\arrow[r]&0\arrow[r]\arrow[d]&N\arrow[r]\arrow[d,"g^{0}"]&0\arrow[r]\arrow[d]&\cdots\\
		\cdots\arrow[r]&I^{-1}\arrow[r,"d^{-1}"]&I^{0}\arrow[r,"d^{0}"]&I^{1}\arrow[r]&\cdots.
	\end{tikzcd}
	\]
	We put $ L=\{x\in I^{0}\,|\,d^{0}(x)=0,xa=0,\forall a\in A^{p},p<0\} $. Then $ L $ is in $ \mathrm{mod}H^{0}A $ and $ g''(N) $ is contained in $ L $. Thus, we have the following commutative diagram 
	\[
	\begin{tikzcd}
		&L\arrow[hookrightarrow]{d}{i}\\
		N\arrow[r,swap,"g''"]\arrow[ur,"g''"]&I_{P}.
	\end{tikzcd}
	\]
	
	Since the subcategory $ \mathrm{mod}_{B}H^{0}(A) $ is functorially finite in $ \mathrm{mod}H^{0}(A) $, there exists an object $ Y $ in $ \mathrm{mod}_{B}H^{0}(A) $ with a right $ \mathrm{mod}_{B}(H^{0}A) $-approximation $ j\colon Y\ra L $. Then there exists a morphism $ k\colon N\ra Y $ such that the following diagram commutes
	\[
	\begin{tikzcd}
		&Y\arrow[d,"j"]\\
		&L\arrow[hookrightarrow]{d}{i}\\
		N\arrow[r,swap,"g''"]\arrow[ur]\arrow[uur,"k"]&I_{P}.
	\end{tikzcd}
	\]
	This shows that $ \mathrm{mod}_{B}H^{0}(A) $ is contravariantly finite in $ \per A $.
	
\end{proof}

\begin{Cor}\label{Hom finite of RCC}
	The relative cluster category $ \mathcal{C}_{n}(A,B) $ is Hom-finite.
\end{Cor}

\section{Silting reduction and relative fundamental domain}\label{section SR and RFD}
\subsection{Silting reduction}
Let $ \mathcal{T} $ be a triangulated category. A full subcategory $ \mathcal{P} $ of $ \mathcal{T} $ is \emph{presilting} if $ \Hom_{\mathcal{T}}(P,\Si^{i}P)=0 $ for any $ i>0 $. It is \emph{silting} if in addition $ \mathcal{T}=\thick\mathcal{P} $. An object $ P $ of $ \mathcal{T} $ is \emph{presilting} if $ \add P $ is a presilting subcategory and \emph{silting} if $ \add P $ is a silting subcategory.

Let $ \mathcal{P} $ be a presilting subcategory of $ \mathcal{T} $. Let $ \mathcal{S} $ be the thick subcategory $\thick\mathcal{P} $ of $ \ct $ and $ \mathcal{U} $ the quotient category $ \mathcal{T}/\mathcal{S} $. We call $ \mathcal{U} $ the \emph{silting reduction} of $ \mathcal{T} $ with respect to $ \mathcal{P} $ (see \cite{AI2012}). For an integer $ l $, there is a bounded co-t-structure $ (\mathcal{S}_{\geqslant l},\mathcal{S}_{\leqslant l}) $ on $ \mathcal{S} $ (see \cite[Proposition 2.8]{IY2018}), where
$$ \mathcal{S}_{\geqslant l}=\mathcal{S}_{>l-1}\coloneqq\bigcup_{i\geqslant0}\Si^{-l-i}\mathcal{P}*\cdots*\Si^{-l-1}\mathcal{P}*\Si^{-l}\mathcal{P},$$
$$ \mathcal{S}_{\leqslant l}=\mathcal{S}_{<l+1}\coloneqq\bigcup_{i\geqslant0}\Si^{-l}\mathcal{P}*\Si^{-l+1}\mathcal{P}\cdots*\Si^{-l+i}\mathcal{P}.$$

Let $ \cz $ be the following subcategory of $ \ct $
$$ \mathcal{Z}={}^{\perp_{\mathcal{T}}}(\mathcal{S}_{<0})\cap(\mathcal{S}_{>0})^{\perp_{\mathcal{T}}}={}^{\perp_{\mathcal{T}}}(\Si^{>0}\mathcal{P})\cap(\Si^{<0}\mathcal{P})^{\perp_{\mathcal{T}}} .$$

\begin{Ex}
	Let $ \ce $ be a Frobenius category. Let $ \ct=\cd^{b}(\ce) $ be its bounded derived category and $ \cp $ the projective-injective subcategory of $ \ce $. Then $ \cz $ is equal to $ \ce\subseteq\cd^{b}(\ce) $.
\end{Ex}

We consider the following mild technical conditions:
\begin{itemize}
	\item[(P1)] $ \mathcal{P} $ is covariantly finite in $ {}^{\perp_{\mathcal{T}}}(\Si^{>0}\mathcal{P}) $ and contravariantly finite in $ (\Si^{<0}\mathcal{P})^{\perp_{\mathcal{T}}} $.
	\item[(P2)] For any $ X\in\mathcal{T} $, we have $ \Hom_{\mathcal{T}}(X,\Si^{l}\mathcal{P})=0=\Hom_{\mathcal{T}}(\mathcal{P},\Si^{l}X) $ for $ l\gg0 $.
	
\end{itemize}

\begin{Prop}\cite[Proposition 3.2]{IY2018}\label{Presilting to co-t-structure}
	The following conditions are equivalent.
	\begin{itemize}
		\item[(a)] The conditions (P1) and (P2) are satisfied.
		\item[(b)] The two pairs $( {}^{\perp_{\mathcal{T}}}\mathcal{S}_{<0} , \mathcal{S}_{\leqslant0} ) $ and $ (\mathcal{S}_{\geqslant0},\mathcal{S}_{>0}^{\perp_{\mathcal{T}}} ) $ are co-t-structures on $ \mathcal{T} $.
	\end{itemize}
	In this case, the co-hearts of these co-t-structures are $ \mathcal{P} $.
\end{Prop}

\begin{Thm}\cite[Theorem 3.1]{IY2018}\label{silting reduction}
	Under the conditions $ (P1) $ and $ (P2) $, the composition $ \mathcal{Z}\subset\mathcal{T}\xrightarrow{\rho}\mathcal{U} $ of natural functors induces an equivalence of additive categories:
	$$ \overline{\rho}\colon\mathcal{Z}/[\mathcal{P}]\longrightarrow\mathcal{U}. $$
\end{Thm}

Moreover, we have the following theorem.

\begin{Thm}\cite[Theorem 4.2]{IYY2008}\label{Tiangule structure for Z/P}
	The category $ \mathcal{Z}/[\mathcal{P}] $ has the structure of a triangulated category with respect to the following shift functor and triangles:
	\begin{itemize}
		\item[(a)] For $ X\in\mathcal{Z} $, we take a triangle $$ X\xrightarrow{l_{X}}P_{X}\longrightarrow X\langle1\rangle\longrightarrow \Si X $$
		with a (fixed) left $ \mathcal{P} $-approximation $ l_{X} $. Then $ \langle1\rangle $ gives a well-defined auto-equivalence of $ \mathcal{Z}/[\mathcal{P}] $, which is the shift functor of $ \mathcal{Z}/[\mathcal{P}] $.
		\item[(b)] For a triangle $ X\longrightarrow Y\longrightarrow Z\longrightarrow \Si X $ with $ X, Y, Z\in\mathcal{Z} $, take the following commutative diagram of triangles
		\begin{align*}
			\xymatrix{
				X\ar[r]^{f}\ar@{=}[d]&Y\ar[r]^{g}\ar[d]&Z\ar[r]^{g}\ar[d]^{a}&\Si X\ar@{=}[d]\\
				X\ar[r]^{l_{X}}&P_{X}\ar[r]&X\langle1\rangle\ar[r]&\Si X.
			}
		\end{align*}
		Then we have a complex $ X\xrightarrow{\overline{f}}Y\xrightarrow{\overline{g}}Z\xrightarrow{\overline{a}}X\langle1\rangle .$ We define triangles in $ \mathcal{Z}/[\mathcal{P}] $ as the complexes which are isomorphic to complexes obtained in this way.
	\end{itemize}
\end{Thm}

\begin{Thm}\cite[Theorem 3.6]{IY2018}\label{Thm:silting reduction} 
	The functor $ \overline{\rho}\colon \mathcal{Z}/[\mathcal{P}]\longrightarrow\mathcal{U} $ in Theorem~\ref{silting reduction} is a triangle equivalence where the triangulated structure of $ \mathcal{Z}/[\mathcal{P}] $ is given by Theorem~\ref{Tiangule structure for Z/P}.
\end{Thm}
\begin{Rem}
	We remark that more general versions of Theorem~\ref{Thm:silting reduction} have been established in~\cite{Naka2017}.
\end{Rem}

In our case, we put $ \mathcal{T}=\mathrm{per}A $, $ \mathcal{P}=\add(eA) $, and $ \mathcal{S}=\thick\mathcal{P}\cong \mathrm{per}(eAe) $. Then it is clear that the categories $ \mathcal{T} $, $ \mathcal{P} $ and $ \mathcal{S} $ satisfy the conditions $ \rm(P1) $ and $ \rm(P2) $. 

\begin{Cor}
	We have the following equivalence of triangulated categories
	$$ p^{*}\colon \mathcal{Z}/[\mathcal{P}]\iso\per A/\langle eA\rangle\iso\per(\overline{A}) ,$$ where $ \mathcal{Z}={}^{\perp_{\per A}}(\Si^{>0}\mathcal{P})\cap(\Si^{<0}\mathcal{P})^{\perp_{\per A}} $.	
\end{Cor}

\subsection{The standard co-$ t $-structure on $ \mathrm{per}A $}\label{subsection:The standard co-t-structure}

\begin{Prop}\cite[Proposition 2.8]{IY2018}
	Let $ \mathcal{T} $ be a triangulated category and $ \mathcal{M} $ a silting subcategory of $ \mathcal{T} $ with $ \mathcal{M}=\add\mathcal{M} $.
	\begin{itemize}
		\item[(a)] Then $ (\mathcal{T}_{\geqslant0},\mathcal{T}_{\leqslant0}) $ is a bounded co-t-structure on $ \mathcal{T} $, where
		$$ \mathcal{T}_{\geqslant0}\coloneqq\bigcup_{k\geqslant0}\Si^{-k}\mathcal{M}*\cdots*\Si^{-1}\mathcal{M}*\mathcal{M}\quad \text{and}\quad\mathcal{T}_{\leqslant0}\coloneqq\bigcup_{k\geqslant0}\mathcal{M}*\Si\mathcal{M}*\cdots*\Si^{k}\mathcal{M}.$$
		\item[(b)] For any integers $ m $ and $ l $, we have
		
		\[\mathcal{T}_{\geqslant l}\cap\mathcal{T}_{\leqslant m}=\left\{
		\begin{aligned}
			\Si^{-m}\mathcal{M}*\Si^{-m+1}\mathcal{M}*\cdots*\Si^{-l}\mathcal{M}&&&&\text{if}\quad l\leqslant m, \\
			0&&&& \text{if}\quad l>m.
		\end{aligned}
		\right.\]	
	\end{itemize}
\end{Prop}

Let $ \Gamma $ be a connective dg algebra. Then $ \Gamma $ is a silting object in $ \mathrm{per}\Gamma $. By the above proposition, the pair $ ((\mathrm{per}\Gamma)_{\geqslant0},(\mathrm{per}\Gamma)_{\leqslant0}) $ is a co-t-structure on $ \mathrm{per}\Gamma $, where $$ (\mathrm{per}\Gamma)_{\geqslant0}\coloneqq\bigcup_{k\geqslant0}\Si^{-k}\mathrm{add}\Gamma*\cdots*\Si^{-1}\mathrm{add}\Gamma*\mathrm{add}\Gamma\quad \text{and}\quad(\mathrm{per}\Gamma)_{\leqslant0}\coloneqq\bigcup_{k\geqslant0}\mathrm{add}\Gamma*\Si\mathrm{add}\Gamma*\cdots*\Si^{k}\mathrm{add}\Gamma.$$
The corresponding co-heart is $ \mathrm{add}\Gamma $.

\subsection{Fundamental domain for generalized cluster categories}

Let $ \mathcal{F} $ be the full subcategory $ \mathcal{D}(\overline{A})^{\leqslant 0}\cap{}^{\perp}\mathcal{D}(\overline{A})^{\leqslant-n}\cap \mathrm{per}(\overline{A})$. In the paper~\cite{Am2008}, it is called the \emph{fundamental domain} of $ \per\overline{A} $. We denote by $ \pi\colon\mathrm{per}\overline{A}\ra\mathcal{C}_{n}(\overline{A}) $ the canonical projection functor.

\begin{Lem}\cite[Lemma 3.2.8]{LYG}\label{addA' approximation}
	For each object $ X $ of $ \mathcal{F} $, there exist $ n-1 $ triangles (which are not unique in general)
	$$P_{1}\to Q_{0}\to X\to \Si P_{1},$$
	$$ P_{2}\to Q_{1}\to P_{1}\to\Si P_{2},$$
	$$ \cdots $$
	$$P_{n-1}\to Q_{n-2}\to P_{n-2}\to\Si P_{n-1}
	,$$
	where $ Q_{0} $, $ Q_{1} $, $ \cdots $, $ Q_{n-2} $ and $ P_{n-1} $ are in $ \mathrm{add}( \overline{A}) $.	
\end{Lem}

\begin{Rem}
	In fact, the fundamental domain $ \mathcal{F} $ is equal to 
	$$ (\mathrm{per}\overline{A})_{\geqslant 1-n}\cap(\mathrm{per}\overline{A})_{\leqslant0}=\mathrm{add}\overline{A}*\Si\mathrm{add}\overline{A}*\cdots*\Si^{n-1}\mathrm{add}\overline{A},$$ where $ ((\mathrm{per}\overline{A})_{\geqslant0},(\mathrm{per}\overline{A})_{\leqslant0}) $ is the canonical co-t-structure on $ \per\overline{A} $.
\end{Rem}

\begin{Prop}\rm\cite[Proposition 4.3.1]{LYG}
	The projection functor $ \pi\colon\mathrm{per}\overline{A}\to\mathcal{C}_{n}(\overline{A}) $ induces a $ k $-linear equivalence between $ \mathcal{F} $ and $ \mathcal{C}_{n}(\overline{A}) $.
\end{Prop}

\subsection{Relative fundamental domain and Higgs category}

\begin{Def}\rm\label{RFD}
	We define the \emph{relative fundamental domain $ \mathcal{F}^{rel} $} of $ \per A $ to be the following full subcategory
	$$ \mathcal{Z}\cap(\mathrm{per}A)_{\geqslant 1-n}\cap(\mathrm{per}A)_{\leqslant0}=\mathcal{Z}\cap (\mathrm{add}A*\Si\mathrm{add}A*\cdots*\Si^{n-1}\mathrm{add}A),$$ where $ ((\mathrm{per}A)_{\geqslant0},(\mathrm{per}A)_{\leqslant0}) $ is the canonical co-t-structure on $ \mathrm{per}A $ and $ \mathcal{Z} $ is the subcategory $${}^{\perp_{\mathrm{per}A}}(\Si^{>0}\mathcal{P})\cap(\Si^{<0}\mathcal{P})^{\perp_{\mathrm{per}A}} $$
	with $ \mathcal{P}=\mathrm{add}(eA) $.
\end{Def}

By the proof of \cite[Lemma 7.2.1]{Am2008} (or \cite[Lemma 3.2.8]{LYG}), we can see that the subcategory $ \mathrm{add}A*\Si\mathrm{add}A*\cdots*\Si^{n-1}\mathrm{add}A $ is equal to $ \mathcal{D}(A)^{\leqslant 0}\cap{}^{\perp}(\mathcal{D}(A)^{\leqslant-n})\cap \mathrm{per}(A) $. Thus, the relative fundamental domain $ \mathcal{F}^{rel} $ is also equal to $ \mathcal{Z}\cap \mathcal{D}(A)^{\leqslant 0}\cap{}^{\perp}(\mathcal{D}(A)^{\leqslant-n})\cap \mathrm{per}(A) $.

\begin{Rem}\rm\label{Rem:triangles in RF}
	The relative fundamental domain $ \mathcal{F}^{rel} $ is also equivalent to the full subcategory of $\mathcal{Z}\subseteq \mathrm{per}(A) $ whose objects are the $ X\in\mathcal{Z} $ such that $ X $ fits into the following $ n-1 $ triangles in $ \mathrm{per}A $
	$$M_{1}\to N_{0}\to X\to \Si M_{1},$$
	$$ M_{2}\to N_{1}\to M_{1}\to\Si M_{2},$$
	$$ \cdots $$
	$$M_{n-1}\to N_{n-2}\to M_{n-2}\to\Si M_{n-1}
	$$
	with $ N_{0} $, $ N_{1} $, $ \cdots $, $ N_{n-2} $ and $ M_{n-1} $ in $ \mathrm{add}(A) $.
	
\end{Rem}

\begin{Prop}
	The relative fundamental domain $ \mathcal{F}^{rel} $ is contained in
	$$ \mathcal{D}(A)_{rel}^{\leqslant 0}\cap{}^{\perp}(\mathcal{D}_{B}(A)_{rel}^{\leqslant-n})\cap \per(A) ,$$ where $ \mathcal{D}_{B}(A)_{rel}^{\leqslant-n}$ is the full subcategory of $ \mathcal{D}(A)_{rel}^{\leqslant -n} $ whose objects are the objects
	$ X $ in $ \mathcal{D}(A)_{rel}^{\leqslant -n}$ whose restriction $ i_{*}(X) $ to $ eAe $ is acyclic.
\end{Prop}
\begin{proof}
	Let $ X $ be an object in $ \mathcal{F}^{rel}=\mathcal{Z}\cap (\mathrm{add}A*\mathrm{add}A[1]*\cdots*\mathrm{add}A[n-1]) $. Since $ A, \Si A,\cdots,\Si^{n-1} A $ are in $ \mathcal{D}(A)_{rel}^{\leqslant 0}\cap{}^{\perp}(\mathcal{D}_{B}(A)_{rel}^{\leqslant-n})\cap \mathrm{per}(A) $, by using the triangles in Remark~\ref{Rem:triangles in RF}, we see that $ X $ also lies in $ \mathcal{D}(A)_{rel}^{\leqslant 0}\cap{}^{\perp}(\mathcal{D}_{B}(A)_{rel}^{\leqslant-n})\cap \mathrm{per}(A) $.
	
\end{proof}

\bigskip

We still denote by $ p^{*} $ the restriction of $ p^{*}\colon\mathrm{per}(A)\to \mathrm{per}(\overline{A}) $ to $ \mathcal{F}^{rel} $.
\begin{Prop}\label{p* to F is dense}
	The functor $ p^{*}\colon\mathcal{F}^{rel}\to\mathcal{F} $ is dense.
\end{Prop}
\begin{proof}
	It is easy to see that $ p^{*} $ is well defined. Let $ Y $ be an object in $ \mathcal{F}\subseteq \mathrm{per}\overline{A} $. By Lemma~\ref{addA' approximation}, there exist $ n-1 $ triangles in $ \mathrm{per}\overline{A} $
	$$P_{1}\xrightarrow{b_{0}} Q_{0}\rightarrow Y\to\Si P_{1},$$
	$$ P_{2}\xrightarrow{b_{1}} Q_{1}\rightarrow P_{1}\to\Si P_{2},$$
	$$ \cdots $$
	$$P_{n-2}\xrightarrow{b_{n-3}} Q_{n-3}\rightarrow P_{n-3}\to\Si P_{n-2} ,$$
	$$P_{n-1}\xrightarrow{b_{n-2}} Q_{n-2}\rightarrow P_{n-2}\to\Si P_{n-1},
	$$
	with $ Q_{0} $, $ Q_{1} $, $ \cdots $, $ Q_{n-2} $ and $ P_{n-1} $ in $ \mathrm{add}( \overline{A}) $.
	
	We start from the last triangle. Recall that $ \cz $ is the following subcategory of $ \per A $
	$$ \cz={}^{\perp}(\Si^{>0}\mathcal{P})\cap(\Si^{<0}\mathcal{P})^{\perp} $$ with $ \mathcal{P}=\mathrm{add}(eA) $. 
	
	Since the functor $ p^{*}\colon\mathrm{add}A\subseteq\mathcal{Z}\longrightarrow \mathrm{add}\overline{A} $ is dense, there exist two objects $ M'_{n-1} $, $ N'_{n-2} $ in $ \add A $ such that $ p^{*}(M'_{n-1})\cong P_{n-1} $ and $ p^{*}(N'_{n-2})\cong Q_{n-2} $. We know that $ p^{*}\colon\mathcal{Z}/[\mathcal{P}]\iso\mathrm{per}A/\langle eA\rangle \rightarrow \mathrm{per}(\overline{A}) $ is fully faithful. Thus we have the following surjective map
	$$ \Hom_{\mathcal{Z}}(M'_{n-1},N'_{n-2})\twoheadrightarrow \Hom_{\mathrm{per}(\overline{A})}(P_{n-1},Q_{n-2}) .$$
	
	We lift the map $ b_{n-2}\colon P_{n-1}\rightarrow Q_{n-2} $ from $ \add(\overline{A}) $ to $ \mathrm{add}(A)\subseteq\mathcal{Z} $. Then we get
	$ g'_{n-2}\colon M'_{n-1}\rightarrow N'_{n-1} $ such that $ p^{*}(g'_{n-2})\cong b_{n-2} $.
	
	Since $ \mathcal{P} $ is covariantly finite and contravariantly finite in $ \mathcal{Z} $, we can find $ h_{n-2}\colon M'_{n-1}\rightarrow W_{n-2} $ a left $ \mathrm{add}({e}A) $-approximation of $ M'_{n-1} $. We define 
	$$ (M_{n-1}\xrightarrow{g_{n-2}} N_{n-2})\coloneqq(M'_{n-1}\xrightarrow{[g'_{n-2},\ h_{n-2}]^{t}}N'_{n-2}\oplus W_{n-2}) .$$ Then we can see that $ p^{*}(g_{n-2})\cong b_{n-2} $ and the following map is surjective
	$$ g_{n-2}^{*}\colon \Hom_{\mathrm{per}A}(N_{n-2},eA) \rightarrow \Hom_{\per A}(M_{n-1},eA) .$$ 
	We form a triangle in $ \mathrm{per}A $
	$$ M_{n-1}\xrightarrow{g_{n-2}} N_{n-2}\rightarrow M_{n-2}\rightarrow \Si M_{n-1} .$$ Then $ p^{*}(M_{n-2}) $ is isomorphic to $ P_{n-2} $.
	
	Since the map $ g_{n-2}^{*}\colon \Hom_{\per A}(N_{n-2},eA) \rightarrow \Hom_{\per A}(M_{n-1},eA) $ is surjective, we can see that $ M_{n-2} $ is an object in $ \mathcal{Z} $.
	
	Next, we consider the penultimate triangle. Repeating the above argument, we get a triangle in $ \mathrm{per}A $
	$$ M_{n-2}\xrightarrow{g_{n-3}} N_{n-3}\rightarrow M_{n-3}\rightarrow \Si M_{n-2} $$ such that $ N_{n-3}\in \add A $, $ p^{*}(N_{n-3})\cong Q_{n-3} $, $ p^{*}(g_{n-3})\cong b_{n-3} $, $ p^{*}(M_{n-3})\cong P_{n-3} $ and 
	$ M_{n-3}\in\mathcal{Z} $.
	
	Then, we keep repeating this argument until the first triangle. We get the following $ n-1 $ triangles in $ \per A $
	$$ M_{1}\xrightarrow{g_{0}}N_{0}\rightarrow X\rightarrow\Si M_{1} ,$$
	$$ M_{2}\xrightarrow{g_{1}}N_{1}\rightarrow M_{1}\rightarrow\Si M_{2} ,$$
	$$ \cdots $$
	$$ M_{n-2}\xrightarrow{g_{n-3}}N_{n-3}\rightarrow M_{n-3}\rightarrow\Si M_{n-2} ,$$
	$$ M_{n-1}\xrightarrow{g_{n-2}}N_{n-2}\rightarrow M_{n-2}\rightarrow\Si M_{n-1} $$
	such that $ M_{n-1},N_{n-2},\cdots,N_{0}\in \mathrm{add}A $, $ X\in\mathcal{Z} $ and $ p^{*}(X)\cong Y $. 
	Thus, the object $ X $ belongs to $ \mathcal{F}^{rel} $ and therefore $ p^{*}\colon\mathcal{F}^{rel}\rightarrow\mathcal{F} $ is dense.
	
\end{proof}

\bigskip

Recall that the fundamental domain $ \mathcal{F} $ of $ \per\overline{A} $ is the full subcategory $ \mathcal{D}(\overline{A})^{\leqslant 0}\cap$  $^{\perp}\mathcal{D}(\overline{A})^{\leqslant-n}\cap \mathrm{per}(\overline{A})$ of $ \per\overline{A} $. For any two objects in $ \cz={}^{\perp}(\Si^{>0}\mathcal{P})\cap(\Si^{<0}\mathcal{P})^{\perp} $ with $ \mathcal{P}=\mathrm{add}(eA) $, we have the following useful lemma to compare their extension groups in $ \per A $ and $ \per\overline{A} $.
\begin{Lem}\label{Extension group in Z}
	Let $ X $ and $ Y $ be two objects in $ \mathcal{Z} $. Let $ l>0 $ be an integer. Then we have
	$$ \Hom_{\mathrm{per}A}(X,\Si^{l}Y)\cong \Hom_{\mathrm{per}(\overline{A})}(p^{*}(X),\Si^{l}p^{*}(Y)). $$ 
\end{Lem}

\begin{proof}
	Let $ X $ and $ Y $ be two objects in $ \mathcal{Z} $. For the object $ Y $, we have the following triangle in $ \mathrm{per}A $
	$$ Y\xrightarrow{f_{1}} P_{Y_{1}}\rightarrow Y_{1}\rightarrow \Si Y ,$$ where $ f_{1}\colon Y\rightarrow P_{Y_{1}} $ is a left $ \mathrm{add}(eA) $-approximation. Since $ f_{1} $ is a a left $ \mathrm{add}(eA) $-approximation, it is not hard to see that $ Y_{1} $ lies in $ \cz $.
	
	Similarly, for the object $ Y_{1} $, we have the following triangle in $ \mathrm{per}A $
	$$ Y_{1}\xrightarrow{f_{2}} P_{Y_{2}}\rightarrow Y_{2}\rightarrow \Si Y_{1} ,$$ where $ f_{2}\colon Y_{1}\rightarrow P_{Y_{2}} $ is a left $ \mathrm{add}(eA) $-approximation and $ Y_{2} $ lies in $ \cz $.
	
	Repeating this process, we can get the following $ l $ triangles in $ \mathrm{per}A $
	$$ Y\xrightarrow{f_{1}} P_{Y_{1}}\rightarrow Y_{1}\rightarrow\Si Y ,$$
	$$ Y_{1}\xrightarrow{f_{2}} P_{Y_{2}}\rightarrow Y_{2}\rightarrow\Si Y_{1} ,$$
	$$ \cdots $$
	$$ Y_{l-1}\xrightarrow{f_{l}} P_{Y_{l}}\rightarrow Y_{l}\rightarrow\Si Y_{l-1} ,$$
	where for each $ 1\leqslant i\leqslant l $, $ f_{i} $ is a left $ \mathrm{add}(eA) $-approximation. Moreover each $ Y_{i} $ lies in $ \cz $.
	
	By the first triangle, we can see that
	$$ \Hom_{\mathrm{per}A}(X,\Si^{l-1}Y_{1})\cong \Hom_{\mathrm{per}A}(X,\Si^{l}Y) .$$
	Similarly, by the second triangle, we can see that
	$$ \Hom_{\mathrm{per}A}(X,\Si^{l-2}Y_{2})\cong \Hom_{\mathrm{per}A}(X,\Si^{l-1}Y_{1}) .$$
	Repeating this argument, we have 
	\begin{equation*}
		\begin{split}
			\Hom_{\mathrm{per}A}(X,\Si^{l}Y)&\cong \Hom_{\mathrm{per}A}(X,\Si^{l-1}Y_{1})\\
			&\cong \Hom_{\mathrm{per}A}(X,\Si^{l-2}Y_{2})\\
			&\cdots\\
			&\cong \Hom_{\mathrm{per}A}(X,\Si Y_{l-1}).
		\end{split}
	\end{equation*}
	
	By the last triangle, it induces a long exact sequence
	$$ \rightarrow \Hom_{\mathrm{per}A}(X,P_{Y_{l}})\xrightarrow{\Phi} \Hom_{\mathrm{per}A}(X,Y_{l})\rightarrow \Hom_{\mathrm{per}A}(X,\Si Y_{l-1})\rightarrow0 .$$
	Thus we have
	\begin{equation*}
		\begin{split}
			\Hom_{\mathrm{per}A}(X,\Si^{l}Y)&\cong \Hom_{\mathrm{per}A}(X,\Si Y_{l-1})\\
			&\cong \Hom_{\mathrm{per}A}(X,Y_{l})/\Ima(\Phi)\\
			&\cong \Hom_{\mathcal{Z}/[\mathcal{P}]}(X,Y\langle l\rangle)\\
			&\cong \Hom_{\mathrm{per}(\overline{A})}(p^{*}(X),\Si^{l}p^{*}(Y)).
		\end{split}
	\end{equation*}
	
\end{proof}

\begin{Prop}\label{p* is dense}
	The functor $ p^{*}\colon\mathrm{per}A\rightarrow \mathrm{per}\overline{A} $ is dense. Thus, we have equivalences $$ \mathcal{Z}/[\mathcal{P}]\simeq \mathrm{per}A/\langle eA\rangle\simeq \mathrm{per}\overline{A} ,$$ where $ \mathcal{Z}={}^{\perp}(\Si^{>0}\mathcal{P})\cap(\Si^{<0}\mathcal{P})^{\perp} $ with $ \mathcal{P}=\mathrm{add}(eA) $.
\end{Prop}
\begin{proof}
	There is a canonical co-t-structure $ (\mathrm{per}(\overline{A})_{\geqslant0},\mathrm{per}(\overline{A}))_{\leqslant0} $ on $ \mathrm{per}(\overline{A}) $, where $$ \mathrm{per}(\overline{A})_{\geqslant0}=\bigcup_{n\geqslant0}\Si^{-n}\mathrm{add}(\overline{A})*\cdots*\Si^{-1}\mathrm{add}(\overline{A})*\mathrm{add}(\overline{A}) ,$$
	$$ \mathrm{per}(\overline{A})_{\leqslant0}=\bigcup_{n\geqslant0}\mathrm{add}(\overline{A})*\Si\mathrm{add}(\overline{A})*\cdots*\Si^{n}\mathrm{add}(\overline{A}) .$$
	
	Let $ Z $ be an object in $ \mathrm{per}(\overline{A}) $. By using the canonical co-t-structure on $ \per(\overline{A}) $, we have a triangle in $ \mathrm{per}(\overline{A}) $
	$$ X\rightarrow Z\rightarrow Y\xrightarrow{h} \Si X ,$$ with $ X\in \mathrm{per}(\overline{A})_{\geqslant0} $ and $ Y\in \mathrm{per}(\overline{A})_{\leqslant0} $.
	
	We will find objects $ U,V\in\mathcal{Z}\subseteq \mathrm{per}A $ such that $ p^{*}(U)\cong X $ and $ p^{*}(V)\cong Y $. Suppose that $ X $ is in $ \Si^{-n_{0}}\add(\overline{A})*\cdots*\Si^{-1}\add(\overline{A})*\add(\overline{A}) $ and $ Y $ is in $ \mathrm{add}(\overline{A})*\Si\mathrm{add}(\overline{A})*\cdots*\Si^{n_{1}}\mathrm{add}(\overline{A}) $. 
	
	For the object $ Y $, if $ n_{1}=0 $, it is easy to find such $ V $. So we can assume that $ n_{1}\geqslant1 $. Thus there are $ n_{1} $ triangles in $ \mathrm{per}(\overline{A}) $
	$$ P_{1}\rightarrow Q_{0}\rightarrow Y\rightarrow\Si P_{1} ,$$
	$$ P_{2}\rightarrow Q_{1}\rightarrow P_{1}\rightarrow\Si P_{2} ,$$
	$$\cdots$$
	$$ P_{n_{1}}\rightarrow Q_{n_{1}-1}\rightarrow P_{n_{1}-1}\rightarrow\Si P_{n_{1}} ,$$ with $ P_{n_{1}},Q_{n_{1}-1},\cdots,Q_{0} \in \mathrm{add}(\overline{A}) $. 
	
	It follows by a similar argument to that for Proposition~\ref{p* to F is dense} that there is an object $ V\in\mathcal{Z}\subseteq \mathrm{per}A $ such that $ p^{*}(V)\cong Y $.
	
	For the object $ X\in \Si^{-n_{0}}\mathrm{add}(\overline{A})*\cdots*\Si^{-1}\mathrm{add}(\overline{A})*\mathrm{add}(\overline{A}) $, we have $ \Si^{n_{0}} X $ in $ \mathrm{add}(\overline{A})*\Si\mathrm{add}(\overline{A})*\cdots*\Si^{n_{0}}\mathrm{add}(\overline{A}) $. Thus there exists an object $ U'\in\mathcal{Z} $ such that $ p^{*}(U')\cong \Si^{n_{0}} X $. 
	
	Since $ \mathcal{P}=\mathrm{add}(eA) $ is covariantly finite and contravariantly finite in $ \mathcal{Z} $, we can take the following $ n_{0}+1 $ triangles in $ \mathrm{per}A $
	$$ U'\langle-1\rangle\rightarrow R_{0}\xrightarrow{f_{0}}U'\rightarrow\Si U'\langle-1\rangle ,$$
	$$ U'\langle-2\rangle\rightarrow R_{-1}\xrightarrow{f_{-1}}U'\langle-1\rangle\rightarrow \Si U'\langle-2\rangle ,$$
	$$ \cdots $$
	$$ U'\langle-n_{0}-1\rangle\rightarrow R_{-n_{0}}\xrightarrow{f_{-n_{0}}}U'\langle-n_{0}\rangle\rightarrow \Si U'\langle-n_{0}-1\rangle $$ with $ f_{i} $ a right $ \mathrm{add}(eA) $-approximation for any $ -n_{0}\leqslant i\leqslant0 $. Then the object $ p^{*}(U')\cong \Si^{n_{0}}X $ is isomorphic to $ \Si^{n_{0}}p^{*}(U'\langle-n_{0}\rangle) $. Let $ U=U'\langle-n_{0}\rangle $. Thus, we have $ p^{*}(U) \iso X $. 
	
	Since $ \mathcal{Z}/[\mathcal{P}]\cong \mathrm{per}A/\langle eA\rangle\rightarrow \mathrm{per}(\overline{A}) $ is fully faithful, the following map is a surjection (see Lemma~\ref{Extension group in Z})
	$$ \Hom_{\mathcal{Z}}(V,\Si U\langle-n_{0}\rangle)\twoheadrightarrow \Hom_{\mathcal{Z}/[\mathcal{P}]}(Y,\Si X)= \Hom_{\mathrm{per}(\overline{A})}(Y,\Si X) .$$ 
	We can lift the following triangle in $ \mathrm{per}(\overline{A}) $
	$$ X\rightarrow Z\rightarrow Y\xrightarrow{h} \Si X ,$$
	to a triangle in $ \mathrm{per}A $
	$$ U\langle-n_{0}\rangle\rightarrow W\rightarrow V\xrightarrow{h'} \Si U\langle-n_{0}\rangle .$$
	
	Therefore, the object $ p^{*}(W) $ is isomorphic to $ Z $. Hence the functor $ p^{*}\colon\mathrm{per}A\rightarrow \mathrm{per}(\overline{A}) $ is dense.

\end{proof}

\begin{Cor}\label{Cor: dense to tri equivalence}
	We have a triangle equivalence
	\begin{align*}
		\xymatrix{
			p^{*}\colon\mathrm{per}(A)/\mathrm{per}(eAe)\ra\mathrm{per}(\overline{A}).
		}
	\end{align*}
\end{Cor}
\begin{proof}
	It follows from Corollary \ref{Recollement of dg algebras} and Proposition \ref{p* is dense}.
\end{proof}

\begin{Cor}\label{Equivalence between fundamental domain}
	We have the following equivalence of $ k $-categories
	$$ p^{*}\colon \mathcal{F}^{rel}/[\mathcal{P}]\iso\mathcal{F} .$$
\end{Cor}
\begin{proof}
	By Proposition~\ref{p* to F is dense}, we know that the quotient functor $ \mathcal{F}^{rel}/[\mathcal{P}]\rightarrow\mathcal{F} $ is dense. Since we have an equivalence $ \mathcal{Z}/[\mathcal{P}]\iso \mathrm{per}\overline{A} $, this quotient functor $ \mathcal{F}^{rel}/[\mathcal{P}]\rightarrow\mathcal{F} $ is also fully faithful. Thus the quotient functor $ \mathcal{F}^{rel}/[\mathcal{P}]\iso \mathcal{F} $ is an equivalence of $ k $-categories.
\end{proof}

\begin{Prop}\cite[proposition 7.2.1]{Am2008}\label{relative full faithful}
	The restriction of the quotient functor $\pi^{rel}\colon \mathrm{per}A\to\mathcal{C}_{n}(A,B) $ to $ \mathcal{F}^{rel} $ is fully faithful.
\end{Prop}
\begin{proof}
	Let $ X $ and $ Y $ be objects in $ \mathcal{F}^{rel}\subseteq \mathcal{D}(A)_{rel}^{\leqslant 0}\cap{}^{\perp}(\mathcal{D}_{B}(A)_{rel}^{\leqslant-n})\cap \mathrm{per}(A) $. By Proposition~\ref{Home space in RCA}, the space $ \Hom_{\mathcal{C}_{n}(A,B)}(\pi^{rel}X,\pi^{rel}Y) $ is isomorphic to the direct limit $ \varinjlim_{l<0} \Hom_{\mathcal{D}(A)}(\tau_{\leqslant l}^{rel}X,\tau_{\leqslant l}^{rel}Y) $. A morphism between $ X $ and $ Y $ in $ \mathcal{C}_{n}(A,B) $ is a diagram of the form
	\begin{align*}
		\xymatrix{
			&\tau_{\leqslant l}^{rel}X\ar"2,1"\ar"2,3"&\\
			X&&Y.
		}
	\end{align*}
	The canonical triangle
	$$ \Si^{-1}\tau_{> l}^{rel}X\to\tau_{\leqslant l}^{rel}X\to X\to\tau_{> l}^{rel}X $$
	yields a long exact sequence:
	$$ \cdots\to \Hom_{\mathcal{D}(A)}(\tau_{> l}^{rel}X,Y)\to \Hom_{\mathcal{D}(A)}(X,Y)\to \Hom_{\mathcal{D}(A)}(\tau_{\leqslant l}^{rel}X,Y)\to \Hom_{\mathcal{D}(A)}(\Si^{-1}\tau_{> l}^{rel}X,Y)\to\cdots  .$$
	
	To remind the reader, the morphism $ f\colon B\ra A $ has a left $ (n+1) $-Calabi--Yau structure (see Assumption~\ref{Relative assumption}). Since $ i_{*}(\tau_{>l}^{rel}X)=0 $, it satisfies the conditions of relative $ (n+1) $-Calabi--Yau duality (see Corollary~\ref{Relative CY duality}), and so the space 
	$$ \Hom_{\mathcal{D}(A)}(\Si^{-1}\tau^{rel}_{> l}X,Y) $$ is isomorphic to the space $ D\Hom_{\mathcal{D}(A)}(Y,\Si^{n}\tau^{rel}_{>l}X) $.

	The object $ X $ is in $ \mathcal{D}(A)_{rel}^{\leqslant 0} $, hence we have
	$ \tau^{rel}_{> l}X\in\mathcal{D}_{B}(A)_{rel}^{\leqslant 0} $ and then the space 
	$ \Hom_{\mathcal{D}(A)}(Y,\Si^{n}\tau^{rel}_{>l}X) $ vanishes. 
	
	For the same reasons, the space $ \Hom_{\mathcal{D}(A)}(\tau_{> l}^{rel}X,Y) $ vanishes. Hence there are bijections
	\begin{align*}
		\xymatrix{
			\Hom_{\mathcal{D}(A)}(\tau_{\leqslant l}^{rel}X,\tau_{\leqslant l}^{rel}Y)\ar[r]^-{\sim}&
			\Hom_{\mathcal{D}(A)}(\tau_{\leqslant l}^{rel}X,Y)\ar[r]^-{\sim}& \Hom_{\mathcal{D}(A)}(X,Y).
		}
	\end{align*}
	Thus, the functor $\pi^{rel}: \mathcal{F}^{rel}\to\mathcal{C}_{n}(A,B) $
	is fully faithful.
\end{proof}

\begin{Cor}
	We have an isomorphism
	$ \mathrm{End}_{\mathcal{C}_{n}(A,B)}(\pi^{rel}A)\simeq \mathrm{End}_{\mathrm{per}(A)}(A)=H^{0}(A) $.
\end{Cor}
\begin{proof}
	This follows from Lemma~\ref{relative full faithful} and the fact that $ A $ itself is in $ \mathcal{F}^{rel} $.
\end{proof}

\begin{Def}\rm\label{Higgs category}
	The \emph{Higgs category} $ \mathcal{H} $ is the image of $ \mathcal{F}^{rel} $ in $ \mathcal{C}_{n}(A,B) $ under the quotient functor $ \pi^{rel}\colon\mathrm{per}A\rightarrow\mathcal{C}_{n}(A,B) $.
\end{Def}
\begin{Rem}
	The reason for the name \enquote{Higgs category} is that this category generalizes the category of modules over the preprojective algebra of a Dynkin quiver and a module over the preprojective algebra can be called a \enquote{Higgs module} (in analogy with a \enquote{Higgs bundle}, which is the same object in a geometric context, see~\cite{NH1987,CS1992}).
\end{Rem}

\subsection{Relation with generalized cluster categories}

In \cite{LYG}, Lingyan Guo generalized Claire Amiot's construction~\cite{Am2008} of the generalized cluster categories to finite-dimensional algebras with global dimension at most $ n $. She studied the category $ \mathcal{C}_{n}(\Gamma) =\mathrm{per}\Gamma/\mathrm{pvd}(\Gamma) $ associated with a dg algebra $ \Gamma $ under the following assumptions:
\begin{assumption}\rm\label{absolut assumption}
	\begin{itemize}
		\item[1)] $ \Gamma $ is homologically smooth.
		\item[2)] $ \Gamma $ is connective, i.e, $ H^{p}(\Gamma) $ is zero for each $ p>0 $.
		\item[3)] $ \Gamma $ is $ (n+1) $-Calabi--Yau as a bimodule, i.e.\ there is an isomorphism in $ \mathcal{D}(\Gamma^{e}) $
		\begin{align*}
			\xymatrix{
				\Si^{n+1}\mathbf{R}\Hom_{\mathcal{D}(\Gamma^{e}) }(\Gamma,\Gamma^{e})\cong \Gamma.
		}	\end{align*}
		\item[4)] The space $ H^{0}(\Gamma) $ is finite-dimensional.
	\end{itemize}
\end{assumption}

\begin{Thm}\cite[Chapter 3]{LYG}
	Let $ \Gamma $ be a dg $ k $-algebra with the four properties above. Then
	\begin{itemize}
		\item[(1)] the category $ \mathcal{C}_{n}(\Gamma)=\mathrm{per}\,\Gamma/\mathrm{pvd}(\Gamma) $ is Hom-finite and $ n $-Calabi--Yau;
		\item[(2)] the object $ T= \pi\Gamma $ is an $ n $-cluster tilting object in $ \mathcal{C}_{n}(\Gamma) $ where $ \pi\colon\mathrm{per}\Gamma\longrightarrow\mathcal{C}_{n}(\Gamma) $ is the canonical quotient functor, i.e.\ we have
		\begin{align*}
			\xymatrix{
				\Hom_{\mathcal{C}_{n}(\Gamma)}(T,\Si^{r}T)=0\,\text{for}\quad r=1,\cdots,n-1,
			}
		\end{align*}
		and for each object $ L $ in $ \mathcal{C}_{n}(\Gamma) $, if the space $ \Hom_{\mathcal{C}_{n}(\Gamma)}(T,\Si^{r}L)$ vanishes for each $ r = 1,\cdots,n-1 $, then L belongs to $ \add\,T $, the full subcategory of $ \mathcal{C}_{n}(\Gamma) $ consisting of direct summands of finite direct sums of copies of $ T $;
		\item[(3)] the endomorphism algebra of $ T $ over $ \mathcal{C}_{n}(\Gamma) $ is isomorphic to $ H^{0}(\Gamma) $.
	\end{itemize}
\end{Thm}

\bigskip

By Proposition \ref{Prop: smooth of cofiber} and \ref{Prop: cofiber is fd}, the dg algebra $ \overline{A} $ satisfies the assumptions~\ref{absolut assumption}. We consider the associated generalized $ n $-cluster category $ \mathcal{C}_{n}(\overline{A})=\mathrm{per}\overline{A}/\mathrm{pvd}(\overline{A}) $.

\begin{Prop}
	We have the following equivalence of triangulated categories
	$$p^{*}\colon \mathcal{C}_{n}(A,B)/\mathrm{per}( eAe)\iso\mathcal{C}_{n}(\overline{A}).$$
\end{Prop}
\begin{proof}
	By Corollary \ref{Cor: dense to tri equivalence}, we have a triangle equivalence $ p^{*}\colon \mathrm{per}(A)/\mathrm{per}(eAe)\iso \mathrm{per}(\overline{A}) $. Thus, it is enough to show that we have an equivalence of triangulated categories $ p^{*}\colon\mathrm{pvd}_{B}(A)\iso\mathrm{pvd}(\overline{A}) $ and the two subcategories $ \mathrm{pvd}_{B}(A) $ and $ \mathrm{per}(eAe) $ are left and right orthogonal to each other.
	
	It is clear that the functor $ p_{*}\colon\mathcal{D}(\overline{A})\to \ker(i_{*}) $ is an equivalence of triangulated categories. Then the restriction of $ p^{*} $ and $ p^{!} $ to $ \ker(i_{*}) $ give a quasi-inverse of $ p_{*}\colon\mathrm{pvd}(\overline{A})\to \mathrm{pvd}_{B}(A) $.
	
	Let $ X $ be an object in $ \mathrm{pvd}_{B}(A) $ and let $ Y $ be an object in $ \mathrm{per}(eAe) $. Then $ i_{*}(X) $ is acyclic. Thus, we have
	$$ \Hom_{\mathcal{D}(A)}(X,i^{*}(Y))\cong \Hom_{\mathcal{D}(eAe)}(i_{*}(X),Y)=0 $$ and 
	$$ \Hom_{\mathcal{D}(A)}(i^{*}(Y),X)\cong D\Hom_{\mathcal{D}(A)}(X,\Si^{n+1}i^{*}(Y))=0 ,$$ where the second isomorphism is due to the relative Calabi--Yau property~\ref{Relative CY duality}. Thus, the categories $ \mathrm{pvd}_{B}(A) $ and $ \mathrm{per}(eAe) $ are left and right orthogonal to each other. 
	
\end{proof}

\begin{Cor}\label{Relative commutative digram}
	We have the following commutative diagram
	\begin{align*}
		\xymatrix{
			&&\mathrm{per}(eAe)\ar@{=}[r]\ar@{^{(}->}[d]&\mathrm{per}(eAe)\ar@{^{(}->}[d]\\
			\mathrm{pvd}_{B}(A)\ar@{^{(}->}"2,3"\ar@{->}_{\simeq}[d]&&\mathrm{per}(A)\ar[r]\ar@{->>}[d]&\mathcal{C}_{n}(A,B)\ar@{->>}[d]\\
			\mathrm{pvd}(\overline{A})\ar@{^{(}->}"3,3"&&\mathrm{per}(\overline{A})\ar[r]&\mathcal{C}_{n}(\overline{A})
		}
	\end{align*} and the rows and columns are exact sequences of triangulated categories.
\end{Cor}

\subsection{Equivalence between the relative shifts of $ \mathcal{F}^{rel} $}
\begin{Def}\rm
	Let $ l\geqslant0 $ be an integer. We define the \emph{relative $ l $-shifted fundamental domain $ \mathcal{F}^{rel}\langle l\rangle $} to be the following full subcategory of $ \mathcal{Z} $
	$$ \mathcal{F}^{rel}\langle l\rangle=\{X\in\mathcal{Z}\ |\ p^{*}(X)\in\Si^{l}\mathcal{F}\subseteq \mathrm{per}(\overline{A})\}, $$
	where $ \mathcal{Z}={}^{\perp_{\per A}}(\Si^{>0}\mathcal{P})\cap(\Si^{<0}\mathcal{P})^{\perp_{\per A}} $ with $ \mathcal{P}=\add(eA) $.	
\end{Def}

\begin{Rem}
	If $ l=0 $, then $ \mathcal{F}^{rel}\langle 0\rangle=\{X\in\mathcal{Z}\ |\ p^{*}(X)\in\mathcal{F}\subseteq \mathrm{per}(\overline{A})\} $ is equal to $ \mathcal{F}^{rel} $.
\end{Rem}

Our aim is to show (Proposition~\ref{Equivalence between the shifts of relative fundamental domain}) that the functor $ \tau^{rel}_{\leqslant -l} $ induces an equivalence
$$ \cf^{rel}\langle l-1\rangle\ra\cf^{rel}\langle l\rangle .$$

\bigskip

\begin{Lem}\label{Stable under relative truncation}
	Let $ l $ be an integer. Then the subcategory $ \mathcal{Z} $ of $ \mathrm{per}A $ is stable under the relative truncation functors $ \tau^{rel}_{\leqslant l}, \tau^{rel}_{>l}\colon\mathrm{per}A\rightarrow \mathrm{per}A $, i.e.\ $ \tau^{rel}_{\leqslant l}(\mathcal{Z})\subseteq\mathcal{Z} $ and $ \tau^{rel}_{> l}(\mathcal{Z})\subseteq\mathcal{Z} $.
\end{Lem}
\begin{proof}
	Let $ l $ be an integer and let $ X $ be an object in $ \mathcal{Z} $. We have a triangle in $ \mathrm{per}A $
	$$ \tau^{rel}_{\leqslant l}X\rightarrow X\rightarrow\tau^{rel}_{> l}X\rightarrow \Si\tau^{rel}_{\leqslant l}X .$$

	By the relative Calabi--Yau property (see Corollary~\ref{Relative CY duality}) and $ i_{*}(\tau^{rel}_{> l}X)=0 $, we have 
	\begin{equation*}
		\begin{split}
			\Hom_{\mathrm{per}A}(\tau^{rel}_{> l}X,\Si^{k}eA)&\simeq D\Hom_{\mathrm{per}A}(\Si^{k}eA,\Si^{n+1}\tau^{rel}_{> l}X)\\
			&\simeq D\Hom_{\mathrm{per}A}(i^{*}(\Si^{k}eAe),\Si^{n+1}\tau^{rel}_{> l}X)\\
			&\simeq D\Hom_{\mathrm{per}(eAe)}(\Si^{k}eAe,\Si^{n+1}i_{*}(\tau^{rel}_{> l}X))\\
			&=0
		\end{split}
	\end{equation*}
for all $ k $ in $ \mathbb{Z} $.
This implies that $ \tau^{rel}_{>l}X $ lies in $ ^{\perp}\Si^{>0}\mathcal{P} $. By the Calabi--Yau property again, the object $ \tau^{rel}_{> l}X $ lies $ (\Si^{<0}\mathcal{P})^{\perp} $. Thus $ \tau^{rel}_{> l}X $ is in $ \mathcal{Z} $.
For any $ k\in\mathbb{Z} $, we have the following exact sequence
	$$ \cdots\rightarrow \Hom_{\mathrm{per}A}(X,\Si^{k}eA)\rightarrow \Hom_{\mathrm{per}A}(\tau^{rel}_{\leqslant l}X,\Si^{k}eA)\rightarrow \Hom_{\mathrm{per}A}(\Si^{-1}\tau^{rel}_{>l}X,\Si^{k}eA)\rightarrow\cdots .$$ We see that $ \Hom_{\mathrm{per}A}(\tau^{rel}_{\leqslant l}X,\Si^{k}eA)=0 $, i.e.\ $ \tau^{rel}_{\leqslant l}X $ lies in $ ^{\perp}(\Si^{>0}\mathcal{P}) $.
	
	Since $ X $ is in $ (\Si^{<0}\mathcal{P})^{\perp} $, the space $ \Hom_{\mathrm{per}A}(eA,\Si^{k}X)=H^{k}(Xe) $ vanishes for any positive integer $ k $. Then by the following computation
	\begin{equation*}
		\begin{split}
			\Hom_{\mathrm{per}A}(\Si^{<0}eA,\tau^{rel}_{\leqslant l}X)&\simeq \Hom_{\mathrm{per}(eAe)}(\Si^{<0}eAe,i_{*}(\tau^{rel}_{\leqslant l}X))\\
			&\simeq \Hom_{\mathrm{per}(eAe)}(\Si^{<0}eAe,i_{*}(X))\\
			&\simeq \Hom_{\mathrm{per}(eAe)}(\Si^{<0}eAe,Xe),
		\end{split}
	\end{equation*}
	we see that $ \Hom_{\mathrm{per}A}(\Si^{<0}eA,\tau^{rel}_{\leqslant l}X) $ vanishes, i.e.\ $ \tau^{rel}_{\leqslant l}X $ is in $ (\Si^{<0}\mathcal{P})^{\perp} $. Thus $ \tau^{rel}_{\leqslant l}X $ is in $ \mathcal{Z} $.
	
\end{proof}

\bigskip

Let $ l $ be a positive integer and $ X $ an object of $ \mathcal{F}^{rel}\langle l\rangle $. Then the object $ p^{*}(X) $ lies in $ \Si^{l}\mathcal{F}\subseteq \mathrm{per}(\overline{A}) $. Hence $ \Si^{1-l}p^{*}(X) $ is in $ \Si\mathcal{F} $. By definition, there are $ n-1 $ triangles related to the object $ \Si^{1-l}p^{*}(X) $, i.e.\ $ \Si^{1-l}p^{*}(X) $ fits into the following $ n-1 $ triangles in $ \mathrm{per}(\overline{A}) $
$$P_{1}\rightarrow \Si Q_{0}\rightarrow \Si^{1-l}p^{*}(X)\xrightarrow{h_{0}} \Si P_{1},$$
$$ P_{2}\rightarrow\Si Q_{1}\rightarrow P_{1}\xrightarrow{h_{1}} \Si P_{2},$$
$$ \cdots $$
$$P_{n-2}\rightarrow\Si Q_{n-3}\rightarrow P_{n-3}\xrightarrow{h_{n-3}} \Si P_{n-2},
$$
$$\Si P_{n-1}\rightarrow\Si Q_{n-2}\rightarrow P_{n-2}\xrightarrow{h_{n-2}} \Si^{2}P_{n-1}
,$$
where $ Q_{0} $, $ Q_{1} $, $ \cdots $, $ Q_{n-2} $ and $ P_{n-1} $ are in $ \mathrm{add}(\overline{A}) $.

We denote by $ \nu=?\otimes_{H^{0}(\overline{A})}D(H^{0}(\overline{A})) $ the Nakayama functor on $ \mathrm{mod}H^{0}(A) $. Then $ \nu H^{0}(P_{n-1}) $ and $ \nu H^{0}(Q_{n-2}) $ are injective $ H^{0}(\overline{A}) $-modules. Let $ M' $ be the kernel of the morphism $ \nu H^{0}(P_{n-1})\rightarrow\nu H^{0}(Q_{n-2}) $. We define $ M $ to be $\Si^{l-1}p_{*}(M') $. Then it is clear that $ M $ belongs to 
$$ \mathcal{D}(A)_{rel}^{\geqslant-l+1}=\{X\in\mathcal{D}(A)|\,i_{*}(X)=0,H^{i}(p^{!}X)\cong H^{i}(X)=0,\forall i<-l\}. $$

\begin{Lem}\label{M is in Z}
	The object $ M=\Si^{l-1}p_{*}(M') $ is in $ \mathcal{Z} $.
\end{Lem}
\begin{proof}
	It is clear that $ M $ belongs to $ \mathrm{pvd}_{B}(A) $. Then $ M $ is an object in $ \mathcal{Z} $ since $ \mathrm{pvd}_{B}(A) $ is a full subcategory of $ \mathcal{Z} $.
	
\end{proof}

\begin{Lem}\cite[Lemma 3.2.9]{LYG}\label{Hom relations in FD}
	
	(1) There are isomorphisms of functors
	$$ \Hom_{\mathcal{D}(\overline{A})}(?,\Si^{2-l}p^{*}(X))|_{\heartsuit(\overline{A})}\cong \Hom_{\mathcal{D}(\overline{A})}(?,\Si^{2}P_{1})|_{\heartsuit(\overline{A})}\cong\cdots $$
	$$ \cdots\cong \Hom_{\mathcal{D}(\overline{A})}(?,\Si^{n-1}P_{n-2})|_{\heartsuit(\overline{A})}\cong \Hom_{\heartsuit(\overline{A})}(?,M') .$$
	
	(2) There is a monomorphism of functors
	$ \mathrm{Ext}^{1}_{\heartsuit(\overline{A})}(?,M')\hookrightarrow \Hom_{\mathcal{D}(\overline{A})}(?,\Si^{n}P_{n-2})|_{\heartsuit(\overline{A})} ,$ where $ \heartsuit(\overline{A})=\mathrm{mod}H^{0}(\overline{A}). $
	
\end{Lem}

By the above Lemma, the following two spaces are isomorphic
$$ \Hom_{\mathrm{per}(\overline{A})}(M',\Si^{2-l}p^{*}(X))\cong \Hom_{\mathrm{per}(\overline{A})}(M',M') .$$

By Lemma~\ref{Extension group in Z}, we have 
\begin{equation*}
	\begin{split}
		\Hom_{\mathrm{per}A}(M,\Si X)&\simeq \Hom_{\mathrm{per}(\overline{A})}(p^{*}(M),\Si p^{*}(X))\\
		&\simeq \Hom_{\mathrm{per}(\overline{A})}(\Si^{l-1}M',\Si p^{*}(X))\\
		&\simeq \Hom_{\mathrm{per}(\overline{A})}(M',\Si^{2-l}p^{*}(X))\\
		&\simeq \Hom_{\mathrm{per}(\overline{A})}(M',M').
	\end{split}
\end{equation*}

Let $ \epsilon $ be the preimage of the identity map on $ M' $ under the isomorphism 
$$ \Hom_{\mathcal{D}(A)}(M,\Si X)\cong \Hom_{\mathrm{per}(\overline{A})}(M',M') .$$ Then we form the corresponding triangle in $ \mathrm{per}A $
\begin{align}\label{Triangle for M}
	\xymatrix{
		X\ar[r]& Y\ar[r]& M\ar[r]^{\varepsilon}& \Si X.
	}
\end{align}

Similarly, let $ \varepsilon' $ be the the preimage of the identity map on $ M' $ under the isomorphism 
$$ \Hom_{\mathcal{D}(\overline{A})}(M',\Si^{2-l}p^{*}(X))\cong \Hom_{\heartsuit(\overline{A})}(M',M') .$$
Then we form the corresponding triangle in $ per(\overline{A}) $
$$ \Si^{1-l}p^{*}(X)\rightarrow Y'\rightarrow M'\xrightarrow{\varepsilon'}\Si^{2-l} p^{*}(X).$$
We see that $ p^{*}(Y) $ is isomorphic to $ \Si^{l-1}Y' $.

\begin{Lem}\cite[Lemma 3.2.11]{LYG}\label{Y' is in FD}
	The object $ Y' $ is in the fundamental domain $ \mathcal{F}\subseteq \mathrm{per}(\overline{A}) $.
\end{Lem}

\begin{Lem}\label{Denseness of relative truncation}
	The object $ Y $ is in $ \mathcal{F}^{rel}\langle l-1\rangle $ and $ \tau^{rel}_{\leqslant l}Y $ is isomorphic to $ X $.
\end{Lem}
\begin{proof}
	
	\emph{Step 1: $ Y $ is an object in $ \mathcal{F}^{rel}\langle l-1\rangle $ .}
	
	By Lemma~\ref{M is in Z}, the object $ M $ is in $ \mathcal{Z} $. By the triangle~(\ref{Triangle for M}), we can see that $ Y $ is in $ \mathcal{Z} $. Then by Lemma~\ref{Y' is in FD}, $ p^{*}(Y)\cong \Si^{l-1}Y' $ belongs to $ \Si^{l-1}\mathcal{F} $. Thus, the object $ Y $ is in $ \mathcal{F}^{rel}\langle l-1\rangle $.

	\emph{Step 2: $ \tau^{rel}_{\leqslant-l}Y $ is isomorphic to $ X $. } 
	
	Since $ X\in\mathcal{D}(A)_{rel}^{\leqslant-l} $ and $ \tau^{rel}_{>-l}(Y)=p_{*}\Si^{l-1}H^{-l+1}(p^{*}(Y))\in\mathcal{D}(A)_{rel}^{\geqslant-l+1} $, the space $ \Hom_{\mathcal{D}(A)}(X,\tau^{rel}_{>-l}(Y)) $ is zero. Hence, we can obtain a commutative diagram of triangles
	
	\begin{align*}
		\xymatrix{
			\tau^{rel}_{\leqslant-l}Y\ar[r]&Y\ar[r]\ar@{=}[d]&\tau^{rel}_{>-l}(Y)\ar[r]&\Si\tau^{rel}_{\leqslant-l}Y\\
			X\ar[r]\ar@{-->}[u]^{\delta_{2}}&Y\ar[r]&M\ar[r]\ar@{-->}[u]^{\delta_{1}}&\Si X\ar@{-->}[u]\,.
		}
	\end{align*}
	
	By the octahedral axiom, we have the following commutative diagram
	\begin{align*}
		\xymatrix{
			Y\ar[r]\ar@{=}[d]&M\ar[r]\ar[d]^{\delta_{1}}&\Si X\ar[r]\ar[d]^{\delta_{2}[1]}&\Si Y\ar@{=}[d]\\
			Y\ar[r]&\tau^{rel}_{>-l}(Y)\ar[r]\ar[d]&\Si\tau^{rel}_{\leqslant-l}Y\ar[r]\ar[d]&\Si Y\\
			&\cone(\delta_{1})\ar@{-->}[r]\ar[d]&\Si \cone(\delta_{2})\ar[d]&\\
			&\Si M\ar[r]&\Si^{2} X
		}
	\end{align*}
	and the object $ \cone(\delta_{1}) $ is isomorphic to $ \Si\cone(\delta_{2}) $ in $ \mathrm{per}A $.
	
	Since $ \tau^{rel}_{\leqslant-l}Y\in\mathcal{D}(A)^{\leqslant-l}_{rel} $ and $ X\in\mathcal{D}(A)^{\leqslant-l}_{rel} $, $ \cone(\delta_{2}) $ is also in $ \mathcal{D}(A)^{\leqslant-l}_{rel} $. Thus $ \Si\cone(\delta_{2}) $ is in $ \mathcal{D}(A)^{\leqslant-l-1}_{rel} $. On the other hand, $ M $ and $ \tau^{rel}_{>-l}(Y) $ are in $ \mathcal{D}_{rel}^{\geqslant-l+1}(A) $. Thus $ \cone(\delta_{1}) $ is in $ \mathcal{D}_{rel}^{\geqslant-l} $. Hence we can conclude that $ \cone(\delta_{1})\cong \Si\cone(\delta_{2}) $ is zero. Thus, the relative truncation $ \tau^{rel}_{\leqslant-l}Y $ of $ Y $ is isomorphic to $ X $.
\end{proof}

\begin{Lem}\label{Fully faithful of the shifts of relative fundamental domain}
	Let $ l>0 $ be an integer. The image of the functor $ \tau^{rel}_{\leqslant-l} $ restricted to $ \mathcal{F}^{rel}\langle l-1\rangle $ is in $ \mathcal{F}^{rel}\langle l\rangle $ and the functor $ \tau^{rel}_{\leqslant-l}\colon\mathcal{F}^{rel}\langle l-1\rangle\rightarrow\mathcal{F}^{rel}\langle l\rangle $ is fully faithful.
\end{Lem}

\begin{proof}
	\emph{Step 1: The image of the functor $ \tau^{rel}_{\leqslant-l} $ restricted to $ \mathcal{F}^{rel}\langle l-1 \rangle $  is in $ \mathcal{F}^{rel}\langle l \rangle $.}
	
	Let $ X $ be an object in $ \mathcal{F}^{rel}\langle l-1 \rangle \subseteq\mathcal{Z} $. By Lemma~\ref{Stable under relative truncation}, $ \tau_{\leqslant-l}^{rel}X $ is still in $ \mathcal{Z} $. It is clear that $ p^{*}(\tau_{\leqslant-l}^{rel}X)\cong\tau_{\leqslant-l}(p^{*}(X)) $ is in $ \mathcal{D}(\overline{A})^{\leqslant-l} $.
	
	We have a triangle in $ \mathrm{per}(\overline{A}) $
	$$ \Si^{-1}\tau_{>-l}(p^{*}(X))\rightarrow\tau_{\leqslant-l}(p^{*}(X))\rightarrow p^{*}(X)\rightarrow \tau_{>-l}(p^{*}(X)).$$
	
	Let $ W $ be an object in $ \mathcal{D}(\overline{A})^{\leqslant-n-l} $. The space $ \Hom_{\mathcal{D}(\overline{A})}(p^{*}(X),W) $ is zero since $ p^{*}(X)\in{}^{\perp}(\mathcal{D}(\overline{A})^{\leqslant-l-n+1}) $. By the Calabi--Yau property, we have 
	\begin{equation*}
		\begin{split}
			\Hom_{\mathcal{D}(\overline{A})}(\Si^{-1}\tau_{>-l}(p^{*}(X)),W)&\cong D\Hom_{\mathcal{D}(\overline{A})}(W,\Si^{n}\tau_{>-l}(p^{*}(X))).
		\end{split}
	\end{equation*}
	The space $ \Hom_{\mathcal{D}(\overline{A})}(W,\Si^{n}\tau_{>-l}(p^{*}(X))) $ vanishes because $ \Si^{n}\tau_{>-l}(p^{*}(X)) \in\mathcal{D}(\overline{A})^{\geqslant-l-n+1} $. Thus $ p^{*}(\tau_{\leqslant-l}^{rel}X) $ is in $ \Si^{l}\mathcal{F}\subseteq \per(\overline{A}) $ and then
	$ \tau_{\leqslant-l}^{rel}X $ belongs to $ \mathcal{F}^{rel}\langle l \rangle $.

	\emph{Step 2: The functor $ \tau^{rel}_{\leqslant-l}\colon\mathcal{F}^{rel}\langle l-1 \rangle\rightarrow\mathcal{F}^{rel}\langle l \rangle $ is fully faithful.}
	
	Let $ X $ and $ Y $ be two objects in $ \mathcal{F}^{rel}\langle l-1 \rangle $ and $ f\colon\tau^{rel}_{\leqslant-l}X\rightarrow\tau^{rel}_{\leqslant-l}Y $ be a morphism.
	\begin{align*}
		\xymatrix{
			\Si^{-1}\tau^{rel}_{>-l}X\ar[r]&\tau^{rel}_{\leqslant-l}X\ar[d]^{f}\ar[r]&X\ar@{.>}[d]\ar[r]&\tau^{rel}_{>-l}X\\
			\Si^{-1}\tau^{rel}_{>-l}Y\ar[r]&\tau^{rel}_{\leqslant-l}Y\ar[r]^{g}&Y\ar[r]&\tau^{rel}_{>-l}Y.
		}
	\end{align*}
	By the relative Calabi--Yau property, the space $ \Hom_{\mathcal{D}(A)}(\Si^{-1}\tau^{rel}_{>-l}X,Y) $ is isomorphic to 
	$ D\Hom_{\mathcal{D}(A)}(Y,\Si^{n}\tau^{rel}_{>-l}X) $. Since $ Y\in{}^{\perp}(\mathcal{D}(A)^{\leqslant-n-l+1}) $ and $ \Si^{n}\tau^{rel}_{>-l}X\in\mathcal{D}(A)^{\leqslant-n-l+1} $, this space is zero. Then the composition $ gf $ factorizes through the canonical morphism $ \tau^{rel}_{\leqslant-l}X\rightarrow X $. Thus the functor $ \tau^{rel}_{\leqslant-l}\colon\mathcal{F}\langle l-1 \rangle\rightarrow\mathcal{F}\langle l \rangle $ is full.
	
	Now let $ X $ and $ Y $ be objects of $ \mathcal{F}^{rel}\langle l-1 \rangle $ and $ f\colon X \rightarrow Y $ a morphism satisfying $ \tau^{rel}_{\leqslant-l}f=0 $. Then it induces a morphism of triangles:
	\begin{align*}
		\xymatrix{
			\Si^{-1}\tau^{rel}_{>-l}X\ar[d]\ar[r]&\tau^{rel}_{\leqslant-l}X\ar[d]^{0}\ar[r]^{\quad h}&X\ar[d]^{f}\ar[r]&\tau^{rel}_{>-l}X\ar[d]\ar@{.>}"2,3"\\
			\Si^{-1}\tau^{rel}_{>-l}Y\ar[r]&\tau^{rel}_{\leqslant-l}Y\ar[r]&Y\ar[r]&\tau^{rel}_{>-l}Y.
		}
	\end{align*}
	
	The composition $ fh $ vanishes, so $ f $ factorizes through $ \tau^{rel}_{>-l}X $. By the relative Calabi--Yau property, the space $ \Hom_{\mathcal(D)(A)}(\tau^{rel}_{>-l}X,Y) $ is isomorphic to $ D\Hom_{\mathcal(D)(A)}(Y,\Si^{n+1}\tau^{rel}_{>-l}X) $ which is zero because $ Y $ lies in $ {}^{\perp}(\mathcal{D}(A)^{\leqslant-n-l+1}) $ and $ \Si^{n+1}\tau^{rel}_{>-l}X\in\mathcal{D}(A)^{\leqslant-l-n}$ . Thus $ f=0 $, i.e,\ the functor $$ \tau^{rel}_{\leqslant-l}\colon\mathcal{F}^{rel}\langle l-1 \rangle\rightarrow\mathcal{F}^{rel}\langle l \rangle $$ is faithful.

\end{proof}

\begin{Prop}\label{Equivalence between the shifts of relative fundamental domain}
	
	For any positive integer $ l $, the functor $ \tau^{rel}_{\leqslant-l} $ induces an equivalence from $ \mathcal{F}^{rel}\langle l-1 \rangle $ to $ \mathcal{F}^{rel}\langle l \rangle $.
\end{Prop}
\begin{proof}
	This follows from Lemma~\ref{Denseness of relative truncation} and Lemma~\ref{Fully faithful of the shifts of relative fundamental domain}.
\end{proof}

\bigskip



\begin{Prop}\label{Extension group in RFD}
	Let $ X $ and $ Y $ be two objects in the Higgs category $ \mathcal{H} $. Let $ l>0 $ be an integer. Then we have
	$$ \Hom_{\mathcal{C}_{n}(A,B)}(X,\Si^{l}Y)\cong \Hom_{\mathcal{C}_{n}(\overline{A})}(p^{*}(X),\Si^{l}p^{*}(Y)) $$
\end{Prop}

\begin{proof}	
	For the object $ Y $, we have the following triangle in $ \mathrm{per}A $
	$$ Y\xrightarrow{f_{1}} P_{Y_{1}}\xrightarrow{g_{1}} Y_{1}\rightarrow \Si Y ,$$ where $ f_{1}\colon Y\rightarrow P_{Y_{1}} $ is a left $ \mathrm{add}(eA) $-approximation and $ Y_{1}\in\mathcal{F}^{rel}\langle1\rangle $ (see proof of Lemma \ref{Extension group in Z}). Since $ \tau^{rel}_{\leqslant-1}\colon\mathcal{F}^{rel}\rightarrow\mathcal{F}^{rel}\langle1\rangle $ is an equivalence, there is an object $ W_{1}\in\mathcal{F}^{rel} $ such that $ \tau^{rel}_{\leqslant-1}W_{1}\cong Y_{1} $. Thus we get a triangle in $ \mathcal{C}_{n}(A,B) $
	$$ Y\xrightarrow{\pi^{rel}(f_{1})}P_{Y_{1}}\xrightarrow{\pi^{rel}(g_{1})} W_{1}\rightarrow\Si Y .$$

	For the object $ W_{1} $, we have the following triangle in $ \mathrm{per}A $
	$$ W_{1}\xrightarrow{f_{2}} P_{Y_{2}}\xrightarrow{g_{2}} Y_{2}\rightarrow \Si W_{1} ,$$ where $ f_{2}\colon W_{1}\rightarrow P_{Y_{2}} $ is a left $ \mathrm{add}(eA) $-approximation and $ Y_{2}\in\mathcal{F}^{rel}\langle1\rangle $. For the same reason, there is an object $ W_{2}\in\mathcal{F}^{rel} $ such that $ \tau^{rel}_{\leqslant-1}W_{2}\cong Y_{2} $. Thus we get a triangle in $ \mathcal{C}_{n}(A,B) $
	$$ W_{1}\xrightarrow{\pi^{rel}(f_{2})}P_{Y_{2}}\xrightarrow{\pi^{rel}(g_{2})} W_{2}\rightarrow \Si W_{1} .$$
	
	Repeating this process, we can get the following $ l $ triangles in $ \mathrm{per}A $
	$$ Y\xrightarrow{f_{1}} P_{Y_{1}}\xrightarrow{g_{1}} Y_{1}\rightarrow \Si Y ,$$
	$$ W_{1}\xrightarrow{f_{2}} P_{Y_{2}}\xrightarrow{g_{1}} Y_{2}\rightarrow \Si W_{1}, $$
	$$ \cdots $$
	$$ W_{l-2}\xrightarrow{f_{l-1}} P_{Y_{l-1}}\xrightarrow{g_{l-1}} Y_{l-1}\rightarrow\Si W_{l-2} ,$$
	$$ W_{l-1}\xrightarrow{f_{l}} P_{Y_{l}}\xrightarrow{g_{l}} Y_{l}\rightarrow \Si W_{l-1} $$
	where for each $ 1\leqslant i\leqslant l $, $ f_{i} $ is a left $ \mathrm{add}(eA) $-approximation, $ Y_{i} $ is in $ \mathcal{F}^{rel}\langle1\rangle $ and $ \tau^{rel}_{\leqslant-1}W_{i}\cong Y_{i} $. 
	
	Thus we get $ l $ triangles in $ \mathcal{C}_{n}(A,B) $
	$$ Y\xrightarrow{\pi^{rel}(f_{1})} P_{Y_{1}}\xrightarrow{\pi^{rel}(g_{1})} W_{1}\rightarrow \Si Y ,$$
	$$ W_{1}\xrightarrow{\pi^{rel}(f_{2})} P_{Y_{2}}\xrightarrow{\pi^{rel}(g_{1})} W_{2}\rightarrow \Si W_{1}, $$
	$$ \cdots $$
	$$ W_{l-2}\xrightarrow{\pi^{rel}(f_{l-1})} P_{Y_{l-1}}\xrightarrow{\pi^{rel}(g_{l-1})} W_{l-1}\rightarrow\Si W_{l-2} ,$$
	$$ W_{l-1}\xrightarrow{\pi^{rel}(f_{l})} P_{Y_{l}}\xrightarrow{\pi^{rel}(g_{l})} W_{l}\rightarrow\Si W_{l-1} .$$
	
	Then we have 
	\begin{equation*}
		\begin{split}
			\Hom_{\mathcal{C}_{n}(A,B)}(X,\Si^{l}Y)&\simeq \Hom_{\mathcal{C}_{n}(A,B)}(X,\Si^{l-1}W_{1})\\
			&\simeq \Hom_{\mathcal{C}_{n}(A,B)}(X,\Si^{l-2}W_{2})\\
			&\cdots\\
			&\simeq \Hom_{\mathcal{C}_{n}(A,B)}(X,\Si W_{l-1})
		\end{split}
	\end{equation*}
	
	By the last triangle, we have the following exact sequence
	$$ \rightarrow \Hom_{\mathcal{C}_{n}(A,B)}(X,P_{Y_{l}})\xrightarrow{\Phi} \Hom_{\mathcal{C}_{n}(A,B)}(X,W_{l})\rightarrow \Hom_{\mathcal{C}_{n}(A,B)}(X,\Si W_{l-1})\rightarrow0. $$
	Thus, we have
	\begin{equation*}
		\begin{split}
			\Hom_{\mathcal{C}_{n}}(A,B)(X,\Si^{l}Y)&\simeq \Hom_{\mathcal{C}_{n}(A,B)}(X,\Si W_{l-1})\\
			&\simeq \Hom_{\mathcal{C}_{n}(A,B)}(X,W_{l})/\Ima(\Phi)\\
			&\simeq \Hom_{\mathrm{per}A}(X,W_{l})/\Ima(\Phi)\\
			&\simeq \Hom_{\mathrm{per}(\overline{A})}(p^{*}(X),p^{*}(W_{l}))\\
			&\simeq \Hom_{\mathcal{C}_{n}(\overline{A})}(p^{*}(X),p^{*}(W_{l}))\\
			&\simeq \Hom_{\mathcal{C}_{n}(\overline{A})}(p^{*}(X),\tau_{\leqslant-1}p^{*}(W_{l}))\\
			&\simeq \Hom_{\mathcal{C}_{n}(\overline{A})}(p^{*}(X),p^{*}(\tau^{rel}_{\leqslant-1}(W_{l})))\\
			&\simeq \Hom_{\mathcal{C}_{n}(\overline{A})}(p^{*}(X),p^{*}(Y_{l}))\\
			&\simeq \Hom_{\mathcal{C}_{n}(\overline{A})}(p^{*}(X),\Si^{l}p^{*}(Y)).
		\end{split}
	\end{equation*}

\end{proof}

\begin{Prop}\cite[Proposition 4.8.1]{LYG}\label{Extension sequence}
	Suppose that $ X $ and $ Y $ are two objects in $\mathcal{F}\subseteq \mathcal{C}_{n}(\overline{A}) $. Then there is a long exact sequence 
	$$ 0\rightarrow \Ext_{\mathcal{D}(\overline{A})}^{1}(X,Y)\rightarrow \Ext_{\mathcal{C}_{n}(\overline{A})}^{1}(X,Y)\rightarrow D\Ext_{\mathcal{D}(\overline{A})}^{n-1}(X,Y) $$
	$$ \rightarrow \Ext_{\mathcal{D}(\overline{A})}^{2}(X,Y)\rightarrow \Ext_{\mathcal{C}_{n}(\overline{A})}^{2}(X,Y)\rightarrow D\Ext_{\mathcal{D}(\overline{A})}^{n-2}(X,Y) $$
	$$ \rightarrow\qquad\cdots\qquad\rightarrow $$
	$$ \rightarrow \Ext_{\mathcal{D}(\overline{A})}^{n-1}(X,Y)\rightarrow \Ext_{\mathcal{C}_{n}(\overline{A})}^{n-1}(X,Y)\rightarrow D\Ext_{\mathcal{D}(\overline{A})}^{1}(X,Y)\rightarrow0. $$
	
\end{Prop}

\begin{Cor}\label{Cor:exact sequence for Higgs}
	Suppose that $ X $ and $ Y $ are two objects in the Higgs category $ \mathcal{H}\subseteq\mathcal{C}_{n}(A,B) $. Then there is a long exact sequence 
	$$ 0\rightarrow \Ext_{\mathcal{D}(A)}^{1}(X,Y)\rightarrow \Ext_{\mathcal{C}_{n}(A,B)}^{1}(X,Y)\rightarrow D\Ext_{\mathcal{D}(A)}^{n-1}(X,Y) $$
	$$ \rightarrow \Ext_{\mathcal{D}(A)}^{2}(X,Y)\rightarrow \Ext_{\mathcal{C}_{n}(A,B)}^{2}(X,Y)\rightarrow D\Ext_{\mathcal{D}(\overline{A})}^{n-2}(X,Y) $$
	$$ \rightarrow\qquad\cdots\qquad\rightarrow $$
	$$ \rightarrow \Ext_{\mathcal{D}(A)}^{n-1}(X,Y)\rightarrow \Ext_{\mathcal{C}_{n}(A,B)}^{n-1}(X,Y)\rightarrow D\Ext_{\mathcal{D}(A)}^{1}(X,Y)\rightarrow0. $$
\end{Cor}

\begin{proof}
	This follows from Proposition~\ref{Extension group in Z}, Proposition~\ref{Extension group in RFD} and Proposition~\ref{Extension sequence}.
	
\end{proof}

\begin{Prop}\label{Extension closed subcategory}
	The Higgs category $ \mathcal{H} $ is an extension closed subcategory of $ \mathcal{C}_{n}(A,B) $.		
\end{Prop}
\begin{proof}
	Let $ X $ and $ Y $ be two objects in $ \mathcal{H}\subseteq\mathcal{C}_{n}(A,B) $.  For the object $ Y $, we take a triangle in $ \mathrm{per}A $
	$$ Y\xrightarrow{f_{1}} P_{Y_{1}}\rightarrow Y_{1}\xrightarrow{\phi_{1}}\Si Y ,$$ where $ f_{1}\colon Y\rightarrow P_{Y_{1}} $	is a fixed left $ \mathrm{add}(eA)$-approximation and $ Y_{1}\in\mathcal{F}^{rel}\langle1\rangle $. Then we can get a triangle in $ \mathcal{C}_{n}(A,B) $
	$$ Y\xrightarrow{\pi^{rel}(f_{1})} P_{Y_{1}}\rightarrow Y_{1}\xrightarrow{\pi^{rel}(\phi_{1})}\Si Y .$$ 
	
	This induces a long exact sequence
	$$ \cdots\rightarrow \Hom_{\mathcal{C}_{n}(A,B)}(X,Y_{1})\rightarrow \Hom_{\mathcal{C}_{n}(A,B)}(X,\Si Y)\rightarrow \Hom_{\mathcal{C}_{n}(A,B)}(X,\Si P_{Y_{1}}) \rightarrow\cdots .$$
	
	Since $ P_{Y_{1}}\in \mathrm{pvd}_{B}(A)^{\perp} $, we have 
	$$ \Hom_{\mathcal{C}_{n}(A,B)}(X,\Si P_{Y_{1}})\cong  \Hom_{\mathrm{per}A}(X,\Si P_{Y_{1}})=0 .$$
	
	Thus we get the following surjective map
	$$ \cdots\rightarrow \Hom_{\mathcal{C}_{n}(A,B)}(X,Y_{1})\rightarrow \Hom_{\mathcal{C}_{n}(A,B)}(X,\Si Y)\rightarrow0. $$
	
	For the object $ X $, we have a canonical triangle in $ \mathrm{per}A $
	$$ \tau^{rel}_{\leqslant-1}X\rightarrow X\rightarrow \tau^{rel}_{\geqslant0}X\rightarrow \Si\tau^{rel}_{\leqslant-1}X.$$
	
	Hence, $ \tau^{rel}_{\leqslant-1}X $ is isomorphic to $ X $ in $ \mathcal{C}_{n}(A,B) $. Then we get the following exact sequence
	$$ \cdots\rightarrow \Hom_{\mathcal{C}_{n}(A,B)}(\tau^{rel}_{\leqslant-1}X,Y_{1})\rightarrow \Hom_{\mathcal{C}_{n}(A,B)}(X,\Si Y)\rightarrow0. $$
	
	It is clear that $ \tau^{rel}_{\leqslant-1}X $ and $ Y_{1} $ are in $ \mathcal{F}^{rel}\langle1\rangle $. Since $ \pi^{rel}\colon \mathcal{F}^{rel}\cong\mathcal{F}^{rel}\langle1\rangle\rightarrow\mathcal{C}_{n}(A,B) $ is also fully faithful, we have that the space $ \Hom_{\mathcal{C}_{n}(A,B)}(\tau^{rel}_{\leqslant-1}X,Y_{1}) $ is isomorphic to $ \Hom_{\per A}(\tau^{rel}_{\leqslant-1}X,Y_{1}) $ and the following sequence is exact
	$$ \cdots\rightarrow \Hom_{\per A}(\tau^{rel}_{\leqslant-1}X,Y_{1})\rightarrow \Hom_{\mathcal{C}_{n}(A,B)}(X,\Si Y)\rightarrow0. $$
	
	Let $ \epsilon $ be an element in $ \Hom_{\mathcal{C}_{n}(A,B)}(X,\Si Y) $. We suppose that the corresponding triangle in $ \mathcal{C}_{n}(A,B) $ is given by 
	$$ Y\rightarrow W\rightarrow X\xrightarrow{\varepsilon}\Si Y .$$
	We need to show that $ W $ is also in $ \mathcal{H} $.
	
	Since the map $ \Hom_{\mathrm{per}A}(\tau^{rel}_{\leqslant-1}X,Y_{1})\rightarrow \Hom_{\mathcal{C}_{n}(A,B)}(X,\Si Y) $ is surjective, there is a morphism $ \epsilon'\colon\tau^{rel}_{\leqslant-1}X\rightarrow Y_{1} $ in $ \mathrm{per}A $ such that $ \pi^{rel}(\phi_{1}\circ\varepsilon')\cong\varepsilon $ in $ \mathcal{C}_{n}(A,B) $.
	
	We take a triangle in $ \mathrm{per}A $
	$$ Y\rightarrow W_{1}\rightarrow\tau^{rel}_{\leqslant-1}X\xrightarrow{\phi_{1}\circ\varepsilon'}\Si Y .$$ Then, the following morphism of triangles in $ \mathcal{C}_{n}(A,B) $ is an isomorphism
	\begin{align*}
		\xymatrix{
			Y\ar[r]\ar[d]_{\boldmath{1}_{Y}}& W_{1}\ar[r]\ar[d]&\tau^{rel}_{\leqslant-1}X\ar"1,5"^{\pi^{rel}(\phi_{1}\circ\varepsilon')}\ar[d]&& \Si Y\ar[d]^{\boldmath{1}_{Y}}\\
			Y\ar[r]& W\ar[r]&X\ar"2,5"^{\varepsilon}&&\Si Y.
		}
	\end{align*}
	In particular, $ W_{1} $ is isomorphic to $ W $ in $ \mathcal{C}_{n}(A,B) $.
	
	Since $ Y $ and $ \tau^{rel}_{\leqslant-1}X $ are in $ \mathcal{Z}\subseteq \mathrm{per}A $, $ W_{1} $ is also in $ \mathcal{Z} $. It is easy to see that 
	$$ p^{*}(Y)\in\mathcal{F}=\mathcal{D}(\overline{A})^{\leqslant0}\cap{}^{\perp}(\mathcal{D}(\overline{A})^{\leqslant-n})\cap \mathrm{per}(\overline{A}) $$ and $$ p^{*}(\tau^{rel}_{\leqslant-1}X)\cong\tau_{\leqslant-1}(p^{*}(X))\in\Si\mathcal{F}=\mathcal{D}(\overline{A})^{\leqslant-1}\cap{}^{\perp}(\mathcal{D}(\overline{A})^{\leqslant-n-1})\cap \mathrm{per}(\overline{A}) .$$
	
	Then by the triangle in $ \mathrm{per}A $
	$$ Y\rightarrow W_{1}\rightarrow\tau^{rel}_{\leqslant-1}X\xrightarrow{\phi_{1}\circ\varepsilon'}\Si Y ,$$ we can see that $ p^{*}(W_{1}) $ is in $ \mathcal{D}(\overline{A})^{\leqslant0}\cap{}^{\perp}(\mathcal{D}(\overline{A})^{\leqslant-n-1})\cap \mathrm{per}(\overline{A}) $.
	
	Next we consider the object $ \tau^{rel}_{\leqslant-1}W_{1}\in \mathrm{per}A $. Since $ W_{1} $ is in $ \mathcal{Z} $, $ \tau^{rel}_{\leqslant-1}W_{1} $ is still in $ \mathcal{Z} $. And we have a canonical triangle in $ \mathrm{per}(\overline{A}) $
	$$ \tau_{\leqslant-1}(p^{*}(W_{1}))\rightarrow p^{*}(W_{1})\rightarrow \tau_{\geqslant 0}(p^{*}(W_{1}))\rightarrow\Si \tau_{\leqslant-1}(p^{*}(W_{1})) .$$
	Because $ p^{*}(W_{1}) $ is in $ \mathcal{D}(\overline{A})^{\leqslant0}\cap{}^{\perp}(\mathcal{D}(\overline{A})^{\leqslant-n-1})\cap \per(\overline{A}) $, we have 
	$$ \tau_{\leqslant-1}(p^{*}(W_{1}))\in \mathcal{D}(\overline{A})^{\leqslant-1}\cap{}^{\perp}(\mathcal{D}(\overline{A})^{\leqslant-n-1})\cap \per(\overline{A})=\Si\mathcal{F} .$$
	
	Thus the object $ \tau^{rel}_{\leqslant-1}W_{1} $ is in $ \mathcal{F}^{rel}\langle1\rangle $. By the equivalence $ \tau^{rel}_{\leqslant-1}\colon\mathcal{F}^{rel}\rightarrow\mathcal{F}^{rel}\langle1\rangle $, there exists an object $ W_{2}\in\mathcal{F}^{rel} $ such that $ \tau^{rel}_{\leqslant-1}W_{2}\cong\tau^{rel}_{\leqslant-1}W_{1} $.
	
	Since $ W_{2} $ and $ W_{1} $ are isomorphic in $ \mathcal{C}_{n}(A,B) $, $ W_{2} $ is isomorphic to $ W $ in $ \mathcal{C}_{n}(A,B) $. Thus $ W $ is an object in $ \mathcal{H}\subseteq\mathcal{C}_{n}(A,B) $. Therefore, $ \mathcal{H} $ is an extension closed subcategory of $ \mathcal{C}_{n}(A,B) $.

\end{proof}

Recall that a full subcategory $ \mathcal{P} $ of a triangulated category $ \mathcal{T} $ is presilting if $ \Hom_{\mathcal{T}}(\mathcal{P},\Si^{>0}\mathcal{P}) = 0 $.

\begin{Prop}\label{P is a presilting in RCA}
	\begin{itemize}
		\item[(1)] $ \mathcal{P}=\add(eA) $ is a presilting subcategory of $ \mathcal{C}_{n}(A,B)=\mathrm{per}A/\mathrm{pvd}_{B}(A) $.
		\item[(2)] $ \mathcal{P} $ is covariantly finite in $ {}^{\perp_{\mathcal{C}_{n}(A,B)}}(\Si^{>0}\mathcal{P}) $ and contravariantly finite in $ (\Si^{<0}\mathcal{P})^{\perp_{\mathcal{C}_{n}(A,B)}} $.
		\item[(3)] For any $ X\in\mathcal{C}_{n}(A,B) $, we have $ \Hom_{\mathcal{C}_{n}(A,B)}(X,\Si^{l}P)=0=\Hom_{\mathcal{C}_{n}(A,B)}(P,\Si^{l} X) $ for $ l\gg0 $.
	\end{itemize}
\end{Prop}
\begin{proof}
	For $ P\in\cp $, $ X\in\per A $ and $ m\in\mathbb{Z} $, we have isomorphisms
	$$ \Hom_{\per A}(P,\Si^{m}X)\iso\Hom_{\cc_{n}(A,B)}(P,\Si^{m}X) $$
	and
	$$ \Hom_{\per A}(X,\Si^{m}P)\iso\Hom_{\cc_{n}(A,B)}(X,\Si^{m}P) $$
	because $ \cp $ is left orthogonal and right orthogonal to $ \mathrm{pvd}_{B}(A) $. This implies $ (1) $, $ (2) $ and $ (3) $.
\end{proof}

\begin{Cor}\label{Cor: category E}
	Let $ \mathcal{E} $ be the following additive subcategory of $ \mathcal{C}_{n}(A,B) $
	$$ \mathcal{E}={}^{\perp_{\mathcal{C}_{n}(A,B)}}(\Si^{>0}\mathcal{P})\cap(\Si^{<0}\mathcal{P})^{\perp_{\mathcal{C}_{n}(A,B)}} .$$ 
	
	Then the composition $ \mathcal{E}\subseteq\mathcal{C}_{n}(A,B)\xrightarrow{p^{*}}\mathcal{C}_{n}(\overline{A}) $ induces a triangle equivalence 
	$$ \mathcal{E}/[\mathcal{P}]\overset{\sim}{\longrightarrow}\mathcal{C}_{n}(\overline{A}) .$$
\end{Cor}
\begin{proof}
	This follows from Proposition~\ref{P is a presilting in RCA} and Theorem~\ref{silting reduction}.
\end{proof}

\begin{Thm}\label{Higgs category is a Silting reduction}
	
	The Higgs category $ \mathcal{H}\subseteq\mathcal{C}_{n}(A,B) $ is equal to $\mathcal{E}={}^{\perp_{\mathcal{C}_{n}(A,B)}}(\Si^{>0}\mathcal{P})\cap(\Si^{<0}\mathcal{P})^{\perp_{\mathcal{C}_{n}(A,B)}} $. In particular, the Higgs category $ \mathcal{H} $ is idempotent complete.
	
\end{Thm}
\begin{proof}
	It is clear that we have the inclusion $ \mathcal{H}\subseteq\mathcal{E} $. Let $ X $ be an object in $ \mathcal{E}={}^{\perp_{\mathcal{C}_{n}(A,B)}}(\Si^{>0}\mathcal{P})\cap(\Si^{<0}\mathcal{P})^{\perp_{\mathcal{C}_{n}(A,B)}} $. We also view it as an object of $ \per A $.
	
	Since $ \mathrm{per}(eAe) $ and $ \mathrm{pvd}_{B}(A) $ are left orthogonal and right orthogonal to each other, we see that $ X $ is in $ \mathcal{Z}={}^{\perp_{\mathrm{per}A}}(\Si^{>0}\mathcal{P})\cap(\Si^{<0}\mathcal{P})^{\perp_{\mathrm{per}A}}\subseteq \mathrm{per}A $. 
	
	For the object $ p^{*}(X)\in \mathrm{per}(\overline{A}) $, there exists a non-negative integer $ r $ such that $ p^{*}(X) $ is in $ {}^{\perp}(\mathcal{D}(\overline{A})^{\leqslant-n-r}) $. We consider the object $ X'=\tau^{rel}_{\leqslant -r}X $. Then $ X' $ becomes isomorphic to $ X $ in $ \mathcal{C}_{n}(A,B) $ and $ X' $ belongs to $ \mathcal{F}^{rel}\langle r\rangle $. By Proposition~\ref{Equivalence between the shifts of relative fundamental domain}, there exists an object $ Y $ in $ \mathcal{F}^{rel} $ such that $ Y $ is isomorphic to $ X' $ in $ \mathcal{C}_{n}(A,B) $. Thus, $ X $ is in the image of $ \mathcal{F}^{rel} $, i.e.\ $ X $ belongs to $ \mathcal{H} $. Hence $ \mathcal{H} $ is equal to $ \mathcal{E} $.
	
	
\end{proof}

\begin{Thm}\label{eA and RFD generate Relative Cluster Category}
	For any object $ X\in\mathcal{C}_{n}(A,B) $, there exists $ l\in\mathbb{Z} $, $ F\in \mathcal{H} $ and $ P\in \mathrm{per}(eAe) $, such that we have a triangle in $ \mathcal{C}_{n}(A,B) $
	\begin{align*}
		\xymatrix{
			\Si^{l}F\ar[r]&X\ar[r]&P\ar[r]&\Si^{l+1}F
		}.
	\end{align*}
	Dually, there exist $ m\in\mathbb{Z} $, $ F'\in\mathcal{H} $ and $ P'\in \mathrm{per}(eAe) $, such that we have a triangle in $ \mathcal{C}_{n}(A,B) $
	\begin{align*}
		\xymatrix{
			P'\ar[r]&X\ar[r]&\Si^{m}F'\ar[r]&\Si P'
		}.
	\end{align*}
\end{Thm}

\begin{proof}	
	We only show the first statement since the second statement can be shown dually. Let $ X $ be an object in $ \mathcal{C}_{n}(A,B) $. We view it as an object in $ \mathrm{per}A $. There exists a positive integer $ r_{1} $ such that the object $ X $ is in $ \mathcal{D}(A)^{\leqslant r_{1}} $. Consider the object $ \Si^{r_{1}}X $. Then $ \Si^{r_{1}}X $ lies in $ \mathcal{D}(A)^{\leqslant 0} $. 
	
	By Proposition~\ref{Presilting to co-t-structure}, the pairs $( {}^{\perp_{\mathcal{T}}}\mathcal{S}_{<0} , \mathcal{S}_{\leqslant0} ) $ and $ (\mathcal{S}_{\geqslant0},\mathcal{S}_{>0}^{\perp_{\mathcal{T}}} ) $ are co-t-structures on $ \ct=\mathrm{per}A $,  where
	$$ \mathcal{S}_{\geqslant l}=\mathcal{S}_{>l-1}\coloneqq\bigcup_{i\geqslant0}\Si^{-l-i}\mathcal{P}*\cdots*\Si^{-l-1}\mathcal{P}*\Si^{-l}\mathcal{P},$$
	$$ \mathcal{S}_{\leqslant l}=\mathcal{S}_{<l+1}\coloneqq\bigcup_{i\geqslant0}\Si^{-l}\mathcal{P}*\Si^{-l+1}\mathcal{P}\cdots*\Si^{-l+i}\mathcal{P} $$ and $ \mathcal{P}=\mathrm{add}(eA) $. Hence we have a triangle
	$$ X'\rightarrow \Si^{r_{1}}X\rightarrow S\rightarrow \Si X', $$ where $ X'\in{}^{\perp}(\mathcal{S}_{<0}) $ and $ S\in\mathcal{S}_{<0}\subseteq \mathcal{D}(A)^{\leqslant-1} $. We can see that $ X' $ belongs to $ \mathcal{D}(A)^{\leqslant0} $.

	\emph{Step 1: The object $ X' $ is in $ \mathcal{Z}= {}^{\perp}(\Si^{>0}\mathcal{P})\cap(\Si^{<0}\mathcal{P})^{\perp}$.}
	
	Since $ X'\in{}^{\perp}(\mathcal{S}_{<0}) $, it is enough to show that $ X' $ is also in $ (\Si^{<0}\mathcal{P})^{\perp} $. For any positive integer $ k $, we have 
	\begin{equation*}
		\begin{split}
			\Hom_{\mathcal{D}(A)}(eA,\Si^{k}X')&\cong \Hom_{\mathcal{D}(A)}(i^{*}(eAe),\Si^{k}X')\\
			&\cong \Hom_{\mathcal{D}(eAe)}(eAe,\Si^{k}i_{*}(X')).
		\end{split}
	\end{equation*}
	
	The space $ \Hom_{\mathcal{D}(A)}(eA,\Si^{k}X') $ vanishes for any positive integer $ k $. Thus the object $ X' $ is in $ \mathcal{Z} $.

	\emph{Step 2: There exists an object $ W\in\mathcal{F}^{rel} $ such that $ W $ is isomorphic to $ X' $ in $\mathcal{C}_{n}(A,B)$.}
	
	By Step 1, the object $ X' $ is in $ \mathcal{Z}\subseteq\per A $. There is a non-negative integer $ r_{2} $ such that $ p^{*}(X')\in{}^{\perp}(\mathcal{D}(\overline{A})^{\leqslant-n-r_{2}}) $. Consider the object $ W'=\tau^{rel}_{\leqslant-r_{2}}X' $. Then $ W' $ is isomorphic to $ X' $ in $ \mathcal{C}_{n}(A,B) $. 
	
	By the definition of $ \tau^{rel}_{\leqslant -r_{2}}X' $, we have the following triangle in $ \per\overline{A} $
	$$ p^{*}(\tau^{rel}_{\leqslant-r_{2}}X')\ra p^{*}(X')\ra\tau_{>-r_{2}}(p^{*}(X'))\ra\Si p^{*}(\tau^{rel}_{\leqslant-r_{2}}X') .$$ Hence $ p^{*}(\tau^{rel}_{\leqslant-r_{2}}X') $ is isomorphic to $ \tau_{\leqslant-r_{2}}(p^{*}(X')) $. For any object $ Y $ in $ \mathcal{D}(\overline{A})^{\leqslant-n-r_{2}} $, by the $ (n+1) $-Calabi--Yau property of $ \overline{A} $, the space $ \Hom_{\cd(\overline{A})}(\Si^{-1}\tau_{>-r_{2}}(p^{*}(X')),Y) $ is isomorphic to $$ D\Hom_{\cd(\overline{A})}(Y,\Si^{n}\tau_{>-r_{2}}p^{*}(X')) $$ and so vanishes. Therefore, $ p^{*}(\tau^{rel}_{\leqslant-r_{2}}X')\cong \tau_{\leqslant-r_{2}}(p^{*}(X')) $ lies in $ {}^{\perp}(\mathcal{D}(\overline{A})^{\leqslant-n-r_{2}}) $. Thus, the object $ p^{*}(\tau^{rel}_{\leqslant-r_{2}}X') $ lies in $ {}^{\perp}(\mathcal{D}(\overline{A})^{\leqslant-n-r_{2}})\cap\cd(\overline{A})^{\leqslant-r_{2}}\cap\per(\overline{A}) $ which is equal to $ \Si^{r_{2}}\cf $. This shows that $ W' $ belongs to $ \mathcal{F}^{rel}\langle r_{2}\rangle $. 
	
	By Proposition~\ref{Equivalence between the shifts of relative fundamental domain}, there exists an object $ W $ in $ \mathcal{F}^{rel} $ such that $ W $ is isomorphic to $ W' $ in $ \mathcal{C}_{n}(A,B) $. Thus, we get the following triangle in $ \mathcal{C}_{n}(A,B) $ 
	$$ W\rightarrow\Si^{r_{1}}X\rightarrow S\rightarrow\Si W'' ,$$ where $ W $ is in $ \mathcal{H} $ and $ S $ is in $ \mathrm{per}(eAe) $. 
	
\end{proof}

\subsection{Frobenius $ m $-exangulated categories}

In this subsection, we describe our results using the framework of $ m $-exangulated categories. We refer to the readers to~\cite{NP2019},~\cite{HLN2021},~\cite{HLN2021part2} and ~\cite{LZ2020} for the relevant definitions and facts concerning $ m $-exangulated categories.

\begin{Def}\rm\cite[Deﬁnition 3.2]{LZ2020}
	Let $ (\mathcal{C},\mathbb{E},\mathfrak{s}) $ be an $ m $-exangulated category.
	\begin{itemize}
		\item[(1)] An object $ P\in\mathcal{C} $ is called \emph{projective} if for any distinguished $ m $-exangle
		$$ A_{0}\xrightarrow{\alpha_{0}} A_{1}\rightarrow\cdots\rightarrow A_{m}\xrightarrow{\alpha_{m}} A_{m+1}\stackrel{\delta}{-\rightarrow} $$ and any morphism $ c $ in $ \mathcal{C}(P,A_{m+1}) $, there exists a morphism $ b\in\mathcal{C}(P, A_{m}) $ satisfying $ \alpha_{m}b=c $. We denote the full subcategory of projective objects in $ \mathcal{C} $ by $ \mathcal{P} $. Dually, the full subcategory of injective objects in $ \mathcal{C} $ is denoted by $ \mathcal{I} $.
		\item[(2)] We say that $ \cc $ has \emph{enough projectives} if for any object $ C\in\mathcal{C}$, there exists a distinguished $ m $-exangle
		$$ B\xrightarrow{\alpha_{0}} P_{1}\rightarrow\cdots\rightarrow P_{m}\xrightarrow{\alpha_{m}} C\stackrel{\delta}{-\rightarrow} $$ satisfying $ P_{1},P_{2},\cdots,P_{m}\in\mathcal{P} $. We can define the notion of \emph{having enough injectives} dually.
		\item[(3)] $ \mathcal{C} $ is said to be \emph{Frobenius} if $ \mathcal{C} $ has enough projectives and enough injectives and if moreover the projectives coincide with the injectives.
	\end{itemize}
\end{Def}

\begin{Rem}
	In the case $ m=1 $, these agree with the usual definitions of extriangulated categories (see \cite[Definition 3.23, Definition 3.25 and Definition 7.1]{NP2019}).
\end{Rem}

\begin{Thm}\label{Higgs is Frobenius extrianglated category}
	The Higgs category $ \mathcal{H} $ carries a canonical structure of Frobenius extriangulated category with projective-injective objects $ \mathcal{P}=\mathrm{add}(eA) $. The functor $ p^{*}\colon\mathcal{C}_{n}(A,B)\rightarrow\mathcal{C}_{n}(\overline{A}) $ induces an equivalence of triangulated categories
	$$ \mathcal{H}/[\mathcal{P}]\overset{\sim}{\longrightarrow}\mathcal{C}_{n}(\overline{A}) .$$
\end{Thm}

\begin{proof}
	\emph{Step 1: $ \mathcal{H} $ is an extriangulated category.}

	By Proposition~\ref{Extension closed subcategory}, the Higgs category $ \mathcal{H} $ is an extension closed subcategory of $ \mathcal{C}_{n}(A,B) $. Then by~\cite[Remark 2.18]{NP2019}, $ \mathcal{H} $ is an extriangulated category and $ (\mathcal{H},\mathbb{E},\mathfrak{s}) $ can be described as follows:
	\begin{itemize}
		\item[(1)] For any two objects $ X,Y\in\!\mathcal{H}\subseteq\mathcal{C}_{n}(A,B) $, the $ \mathbb{E} $-extension space $ \mathbb{E}(X,Y) $ is given by $$ \Hom_{\mathcal{C}_{n}(A,B)}(X,\Si Y) .$$
		\item[(2)] For any $ \delta\in\mathbb{E}(X,Y)=\Hom_{\mathcal{C}_{n}(A,B)}(Z,\Si X) $, take a distinguished triangle $$ X\xrightarrow{f} Y\xrightarrow{g} Z\xrightarrow{\delta} \Si X $$ and define $ \mathfrak{s}(\delta)=[X\xrightarrow{f} Y\xrightarrow{g} Z] $. This $ \mathfrak{s}(\delta) $ does not depend on the choice of the distinguished triangle above.
	\end{itemize}

	\emph{Step 2: $ \mathcal{H} $ has has enough injectives and the full subcategory of injective objects in $ \mathcal{H} $ is $ \mathcal{P}=\mathrm{add}(eA) $.}
	
	Let $ I $ be an object in $ \mathrm{add}(eA) $. For any distinguished triangle in $ \mathcal{H} $
	$$ X\rightarrow Y\rightarrow Z\stackrel{\delta}{-\rightarrow} ,$$ 
	the space $ \Hom_{\mathcal{C}_{n}(A,B}(\Si^{-1}Z,I)\cong \Hom_{\per A}(Z,\Si I) $ vanishes since $ Z\in\mathcal{Z}={}^{\perp}(\Si^{>0}\mathcal{P})\cap(\Si^{<0}\mathcal{P})^{\perp}\subseteq\per A $. Thus, we have the following exact sequence
	$$ \Hom_{\mathcal{C}_{n}(A,B)}(Y,I)\rightarrow \Hom_{\mathcal{C}_{n}(A,B)}(X,I)\rightarrow0 .$$
	Thus, any object in $ \add(eA) $ is injective.
	
	Now let $ X $ be an object in $ \mathcal{H}\subseteq\mathcal{C}_{n}(A,B) $. Then $ X $ is an object in $ \mathcal{Z}\subseteq\per A $. We take a triangle in $ \per A $
	$$ X\xrightarrow{l_{X}} P_{X}\rightarrow X_{1}\rightarrow \Si X $$ with a left $ \mathcal{P}=\mathrm{add}(eA) $-approximation $ l_{X} $ and $ X_{1}\in\mathcal{Z} $. It is easy to see that $ X_{1} $ is in $ \mathcal{F}^{rel}\langle1\rangle $. By Proposition~\ref{Equivalence between the shifts of relative fundamental domain}, there is an object $ X_{2}\in\mathcal{F}^{rel} $ such that $ \tau^{rel}_{\leqslant-1}X_{2}\cong X_{1} $. Thus, we have a triangle in $ \mathcal{C}_{n}(A,B) $
	$$ X\xrightarrow{l_{X}} P_{X}\rightarrow X_{2}\rightarrow\Si X $$ with $ P_{X} $ in $ \mathrm{add}(eA) $ and $ l_{X} $ an inflation. Therefore, $ \mathcal{H} $ has has enough injectives. 
	
	It remains to show that any injective object is in $ \mathrm{add}(eA) $. Let $ J $ be an injective object in $ \mathcal{H} $. We take a triangle in $ \mathrm{per}A $
	$$ J\xrightarrow{l_{J}} P_{J}\rightarrow J_{1}\rightarrow\Si J $$ with a left $ \mathcal{P}=\mathrm{add}(eA) $-approximation $ l_{J} $ and $ J_{1}\in\mathcal{Z} $. Since $ J $ is injective, the morphism $ l_{J}\colon J\rightarrow P_{J} $ is split in $ \mathcal{H}\subseteq\mathcal{C}^{rel}_{n}(A,B) $. Thus $ l_{J} $ is also split in $ \mathcal{F}^{rel}\subseteq\mathcal{Z}\subseteq \mathrm{per}A $. 
	
	Therefore, $ J $ belongs to $ \add(eA) $ and the subcategory of injective objects in $ \mathcal{H} $ is $ \mathcal{P}=\mathrm{add}(eA) $.

	\emph{Step 3: $ \mathcal{H} $ has has enough projectives and the full subcategory of projective objects in $ \mathcal{H} $ is $ \mathcal{P}=\mathrm{add}(eA) $.}
	
	This follows from the dual of the argument in Step 2.

	\emph{Step 4: $ \mathcal{H} $ is a Frobenius extriangulated category.}
	
	By Steps 1, 2, and 3, the Higgs category $ \mathcal{H} $ is a Frobenius extriangulated category with projective-injective objects $ \mathcal{P}=\mathrm{add}(eA) $. By Corollary~\ref{Equivalence between fundamental domain}, we have the equivalence between triangulated categories
	$$ \mathcal{H}/[\mathcal{P}]\cong\mathcal{F}\cong\mathcal{C}_{n}(\overline{A}) .$$
	
\end{proof}

\subsection{Higher extensions in an extriangulated category}

Let $ (\mathcal{C},\mathbb{E},\mathfrak{s}) $ be an extriangulated category. Assume that it has enough projectives and injectives, and let $ \mathcal{P}\subseteq\mathcal{C} $ (respectively, $ \mathcal{I}\subseteq\mathcal{C} $) denote the full subcategory of projectives (resp. injectives). We denote the ideal quotients $ \mathcal{C}/[\mathcal{P}] $ and $ \mathcal{C}/[\mathcal{I}] $ by $ \underline{\mathcal{C}} $ and $ \overline{\mathcal{C}} $, respectively. The extension group bifunctor $ \mathbb{E}\colon\mathcal{C}^{op}\times\mathcal{C}\rightarrow \mathrm{Ab} $ induces $ \mathbb{E}\colon\underline{\mathcal{C}^{op}}\times\overline{\mathcal{C}}\rightarrow \mathrm{Ab} $, which we denote by the same symbol. To define the higher extension groups, we need the following assumptions
\begin{assumption}\rm
	Each object $ A\in\mathcal{C} $ is assigned the following data (i) and (ii).
	\begin{itemize}
		\item[(i)] A pair $ (\Sigma A,l^{A}) $ of an object $ \Sigma A\in\mathcal{C} $ and an extension $ l^{A}\in\mathbb{E}(\Sigma A,A) $, for
		which $ \mathfrak{s}(l^{A})=[A\rightarrow I\rightarrow \Sigma A] $ satisfies $ I\in\mathcal{I} $.
		\item[(ii)] A pair $ (\Omega A,\omega^{A}) $ of an object $ \Omega A\in\mathcal{C} $ and an extension $ \omega^{A}\in\mathbb{E}(A,\Omega A) $, for
		which $ \mathfrak{s}(\omega^{A})=[\Omega A\rightarrow P\rightarrow A] $ satisfies $ P\in\mathcal{P} $. 
	\end{itemize}
\end{assumption}

\begin{Def}\rm\cite[Definition 3.6]{HLN2021part2}\label{Def:Higher extension group}
	Let $ i\geqslant1 $ be any integer. Define a biadditive functor $ \mathbb{E}^{i}\colon\mathcal{C}^{op} \times\mathcal{C}\rightarrow Ab $ to be the composition of
	$$ \mathcal{C}^{op} \times\mathcal{C}\rightarrow\underline{\mathcal{C}^{op}} \times\overline{\mathcal{C}}\xrightarrow{\mathrm{Id}\times \Sigma^{i-1}}\underline{\mathcal{C}^{op}} \times\overline{\mathcal{C}}\xrightarrow{\mathbb{E}}\mathrm{Ab}, $$
	where $ \Sigma^{i-1} $ is the $ (i − 1) $-times iteration of the endfunctor $ \Sigma $.
	
\end{Def}

%
%
%

By Theorem~\ref{Higgs is Frobenius extrianglated category}, the Higgs category $ \mathcal{H} $ is a Frobenius extriangulated category (or Frobenius 1-exangulated category) with projective-injective objects $ \mathcal{P}=\mathrm{add}eA $. Thus the higher extension can be computed as follows:

Let $ X $ and $ Y $ be two objects in $ \mathcal{H} $. Let $ l>0 $ be an integer. We have 
$$ \mathbb{E}^{l}(X,Y)=\Hom_{\mathcal{Z}/[\mathcal{P}]}(X,Y\langle l\rangle)\cong \Hom_{\mathcal{C}_{n}(\overline{A})}(p^{*}(A),\Si^{l}p^{*}(Y))\cong \Hom_{\mathcal{C}_{n}(A,B)}(X,\Si^{l} Y) .$$

\begin{Def}\rm\cite[Definition 3.21]{HLN2021part2}
	Let $ (\mathcal{C},\mathbb{E},\mathfrak{s}) $ be a Frobenius extriangulated category. Let $ \mathcal{T}\subseteq\mathcal{C} $ be a full additive subcategory closed under isomorphisms and direct summands. Then $ \mathcal{T} $ is called an \emph{n-cluster tilting subcategory} of $ \mathcal{C} $, if it satisfies the following conditions.
	\begin{itemize}
		\item[(1)] $ \mathcal{T}\subseteq\mathcal{C} $ is functorially finite.
		\item[(2)] For any $ C\in\mathcal{C} $, the following are equivalent.
		\subitem(i) $ C\in\mathcal{T} $,
		\subitem(ii) $ \mathbb{E}^{i}(C,\mathcal{T})=0 $ for any $ 1\leqslant i\leqslant n-1 $,
		\subitem(iii) $ \mathbb{E}^{i}(\mathcal{T},C)=0 $ for any $ 1\leqslant i\leqslant n-1 $.
	\end{itemize}
\end{Def}

\begin{Prop}\label{addA is an n-cluster tiliting}
	The category $ \add(\pi^{rel}A) $ is an $ n $-cluster-tilting subcategory of $ \mathcal{H} $.
\end{Prop}
\begin{proof}
	Since $ \ch $ is Hom-finite, it is clear that $ \add(\pi^{rel}A) $ is functorially finite in $ \mathcal{H} $. 
	
	\emph{Step 1: $ \pi^{rel}(A) $ is an $ n $-rigid object in $ \mathcal{H} $.}
	
	By Proposition~\ref{Extension group in RFD}, \cite[Theorem 7.1]{Am2008} and \cite[Theorem 3.2.2]{LYG}, we have that
	\begin{equation*}
		\begin{split}
			\Hom_{\mathcal{C}_{n}(A,B)}(\pi^{rel}(A),\Si^{i}\pi^{rel}(A))&\simeq \Hom_{\mathcal{C}_{n}(\overline{A})}(\overline{A},\Si^{i}\overline{A})\\
			&\simeq0
		\end{split}
	\end{equation*}
	for any $ 1\leqslant i\leqslant n-1 $.

	\emph{Step 2: Let $ X $ be an object in $ \mathcal{H} $ satisfying $ \mathbb{E}^{i}(X,\mathrm{add}A)=0 $ for $ 1\leqslant i\leqslant n-1 $. Then $ X $ is in $ \mathrm{add}A $.}
	
	Since $ \mathbb{E}^{i}(X,\mathrm{add}A)=0 $ for $ 1\leqslant i\leqslant n-1 $, we have $ \Hom_{\mathcal{C}_{n}(\overline{A})}(p^{*}(X),\mathrm{add}(\overline{A}))=0 $ for $ 1\leqslant i\leqslant n-1 $. We know that $ \mathrm{add}(\overline{A}) $ is an $ n $-cluster tilting subcategory of $ \mathcal{F}\cong\mathcal{C}_{n}(\overline{A}) $ (see \cite{Am2008,LYG}). Hence $ p^{*}(X) $ is in $ \add(\overline{A}) $. By the equivalence $ p^{*}\colon\mathcal{Z}/[\mathcal{P}]\iso\mathcal{F}\iso\mathcal{C}_{n}(\overline{A}) $, the object $ X $ is in $ \mathrm{add}A $.

	\emph{Step 3: Let $ X $ be an object in $ \mathcal{H} $ satisfying $ \mathbb{E}^{i}(\mathrm{add}A,X)=0 $ for $ 1\leqslant i\leqslant n-1 $. Then $ X $ is in $ \mathrm{add}A $.}
	
	This follows by a similar argument to that in Step 2.
	
	Thus, the category $ \mathrm{add}A $ is an $ n $-cluster tilting subcategory of $ \mathcal{H} $.
\end{proof}

\begin{Prop}\label{Condition 1 for Frobenius n-exangulated}
	Suppose that the $ n $-cluster tilting category $ \add\overline{A} $ of $ \cc_{n}(\overline{A}) $ satisfies $ \Si^{n}(\add\overline{A})=\add\overline{A} $. Then we have:
	\begin{itemize}
		\item[(1)] If $ X\in \mathcal{H} $ satisfies $ \mathbb{E}^{n-1}(\mathrm{add}A,X)=0 $, then there is $ \mathfrak{s} $-triangle
		$$ Y\xrightarrow{f} P\rightarrow X\stackrel{}{-\rightarrow} \quad (P\in\mathcal{P}=\mathrm{add}(eA)) $$ for which
		$$ \mathcal{H}(T,f)\colon\Hom_{\mathcal{H}}(T,Y)\rightarrow \Hom_{\mathcal{H}}(T,P) $$ is injective for any $ T\in \mathrm{add}(A) $.
		\item[(2)] Dually, if $ Z\in \mathcal{H} $ satisfies $ \mathbb{E}^{n-1}(Z,\mathrm{add}A)=0 $, then there is $ \mathfrak{s} $-triangle
		$$ Z\rightarrow I\xrightarrow{g} W\stackrel{}{-\rightarrow} \quad (I\in\mathcal{P}=\mathrm{add}(eA))  $$ for which
		$$ \mathcal{H}(g,T)\colon\Hom_{\mathcal{H}}(W,T)\rightarrow \Hom_{\mathcal{H}}(I,T) $$
		is injective for any $ T\in \mathrm{add}(A) $. 
	\end{itemize}
\end{Prop}
\begin{proof}
	We only show the first statement since the second statement can be shown dually. Let $ X $ be an object in $ \mathcal{H} $ which satisfies $$ \mathbb{E}^{n-1}(\mathrm{add}A,X)=\Hom_{\mathcal{C}_{n}(A,B)}(\add(A),\Si^{n-1}X)\simeq\Hom_{\mathcal{C}_{n}(\overline{A})}(\add(\overline{A}),\Si^{n-1}p^{*}X)=0 .$$ 
	Since $ \ch $ is a Frobenius extriangulated category, there is an $ \mathfrak{s} $-triangle
	$$ Y\xrightarrow{f} P\rightarrow X\stackrel{}{-\rightarrow} \quad (P\in\mathcal{P}=\mathrm{add}(eA)) $$ with $ Y $ in $ \ch $ and $ P $ in $ \cp=\add(eA) $. Then it is enough to show that $ \Hom_{\cc_{n}(A,B)}(\Si\add(A),X)=0 $. 
	
	By Proposition~\ref{Extension group in RFD}, we have 
	\begin{equation*}
		\begin{split}
			\Hom_{\cc_{n}(A,B)}(\Si\add(A),X)\simeq&\Hom_{\cc_{n}(\overline{A})}(\Si\add(\overline{A}),p^{*}X)\\
			\simeq&\Hom_{\cc_{n}(\overline{A})}(\Si^{n}\add(\overline{A}),\Si^{n-1}p^{*}X)\\
			\simeq&\Hom_{\cc_{n}(\overline{A})}(\add(\overline{A}),\Si^{n-1}p^{*}X)\\
			=&0.
		\end{split}
	\end{equation*}
	Thus, the space $ \Hom_{\cc_{n}(A,B)}(\Si\add(A),X) $ vanishes.

\end{proof}

\begin{Prop}\label{Condition 2 for Frobenius n-exangulated}
	By Theorem~\ref{Higgs is Frobenius extrianglated category}, $ (\mathcal{H},\mathbb{E},\mathfrak{s}) $ is an extriangulated category. Let $ f\in \mathcal{H}(X,Y) $, $ g\in \mathcal{H}(Y,Z) $ be any pair of morphisms. We have
	\begin{itemize}
		\item[(1)] If $ g\circ f $ is an $ \mathfrak{s} $-inflation (\cite[Definition 2.23]{HLN2021}), then so is $ f $.
		\item[(2)] If $ g\circ f $ is an $ \mathfrak{s} $-deflation (\cite[Definition 2.23]{HLN2021}), then so is $ g $.
	\end{itemize}
\end{Prop}
\begin{proof}
	We only show the first statement since the second statement can be shown dually. Suppose that $ g\circ f $ is an $ \mathfrak{s} $-inflation, i.e.\ there is a triangle in $ \mathcal{C}_{n}(A,B) $
	$$ X\xrightarrow{g\circ f}Z\rightarrow W\rightarrow \Si X $$ such that $ W $ is also in $ \mathcal{H} $. By the octahedral axiom, we have the following commutative diagram in $ \mathcal{C}_{n}(A,B) $
	\begin{align*}
		\xymatrix{
			X\ar@{=}[d]\ar[r]^{f}&Y\ar[d]^{g}\ar[r]&U\ar[d]\ar[r]&\Si X\ar@{=}[d]\\
			X\ar[r]^{gf}&Z\ar[r]\ar[d]&W\ar[r]\ar[d]&\Si X\\
			&M\ar@{=}[r]\ar[d]&M\ar[d]\\
			&\Si Y\ar[r]&\Si U\\
		}
	\end{align*}
	and the upper middle commutative diagram is a homotopy bi-cartesian square. Thus, there is a triangle in $ \mathcal{C}_{n}(A,B) $
	$$ Y\rightarrow U\oplus Z\rightarrow W\rightarrow\Si Y .$$
	
	Since $ \mathcal{H} $ is an extension closed subcategory of $ \mathcal{C}_{n}(A,B) $, the sum $ U\oplus Z $ is in $ \mathcal{H} $. By Theorem~\ref{Higgs category is a Silting reduction}, $ \mathcal{H} $ is closed under taking direct summands. Thus $ U $ is in $ \mathcal{H} $. We conclude that $ f\colon X\rightarrow Y $ is an $ \mathfrak{s} $-inflation.
	
\end{proof}

\begin{Rem}
	The Proposition above shows that $ \ch $ satisfies Nakaoka--Palu’s WIC condition (see~\cite[Condition 5.8]{NP2019}), which is equivalent to $ \ch $ being weakly idempotent complete in the usual sense (see~\cite[Proposition 2.7]{Klapproth2023}).
\end{Rem}

\begin{Thm}\label{addA is Frobenius n-exangulated category}
	Suppose that the $ n $-cluster tilting category $ \add\overline{A} $ of $ \cc_{n}(\overline{A}) $ satisfies $$ \Si^{n}\add\overline{A}=\add\overline{A} .$$ Then the $ n $-cluster-tilting subcategory $ \add\overline{A} $ of $ \cc_{n}(\overline{A}) $ carries a canonical $ (n+2) $-angulated structure. Moreover, the $ n $-cluster-tilting subcategory $ \mathrm{add}A $ of $ \ch $ carries a canonical structure of Frobenius $ n $-exangulated category with projective-injective objects $ \mathcal{P}=\mathrm{add}(eA) $. The quotient functor $ p^{*}\colon\mathcal{C}_{n}(A,B)\rightarrow\mathcal{C}_{n}(\overline{A}) $ induces an equivalence of $ (n+2) $-angulated categories
	$$ \mathrm{add}A/[\mathcal{P}]\iso\mathrm{add}(\overline{A}). $$
\end{Thm}
\begin{proof}

	\emph{1. The canonical $ (n+2) $-angulated structure on $ \add\overline{A} $.}
	
	Since $ \add \overline{A} $ is closed under the $ n $-th power of the shift functor in $ \mathcal{C}_{n}(\overline{A}) $, by \cite[Theorem 1]{GKO2013}, the $ n $-cluster-tilting subcategory $ \add\overline{A} $ carries a canonical $ (n+2) $-angulated structure $ (\add\overline{A},\Si^{n},\pentagon) $, where $ \pentagon $ is the class of all $ (n + 2) $-sequences in $ \add\overline{A} $
	$$ M\xrightarrow{\alpha_{0}}X_{1}\xrightarrow{\alpha_{1}}X_{2}\xrightarrow{\alpha_{2}}\cdots\xrightarrow{\alpha_{n-1}}X_{n}\xrightarrow{\alpha_{n}}N\xrightarrow{\delta}\Si^{n}M $$ such that there exists a diagram
	\[
	\begin{tikzcd}
		&X_{1}\arrow[dr]\arrow[rr,"\alpha_{2}"]&&X_{2}\arrow[dr]&&\cdots&&X_{n}\arrow[dr,"\alpha_{n}"]\\
		M\arrow[ur,"\alpha_{0}"]&&X_{1.5}\arrow[ur]\arrow[ll,"\qquad\qquad\qquad |" marking]&&X_{2.5}\arrow[ll,"\qquad\qquad\qquad |" marking]&\cdots& X_{n-0.5}\arrow[ur]&&N\arrow[ll,"\qquad\qquad\qquad |" marking]
	\end{tikzcd}
	\] with $ X_{i}\in\cc_{n}(\overline{A}) $ for $ i\notin\mathbb{Z} $, such that all oriented triangles are triangles in $ \cc_{n}(\overline{A}) $, all non-oriented triangles commute, and $ \delta $ is the composition along the lower edge of the diagram.
	
	For any two objects $ M,N $ in $ \add A $, the category $ \mathbf{T}^{n+2}_{M,N} $ (see~\cite[Definition 2.17]{HLN2021}) is defined as follows$\colon$
	\begin{itemize}
		\item[(a)] An object in $ \mathbf{T}^{n+2}_{M,N} $ is a complex $ X^{\bullet}=({X^{i},d_{X}^{i}}) $ of the form
		$$ X^{0}\xrightarrow{d_{X}^{0}}X^{1}\xrightarrow{d_{X}^{1}}\cdots\xrightarrow{d_{X}^{n-1}}X^{n}\xrightarrow{d_{X}^{n}}X^{n+1} $$ with all $ X^{i} $ in $ \add A $ and $ X^{0}=M $, $ X^{n+1}=N $.
		\item[(b)] For any $ X^{\bullet}, Y^{\bullet}\in\mathbf{T}^{n+2}_{M,N} $, a morphism $ f $ between $ X^{\bullet} $ and $ Y^{\bullet} $ is a chain map $ f=(f^{0},\cdots,f^{n+1}) $ such that $ f^{0}=\boldmath{1}_{M} $ and $ f^{n+1}=\boldmath{1}_{N} $. Two morphisms $ f^{\bullet} $ and $ g^{\bullet}\in\mathbf{T}^{n+2}_{M,N}(X^{\bullet},Y^{\bullet}) $ are \emph{homotopic} if there is a sequence of morphisms $ h^{\bullet}=(h^{1},\cdots,h^{n}) $ satisfying
		\begin{equation*}
			\begin{split}
				0=&h^{1}\circ d_{X}^{0},\\
				g^{i}-f^{i}=&d_{Y}^{i-1}\circ h^{i}+h^{i+1}\circ d_{X}^{i}\quad(1\leqslant i\leqslant n),\\
				0=&d_{Y}^{n}\circ h^{n+1}.
			\end{split}
		\end{equation*}
		
	\end{itemize}

	\bigskip
	\emph{2. The canonical Frobenius $ n $-exangulated structure on $ \add A $.}
	
	By Propositions~\ref{Condition 1 for Frobenius n-exangulated} and \ref{Condition 2 for Frobenius n-exangulated}, the $ n $-cluster-tilting subcategory $ \add A\subseteq\ch $ satisfies the conditions in \cite[Theorem 3.41]{HLN2021part2}. Thus, it carries a canonical $ n $-exangulated structure $ (\add A,\mathbb{E}^{n},\mathfrak{s}^{n}) $ which is given by
	\begin{itemize}
		\item[(1)] For any two objects $ M,N $ in $ \add A $, the group $ \mathbb{E}^{n}(M,N) $ is the higher extension group defined in Definition~\ref{Def:Higher extension group}, i.e.\ $ \mathbb{E}^{n}(N,M)=\Hom_{\cc_{n}(A,B)}(N,\Si^{n}M)\simeq\Hom_{\cc_{n}(\overline{A})}(p^{*}(N),\Si^{n}p^{*}(M)) $;
		\item[(2)] For any $ M,N $ in $ \add A $ and any $ \delta\in\mathbb{E}^{n}(N,M) $, define $$ \mathfrak{s}^{n}(\delta)=[X^{\bullet}] $$ to be the homotopy equivalence class of $ X^{\bullet} $ in $ \mathbf{T}^{n+2}_{M,N} $, where $ X^{\bullet} $ is given by an $ (n+2) $-sequence in $ \add A $
		$$ M\xrightarrow{\alpha_{0}}X_{1}\xrightarrow{\alpha_{1}}X_{2}\xrightarrow{\alpha_{2}}\cdots\xrightarrow{\alpha_{n-1}}X_{n}\xrightarrow{\alpha_{n}}N\xrightarrow{\delta}\Si^{n}M $$ such that there exists a diagram
		\[
		\begin{tikzcd}
			&X_{1}\arrow[dr]\arrow[rr,"\alpha_{2}"]&&X_{2}\arrow[dr]&&\cdots&&X_{n}\arrow[dr,"\alpha_{n}"]\\
			M\arrow[ur,"\alpha_{0}"]&&X_{1.5}\arrow[ur]\arrow[ll,"\qquad\qquad\qquad |" marking]&&X_{2.5}\arrow[ll,"\qquad\qquad\qquad |" marking]&\cdots& X_{n-0.5}\arrow[ur]&&N\arrow[ll,"\qquad\qquad\qquad |" marking]
		\end{tikzcd}
		\] with $ X_{i}\in\ch $ for $ i\notin\mathbb{Z} $, such that all oriented triangles are triangles in $ \cc_{n}(A,B) $, all non-oriented triangles commute, and $ \delta $ is the composition along the lower edge of the diagram. 
	\end{itemize}
	
	Next, we will show that $ \add A $ carries a canonical structure of Frobenius $ n $-exangulated category with projective-injective objects $ \mathcal{P}=\mathrm{add}(eA) $.
	
	Firstly, we show that $ \cp=\add (eA) $ consists of projective-injective objects in $ \add A $. Let $ P $ be an object in $ \add(eA) $. We take a distinguished $ n $-exangle in $ \add A $
	$$ Y_{0}\xrightarrow{\alpha_{0}}Y_{1}\ra\cdots\ra Y_{n}\xrightarrow{\alpha_{n}}Y_{n+1}\xrightarrow{\delta}\Si^{n}Y_{0} .$$ Then we have a distinguished triangle in $ \cc_{n}(A,B) $
	$$ X\ra Y_{n}\xrightarrow{\alpha_{n}} Y_{n+1}\ra\Si X $$ such that $ X $ is in $ \ch $. Let $ c\colon P\ra Y_{n+1} $ be a morphism in $ \add A $. It induces the following long exact sequence
	$$ \cdots\ra\Hom_{\cc_{n}(A,B)}(P,Y_{n})\ra\Hom_{\cc_{n}(A,B)}(P,Y_{n+1})\ra\Hom_{\cc_{n}(A,B)}(P,\Si X)\ra\cdots .$$
	Since $ X $ is in $ \ch $ and $ \cp $ is projective in $ \ch $, the space $ \Hom_{\cc_{n}(A,B)}(P,\Si X) $ vanishes. Thus, there exists a morphism $ b\colon P\ra Y_{n} $ in $ \add A $ satisfying $ \alpha_{n}\circ b=c $. This shows that $ P $ is projective. Dually, we can show that $ P $ is injective.
	
	Let $ N $ be an object in $ \add A $. Since $ \add (eA) $ is functorially finite in $ \cc_{n}(A,B) $, there exists a distinguished triangle in $ \cc_{n}(A,B) $
	$$ Q_{n}\xrightarrow{a_{n}} P_{n}\xrightarrow{b_{n}} N\xrightarrow{c_{n}}\Si Q_{n} $$
	with $ P_{n} $ in $ \add(eA) $. We see that $ p^{*}(N)\simeq \Si p^{*}(Q_{n}) \in\add(\overline{A}) $ in $ \cc_{n}(\overline{A}) $.
	
	For the object $ Q_{n} $, we also have a distinguished triangle in $ \cc_{n}(A,B) $
	$$ Q_{n-1}\xrightarrow{a_{n-1}} P_{n-1}\xrightarrow{b_{n-1}} Q_{n}\xrightarrow{c_{n-1}} \Si Q_{n-1} $$
	with $ P_{n-1} $ in $ \add(eA) $. We see that $ p^{*}(N)\simeq \Si^{2} p^{*}(Q_{n-1}) \in\add(\overline{A}) $ in $ \cc_{n}(\overline{A}) $.
	
	Repeating the process, we get the following triangles in $ \cc_{n}(A,B) $
	$$ Q_{n}\xrightarrow{a_{n}} P_{n}\xrightarrow{b_{n}} N\xrightarrow{c_{n}} \Si Q_{n} ,$$
	$$ Q_{n-1}\xrightarrow{a_{n-1}} P_{n-1}\xrightarrow{b_{n-1}} Q_{n}\xrightarrow{c_{n-1}} \Si Q_{n-1} ,$$
	$$ \cdots $$
	$$ Q_{0}\xrightarrow{a_{0}} P_{0}\xrightarrow{b_{0}} Q_{1}\xrightarrow{c_{0}} \Si Q_{0} $$ such that all $ P_{i}\,,\,0\leqslant i\leqslant n $, are in $ \add(eA) $ and $ p^{*}(N)\simeq\Si^{n}p^{*}(Q_{0})\in\add(\overline{A}) $.
	
	By our assumption $ \Si^{n}\add\overline{A}=\add\overline{A} $, we see that $ p^{*}(Q_{0}) $ is in $ \add\overline{A} $. Thus, the object $ Q_{0} $ is in $ \add A $. Then we get a distinguished $ n $-exangle in $ \add A $
	$$ Q_{0}\xrightarrow{a_{0}} P_{0}\xrightarrow{a_{1}\circ b_{0}} P_{1}\ra\cdots\ra P_{n}\xrightarrow{b_{n}} N\xrightarrow{\delta}\Si^{n}Q_{0} ,$$ where $ \delta $ is the composition
	$$ N\xrightarrow{c_{n}}\Si Q_{n}\xrightarrow{\Si c_{n-1}}\Si^{2}Q_{n-1}\ra\cdots\ra Q_{1}\xrightarrow{\Si^{n-1}c_{0}}\Si^{n}Q_{0} .$$
	
	Thus, this shows that $ \add A $ has enough projectives. Dually, we can show that $ \add A $ has enough injectives. Moreover, projective-injective objects form exactly the subcategory $ \cp=\add(eA) $. Therefore, we have shown that $ \add A $ carries a canonical structure of Frobenius $ n $-exangulated category with projective-injective objects $ \mathcal{P}=\mathrm{add}(eA) $.
	
	\bigskip
	\emph{3. The canonical $ (n+2) $-angulated structure on $ \add A/[\cp] $.}
	
	The stable category $ \add A/[\cp] $ has the same objects as $ \add A $. For any two objects $ M $ and $ N $, the morphism space is given by the quotient group 
	$$ \Hom_{\add A}(M,N)/[\cp](M,N) ,$$ where $ [\cp](M,N) $ is the subgroup of $ \Hom_{\add A}(M,N) $ consisting of those morphisms which factor through an object in $ \cp=\add eA $.
	
	For any object $ M $ in $ \add A $, we have the following triangles in $ \cc_{n}(A,B) $
	$$ M\xrightarrow{a_{0}}I_{0}\xrightarrow{b_{0}}Q_{0}\xrightarrow{c_{0}}\Si M ,$$
	$$ Q_{0}\xrightarrow{a_{1}}I_{1}\xrightarrow{b_{1}}Q_{1}\xrightarrow{c_{1}}\Si Q_{0} ,$$
	$$ \cdots $$
	$$ Q_{n-1}\xrightarrow{a_{n}}I_{n}\xrightarrow{b_{n}}Q_{n}\xrightarrow{c_{n}}\Si Q_{n-1} $$ such that all $ I_{i} $, $ 0\leqslant i\leqslant n $, are in $ \add(eA) $ and $ Q_{n} $ is in $ \add A $. Those triangles induce a distinguished $ n $-exangle in $ \add A $
	$$ M\xrightarrow{a_{0}}I_{0}\xrightarrow{a_{1}\circ b_{0}}I_{1}\ra\cdots\ra I_{n}\xrightarrow{c_{n}}Q_{n}\xrightarrow{\delta}\Si^{n}M ,$$ where $ \delta $ is the composition
	$$ Q_{n}\xrightarrow{c_{n}}\Si Q_{n-1}\ra\cdots\ra \Si^{n-1}Q_{0}\xrightarrow{\Si^{n-1}c_{0}}\Si^{n}M .$$
	We define the functor $ S\colon \add A/[\cp]\ra\add A/[\cp] $ such that it takes $ M $ to $ Q_{n} $. By~\cite[Proposition 3.7]{LZ2020}, the $ S $ functor is well defined and it is an auto-equivalence. It is easy to see that $ S(M) $ is isomorphic to $ \Si^{n}p^{*}(M) $ in $ \cc_{n}(\overline{A}) $.
	
	Thus, by~\cite[Theorem 3.13]{LZ2020}, the stable category $ \add A/[\cp] $ carries a canonical $ (n+2) $-angulated structure $ (\add A/[\cp],S,\square_{S}) $ which is given by
	\begin{itemize}
		\item[(1)] The functor $ S $ defined as above.
		\item[(2)] For any two objects $ M,N $ in $ \add A $, there is a one-to-one correspondence between $ \mathbb{E}^{n}(N,M)=\Hom_{\cc_{n}(A,B)}(N,\Si^{n}M) $ and $ \Hom_{\add/[\cp]}(N,S(M))\simeq\mathbb{E}^{n}(N,M) $ (see~\cite[Lemma 3.12]{LZ2020}). Any distinguished $ n $-exangle
		$$ M\xrightarrow{\alpha_{0}}X_{1}\xrightarrow{\alpha_{1}}X_{2}\xrightarrow{\alpha_{2}}\cdots\xrightarrow{\alpha_{n-1}}X_{n}\xrightarrow{\alpha_{n}}N\xrightarrow{\delta}\Si^{n}M $$ in $ \add A $ induces an $ (n+2) $-sequence 
		$$ M\xrightarrow{\overline{\alpha}_{0}}X_{1}\xrightarrow{\overline{\alpha}_{1}}X_{2}\xrightarrow{\alpha_{2}}\cdots\xrightarrow{\overline{\alpha}_{n-1}}X_{n}\xrightarrow{\overline{\alpha}_{n}}N\xrightarrow{\overline{\delta}}S(M) $$ in $ \add A/[\cp] $. We call such sequence an \emph{$ (n + 2) $-$ S $-sequence}. We denote by $ \square_{S} $ the class of $ (n+2) $-$ S $-sequences.
	\end{itemize}
For any object $ M $ of $ \mathrm{add}A/[\mathcal{P}] $, we have $ S(M)\cong p^{*}(\Si^{n}M)\cong\Si^{n}(p^{*}(M)) $. Moreover, the $ k $-equivalence $ p^{*}\colon\mathrm{add}A/[\mathcal{P}]\iso\mathrm{add}(\overline{A}) $ maps $ \square_{S} $ to $ \pentagon $.
	Thus $ p^{*} $ induces an equivalence of $ (n+2) $-angulated categories
	$$ (\mathrm{add}A/[\mathcal{P}],S,\square_{S})\iso(\mathrm{add}(\overline{A}),\Si^{n},\pentagon). $$
\end{proof}

\begin{Rem}
	If $ \add A $ is stable under $ \Si^{n} $ in $ \cc_{n}(A,B) $, then the algebra $ B $ is zero. So the $ n $-cluster-tilting subcategory $ \add A\subseteq\cc_{n}(A,B) $ can only carry an $ n $-angulated structure with higher suspension $ \Si^{n} $ if $ B=0 $ (see \cite[Theorem 1]{GKO2013}).
\end{Rem}

\section{The case when $ A $ is concentrated in degree 0}\label{section A is stalk}
Let $f\colon B\to A $ be a morphism (not necessarily preserving the identity element) between two differential graded (=dg) $ k $-algebras. We assume that $ f $ satisfies the assumptions~\ref{Relative assumption} and moreover, $ A $ is concentrated in degree 0. In particular, $ f $ carries a relative $ (n+1) $-Calabi--Yau structure. 
\begin{Prop}\label{Smooth is finite gl.dim}
	Under the assumption above, the $ k $-algebra $ H^{0}(A) $ is a finite-dimensional algebra with $ \mathrm{gldim}\,H^{0}(A)\leqslant n+1 $.
\end{Prop}
\begin{proof}
	By assumptions~\ref{Relative assumption}, the algebra $ H^{0}(A) $ is finite-dimensional. Suppose that $ \boldmath{1}_{H^{0}(A)} $ has decomposition
	$$ \boldmath{1}_{H^{0}(A)}=e_{1}+e_{2}+\cdots+e_{n} $$ into primitive orthogonal idempotents such that
	$$ e=f(\boldmath{1}_{B})=e_{1}+\cdots+e_{k} $$ for an integer $ 0\leqslant k\leqslant n $. Here we regard $ e $ as an element of $ H^{0}(A) $. By Proposition~\ref{Relative Hom}, $ \mathrm{pvd}_{B}(A) $ is an $ (n+1) $-Calabi–Yau triangulated category. Thus for each simple module $ S_{i} $, $ k+1\leqslant i\leqslant n $, we have $ \mathrm{pdim}\,S_{i}\leqslant n+1 $.

	Let $ M $ be a finite-dimensional $ H^{0}(A) $-module. For each simple module $ S_{i} $, $ 1\leqslant i\leqslant k $, by Proposition~\ref{Relative Hom}, we have the following isomorphism of triangles
	\begin{align*}
		\xymatrix{
			\mathcal{C}(S_{i},\Si^{-1}M)\ar[r]\ar[d]^{\simeq}&\mathbf{R}\Hom_{A}(S_{i},M)\ar[r]\ar[d]^{\simeq}&\mathbf{R}\Hom_{B}(S_{i}|_{B},M|_{B})\ar[r]\ar[d]^{\simeq}&\\
			D\mathbf{R}\Hom_{A}(M,\Si^{n+1}S_{i})\ar[r]&D\mathcal{C}(M,\Si^{n}S_{i})\ar[r]&D\mathbf{R}\Hom_{B}(M|_{B},\Si^{n}S_{i}|_{B})\ar[r]&.
		}
	\end{align*}
Recall that $ \mathcal{C}(S_{i},M) $ is defined as $ \cone(\mathbf{R}\Hom_{A}(S_{i},M)\ra\mathbf{R}\Hom_{B}(S_{i}|_{B},M|_{B})) $. Hence we have $ \mathcal{C}(S_{i},\Si^{k}M)=\Si^{k}\mathcal{C}(S_{i},M)$ for any $ i\in\mathbb{Z} $.

For each integer $ p\geqslant n+2 $, we have 
	$$ \Hom_{\mathcal{D}(B)}(S_{i}|_{B},\Si^{p}(M|_{B}))=0 $$ because $ B $ is $ n $-Calabi--Yau.
	
	Thus, we have 
	\begin{align*}
		\xymatrix{
			H^{0}(\mathcal{C}(S_{i},\Si^{p-1}M))\ar[r]\ar[d]^{\simeq}&\Ext_{H^{0}(A)}^{p}(S_{i},M)\ar[r]\ar[d]^{\simeq}&0\\
			0=D\Ext_{H^{0}(A)}^{n+1-p}(M,S_{i})\ar[r]&DH^{0}(\mathcal{C}(M,\Si^{n-1-p}S_{i}))\ar[r]&0.
		}
	\end{align*}
	We see that the space $ \Ext_{H^{0}(A)}^{p}(S_{i},M) $ vanishes for each $ p\geqslant n+2 $.
	Then $ \mathrm{pdim}\,S_{i}\leqslant n+1 $ for each $ 1\leqslant i\leqslant k $. Therefore, we have $ \mathrm{gldim}\,H^{0}(A)\leqslant n+1 $.
\end{proof}

\bigskip

Let $ \ce $ be a Frobenius category and $ \cm $ a full subcategory of $ \ce $ which contains the full subcategory $ \cp $ of $ \ce $ formed by the projective-injective objects. We denote by $ \ck^{b}(\ce)  $ and $ \cd^{b}(\ce) $ respectively the bounded homotopy category and the bounded derived category of $ \ce $. 

We say that a complex $ X\colon\cdots\ra X^{i-1}\ra X^{i}\ra X^{i+1}\ra\cdots $ in $ \ck^{b}(\ce) $ is \emph{$ \ce $-acyclic} if there are conflations $
\begin{tikzcd}
Z^{i}\arrow[r,tail,"l^{i}"]&X^{i}\arrow[r,two heads,"\pi^{i}"]& Z^{i+1}	
\end{tikzcd} $ such that $ d_{X}^{i}=l^{i+1}\circ\pi^{i} $ for each $ i\in\mathbb{Z} $.

We also denote by $ \ck_{\ce-ac}^{b}(\ce) $, $ \ck^{b}(\cp) $, $ \ck^{b}(\cm) $ and $ \ck^{b}_{\ce-ac}(\cm) $ the full subcategory of $ \ck^{b}(\ce)  $ whose objects are the $ \ce $-acyclic complexes, the complexes of projective objects in $ \ce $, the complexes of objects of $ \cm $ and the $ \ce $-acyclic complexes of objects of $ \cm $ respectively.

\begin{Thm}\label{Higgs of stalk}
Let $ f\colon B\ra A $ be a dg algebra morphism which satisfies the assumptions~\ref{Relative assumption}. Let $ e=f(\boldmath{1}_{B}) $. Moreover, we assume that $ A $ is concentrated in degree $ 0 $. Then we have
	\begin{itemize}
		\item[a)] The algebra $ B'=eH^{0}(A)e $ is Iwanaga-Gorenstein of injective dimension at most $ n+1 $ as a $ B' $-module.
		\item[b)] Under the equivalence $ \mathcal{D}^{b}(\mathrm{mod}H^{0}A)\simeq\per A $, the subcategory $ \mathcal{F}^{rel} $ corresponds to the subcategory $ \mathrm{mod}_{n-1}(H^{0}A) $ of $ H^{0}A $-modules of projective dimension at most $ n-1 $.
		\item[c)] Via the equivalence $ \mathrm{res}\colon\mathcal{D}^{b}(\mathrm{mod}H^{0}A)\overset{\sim}{\longrightarrow} \per A $, the localization $ \pi^{rel}\colon\per A\rightarrow\mathcal{C}_{n}(A,B) $ identifies with the restriction functor $ \mathcal{D}^{b}(\mathrm{mod}H^{0}A)\rightarrow\mathcal{D}^{b}(\mathrm{mod}B') $, i.e.\ we have a commutative square
		\[
		\begin{tikzcd}
			\mathcal{D}^{b}(\mathrm{mod}H^{0}A) \arrow{r}\isoarrow{d} &\mathcal{D}^{b}(\mathrm{mod}B')\isoarrow{d}\\
			\per A\arrow{r} &\mathcal{C}_{n}(A,B). 
		\end{tikzcd}
		\]
		\item[d)] Under the equivalence $ \mathcal{D}^{b}(\mathrm{mod}B')\overset{\sim}{\longrightarrow}\mathcal{C}_{n}(A,B) $, the Higgs category $ \mathcal{H}\subseteq\mathcal{C}_{n}(A,B) $ corresponds to the subcategory $ \mathrm{gpr}B' $ of Gorenstein projective modules over $ B'=eH^{0}(A)e $. In particular, when $ B' $ is self injective, we have $ \mathcal{H}\cong\mathrm{mod}B'$.
		\item[e)] Let $ \cm=\add A\subset\ch $. Then the exact sequence of triangulated categories 
		$$ 0\ra\pvd_{B}(A)\ra\per A\ra\cc_{n}(A,B)\ra0 $$
		is equivalent to
		$$ 0\ra\ck^{b}_{\ch-ac}(\cm)\ra\ck^{b}(\cm)\ra\cd^{b}(\ch)\ra0. $$ In particular, the relative cluster category $ \cc_{n}(A,B) $ is equivalent to the bounded derived category $ \cd^{b}(\ch) $ of $ \ch $.
	\end{itemize}
\end{Thm}

\begin{proof}
	Since $ H^{0}(A) $ is of finite global dimension, the restriction along the quasi isomorphism $$ A\ra H^{0}(A) $$ induces a triangle equivalence 
	$$ \mathcal{D}^{b}(\mathrm{mod}H^{0}(A))\iso\mathrm{per}A .$$
	Under this equivalence, $ \mathrm{pvd}_{B}(A) $ identifies with
	\begin{equation*}
		\begin{split}
			\mathcal{D}_{\mathcal{N}}^{b}(\mathrm{mod}H^{0}(A))=&\{X\in\mathcal{D}^{b}(\mathrm{mod}H^{0}(A))|\ H^{l}(X)|_{B'}=0,\forall l\in\mathbb{Z} \}\\
			=&\{X\in\mathcal{D}^{b}(\mathrm{mod}H^{0}(A))|H^{l}(A)\in\mathcal{N},\forall l\in\mathbb{Z}\},	
		\end{split}
	\end{equation*}
	where $ \mathcal{N}=\{M\in \mathrm{mod}H^{0}(A)\  |\ M|_{B'}=0\} $. Clearly, the category $ \mathcal{N} $ is a Serre subcategory of $ \mathrm{mod}H^{0}(A) $ and the restriction $ \mathrm{mod}H^{0}(A)\rightarrow \mathrm{mod}B' $ induces an exact sequence of abelian categories
	$$ 0\rightarrow\mathcal{N}\rightarrow \mathrm{mod}H^{0}(A)\rightarrow \mathrm{mod}B'\rightarrow0. $$ This exact sequence induces an exact sequence of triangulated categories
	$$ 0\rightarrow \mathcal{D}_{\mathcal{N}}^{b}(\mathrm{mod}H^{0}A)\rightarrow\mathcal{D}^{b}(\mathrm{mod}H^{0}(A))\rightarrow\mathcal{D}^{b}(\mathrm{mod}B')\rightarrow0. $$
	Thus, the restriction $ \mathrm{per}A\cong\mathcal{D}^{b}(\mathrm{mod}H^{0}(A))\rightarrow\mathcal{D}^{b}(\mathrm{mod}B') $ induces an equivalence 
	$$ \mathcal{C}_{n}(A,B)\iso\mathcal{D}^{b}(\mathrm{mod}B') .$$
	By inspecting the definition of $ \mathcal{F}^{rel} $, it is equivalent to the following subcategory
	$$\mathrm{mod}_{n-1}(H^{0}A)=\{M\in \mathrm{mod}H^{0}(A)\ |\ \mathrm{pdim}M\leqslant n-1\}.$$
	
	The Higgs category $ \ch $ is contained in $ \mathrm{mod}B' $ and stable under extensions in $ \cc_{n}(A,B)\simeq\cd(B') $. Thus, it is a fully exact subcategory of $ \mathrm{mod}B' $ with the induced exact structure. Moreover, $ \mathcal{H} $ is a Frobenius exact category with projective-injective objects $ \mathcal{P}=\mathrm{proj}(B') $ and $ \mathcal{H} $ contains an $ n $-cluster-tilting object $ T $, namely the image of $ A $, such that $ \mathrm{End}_{\mathcal{H}}(T)\cong H^{0}(A) $ with $ \mathrm{gldim}\,H^{0}(A)\leqslant n+1 $. 
	
	By Theorem~\ref{Higgs category is a Silting reduction}, the Higgs category is idempotent complete. Thus, we can apply Iyama--Kalck--Wemyss--Yang's structure theorem for Frobenius exact categories with an $ n $-cluster-tilting object (see \cite[Theorem 2.7]{KIWY15}), to conclude that 	$ B' $ is Iwanaga-Gorenstein of injective dimension at most $ n+1 $ as a $ B' $-module and that restriction to $ B' $ is an equivalence from $ \ch\subseteq\mathrm{mod}H^{0}(A) $ to the category $ \mathrm{gpr}(B') $ of Gorenstein projective $ B' $-modules, i.e.\ we have an equivalence
	$$ \mathcal{H}\iso \mathrm{gpr}(B')=\{M\in \mathrm{mod}B'\ |\ \Ext_{B'}^{i}(M,B')=0,\ \forall i>0\} .$$
	This shows $ a)$, $ b) $, $ c) $ and $ d) $. It remains to show $ e) $. 
	
	Let $ \per\cm $ be the full subcategory of the derived category of modules over $ \cm $ generated by all representable functors and let $ \per_{\underline{\cm}} $ be its full subcategory consisting of complexes whose cohomologies are in $ \mathrm{mod} \underline{\cm}\cong\mathrm{mod}\overline{A} $. By \cite[Lemma 7]{Palu2009}, the Yoneda equivalence of triangulated categories $ \ck^{b}(\cm)\ra\per\cm\cong\per A $ induces a triangle equivalence $$ \ck^{b}_{\ch-ac}(\cm)\ra\per_{\underline{\cm}}\cm\cong\pvd_{B}(A).$$ Thus, we finish the proof.
\end{proof}

\subsection{Relation with Pressland's works}

\begin{Def}\rm\cite{MP2017}\label{Def:internally cy}
	Let $ A $ be a $ k $-algebra, $ e $ an idempotent of $ A $, and $ d $ a non-negative integer. We say that $ A $ is \emph{internally $ d $-Calabi–Yau} with respect to $ e $ if
	\begin{itemize}
		\item[(1)] $ \mathrm{gldim}A \leqslant d $, and
		\item[(2)] for each $ i\in \mathbb{Z} $, there is a functorial duality
		\begin{align*}
			\xymatrix{
				D\Ext^{i}_{A}(M,N)\cong \Ext^{d-i}_{A}(N,M)
			}
		\end{align*}
		where $ M $ and $ N $ are quasi-isomorphic to a bounded complex of finitely generated projective $ A $-modules such that $ M $ is also a finite-dimensional
		$ A/AeA $-module.
	\end{itemize}
\end{Def}

Let $ A $ be an algebra and $ e $ an idempotent of $ A $. We denote the corresponding quotient algebra by $ \overline{A}\coloneqq A/\langle e\rangle $. Let $ \mathcal{D}(A) $ be the unbounded derived category of $ A $, $ \mathcal{D}_{e}(A) $ the full subcategory of $ \mathcal{D}(A) $
consisting of complexes with homology groups in $ \mathrm{Mod}(\overline{A}) $, and  $ \mathrm{pvd}_{e}(A) $ the full subcategory of $ \mathcal{D}_{e}(A) $ consisting of objects with finite dimensional total cohomology. 

Recall that the inverse dualizing bimodule of $ A $ is defined as $ A^{\vee}=\RHom_{A^{e}}(A,A^{e}) $.

\begin{Def}\rm\cite{MP2017}\label{Def:bimodule internall cy}
	An algebra $ A $ is \emph{bimodule internally $ n $-Calabi–Yau} with respect to an idempotent $ e \in A $ if
	\begin{itemize}
		\item $ \mathrm{pdim}_{A^{e}}A\leqslant n $,
		\item $ A\in \mathrm{per}A^{e} $, and 
		\item there exists a triangle
		\begin{align*}
			\xymatrix{
				A\ar[r]&\Si^{n}A^{\vee}\ar[r]&C\ar[r]&\Si A
			}
		\end{align*}
		in $ \mathcal{D}(A^{e}) $, such that $ \mathbf{R}\Hom_{A}(C,M)=0=\mathbf{R}\Hom_{A^{op}}(C,N) $ for any $ M\in \mathrm{pvd}_{e}(A) $ and $ N\in \mathrm{pvd}_{e}(A^{op}) $.
	\end{itemize}
\end{Def}

\begin{Prop}\cite[Corollary 5.12]{MP_thesis}
	If $ A $ is internally bimodule $ n $-Calabi–Yau with respect to an idempotent $ e $ of $ A $, then it is internally $ n $-Calabi--Yau with respect to $ e $.
\end{Prop}

\begin{Prop}\label{Relative CY to internally}
	Let $ f\colon B\to A $ be a morphism between dg $ k $-algebras. Suppose that $ f $ satisfies the assumptions~\ref{Relative assumption} and moreover, $ A $ is concentrated in degree $ 0 $. Then $ A $ and $ A^{op} $ are internally bimodule $ (n+1) $-Calabi–Yau with respect to $ e=f(\boldmath{1}_{B}) $. Hence, the algebras $ A $ and $ A^{op} $ are internally $ (n+1) $-Calabi--Yau with respect to $ e=f(\boldmath{1}_{B}) $.
\end{Prop}
\begin{proof}
	By the definition of a relative $ (n+1) $-Calabi--Yau structure, we have the following triangle in $ \mathcal{D}(A^{e}) $
	\begin{align*}
		\xymatrix{
			A\ar[r]&\Si^{n+1}A^{\vee}\ar[r]&\Si^{n+1}\mathbi{L}f^{*}(B^{\vee})\ar[r]&\Si A,
		}
	\end{align*}
	where $ A^{\vee}=\mathbf{R}\Hom_{A^{e}}(A,A^{e}) $, $ B^{\vee}=\mathbf{R}\Hom_{B^{e}}(B,B^{e}) $ and $ \mathbf{L}f^{*}(B^{\vee})\cong A\lten_{B}B^{\vee}\lten_{B}A $.
	
	Let $ M $ be an object in $ \mathrm{pvd}_{e}(A) $. We have
	\begin{equation*}
		\begin{split}
			\mathbf{R}\Hom_{A}(\mathbf{L}f^{*}(B^{\vee}),M)&=\mathbf{R}\Hom_{A}(A\lten_{B}B^{\vee}\lten_{B}A,M)\\
			&\simeq\mathbf{R}\Hom_{A}(A\lten_{B}B^{\vee},\RHom_{B}(A,M|_{B}))\\
			&=0
		\end{split}
	\end{equation*}
	Similarly, we have $ \mathbf{R}\Hom_{A^{op}}(\mathbf{L}f^{*}(B^{\vee}),N)=0 $ for any $ N\in \mathrm{pvd}_{e}(A^{op}) $. Thus, the algebra $ A $ is bimodule internally $ (n+1) $-Calabi–Yau with respect to the idempotent $ e=f(\boldmath{1}_{B}) $. In the same way, we can show that $ A^{op} $ is bimodule internally $ (n+1) $-Calabi–Yau with respect to the idempotent $ e $.
	
\end{proof}

\section{Relative cluster categories for Jacobi-finite ice quivers with potential}\label{section FQP}

\subsection{Ice quivers with potential}\label{FQP}

\begin{Def}\rm
	A \emph{quiver} is a tuple $ Q=(Q_{0},Q_{1},s,t) $, where $ Q_{0} $ and $ Q_{1} $ are sets, and $s,t\colon Q_{1}\to Q_{0} $ are functions. Each $ \alpha\in Q_{1} $ is realised as an arrow $ \alpha\colon s(\alpha)\to t(\alpha) $. We call $ Q $ finite if $ Q_{0} $ and $ Q_{1} $ are finite sets.
\end{Def}

\begin{Def}\rm
	Let $ Q $ be a quiver. A quiver $ F=(F_{0},F_{1},s',t') $ is called a \emph{subquiver} of $ Q $ if $ F_{0}\subseteq Q_{0} $, $ F_{1}\subseteq Q_{1} $ and $ s',t' $ are the restrictions of $ s,t $ to $ F_{1} $. We call $ F $ a \emph{full subquiver} of $ Q $ if $ F $ is a subquiver and $ F_{1}=\{\alpha\in Q_{1}\colon s(\alpha),t(\alpha)\in F_{0}\} $. 
\end{Def}

\begin{Def}\rm
	An \emph{ice quiver} is a pair $ (Q,F) $, where $ Q $ is a quiver, and $ F $ is a subquiver of $ Q $.
\end{Def}

Let $ Q $ be a finite quiver. For each arrow $ a $ of $ Q $, we define the cyclic derivative with respect to $ a $ as the unique linear map
\begin{align*}
	\xymatrix{
		\partial_{a}\colon kQ/[kQ,kQ]\ar[r]&kQ
	}
\end{align*}
which takes the class of a path $ p $ to the sum $ \sum_{p=uav}vu $ taken over all decompositions of the path $ p $.

\begin{Def}\rm
	An element of $ kQ/[kQ,kQ] $ is called a \emph{potential} on $ Q $. It is given by a linear combination of cycles in $ Q $. An ice quiver with potential is a tuple $ (Q,F,W) $ in which $ (Q,F) $ is a finite ice quiver without loops, and $ W $ is a potential on $ Q $. If $ F $ is the empty quiver $ \emptyset $, then $ (Q,\emptyset,W)\coloneqq(Q,W) $ is simply called a quiver with potential.
\end{Def}

\subsection{Relative Ginzburg algebras and relative Jacobian algebras}	
\begin{Def}\rm\label{RGD}
	Let $ (Q,F,W) $ be a finite ice quiver with potential. Let $  \widetilde{Q} $ be the graded quiver with the same vertices as $ Q $ and whose arrows are
	\begin{itemize}
		\item the arrows of $ Q $,
		\item an arrow $ a^{*}\colon j\to i $ of degree $ -1 $ for each arrow $ a $ of $ Q $ not belonging to $ F $,
		\item a loop $ t_{i}\colon i\to i $ of degree $ -2 $ for each vertex $ i $ of $ Q $ not belonging to $ F $.
	\end{itemize}
	The \emph{relative Ginzburg dg algebra} $ \bm{\Gamma}_{rel}(Q,F,W) $ is the dg algebra whose underlying graded vector space is the graded path algebra $ k\widetilde{Q} $. Its differential is the unique linear endomorphism of degree 1 which satisfies the Leibniz rule
	\begin{align*}
		\xymatrix{
			d(u\circ v)=d(u)\circ v+(-1)^{p}u\circ d(v)
		}
	\end{align*}
	for all homogeneous $ u $ of degree $ p $ and all $ v $, and takes the following values on the arrows of $ \widetilde{Q} $:
	\begin{itemize}
		\item $ d(a)=0 $ for each arrow $ a $ of $ Q $,
		\item $d(a^{*})=\partial_{a}W$ for each arrow $ a $ of $ Q $ not belonging to $ F $,
		\item $ d(t_{i})=e_{i}(\sum_{a\in Q_{1}}[a,a^{*}])e_{i} $ for each vertex $ i $ of $ Q $ not belonging to $ F $, where $ e_{i} $ is the lazy path corresponding to the vertex $ i $.
	\end{itemize}
	
\end{Def}

\begin{Def}\rm
	Let $ (Q,F,W) $ be a finite ice quiver with potential. The \emph{relative (or frozen) Jacobian algebra} $ J(Q,F,W) $ is the zeroth cohomology of the relative Ginzburg algebra $ \bm{\Gamma}_{rel}(Q,F,W) $. It is the quotient algebra
	\begin{align*}
		kQ/\langle \partial_{a}W,a\in Q_{1}\setminus F_{1}\rangle
	\end{align*}
	where $ \langle \partial_{a}W,a\in Q_{1}\setminus F_{1}\rangle $ is the two-sided ideal generated by $ \partial_{a}W $ with $ a\in Q_{1}\setminus F_{1} $.
\end{Def}

Let $ (Q,F,W) $ be a finite ice quiver with potential. Since $ W $ can be viewed as an element in $ HC_{0}(kQ) $, $ c=B(W) $ is an element in $ HH_{1}(kQ) $, where $ B $ is the Connes connecting map (see~\cite[Section 6.1]{BK2011})
$$ B\colon H\!C_{n}(kQ)\ra H\!H_{n+1}(kQ) .$$
Then $ \xi=(0,c) $ is an element of $ H\!H_{0}(G) $ which provides the deformation parameter for the relative 3-Calabi--Yau completion of $ G\colon kF\hookrightarrow kQ $, namely the functor $$ \bm{G}_{rel}\colon\bm{\Pi}_{2}(F)\ra\bm{\Pi}^{red}_{3}(Q,F,\xi) $$ defined in Proposition~\ref{reduced completion}. An easy check shows that the dg algebra $ \bm{\Pi}^{red}_{3}(kQ,kF,\xi) $ is isomorphic to $ \bm{\Gamma}_{rel}(Q,F,W) $ and that the dg functor $ \bm{G}_{rel} $ takes the following values as follows:
\begin{itemize}
	\item $ \bm G_{rel}(i)=i $ for each frozen vertex $ i\in F_{0} $,
	\item $ \bm G_{rel}(a)=a $ for each arrow $ a\in F_{1} $,
	\item $ \bm G_{rel}(\tilde{a})=-\partial_{a}W $ for each arrow $ a\in F_{1} $,
	\item $ \bm G_{rel}(r_{i})=e_{i}(\sum_{a\in Q_{1}\setminus F_{1}}[a,a^{*}])e_{i} $ for each frozen vertex $ i\in F_{0} $.
\end{itemize}		
We call $ \bm G_{rel} $ the \emph{Ginzburg functor} associated with $ G\colon kF\hookrightarrow kQ $ and $ W $.
\begin{Rem}
	If we apply $ H^{0} $ to $ \bm G_{rel} $, we recover Proposition 8.1 of \cite{MP2020}.
\end{Rem}

\begin{Prop}\label{Homtopy Cofiber sequence of (Q,F,W)}
	Let $ (Q,F,W) $ be a finite ice quiver with potential. Let $ \overline{Q} $ be the quiver obtained from $ Q $ by deleting all vertices in $ F $ and all arrows incident with vertices in $ F $. Let $ \overline{W} $ be the potential on $ \overline{Q} $ obtaining by deleting all cycles passing through vertices of $ F $ in $ W $. Then $$ \bm{\Pi}_{2}(F)\xrightarrow{\bm G_{rel}} \bm{\Gamma}_{rel}(Q,F,W)\to \bm{\Gamma}(\overline{Q},\overline{W}) $$ is a
	homotopy cofiber sequence of dg categories, where $ \bm{\Gamma}(\overline{Q},\overline{W}) $ is the Ginzburg algebra (see \cite{BK2011}) associated with quiver with potential $ (\overline{Q},\overline{W}) $.
	
\end{Prop}
\begin{proof}
	By Proposition~\ref{reduced completion}, the homotopy cofiber of $ \bm{G}_{rel} $ is isomorphic to that of $ \tilde{G} $. Since $ \tilde{G} $ is a cofibration, the dg quotient identifies with the quotient of $ \bm{\Pi}_{3}(kQ,kF,\xi) $ by the 2-sided ideal generated by the image of $ \tilde{G} $. This quotient is isomorphic to $ \bm{\Gamma}(\overline{Q},\overline{W}) $ as a dg category.
\end{proof}

\subsection{Jacobi-finite ice quivers with potential}

An ice quiver with potential $ (Q,F,W) $ is called \emph{Jacobi-finite} if the relative Jacobian algebra $ J(Q,F,W) $ is finite-dimensional.

\begin{Def}\rm
	Let $ (Q,F,W) $ be a Jacobi-finite ice quiver with potential. Denote by $ \bm\Gamma_{rel} $ the relative Ginzburg dg algebra $  \bm\Gamma_{rel}(Q,F,W) $. Let $ e=\sum_{i\in F}e_{i} $ be the idempotent associated with all frozen vertices. Let $ \mathrm{pvd}_{e}(\bm\Gamma_{rel}) $ be the full subcategory of $ \mathrm{pvd}(\bm\Gamma_{rel}) $ consisting of the dg $ \bm\Gamma_{rel} $-modules whose restriction to frozen vertices is acyclic. 
	
	Then the \emph{relative cluster category} $ \mathcal{C}(Q,F,W) $ associated to $ (Q,F,W) $ is defined as the Verdier quotient of triangulated categories $$ \mathrm{per}(\bm\Gamma_{rel})/\mathrm{pvd}_{e}(\bm\Gamma_{rel}) .$$ 
	The relative fundamental domain $ \mathcal{F}^{rel} $ associated to $ (Q,F,W) $ is defined as the following subcategory of $ \per\,\bm{\Gamma}_{rel} $
	$$ \mathcal{F}^{rel}\coloneqq\{\cone(X_{1}\xrightarrow{f} X_{0})\ |\ X_{i}\in \add(\bm\Gamma_{rel})\,\,\text{and}\,\,\Hom(f,I)\ \text{is surjective},\ \forall\ I\in\mathcal{P}=\mathrm{add}(e\bm\Gamma_{rel}) \}. $$
	
	We have a fully faithful embedding $ \pi^{rel}\colon\mathcal{F}^{rel}\subseteq\mathrm{per}\,\bm{\Gamma}_{rel}\rightarrow\mathcal{C}(Q,F,W) $. Then the Higgs category $ \mathcal{H} $ associated to $ (Q,F,W) $ is the image of $ \mathcal{F}^{rel} $ in $ \mathcal{C}(Q,F,W) $ under the functor $ \pi^{rel} $.
\end{Def}

Combining Theorem~\ref{Higgs is Frobenius extrianglated category} and Proposition~\ref{addA is an n-cluster tiliting}, we get the result.

\begin{Thm}\label{Higgs for frozen quiver}
	Let $ (Q,F,W) $ be a Jacobi-finite ice quiver with potential. Then the relative cluster category $ \mathcal{C}(Q,F,W) $ is Hom-finite, and the Higgs category $ \mathcal{H} $ is a Frobenius 2-Calabi--Yau extriangulated category with projective-injective objects $ \mathcal{P}=\mathrm{add}(e\bm\Gamma_{rel}) $. The free module $ \bm\Gamma_{rel} $ in $ \mathcal{H} $ is a cluster-tilting object. Its endomorphism algebra is isomorphic to the relative Jacobian algebra $ J(Q,F,W) $. 
	
	Moreover, the stable category of $ \mathcal{H} $ is equivalent to the usual cluster category
	$$ \underline{\mathcal{H}}=\mathcal{H}/[\mathcal{P}]\overset{\sim}{\longrightarrow}\mathcal{C}(\overline{Q},\overline{W})=\per({\bm\Gamma}(\overline{Q},\overline{W}))/\pvd({\bm\Gamma}(\overline{Q},\overline{W})) $$
	
	and the following diagram commutes
	\begin{align*}
		\xymatrix{
			&&\mathrm{per}(e\bm\Gamma_{rel}e)\ar@{=}[r]\ar@{^{(}->}[d]&\mathrm{per}(e\bm\Gamma_{rel}e)\ar@{^{(}->}[d]\\
			\mathrm{pvd}_{e}(\bm\Gamma_{rel})\ar@{^{(}->}"2,3"\ar@{->}_{\cong}[d]&&\mathrm{per}(\bm\Gamma_{rel})\ar[r]\ar@{->>}[d]&\mathcal{C}(Q,F,W)\ar@{->>}[d]\\
			\mathrm{pvd}({\bm\Gamma}(\overline{Q},\overline{W}))\ar@{^{(}->}"3,3"&&\mathrm{per}({\bm\Gamma}(\overline{Q},\overline{W}))\ar[r]&\mathcal{C}(\overline{Q},\overline{W})
		}
	\end{align*}
	where the rows and columns are exact sequences of triangulated categories.
\end{Thm}

\bigskip

Let $ (Q,F,W) $ be an ice quiver with potential. For simplicity of notation, we write $ J_{rel} $ and $ \mathbf{\Gamma}_{rel} $ respectively instead of $ J(Q,F,W) $ and $ \mathbf{\Gamma}_{rel}(Q,F,W) $. Let $ Q_{0}^{m}=Q_{0}\setminus F_{0} $ and $ Q_{1}^{m}=Q_{1}\setminus F_{1} $. Let $ S $ be the semisimple $ k $-algebra $ \prod_{i\in Q_{0}}ke_{i} $. We denote  by $ S^{m} $, $ V $ and $ V^{m} $ the $ S $-bimodules generated by $ Q_{0}^{m} $, $ Q_{1} $ and $ Q_{1}^{m} $ respectively. Let $ V^{m*} $ be the dual bimodule $ \Hom_{S^{e}}(V^{m},S^{e}) $.

There is a canonical short exact sequence of $ \mathbf{\Gamma}_{rel} $-bimodules
\begin{align}\label{Cone resolution}
	\xymatrix{
		0\ra\ker(m)\xrightarrow{\rho}\mathbf{\Gamma}_{rel}\ten_{S}\mathbf{\Gamma}_{rel}\xrightarrow{m}\mathbf{\Gamma}_{rel}\ra0,
	}
\end{align}
where the map $ m $ is induced by the multiplication of $ \mathbf{\Gamma}_{rel} $. The mapping cone $ \cone(\rho) $ of $ \rho $ is a cofibrant resolution of $ \mathbf{\Gamma}_{rel} $ as a bimodule over itself.

Then $ P(J_{rel})=J_{rel}\otimes_{\bm\Gamma_{rel}}\cone(\rho)\otimes_{\bm\Gamma_{rel}}J_{rel} $ is the following complex
\begin{align*}
	\xymatrix{
		0\ar[r]&J_{rel}\otimes_{S}\otimes S^{m}\otimes_{S}J_{rel}\ar[r]^{m_{3}}&J_{rel}\otimes_{S}\otimes V^{m*}\otimes_{S}J_{rel}\ar[r]^{m_{2}}&J_{rel}\otimes_{S}\otimes V\otimes_{S}J_{rel}\ar[r]^-{m_{1}}&J_{rel}\otimes_{S}J_{rel}\ar[r]&0,
	}
\end{align*} 
where $ m_{3} $, $ m_{2} $ and $ m_{1} $ are as follows:
$$ m_{1}(x\otimes a\otimes y)=xa\otimes y-x\otimes ay $$ and  
$$ m_{3}(x\otimes t_{i}\otimes y)=\sum_{a,t(a)=t_{i}}xa\otimes a^{*}\otimes y-\sum_{b,s(b)=t_{i}}x\otimes b^{*}\otimes by. $$

For any path $ p=a_{m}\cdots a_{1} $ of $ kQ $, we define
\begin{align*}
	\triangle_{a}(p)=\sum_{a_{i}=a}a_{m\cdots}a_{i+1}\otimes a_{i}\otimes a_{i-1}\cdots a_{1},
\end{align*}
and extend by linearity to obtain a map $ \triangle_{a}\colon kQ\to J_{rel}\otimes_{S} kQ_{1}\otimes_{S} J_{rel} $. Then $ m_{2} $ is given by:
$$ m_{2}(x\otimes a^{*}\otimes y)=\sum_{b\in Q_{1}}x\triangle_{b}(\partial_{a}W)y.$$

There is a canonical morphism $ P(J_{rel}) \to J_{rel} $, which is induced by the multiplication map $ m $ in $ J_{rel} $
\begin{align}\label{extended cotangent complex}
	\xymatrix{
		0\ar[r]&J_{rel}\otimes_{S}\otimes S^{m}\otimes_{S}J_{rel}\ar[r]^{m_{3}}\ar[d]&J_{rel}\otimes_{S}\otimes V^{m*}\otimes_{S}J_{rel}\ar[r]^{m_{2}}\ar[d]&J_{rel}\otimes_{S}\otimes V\otimes_{S}J_{rel}\ar[r]^-{m_{1}}\ar[d]^{\alpha}&J_{rel}\otimes_{S}J_{rel}\ar[r]\ar[d]^{m}&0
		\\
		&0\ar[r]&0\ar[r]&0\ar[r]&J_{rel}.
	}
\end{align} 

\begin{Rem}\label{exact at 2 places}
	When $ F=\emptyset $, the complex~(\ref{extended cotangent complex}) defined above is the complex associated to $ (Q, W) $ defined by Ginzburg in~\cite[Section 5]{VG}. In general, it is exactly the complex $ P(J_{rel}) $ defined by Pressland in~\cite[Deﬁnition 5.21]{MP_thesis}. Moreover, it has already appeared in the work of Amiot–Reiten–Todorov (see~\cite[Propostion 2.2]{Am2011}). By a result of Butler–King \cite[1.2]{BuKing1999}, the vertical maps $ \alpha $ and $ m $ induce isomorphisms on cohomology. Thus, the canonical morphism $ P(J_{rel})\ra J_{rel} $ is a quasi-isomorphism if and only if the cohomology of $ P(J_{rel}) $ vanishes at $ J_{rel}\ten_{S}V^{m*}\ten_{S}J_{rel} $ and $ J_{rel}\ten_{S}R^{m}\ten_{S}J_{rel} $.
\end{Rem}

Applying the functor $ S\ten_{J_{rel}}? $ to the complex~(\ref{extended cotangent complex}), we get another complex
$$ 0\ra S\ten_{J_{rel}}P(J_{rel})\ra S\ra0 .$$

The complex above decomposes along with $ S $, so its exactness is equivalent to the exactness of the following complex
\begin{align}\label{Complex at simples}
	\xymatrix{
		0\ra S_{v}\ten_{J_{rel}}P(J_{rel})\ra S_{v}\ra 0
	}
\end{align} for each $ v\in Q_{0} $, where $ S_{v} $ is the simple $ J_{rel} $-module at the vertex $ v $. If $ v $ is an unfrozen vertex, the complex~(\ref{Complex at simples}) is the following complex
\begin{align}\label{Complex at unfrozen simples}
	\xymatrix{
		0\ar[r]&e_{v}J_{rel}\ar[r]^-{(b)}&\displaystyle\bigoplus_{b\in Q_{1},s(b)=v}e_{t(b)}J_{rel}\ar[rr]^-{(a^{-1}(\partial_{b}W))}&&\displaystyle\bigoplus_{a\in Q_{1},t(a)=v}e_{s(a)}J_{rel}\ar[r]^-{(a)}&e_{v}J_{rel}\ar[r]&S_{v}\ar[r]&0.
}
\end{align}
If $ v $ is a frozen vertex, the complex~(\ref{Complex at simples}) is the following complex
\begin{align}\label{Complex at frozen simpls}
	\xymatrix{
		0\ar[r]&\displaystyle\bigoplus_{b\in Q_{1}\setminus F_{1},s(b)=v}e_{t(b)}J_{rel}\ar[rr]^-{(a^{-1}(\partial_{b}W))}&&\displaystyle\bigoplus_{a\in Q_{1},t(a)=v}e_{s(a)}J_{rel}\ar[r]^-{(a)}&e_{v}J_{rel}\ar[r]&S_{v}\ar[r]&0.
	}
\end{align}
When $ c $ is a path in $ Q $, we write $ a^{-1}c=b $ if $ b = ac $ in $ kQ $ and $ a^{-1}c=0 $ otherwise.

\begin{Thm}\rm
	Let $ (Q,F,W) $ be an ice quiver with potential. Suppose that $ F $ is a full subquiver of $ Q $. Let $ \bm\Gamma_{rel}(Q,F,W) $ be the complete relative Ginzburg algebra (see~\cite[Definition 4.20]{YW2021}). The following statements are equivalent.
	\begin{itemize}
		\item[(i)] $ \bm\Gamma_{rel} $ is a stalk algebra, i.e.\ its homology is concentrated in degree zero.
		\item[(ii)] The canonical morphism (\ref{extended cotangent complex}) is a quasi-isomorphism.
		\item[(iii)] For each vertex $ v\in Q_{0} $, the corresponding complex~(\ref{Complex at unfrozen simples}) or (\ref{Complex at frozen simpls}) is exact.
		\item[(iv)] The canonical map $ g_{rel}\colon \bm\Pi_{2}(F)\xrightarrow{\bm{G}_{rel}}\bm\Gamma_{rel}(Q,F,W)\ra J_{rel} $ has a left 3-Calabi--Yau structure which is induced by the canonical left 3-Calabi--Yau structure on $ \bm{G}_{rel} $.
	\end{itemize}
Moreover, if $ (Q,F,W) $ is Jacobi-finite, then the above statements are also equivalent to
\begin{itemize}
	\item[(v)] The extriangulated structure on $ \ch $ is exact (see Theorem~\ref{Higgs for frozen quiver}) and for every frozen vertex $ v\in F_{0} $, the module $ \mathrm{rad}(e_{v}J_{rel}) $ is of the form $ \Hom_{\cc(Q,F,W)}(\bm\Gamma_{rel},T_{v}) $ for some $ T_{v} $ in $ \ch $.
\end{itemize} 
\end{Thm}
\begin{proof}
$ (i)\Rightarrow(ii)\Rightarrow(iii) $ and $ (i)\Rightarrow(iv) $ are clear. To prove that $ (iii)\Rightarrow(i) $, let $ \cd_{pc}(\bm\Gamma_{rel}) $ be the pseudocompact derived category of $ \bm\Gamma_{rel} $ (see \cite[Section 7.11]{BKDY2011}). By~\cite[Proposition 7.14]{BKDY2011}, $ (\cd_{pc}(\bm\Gamma_{rel}))^{op} $ is compactly generated by $ \{S_{v}\,|\,v\in Q_{0}\} $ and similarly for $ (\cd_{pc}(J_{rel}))^{op} $. The restriction functor $$ \theta\colon \cd_{pc}(J_{rel})\ra\cd_{pc}(\bm\Gamma_{rel}) $$ takes $ S_{v} $ to $ S_{v} $. We can conclude that $ \theta $ is an equivalence if it induces isomorphisms
$$ \Ext^{*}_{J_{rel}}(S_{i},S_{j})\iso\Ext^{*}_{\bm\Gamma_{rel}}(S_{i},S_{j})\,\forall i,j\in Q_{0}.$$ 

The complexes (\ref{Complex at unfrozen simples}), (\ref{Complex at frozen simpls}) are exact, i.e.\ $ S_{i}\ten_{J_{rel}}\cone(\rho)\ten_{J_{rel}}J_{rel} \ra S_{i} $ is exact for each vertex $ i $ (see~(\ref{Cone resolution})). So we can use it to compute $ \RHom $:
\begin{equation*}
	\begin{split}
		\RHom_{J_{rel}}(S_{i},S_{j})&=\Hom_{J_{rel}}(S_{i}\ten_{J_{rel}}\cone(\rho)\ten_{J_{rel}}J_{rel},S_{j})\\
		&=\Hom_{\bm\Gamma_{rel}}(S_{i}\ten_{J_{rel}}\cone(\rho),\Hom_{J_{rel}}(J_{rel},S_{j}))\\
		&=\Hom_{\bm\Gamma_{rel}}(S_{i}\ten_{\bm\Gamma_{rel}}\cone(\rho),S_{j})\\
		&=\RHom_{\bm\Gamma_{rel}}(S_{i},S_{j}).
	\end{split}
\end{equation*}
Thus, the restriction functor $ \theta\colon \cd_{pc}(J_{rel})\ra\cd_{pc}(\bm\Gamma_{rel}) $ is an equivalence. It follows that $ \bm\Gamma_{rel}\ra J_{rel} $ is a quasi-isomorphism. Thus, the complete relative Ginzburg algebra $ \bm\Gamma_{rel} $ is concentrated in degree 0.

To prove that $ (iv)\Rightarrow(iii) $, suppose that the canonical map $$ g_{rel}\colon \bm\Pi_{2}(F)\xrightarrow{\bm{G}_{rel}}\bm\Gamma_{rel}(Q,F,W)\ra J_{rel} $$ has a canonical left 3-Calabi--Yau structure. Let $ k $ be an unfrozen vertex. By the proof of Proposition~\ref{Smooth is finite gl.dim}, we see that the simple $ J_{rel} $-module $ S_{k} $ has projective dimension at most 3. Therefore, by Remark~\ref{exact at 2 places}, there is a projective resolution of length 3
\[\begin{tikzcd}
	\displaystyle\bigoplus_{j\in Q_{0}}(e_{j}J_{rel})^{m_{kj}}\arrow[r,hook]&\displaystyle\bigoplus_{b\in Q_{1},s(b)=k}e_{t(b)}J_{rel}\arrow[rr,"(a^{-1}(\partial_{b}(W)))"]&&\displaystyle\bigoplus_{a\in Q_{1},t(a)=k}e_{s(a)}J_{rel}\arrow[r,"(a)"]&e_{k}J_{rel}\arrow[r,twoheadrightarrow]&S_{k}.
\end{tikzcd}
\]
Using the relative Calabi--Yau property (see Corollary~\ref{Relative CY duality}) and comparing dimensions, we conclude that
$$ m_{kj}=\dim\Ext^{3}_{J_{rel}}(S_{k},S_{j})=\dim\Hom_{J_{rel}}(S_{j},S_{k})=\delta_{jk} .$$
Thus, the above resolution is exactly the complex~(\ref{Complex at unfrozen simples}).

Let $ v $ be a frozen vertex. Using the same proof as for Proposition~\ref{Smooth is finite gl.dim}, we show that the frozen simple $ J_{rel} $-module $ S_{v} $  has projective dimension at most 2. Therefore, there is a projective resolution of length 2
\[\begin{tikzcd}
		\displaystyle\bigoplus_{j\in Q_{0}}(e_{j}J_{rel})^{m_{vj}}\arrow[r,hook]&\displaystyle\bigoplus_{a,t(a)=v}e_{s(a)}J_{rel}\arrow[r,"(a)"]&e_{v}J_{rel}\arrow[r,twoheadrightarrow]&S_{v}.
\end{tikzcd}
\]
Similarly, using the relative Calabi--Yau property (see Corollary~\ref{Relative CY duality}) and comparing dimensions, we conclude that
\begin{itemize}
	\item For any unfrozen vertex $ k $, we have
	 \begin{equation*}
	 	\begin{split}
	 		\dim\Ext^{2}_{J_{rel}}(S_{v},S_{k})&=\dim\Ext^{1}_{J_{rel}}(S_{k},S_{v})\\
	 		&=\sharp\{a\in Q_{1}\colon v\to k\}\\
	 		&=\sharp\{a\in Q_{1}^{m}\colon v\to k\}.
	 	\end{split}
	 \end{equation*}
 \item For any frozen vertex $ w\neq v $, by Proposition~\ref{Relative Hom}, we have an exact sequence
 $$ D\Ext^{1}_{\bm\Pi_{2}(F)}(S_{w},S_{v})\ra D\Ext^{1}_{J_{rel}}(S_{w},S_{v})\ra\Ext^{2}_{J_{rel}}(S_{v},S_{w})\ra0 .$$
 Since $ F $ is a full subquiver of $ Q $, the space $ \Ext^{2}_{J_{rel}}(S_{v},S_{w}) $ vanishes. Thus, we can conclude that
 $$ \displaystyle\bigoplus_{j\in Q_{0}}(e_{j}J_{rel})^{m_{vj}}=\displaystyle\bigoplus_{b\in Q_{1}^{m},s(b)=v}e_{t(b)}J_{rel} .$$ Hence the complex~(\ref{Complex at simples}) is exact. This finishes the proof of $ (iv)\Rightarrow(iii) $.
	
\end{itemize}

Now, suppose that $ (Q,F,W) $ is Jacobi-finite. By Theorem~\ref{Higgs of stalk}, $ (i)\Rightarrow (v) $ is clear. To prove that $ (v)\Rightarrow(iii) $, suppose that the extriangulated structure on $ \ch $ is exact. All inflations in $ \ch $ are monomorphisms and all deflations in $ \ch $ are epimorphisms. Let $ j $ be an unfrozen vertex. We have exchange conflations in $ \ch $ (\cite[Proposition 1.1 and 1.2]{Am2011})
$$ e_{j}\bm\Gamma_{rel}\rightarrowtail\bigoplus_{j\to k\in Q_{1}}e_{k}\bm\Gamma_{rel}\twoheadrightarrow T_{j}^{*} \quad\text{and}\quad T_{j}^{*}\rightarrowtail\bigoplus_{i\to j\in Q_{1}}e_{i}\bm\Gamma_{rel}\twoheadrightarrow e_{j}\bm\Gamma_{rel} .$$ Thus, we get a 2-almost-split sequence
\begin{align*}
	\xymatrix{
		0\ra e_{j}\bm\Gamma_{rel}\ra\displaystyle\bigoplus_{j\to k\in Q_{1}}e_{k}\bm\Gamma_{rel}\ra\displaystyle\bigoplus_{i\to j\in Q_{1}}e_{i}\bm\Gamma_{rel}\ra e_{j}\bm\Gamma_{rel}\ra 0.
	}
\end{align*}
Applying the functor $ \Hom_{\ch}(\bm\Gamma_{rel},?) $ to the above 2-almost-split sequence, we get a projective resolution of $ S_{j} $ which is exactly the complex~\ref{Complex at unfrozen simples}. Similarly, applying the functor $ D\Hom_{\ch}(?,\bm\Gamma_{rel}) $ to the 2-almost split sequence, we get an injective resolution of $ S_{j} $
$$ S_{j}\hookrightarrow D(J_{rel}e_{j})\xrightarrow{(b)}\displaystyle\bigoplus_{b\in Q_{1},s(b)=j}D(J_{rel}e_{t(b)})\ra\displaystyle\bigoplus_{a\in Q_{1},t(a)=j}D(J_{rel}e_{s(a)})\twoheadrightarrow D(J_{rel}e_{j}) .$$

Suppose that for every frozen vertex $ v\in F_{0} $, the functor $$ \mathrm{rad}(?,e_{v}\bm\Gamma_{rel})|_{\ch}\colon\ch\ra\mathrm{mod}J_{rel} $$ is representable in $ \ch $.
Let $ v $ be a frozen vertex. By the assumption, there exists an object $ T_{1}\in\add\bm\Gamma_{rel} $ such that we have a conflation in $ \ch $
$$ T_{1}\rightarrowtail \displaystyle\bigoplus_{a,t(a)=v}e_{s(a)}\bm\Gamma_{rel} \twoheadrightarrow e_{v}\bm\Gamma_{rel} .$$ This shows that $ \mathrm{pdim}S_{v}=2 $. Let
\[\begin{tikzcd}
	\displaystyle\bigoplus_{j\in Q_{0}}(e_{j}J_{rel})^{m_{vj}}\arrow[r,hook]&\displaystyle\bigoplus_{a,t(a)=v}e_{s(a)}J_{rel}\arrow[r,"(a)"]&e_{v}J_{rel}\arrow[r,twoheadrightarrow]&S_{v}.
\end{tikzcd}
\] be a minimal projective resolution of $ S_{v} $. Since $ F $ is full, there are no direct summands $ \{e_{v}J_{rel}\,|\,v\in F_{0}\} $ in $ \displaystyle\bigoplus_{j\in Q_{0}}(e_{j}J_{rel})^{m_{kj}} $. For each unfrozen vertex $ j $, we have
\begin{equation*}
	\begin{split}
		\dim\Ext^{1}_{J_{rel}}(S_{v},S_{j})&=\dim\Ext^{1}_{\bm\Gamma_{rel}}(S_{v},S_{j})\\
		&=\sharp\{a\in Q_{1}\colon j\to v\}
	\end{split}
\end{equation*}
and
\begin{equation*}
	\begin{split}
		m_{vj}&=\dim\Ext^{2}_{J_{rel}}(S_{v},S_{j})\\
		&=\dim(\displaystyle\bigoplus_{v\to j}(S_{v},D(J_{rel}e_{j})))\\
		&=\sharp\{b\in Q_{1}^{m}\colon v\to j\}.
	\end{split}
\end{equation*}

Thus, we conclude that
$$ \displaystyle\bigoplus_{j\in Q_{0}}(e_{j}J_{rel})^{m_{vj}}=\displaystyle\bigoplus_{b\in Q_{1}^{m},s(b)=v}e_{t(b)}J_{rel} .$$ Hence we get the resolution complex~(\ref{Complex at frozen simpls}).
	
\end{proof}

\section{Relative Calabi--Yau structures in higher Auslander--Reiten theory}\label{section 8}
\subsection{For algebras of finite global dimension}
Let $ n $ be a non-negative integer. Let $ B_{0} $ be a finite dimensional algebra with global dimension at most $ n $. Let $ \mathbb{S}_{B_{0}}=?\lten_{B_{0}}DB_{0} $ be the Serre functor of $ \mathcal{D}^{b}(\mathrm{mod}B_{0}) $. The corresponding inverse Serre functor is given by $ \mathbb{S}^{-1}_{B_{0}}=?\lten_{B_{0}}\mathbf{R}\Hom_{B^{e}}(B_{0},B_{0}^{e})$. Moreover, the \emph{Nakayama functor} $ \nu_{B_{0}} $ for $ \mathrm{mod}B_{0} $ is given by $ \nu_{B_{0}}=D\Hom_{B_{0}}(?,B_{0}) $. 

\begin{Def}\rm\cite{OI2007}
	The \emph{higher inverse Auslander-Reiten translation} $ \tau^{-1}_{n} $ of $ \mathrm{mod}B_{0} $ is defined to be the following composition
	\begin{align*}
		\xymatrix{
			\tau^{-1}_{n}\colon \mathrm{mod}B_{0}\ar@{^{(}->}[r]&\mathcal{D}^{b}(B_{0})\ar[r]^{\Sigma^{n}\mathbb{S}^{-1}_{B_{0}}}&\mathcal{D}^{b}(B_{0})\ar[r]^{H^{0}}&\mathrm{mod}B_{0}.
		}
	\end{align*}
\end{Def}

\begin{Def}\rm
	Let $ f\colon\mathcal{B}\rightarrow\mathcal{A} $ be a dg functor. The \emph{relative inverse Serre functor} for $ \mathcal{D}(\mathcal{A}) $ is defined as 
	\begin{align*}
		\xymatrix{
			\mathbb{S}_{\mathcal{A},\mathcal{B}}^{-1}=?\lten_{\ca}\Theta_{f}\colon\mathcal{D}(\mathcal{A})\to\mathcal{D}(\mathcal{A})
		},
	\end{align*}
	where $ \Theta_{f}=\mathbf{R}\Hom_{\mathcal{A}^{e}}(\cone(\mathcal{A}\lten_{\cb}\mathcal{A}\to \mathcal{A}),\mathcal{A}^{e})\in\mathcal{D}(\mathcal{A}^{e}) $.
\end{Def}

\begin{Rem}
	It is clear that we have an isomorphism $ \bm{\Pi}_{n+2}(\mathcal{A},\mathcal{B})\simeq\bigoplus_{i\geqslant 0}(\Si^{n+1}\Theta_{f})^{\lten_{i}} $ in $ \mathcal{D}(\mathcal{A}) $.
\end{Rem}

\begin{Def}\rm\cite{OISO2011}
	Let $ B_{0} $ be an algebra of global dimension at most $ n $. Then the $ (n + 1) $-\emph{preprojective algebra} of $ B_{0} $ is defined as
	$$ \widetilde{B_{0}}=T_{B_{0}}(\Ext_{B_{0}}^{n}(DB_{0},B_{0})), $$ i.e, the tensor algebra of the $ B_{0} $-bimodule $ \Ext_{B_{0}}^{n}(DB_{0},B_{0}) $ over $ B_{0} $. Then
	$ \widetilde{B_{0}} $ is isomorphic to $\bigoplus_{i\geqslant 0}\tau_{n}^{-i}B_{0} $ as a $ B_{0} $-module.
\end{Def}

\begin{Rem}
	In~\cite[Section 4]{BK2011}, Keller introduced the notion of the derived $ (n + 1) $-preprojective algebra $ \bm{\Pi}_{n+1}(B_{0})$ (also called $ (n+1) $-Calabi--Yau completion of $ B_{0} $). The $ (n + 1) $-preprojective algebras are the $ 0 $-th homology of his derived $ (n + 1) $-preprojective algebras.
\end{Rem}

We denote by $ \mathcal{B}\coloneqq\mathrm{proj}B_{0}\subseteq \mathrm{mod}B_{0} $ the projective modules. Let $ \mathcal{A} $ be a subcategory of $ \mathrm{mod}B_{0} $ which contains $ \mathcal{B} $ as a full subcategory. Then there is a natural dg inclusion functor
$$ f_{0}\colon\mathcal{B}\hookrightarrow\mathcal{A} .$$

For any $ X\in\mathcal{A} $, we put $ X^{\wedge}\coloneqq\Hom_{B_{0}}(?,X)|_{\mathcal{A}}\in \mathrm{proj}\mathcal{A} $. For simplicity of notation, let $ \Theta_{\cb} $ (respectively $ \Theta_{\ca} $) stand for $ \cb^{\vee}=\mathbf{R}\Hom_{\mathcal{B}^{e}}(\mathcal{B},\mathcal{B}^{e}) $ (respectively $ \ca^{\vee} $) in the next Proposition.
\begin{Prop}\label{Relative Serre formular}
	Assume that $ \mathcal{A} $ is homologically smooth and is an $ n $-rigid subcategory of $ \mathrm{mod}B_{0} $, i.e.\ $ \Ext^{k}_{B_{0}}(\mathcal{A},\mathcal{A})=0 $ for $ 1\leqslant k\leqslant n-1 $. Then for $ X\in\mathcal{A} $, we have a functorial isomorphism $ X^{\wedge}\lten_{\ca}\Si^{n+1}\Theta_{f_{0}}\cong (\tau_{n}^{-1}X)^{\wedge} $.
\end{Prop}

\begin{proof}
	Let $ X $ be an object in $ \ca $. We will show that
	$$ X^{\wedge}\lten_{\ca}\Si^{n+1}\Theta_{f_{0}}\cong (\tau_{n}^{-1}X)^{\wedge} ,$$ where $ \Theta_{f_{0}}=\mathbf{R}\Hom_{\mathcal{A}^{e}}(\cone(\mathcal{A}\lten_{\cb}\mathcal{A}\to \mathcal{A}),\mathcal{A}^{e}) $.

	\emph{Step 1.
		We compute the image of $ X^{\wedge} $ under the functor $ ?\lten_{\ca}\RHom_{\ca^{e}}(\mathcal{A}\lten_{\cb}\mathcal{A},\mathcal{A}^{e})\colon\cd(\ca)\ra\cd(\ca) .$}

	Since $ \mathcal{B} $ is smooth, we have
	\begin{equation*}
		\begin{split}
			\mathbf{R}\Hom_{\mathcal{A}^{e}}(\mathcal{A}\lten_{\cb}\mathcal{A},\mathcal{A}^{e})\cong&\ \mathbf{R}\Hom_{\mathcal{A}^{e}}(\mathcal{A}\lten_{\cb}\mathcal{B}\lten_{\cb}\mathcal{A},\mathcal{A}^{e})\\
			\cong&\RHom_{\cb^{e}}(\cb,\ca^{e})\\
			\cong&\Theta_{\cb}\lten_{\cb^{e}}\ca^{e}\\
			\cong&\mathcal{A}\lten_{\cb}\Theta_{\cb}\lten_{\cb}\mathcal{A},
		\end{split}
	\end{equation*} 
where $ \Theta_{\cb}=\mathbf{R}\Hom_{\mathcal{B}^{e}}(\mathcal{B},\mathcal{B}^{e}) $. Here we use the smoothness of $ \cb $ in the third isomorphism.
	
	Then we have
	\begin{equation*}
		\begin{split}
			X^{\wedge}\lten_{\ca}\Si^{n+1}( \mathcal{A}\lten_{\cb}\Theta_{\cb}\lten_{\cb}\mathcal{A})&\cong (X^{\wedge}\lten_{\cb}\Theta_{\cb})\lten_{\cb}\Si^{n+1}\mathcal{A}\\
			&\cong \mathbb{S}_{B_{0}}^{-1}(\Si^{n+1}X)^{\wedge}\lten_{\cb}\mathcal{A}.
		\end{split}
	\end{equation*}
	
	Fix a minimal injective resolution of $ X $
	$$ 0\to X\to I^{0}\to I^{1}\cdots\to I^{n}\to 0 .$$ Then $ \mathbb{S}_{B_{0}}^{-1}(\Si^{n+1}X)=\nu_{B_{0}}^{-1}(I^{\bullet}) $ is the following complex
	$$ 0\to P_{0}\to P_{1}\cdots\to P_{n}\to 0 ,$$ where $ P_{i} $ is in degree $ i-n-1 $ and $ P_{i}=\nu_{B_{0}}^{-1}(I^{i}) \in \mathrm{proj}B_{0} $. After applying the functor $$ ?\lten_{\cb}\mathcal{A}\colon\mathcal{D}(\mathcal{B})\to \mathcal{D}(\mathcal{A}) ,$$ we get
	$$ 0\to P_{0}^{\wedge}\to P_{1}^{\wedge}\to\cdots\to P_{n}^{\wedge}\to 0, $$ where $ P_{i}^{\wedge}=\Hom_{\mathrm{mod}B_{0}}(?,P_{i})|_{\mathcal{A}}\in\mathrm{proj}(\mathcal{A}) $. 
	
	Thus the image of $ X^{\wedge} $ under the functor $ ?\lten_{\ca}\mathbf{R}\Hom_{\mathcal{A}^{e}}(\mathcal{A}\lten_{\cb}\mathcal{A},\mathcal{A}^{e})\colon\mathcal{D}(\mathcal{A})\rightarrow\mathcal{D}(\mathcal{A}) $ is 
	$$ 0\to P_{0}^{\wedge}\to P_{1}^{\wedge}\to\cdots\to P_{n}^{\wedge}\to 0, $$ where $ P_{i}^{\wedge}=\Hom_{\mathrm{mod}B_{0}}(?,P_{i})|_{\mathcal{A}}\in\mathrm{proj}(\mathcal{A}) $.
	
	\emph{Step 2. We compute the image of $ X^{\we} $ under the functor $ ?\lten_{\ca}\Si^{n+1}\Theta_{\ca}=\Si^{n+1}\mathbb{S}_{\mathcal{A}}^{-1}\colon\mathcal{D}(\mathcal{A})\to\mathcal{D}(\mathcal{A}) $.}

	We have the minimal injective resolution in Step 1
	$$ 0\to X\to I^{0}\to I^{1}\cdots\to I^{n}\to 0 .$$ Then $ \Si^{n+1}\mathbb{S}_{\mathcal{A}}^{-1}(X^{\wedge}) $ is the following complex
	$$ 0\to (?,X)\to (?,I^{0})\to (?,I^{1})\cdots\to (?,I^{n}) .$$
	For each $ 1\leqslant i\leqslant n-1 $, the cohomology at $ (?,I^{i}) $ is $ \Ext^{i}_{B_{0}}(?,X)=0 $ because $ \mathcal{A} $ is an $ n $-rigid subcategory of $ \mathrm{mod}B_{0} $. The cohomology at $ (?,I^{n}) $ is $ \Ext^{n}_{B_{0}}(?,X) $. For any object $ L $ in $ \mathrm{mod}B_{0} $, we have
	\begin{equation*}
		\begin{split}
			D\Ext^{n}_{B_{0}}(L,X) &\simeq D\Hom_{\mathcal{D}(B_{0})}(L,\Si^{n}X)\\
			&\simeq \Hom_{\mathcal{D}(B_{0})}(\Si^{n}X,\mathbb{S}_{B_{0}}(L))\\
			&\simeq \Hom_{\mathcal{D}(B_{0})}(\Si^{n}(\mathbb{S}_{B_{0}}^{-1}X),L)\\
			&\simeq \Hom_{\mathcal{D}(B_{0})}(H^{0}(\Si^{n}(\mathbb{S}_{B_{0}}^{-1}X)),L)\\
			&\simeq \Hom_{\mathcal{D}(B_{0})}(\tau_{n}^{-1}X,L)\\
			&\simeq \Hom_{B_{0}}(\tau_{n}^{-1}X,L),
		\end{split}
	\end{equation*}
	where the fourth equivalence follows from $ \mathbb{S}_{B_{0}}^{-1}=\mathbf{R}\Hom_{B_{0}}(DB_{0},?) $ and $ \mathrm{gldim}(B_{0})\leqslant n $.	Then the cohomology at $ (?,I^{n}) $ is isomorphic to $ D\Hom_{B_{0}}(\tau_{n}^{-1}X,?) $.
	
	Since we have isomorphisms $ (?,I^{i})\simeq D(P_{i},?) $ for all $ 1\leqslant i\leqslant n-1 $, we get the following injective resolution of $ (?,X) $
	$$ 0\to (?,X)\to D(P_{0},?)\to D(P_{1},?)\cdots\to D(P_{n},?)\to D(\tau_{n}^{-1}X,?)\to 0 .$$

	Applying the functor $ \Si^{n+1}\mathbb{S}_{\mathcal{A}}^{-1} \colon\cd(\ca)\ra\cd(\ca) $ to the above complex, we get a complex
	$$ 0\to(?,P_{0})\to(?,P_{1})\to\cdots\to(?,P_{n})\to (?,\tau_{n}^{-1}X)\to 0, $$ where $ (?,P_{i}) $ is in degree $ i-n-1 $ and $ (?,\tau_{n}^{-1}X) $ is in degree 0. This is because $ \mathbb{S}_{\mathcal{A}}^{-1}(D(M,?))=\nu_{\mathrm{mod}\mathcal{A}}^{-1}(D(M,?))=(?,M) $	
	for any object $ M $ in $ \mathcal{A} $.

	\emph{Step 3.} From the computations in step 1 and step 2, the object $ X^{\wedge}\lten_{\ca}\Si^{n+1}\Theta_{f_{0}} $ is equal to the homotopy fiber of the following morphism of complexes
	\begin{align*}
		\xymatrix{ 0\ar[r]&(?,P_{0})\ar[r]\ar@{=}[d]&(?,P_{1})\ar[r]\ar@{=}[d]&\cdots\ar[r]&(?,P_{n})\ar[r]\ar@{=}[d]&(?,\tau_{n}^{-1}X)\ar[d]\ar[r]&0\\
			0\ar[r]&(?,P_{0})\ar[r]&(?,P_{1})\ar[r]&\cdots\ar[r]&(?,P_{n})\ar[r]&0\ar[r]&0.
		}
	\end{align*}
	Thus, the object $ \Si^{n+1}(\mathbb{S}_{\mathcal{A},\mathcal{B}})(X^{\wedge}) $ is quasi-isomorphic to $ \tau_{n}^{-1}(X)^{\wedge} $.	
\end{proof}

\begin{Cor}\label{Relative CY completion is concentrated in degree 0}
	
	Let $ \mathcal{B}=\mathrm{proj}B_{0}\subseteq \mathrm{mod}B_{0} $ be the subcategory of projectives and let $ \mathcal{A} $ be a subcategory of $ \mathrm{mod}B_{0} $ which contains $ \mathcal{B} $ as a full subcategory. Suppose that $ \mathcal{A}  $ is homologically smooth and is $ n $-rigid in $ \mathrm{mod}B $. Then the relative $ (n+2) $-Calabi--Yau completion of $ f_{0}\colon\mathcal{B}\hookrightarrow\mathcal{A} $
	
	$$ f\colon\bm{\Pi}_{n+1}(\mathcal{B})\longrightarrow\bm{ \Pi}_{n+2}(\mathcal{A},\mathcal{B}) $$ can be described as follows:
	\begin{itemize}
		\item The objects in $ \bm{\Pi}_{n+2}(\mathcal{A},\mathcal{B}) $ are the same as those of $ \ca $;
		\item For any two objects $ L $, $ M $ in $ \mathcal{A} $, the space $ \Hom_{\bm{\Pi}_{n+2}(\mathcal{A},\mathcal{B})}(L,M) $ is given by
		\begin{equation*}
			\begin{split}
				\Hom_{\bm{\Pi}_{n+2}(\mathcal{A},\mathcal{B})}(L,M)&\cong \RHom_{\mathcal{A}}(L^{\wedge},\bigoplus_{i\geqslant0}M^{\wedge}\lten_{\ca}(\Si^{n+1}\Theta_{f_{0}})^{\lten_{i}})\\
				&\cong \RHom_{\mathcal{A}}(L^{\wedge},\bigoplus_{i\geqslant0}(\tau_{n}^{-i}M)^{\wedge})\\
				&\cong \Hom_{B_{0}}(L,\bigoplus_{i\geqslant0}(\tau_{n}^{-i}M)).
			\end{split}
		\end{equation*}
	\end{itemize}
	In particular, the dg category $ \bm{\Pi}_{n+2}(\mathcal{A},\mathcal{B}) $ is concentrated in degree 0 and we have a fully faithful functor
	$$ H^{0}(f)\colon H^{0}(\bm{\Pi}_{n+1}(\mathcal{B}))\cong\widetilde{B_{0}}\hookrightarrow\bm{\Pi}_{n+2}(\mathcal{A},\mathcal{B}). $$
\end{Cor}

\subsection{$ n $-representation-finite algebras}
Let $ n\geqslant0 $ be an integer. Let $ B_{0} $ be a finite dimensional algebra with global dimension at most $ n $.
\begin{Def}\rm\cite{OI2011}
	We say that $ B_{0} $ is \emph{$ \tau_{n} $-finite} if $ \tau_{n}^{-i}B_{0}=0 $ for sufficiently large $ i $. We say that $ B_{0} $ is \emph{$ n $-representation-finite} if $ \mathrm{mod}B_{0} $ contains an $ n $-cluster tilting object.
\end{Def}

\begin{Rem}
	If $ B_{0} $ is $ n $-representation-finite, then it is $ \tau_{n} $-finite.
\end{Rem}

\begin{Thm}\cite[Proposition 1.3]{OI2011}
	Suppose that $ B_{0} $ is an $ n $-representation-finite algebra. Then
	$\widetilde{B_{0}}\cong\bigoplus_{i\geqslant 0}\tau_{B_{0},n}^{-i}B_{0} $ is the unique basic $ n $-cluster tilting object in $ \mathrm{mod}B_{0} $.
	
\end{Thm}

\begin{Thm}\cite[Theorem 0.2]{OI20071}
	Let $ B_{0} $ be $ n $-representation-finite. Then
	$$ \mathrm{gldim}\, \mathrm{End}_{B_{0}}(\widetilde{B_{0}})\leqslant n+1. $$
\end{Thm}

Now we assume that $ B_{0} $ is $ n $-representation-finite. The corresponding \emph{$ n $-Auslander algebra} is given by $ \End_{B_{0}}(\bigoplus_{i\geqslant 0}\tau_{n}^{-i}B_{0})  $. We denote it by $ A_{0} $. Then there is a natural fully faithful morphism
\begin{align*}
	\xymatrix{
		f_{0}\colon B_{0}\ar@{^{(}->}[r]& A_{0}=\End_{B_{0}}(\bigoplus_{i\geqslant 0}\tau_{n}^{-i}B_{0}).
	}
\end{align*}

\begin{Lem}\label{Dg quorient of algebra}
	Let $ \mathcal{A} $ be a dg category. Let $ \mathcal{P}\subseteq\mathcal{A} $ and $ \mathcal{B}\subseteq\mathcal{A} $ be two full dg subcategories such that $ \obj(\mathcal{A})=\obj(\mathcal{P})\sqcup \obj(\mathcal{B}) $ and $ \Hom_{\mathcal{A}}(B,P) $ is acyclic for all $ B\in\mathcal{B} $, $ P\in\mathcal{P} $. Then the Drinfeld dg quotient $ \mathcal{A}/\mathcal{P} $ is Morita equivalent to $ \mathcal{B} $.
\end{Lem}
\begin{proof}
	The restriction functor $ f_{*}\colon \mathcal{D}(\mathcal{A})\rightarrow\mathcal{D}(\mathcal{B}) $ induced by $ f\colon \mathcal{B}\hookrightarrow\mathcal{A} $ is a localization functor. Moreover, its kernel is generated (as a localizing subcategory) by its intersection with $ \per\mathcal{A} $. Since the space $ \Hom_{\mathcal{A}}(B,P) $ is acyclic for all $ B\in\mathcal{B} $ and $ P\in\mathcal{P} $, the induction functor $$ \mathcal{D}(\mathcal{P})\hookrightarrow\mathcal{D}(\mathcal{A}) $$ induces an equivalence between $ \ker(f_{*}) $ and $ \mathcal{D}(\mathcal{P}) $. Thus, we have an exact sequence of triangulated categories 
	$$ 0\rightarrow\mathcal{D}(\mathcal{P})\rightarrow\mathcal{D}(\mathcal{A})\rightarrow\mathcal{D}(\mathcal{B})\rightarrow0 .$$
	It follows that the Drinfeld dg quotient $ \mathcal{A}/\mathcal{P} $ is Morita equivalent to $ \mathcal{B} $.

\end{proof}

\begin{Prop}\label{Homtopy cofiber of Auslander algebra}
	Let $ e=f_{0}(\boldmath{1}_{B_{0}}) $. The homotopy cofiber of $ f_{0}\colon B_{0}\rightarrow A_{0} $ is equal to the usual quotient $ A_{0}/A_{0}eA_{0} $, i.e.\ the \emph{stable Auslander algebra} of $ B_{0} $.
	
\end{Prop}

\begin{proof}
	Let $ \mathcal{A} $, $ \mathcal{P} $ and $ \mathcal{B} $ be the following full subcategories of $ \mathrm{mod}B_{0} $
	$$ \mathcal{A}=\mathrm{ind}(\mathrm{add}(\{\tau_{n}^{-i}(B_{0})\ |\ i\geqslant0\})),$$
	$$ \mathcal{P}=\mathrm{ind}(\mathrm{proj}B_{0}) ,$$
	$$ \mathcal{B}=\{M\in\mathcal{A}\ |\ M\notin\mathcal{P} \} .$$
	In particular, if $ n=0 $ then $ \ca=\cp $.

	For $ P\in\mathcal{P} $ and $ M\in\mathcal{B} $, we have
	\begin{equation*}
		\begin{split}
			D\Hom_{B_{0}}(M,P)\cong&\Hom_{\mathcal{D}(B_{0})}(P,\mathbb{S}_{B_{0}}(M))\\
			=&\Hom_{\mathcal{D}(B_{0})}(P,\Si^{n}(\tau_{n}(M))).
		\end{split}
	\end{equation*}
	The above space vanishes since $ P\in\mathcal{P} $ and $ \tau_{n}(M)\in \mathrm{mod}B_{0} $. Then, by Lemma~\ref{Dg quorient of algebra} above, the homotopy cofiber of $ f_{0}\colon B_{0}\rightarrow A_{0} $ is equal to the usual quotient $ A_{0}/A_{0}eA_{0} $.

\end{proof}

\begin{Prop}\label{n-reprsentation finite case}
	Via the relative $ (n+2) $-Calabi--Yau completion of $ f_{0}\colon B_{0}\hookrightarrow A_{0} $, we get the following dg functor which has a canonical left $ (n+2) $-Calabi--Yau structure
	$$ f\colon B=\bm{\Pi}_{n+1}(B_{0})\longrightarrow A=\bm{\Pi}_{n+2}(A_{0},B_{0}) .$$ Then 
	\begin{itemize}
		\item[1)] The dg algebra $ A $ is concentrated in degree 0.
		\item[2)] $ H^{0}(A) $ is a finite-dimensional algebra with finite global dimension at most $ n+2 $.
		\item[3)] The homotopy cofiber of $ f $ is equal to $ \bm{\Pi}_{n+2}(A_{0}/A_{0}eA_{0}) $ where $ e=f(\boldmath{1}_{B_{0}}) $.
		\item[4)] The functor $ H^{0}(f)\colon H^{0}(B)=\widetilde{B_{0}}\to H^{0}(A) $ is fully  faithful.
		\item[5)] $ A $ is internally bimodule $ (n+2) $-Calabi--Yau with respect to $ e=f(\boldmath{1}_{B_{0}}) $.
	\end{itemize}	
\end{Prop}
\begin{proof}
	The first and fourth statement follows from Corollary~\ref{Relative CY completion is concentrated in degree 0}. The third statement follows from Proposition~\ref{Relation with absolut CY completion} and the last statement follows from Proposition~\ref{Relative CY to internally}. It remains to show the second one. 
	
	By Corollary~\ref{Relative CY completion is concentrated in degree 0} and the fact that $ B_{0} $ is $ \tau_{n} $-finite, the algebra $ H^{0}(\bm{\Pi}_{n+2}(A_{0},B_{0})) $ is finite-dimensional. By Proposition~\ref{Smooth is finite gl.dim}, the algebra $ H^{0}(\bm{\Pi}_{n+2}(A_{0},B_{0})) $ has finite global dimension at most $ n+2 $.
\end{proof}

\bigskip

We give an example where the relative Ginzburg algebra is not concentrated in degree 0. 
\begin{Ex}
	Let $ (Q,F) $ be the following ice quiver
	
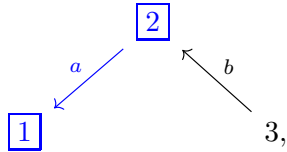
\[
	\begin{tikzcd}
		&\color{blue}\boxed{2}\arrow[dl,blue,"a",swap]&\\
		\color{blue}\boxed{1}&&3\arrow[ul,"b",swap],
	\end{tikzcd}
	\]
	where the ice part $ F $ is give by the blue vertices and blue arrows. The underlying graded quiver of the corresponding relative Ginzburg algebra $ \mathbf{\Gamma}_{rel}(Q,F) $ is given as follows
	
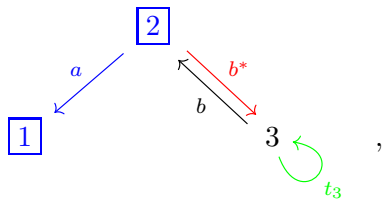
\[
	\begin{tikzcd}
		&\color{blue}\boxed{2}\arrow[dl,blue,"a",swap]\arrow[dr,shift left=0.5ex,red,"b^{*}"]&\\
		\color{blue}\boxed{1}&&3\arrow[ul,shift left=0.5ex,"b"]\arrow[out=290,in=350,loop,green,"t_{3}",swap]&,
	\end{tikzcd}
	\] 
	where $ |b^{*}|=-1 $ and $ |t_{3}|=-2 $. The differential $ d $ takes the following values
	$$ d(a)=0=d(b),\quad d(b^{*})=0,\quad d(t_{3})=-b^{*}b .$$ It is easy to see that $ \mathbf{\Gamma}_{rel}(Q,F) $ is not concentrated in degree 0. In fact, the ice quiver $ (Q,F) $ is the mutation (see~\cite{YW2021}) of the following ice quiver with potential $ (Q',F',W') $ at the frozen vertex $ 2 $
	
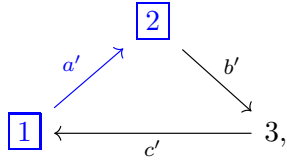
\[
	\begin{tikzcd}
		&\color{blue}\boxed{2}\arrow[dr,"b'"]&\\
		\color{blue}\boxed{1}\arrow[ur,"a'",blue]&&3\arrow[ll,"c'"],
	\end{tikzcd}
	\]
	where the ice part $ F' $ is given by the blue vertices and arrows and the potential $ W' $ is $ c'b'a' $. However, the relative Ginzburg algebra $ \mathbf{\Gamma}_{rel}(Q',F',W') $ associated with $ (Q',F',W') $ is concentrated in degree $ 0 $.
	
\end{Ex}

\bigskip

Suppose that $ \boldmath{1}_{B_{0}} $ has decomposition
$$ \boldmath{1}_{B_{0}}=e_{1}+e_{2}+\cdots+e_{n} $$into primitive orthogonal idempotents. We denote by $ P_{i}=e_{i}B_{0} $ the projective $ B_{0} $-module associated with the idempotent $ e_{i} $. Let $ \cu $ be the following full subcategory of $ \cd^{b}(\mathrm{mod}B_{0}) $ (see~\cite[Theorem 2.16]{OISO2011})
$$ \cu=\add\{\mathbb{S}_{n}^{i}B_{0}\,|\,i\in\mathbb{Z}\}\subseteq\cd^{b}(\mathrm{mod}B_{0}) ,$$ where $ \mathbb{S}_{n}=\Si^{-n}\mathbb{S}_{B_{0}} $ and $ \mathbb{S}_{B_{0}}=?\lten_{B_{0}}B_{0}^{\vee} $ is the Serre functor of $ \cd^{b}(\mathrm{mod}B_{0}) $. By~\cite[Theorem 2.16]{OISO2011}, $ \cu $ is an $ n $-cluster tilting subcategory of $ \cd^{b}(\mathrm{mod}B_{0}) $. 

Let $ \Si^{\mathbb{Z}}\cu=\add\{\Si^{p}\mathbb{S}_{n}^{k}B_{0}\,|\,p,k\in\mathbb{Z}\} $ be the $ \Si $ closure of $ \cu $ in $ \cd^{b}(\mathrm{mod}B_{0}) $. It is a bigraded category where the gradings are given by $ \mathbb{S}_{n} $ and $ \Si $.

The dg algebra $$ \mathbf{\Pi}_{n+1}(B_{0})=T_{B_{0}}(\Si^{n}B_{0}^{\vee}) $$ is Adams graded with $ |\Si^{n}B_{0}^{\vee}|_{a}=1 $. Then the homology algebra $ H^{*}(\mathbf{\Pi}_{n+1}(B_{0})) $ is naturally bigraded. We write $ H^{p}_{k}(\mathbf{\Pi}_{n+1}(B_{0})) $ for the homogeneous component in bidegree $ (k,p) $, where $ k $ is the Adams degree and $ p $ is the cohomological degree. We denote by $ P^{\mathbf{\Pi}}_{i}=e_{i}\mathbf{\Pi}_{n+1}(B_{0}) $ the cofibrant dg $ \mathbf{\Pi}_{n+1}(B_{0}) $-module associated with $ e_{i} $.

By the definition of $ \mathbf{\Pi}_{n+1}(B_{0}) $, for each $ P^{\mathbf{\Pi}}_{i} $, $ P^{\mathbf{\Pi}}_{j} $ and $ p\in\mathbb{Z} $, we have 
\begin{equation*}
	\begin{split}
		\Hom_{\cd(\mathbf{\Pi}_{n+1}(B_{0}))}(P^{\mathbf{\Pi}}_{i},\Si^{p}P^{\mathbf{\Pi}}_{j})&=\bigoplus_{k\geqslant0}\Hom_{\cd(B_{0})}(P_{i},\Si^{p}P_{j}\lten_{B_{0}}(\Si^{n}B_{0}^{\vee})^{\lten_{k}})\\
		&=\bigoplus_{k\geqslant0}\Hom_{\cd(B_{0})}(P_{i},\mathbb{S}_{n}^{-k}(\Si^{p}P_{j})).
	\end{split}
\end{equation*}
On the other hand, the space $ \Hom_{\cd(\mathbf{\Pi}_{n+1}(B_{0}))}(P^{\mathbf{\Pi}}_{i},\Si^{p}P^{\mathbf{\Pi}}_{j}) $ is computed by $ e_{j}H^{p}(\mathbf{\Pi}_{n+1}(B_{0}))e_{i} $.

For any integer $ k $, let $ P^{\mathbf{\Pi}}_{i}\langle k\rangle $ be the shift of $ P^{\mathbf{\Pi}}_{i} $ by degree $ k $ with respect to the Adams grading. We say a dg $ \mathbf{\Pi}_{n+1}(B_{0}) $-module $ M $ has \emph{bidegree $ (k,p) $} if it has Adams degree $ k $ and cohomological degree $ p $. Denote by $ \cc^{\mathbb{Z}}(\mathbf{\Pi}_{n+1}(B_{0})) $ the category of Adams graded dg $ \mathbf{\Pi}_{n+1}(B_{0}) $-modules with morphisms of bidegree $ (0,0) $. The corresponding derived category is denoted by $ \cd^{\mathbb{Z}}(\mathbf{\Pi}_{n+1}(B_{0})) $. We denote by $ T $ the action of $ \langle1\rangle $ on $ \cd^{\mathbb{Z}}(\mathbf{\Pi}_{n+1}(B_{0})) $.

For any two objects $ P^{\mathbf{\Pi}}_{i} $, $ P^{\mathbf{\Pi}}_{j} $ in $ \cd^{\mathbb{Z}}(\mathbf{\Pi}_{n+1}(B_{0})) $, we have
\begin{equation*}
	\begin{split}
		\Hom_{\cd^{\mathbb{Z}}(\mathbf{\Pi}_{n+1}(B_{0})}(P^{\mathbf{\Pi}}_{i},\Si^{p}P^{\mathbf{\Pi}}_{j}\langle k\rangle)&\cong\Hom_{\cd(\mathrm{mod}B_{0})}(P_{i},\mathbb{S}_{n}^{-k}\Si^{p}P_{j})\\
		&\cong e_{j}H^{p}_{k}(\mathbf{\Pi}_{n+1}(B_{0}))e_{i} .
	\end{split}
\end{equation*}

We have an equivalence of bigraded categories
$$ 
\add\{\Si^{p}P^{\mathbf{\Pi}}_{i}\langle k\rangle\,|\,p,k\in\mathbb{Z}\}\iso\Si^{\mathbb{Z}}\cu $$ which maps $ \Si^{p}P^{\mathbf{\Pi}}_{i}\langle k\rangle\ $ to $ \mathbb{S}_{n}^{-k}\Si^{p}P_{i} $.

Via taking the orbit categories with respect to $ T $ and $ \mathbb{S}_{n} $ respectively, we have the following equivalence of graded categories
$$ \add(\Si^{p}P^{\mathbf{\Pi}}_{i}\langle k\rangle\\,|\,i,k\in\mathbb{Z})/T\iso\Si^{\mathbb{Z}}\cu/\mathbb{S}_{n}.$$

We denote by $ \mathrm{proj}^{\mathbb{Z}}(H^{*}\mathbf{\Pi}_{n+1}(B_{0})) $ the category of (cohomological) graded projective modules over $ H^{*}(\mathbf{\Pi}_{n+1}(B_{0})) $. There is an equivalence of graded categories $$ \add(\Si^{p}P^{\mathbf{\Pi}}_{i}\langle k\rangle|\,i,k\in\mathbb{Z})/T\simeq\mathrm{proj}^{\mathbb{Z}}(H^{*}\mathbf{\Pi}_{n+1}(B_{0})) .$$
Thus we have an equivalence of graded categories
$$ \mathrm{proj}^{\mathbb{Z}}(H^{*}\mathbf{\Pi}_{n+1}(B_{0}))\simeq\Si^{\mathbb{Z}}\cu/\mathbb{S}_{n} .$$

Since $ B_{0} $ is $ n $-representation-finite, by \cite[Theorem 3.1 and Proposition 3.6]{OISO2011} we have $$ \Hom_{\cd(\mathrm{mod}B_{0})}(\Si^{i}\cu,\cu)=0 $$ for $ 1\leqslant i\leqslant n-1 $. By the above equivalence of graded categories, the $ i $-th homology $ H^{i}(\bm{\Pi}_{n+1}(B_{0})) $ vanishes for $ i=-1,\ldots,-n+1 $.

\begin{Lem}\label{Lemma:self-injective}
	The higher preprojective algebra $ \widetilde{B_{0}} $ is self-injective.
\end{Lem}
\begin{proof}
	It is enough to show that $ \widetilde{B_{0}}=H^{0}(B) $ is injective as a right $ \widetilde{B_{0}} $-module. The category $ \pvd(B) $ has a canonical $ t $-structure $$ (\pvd(B)_{\leqslant0},\pvd(B)_{\geqslant0}) ,$$ where $ \pvd(B)_{\leqslant0} $ is the full subcategory of $ \pvd(B) $ whose objects are the dg modules $ X $ such that $ H^{p}(X) $ vanishes for all $ p>0 $ and $ \pvd(B)_{\geqslant0} $ is the full subcategory of $ \pvd(B) $ whose objects are the dg modules $ X $ such that $ H^{p}(X) $ vanishes for all $ p<0 $. The corresponding heart is equivalent to $ \mathrm{mod}\widetilde{B_{0}} $. Moreover, by Section 3.1.7 of~\cite{BBD1982}, for all $ X $ and $ Y $ in $ \mathrm{mod}\widetilde{B_{0}} $, we have an isomorphism $$ \Ext^{1}_{\widetilde{B_{0}}}(X,Y)\simeq\Hom_{\cd(B)}(X,\Si Y) .$$
	
	Let $ M $ be an object in $ \mathrm{mod}\widetilde{B_{0}} $. By the $ (n+1) $-Calabi--Yau property of $ \pvd(B) $ and the above isomorphism, we have
	\begin{equation*}
		\begin{split}
			\Ext^{1}_{\widetilde{B_{0}}}(M,\widetilde{B_{0}})&\simeq\Hom_{\cd(B)}(M,\Si\widetilde{B_{0}})\\
			&\simeq D\Hom_{\cd(B)}(\widetilde{B_{0}},\Si^{n}M).
		\end{split}
	\end{equation*}
	
	If $ n=1 $, we have $ \Hom_{\cd(B)}(\widetilde{B_{0}},\Si M)\simeq\Ext^{1}_{\widetilde{B_{0}}}(\widetilde{B_{0}},\Si M)=0 $. Suppose that $ n>1 $. There exists a canonical triangle in $ \cd(B) $
	$$ \tau_{\leqslant-1}B\ra B\ra\widetilde{B_{0}}\ra\Si\tau_{\leqq-1}B .$$ Since the spaces $ \Hom_{\cd(B)}(B,\Si^{n}M) $ and $ \Hom_{\cd(B)}(\Si B,\Si^{n}M) $ vanish, it follows that
	$$ \Hom_{\cd(B)}(\widetilde{B_{0}},\Si^{n}M)\simeq\Hom_{\cd(B)}(\tau_{\leqslant-1}B,\Si^{n-1}M).$$
	
	By the previous discussion, the $ i $-th homology $ H^{i}(B) $ vanishes for $ i=-1,\ldots,-n+1 $. We see that $ \tau_{\leqslant-1}B $ lies $ \cd(B)_{\leqslant-n} $. Thus, the space $ \Hom_{\cd(B)}(\tau_{\leqslant-1}B,\Si^{n-1}M) $ vanishes and then $ \Ext_{B_{0}}^{1}(M,\widetilde{B_{0}}) $ vanishes. It follows that $ \widetilde{B_{0}} $ is injective.
	
\end{proof}

\bigskip
By Propositions~\ref{addA is an n-cluster tiliting}, \ref{Higgs of stalk}, \ref{n-reprsentation finite case} and the above Lemma which was first proved in~\cite[Corollary 3.4]{OISO2011}, we get the following Theorem.
\begin{Thm}\label{CTO in n-finite case}\cite[Theorem 1.1]{OISO2011}
	Consider the relative cluster category $ \mathcal{C}_{n+1}(A,B) $ associated with
	$$ f\colon B=\bm{\Pi}_{n+1}(B_{0})\longrightarrow A=\bm{\Pi}_{n+2}(A_{0},B_{0}) .$$
	\begin{itemize}
		\item[a)] The Higgs category $ \mathcal{H}\subseteq\mathcal{C}_{n+1}(A,B) $ is equivalent to $ \mathrm{mod}(\widetilde{B_{0}}) $ and the image of $ A $ in $\mathcal{H}$ is an $ (n+1) $-cluster-tilting object.
		\item[b)] We have a triangle equivalence $ \underline{\mathrm{mod}}(\widetilde{B_{0}})\cong\mathcal{C}_{n+1}(A_{0}/A_{0}eA_{0}) $, where $ e=f_{0}(\boldmath{1}_{B_{0}}) $.
	\end{itemize}
\end{Thm}

\begin{Rem}
	Above, we have used a different method to reprove Iyama--Oppermann's results in~\cite{OISO2011}. Notice that the algebra $ H^{0}(\mathbf{\Pi}_{n+2}(A_{0},B_{0})) $, which is quasi-isomorphic to $ \mathbf{\Pi}_{n+2}(A_{0},B_{0}) $, is isomorphic to the (non stable) endomorphism algebra of the $ (n+1) $-cluster-tilting object $ T $ given by the image of $ A $ in $ \ch $. This algebra does not appear explicitly in~\cite{OISO2011}.
\end{Rem}

\begin{Ex}\rm
	Let $ Q $ be a Dynkin quiver and let $ \Aus(kQ) $ be the Auslander algebra of the path algebra $ kQ $. We consider the following canonical dg inclusion
	$$ f_{0}\colon kQ\hookrightarrow \Aus(kQ) ,$$ which maps each vertex $ i $ to the corresponding projective module $ P_{i}=e_{i}kQ $.

	We know that $ \mathrm{gl.dim}(kQ)=1 $ and $ \mathrm{gl.dim}(\Aus(kQ))\leqslant2 $. Moreover, the homotopy cofiber of $ f $ is the stable Auslander algebra $ \underline{\Aus}(kQ)=\Aus(kQ)/\langle e\rangle $, where $ e=f(\boldmath{1}_{kQ}) $ (see Proposition~\ref{Homtopy cofiber of Auslander algebra}).
	
	Applying the relative 3-Calabi--Yau completion to the functor $ f_{0} $, we get the following dg functor $ f $ which has a canonical left 3-Calabi--Yau structure and $ \bm{\Pi}_{3}(\Aus(kQ),kQ) $ is concentrated in degree 0
	$$ f\colon\bm{\Pi}_{2}(kQ)\rightarrow\bm{\Pi}_{3}(\Aus(kQ),kQ) .$$
	
	On the level of $ H^{0} $, we get a fully faithful inclusion
	$$ H^{0}(f)\colon \widetilde{kQ}\hookrightarrow\bm{\Pi}_{3}(\Aus(kQ),kQ) ,$$
	where $ \widetilde{kQ} $ is the preprojective algebra of $ Q $ and hence is self-injective. So the Higgs category $ \mathcal{H}=\mathrm{gpr}(\widetilde{kQ}) $ is equivalent to $ \mathrm{mod}(\widetilde{kQ}) $.
	By Theorem~\ref{CTO in n-finite case}, we get a triangle equivalence $$ \underline{\mathrm{mod}}(\widetilde{kQ})\iso\mathcal{C}_{\underline{\Aus}(kQ)}=\mathrm{per}\bm{\Pi}_{3}(\underline{\Aus}(kQ))/\mathrm{pvd}(\bm{\Pi}_{3}(\underline{\Aus}(kQ))) .$$
	Thus, we have reproved that $ \mathrm{mod}(\widetilde{kQ}) $ contains a canonical cluster-tilting object (see \cite{GLS2006}) and that $ \underline{\mathrm{mod}}(\widetilde{kQ}) $ is triangle equivalent to $ \mathcal{C}_{\underline{\Aus}(kQ)} $ (see \cite{Am2011}).
	
\end{Ex}

\end{document}